\documentclass[a4paper,11pt,reqno]{amsart}
\usepackage{mathtools}
\usepackage{enumerate}
\usepackage{amsmath}
\usepackage{amssymb}
\usepackage{enumitem}
\usepackage{bbm}
\usepackage{tikz}
\usetikzlibrary{positioning}
\usepackage{amsthm}
\usepackage{booktabs}
\usepackage[textfont=normalsize, margin=25pt, skip=5pt]{caption}
\usepackage[english]{babel}
\usepackage[enableskew]{youngtab}
\usetikzlibrary{patterns}
\usepackage{tikzit}
\usepackage{diagbox}
\usepackage{nicematrix}

\theoremstyle{plain}
\newtheorem{theorem}{Theorem}[section]
\newtheorem{lemma}[theorem]{Lemma}
\newtheorem{proposition}[theorem]{Proposition}
\newtheorem{corollary}[theorem]{Corollary}

\theoremstyle{definition}
\newtheorem{definition}[theorem]{Definition}
\newtheorem{example}[theorem]{Example}

\theoremstyle{remark}
\newtheorem{remark}[theorem]{Remark}

\makeatletter
\@namedef{subjclassname@2020}{%
	\textup{2020} Mathematics Subject Classification}
\makeatother

\newcommand{\Ind}{\big\uparrow}
\newcommand{\Res}{\big\downarrow}
\newcommand{\ind}{\!\uparrow}
\newcommand{\res}{\!\downarrow}

\newcommand{\charwr}[3]{(#1)^{\widetilde{\times #3}} #2}
\newcommand{\charwrnb}[3]{#1^{\widetilde{\times #3}} #2}
\newcommand{\charwrdb}[3]{(#1)^{\widetilde{\times #3}} \left( #2\right) }

\DeclareMathOperator{\sgn}{sgn}
\DeclareMathOperator{\Gal}{Gal}
\DeclareMathOperator{\CF}{CF}

\DeclareMathOperator{\End}{End}

\newcommand{\List}[2]{#1_1,#1_2,\dots ,#1_{#2}}

\tikzstyle{Excluded}=[fill={rgb,255: red,191; green,191; blue,191}, draw=none, shape=rectangle, minimum width=0.875cm, minimum height=0.875cm]

\tikzstyle{Small excluded}=[fill={rgb,255: red,191; green,191; blue,191}, draw=none, shape=rectangle, minimum width=0.375cm, minimum height=0.375cm]

\tikzstyle{Big excluded}=[fill={rgb,255: red,191; green,191; blue,191}, draw=none, shape=rectangle, minimum width=1.75cm, minimum height=1.75cm]

\tikzstyle{Without excluded}=[fill={rgb,255: red,191; green,191; blue,191}, draw=none, shape=rectangle, minimum width=0.5cm, minimum height=0.5cm]

\tikzstyle{hdomino}=[fill=none, draw=black, shape=rectangle, minimum width=1cm, minimum height=0.5cm]

\tikzstyle{vdomino}=[fill=none, draw=black, shape=rectangle, minimum width=1cm, minimum height=0.5cm, rotate=90]

\tikzstyle{hnew}=[fill={rgb,255: red,197; green,197; blue,197}, draw=black, shape=rectangle, minimum width=1cm, minimum height=0.5cm]

\tikzstyle{vnew}=[fill={rgb,255: red,197; green,197; blue,197}, draw=black, shape=rectangle, minimum width=1cm, minimum height=0.5cm, rotate=90]

\tikzstyle{Rotated none}=[fill=none, draw=none, shape=circle, rotate=90]

\tikzstyle{h2domino} = [rectangle, draw=black, minimum height=0.5cm, minimum width=1cm]

\tikzstyle{v2domino} = [shape=rectangle, draw=black, minimum height=1cm, minimum width=0.5cm]

\tikzstyle{Border edge}=[-, thick, fill=none]

\tikzstyle{measuredots}=[<->]

\tikzstyle{Grey diagram}=[-, fill={rgb,255: red,122; green,122; blue,122}, thick]

\tikzstyle{Light grey column}=[-, fill={rgb,255: red,197; green,197; blue,197}]

\tikzstyle{Extra box}=[-, fill=none, dashed]

\tikzstyle{normal green}=[-, fill={rgb,255: red,132; green,255; blue,65}]

\tikzstyle{Move it}=[->]

\tikzstyle{special column}=[-, fill={rgb,255: red,8; green,243; blue,255}]

\tikzstyle{Temp Gray}=[-, fill={rgb,255: red,191; green,191; blue,191}]

\title{Multiplicity-free induced characters of symmetric groups}
\author{Pavel Turek}
\date{August 13, 2025}
\address{Department of Mathematics, Royal Holloway, University of London, Egham, Surrey TW20 0EX, UK}
\email{pkah149@live.rhul.ac.uk}
\begin{document}
	\maketitle	 
	\begin{abstract}
		Let $n$ be a non-negative integer. Combining algebraic and combinatorial techniques, we investigate for which pairs $(G,\rho)$ of a subgroup $G$ of the symmetric group $S_n$ and an irreducible character $\rho$ of $G$ the induced character $\rho\ind^{S_n}$ is multiplicity-free. As a result, for $n\geq 66$, we classify all subgroups $G\leq S_n$ which give rise to such a pair. Moreover, for the majority of these groups $G$ we identify all the possible choices of the irreducible character $\rho$, assuming $n\geq 73$.  
	\end{abstract} 
	
	\thispagestyle{empty}	
	\section{Introduction}
	
	Let $H$ be a finite group. We say that an ordinary character of $H$ is \textit{multiplicity-free} if each of its irreducible constituents appears with multiplicity one. Wildon \cite{WildonMultiplicity-free09} and independently Godsil and Meagher \cite{GodsilMeagherMultiplicity-free10} classified all multiplicity-free permutation characters of the symmetric group $S_n$, assuming $n\geq 66$ in the former. In this paper we solve, up to two exceptional families described below, the problem of classifying all multiplicity-free characters of symmetric groups induced from an arbitrary irreducible character of an arbitrary subgroup.
	
	We set the scene as follows. Fix a non-negative integer $n$. Throughout, given a subgroup $G$ of $S_n$, we say that a character $\rho$ of $G$ is \textit{induced-multiplicity-free} if $\rho\ind^{S_n}$ is multiplicity-free. Moreover, we say that $G$ is a \textit{multiplicity-free} subgroup of $S_n$ if there exists an induced-multiplicity-free character of $G$. Clearly, without altering the definition, we can require such a character to be irreducible.
	
	From the transitivity of induction, one concludes for any chain of subgroups $G\leq H\leq S_n$ that if $G$ is multiplicity-free, so is $H$. Another observation is that the induced character $\rho\ind^{S_n}$ does not change if we conjugate $G$ and $\rho$ by the same element of $S_n$. Therefore we work up to conjugation in $S_n$ throughout.
	
	Our first main result is the classification of all multiplicity-free subgroups of $S_n$ for $n\geq 66$. In the statement of the result and the rest of this paper we use the following notation for particular index two subgroups of groups $S_k\times S_m\wr S_h, S_m\wr S_h$ and $S_k\times S_2\wr S_2$.
	
	\begin{definition}\label{Defn important subgroups}
		Let $k,m$ and $h$ be positive integers and define $\psi$ to be a sign character $\sgn$ if $m$ is even and a trivial character $\mathbbm{1}$ otherwise. We define the subgroups $T_{k,m,h}\leq S_k\times S_m\wr S_h, T_{m,h}\leq S_m\wr S_h$ and $N_k\leq S_k\times S_2\wr S_2$ to be the kernels of $\sgn\boxtimes \left( \charwrnb{\sgn}{\psi}{h}\right), \charwrnb{\sgn}{\psi}{h}$ and $\sgn\boxtimes \left( \charwrnb{\mathbbm{1}}{\sgn}{2}\right)$, respectively.
	\end{definition}
	
	See \S\ref{Sec plethysms} for the notation used for characters of wreath products.
	\begin{theorem}\label{Theorem main}
		Let $n\geq 66$. A subgroup $G\leq S_n$ is multiplicity-free if and only if it belongs to the list (throughout $k,l\geq 1$, $m\geq 2$ and $h\geq 3$):
		\begin{enumerate}[label=\textnormal{(\roman*)}]
			\item $S_n$ and $A_n$,
			\item $S_k\times S_l$, $(S_k\times S_l)\cap A_{k+l}$ and for $k\neq 2$, respectively, $k,l\neq 2$ also $A_k\times S_l$, respectively, $A_k\times A_l$,
			\item $S_k\times S_m\wr S_2$ and for $k\notin\left\lbrace 2m-3,2m-2,2m-1,2m \right\rbrace $ also the groups $\left( S_k\times S_m\wr S_2\right) \cap A_{k+2m}$ and $T_{k,m,2}$,
			\item $S_k\times S_2\wr S_h$ and for $k\geq 2h+2$ also $\left( S_k\times S_2\wr S_h\right)\cap A_{k+2h} $,
			\item  $S_m\wr S_2, \left( S_m\wr S_2\right) \cap A_{2m}, A_m\wr S_2$ and $T_{m,2}$,
			\item $S_2\wr S_h$ and $(S_2\wr S_h)\cap A_{2h}$,
			\item $S_m\wr S_3, \left( S_m\wr S_3\right) \cap A_{3m}$ and $T_{m,3}$,
			\item $S_k\times L$, $A_k\times L$ and $(S_k\times L)\cap A_n$ where $L$ is one of $\mathrm{P}\Gamma\mathrm{L}_2(\mathbb{F}_8)\leq S_9, \mathrm{ASL}_3(\mathbb{F}_2)\leq S_8, \mathrm{PGL}_2(\mathbb{F}_5)\leq S_6$ and $\mathrm{AGL}_1(\mathbb{F}_5)\leq S_5$,
			\item $A_k\times S_2\wr S_2, N_k$ and $T_{k,2,h}$ with $h\in \left\lbrace 3,4 \right\rbrace$,
			\item $A_m\wr S_2$ embedded in $S_{2m+1}$ by the natural inclusion of $S_{2m}$ in $S_{2m+1}$, provided $m$ is a square.
		\end{enumerate}
	\end{theorem}
	
	We have structured the statement such that in all families apart from (ix) all groups $G$ are subgroups of index $1,2$ or $4$ of the first group in the family. The subgroups in (ix) are then index two subgroups of groups of the form $S_k\times S_2\wr S_h$ with $h\in \left\lbrace 2,3,4\right\rbrace $.
	
	From our result, one can spot the slightly surprising fact that there is a multiplicity-free subgroup $G$ of a symmetric group such that its corresponding permutation character is not multiplicity-free. An example of such $G$ is $S_k\times S_2 \wr S_h$ with $k,h\geq 2$: the corresponding permutation character is not multiplicity-free by, for instance, \cite[Theorem~2]{WildonMultiplicity-free09}, but the group is multiplicity-free, which has already been shown in \cite[Lemma~2]{InglisRichardsonSaxlModel90}. Comparing our classifications to the classifications in \cite{WildonMultiplicity-free09} or \cite{GodsilMeagherMultiplicity-free10}, we can list all such subgroups $G$ for $n\geq 66$: they are $A_k\times A_k$, $A_k\times A_{k+1}$, $S_k\times\mathrm{ASL}_3(\mathbb{F}_2)$, $A_k\times\mathrm{ASL}_3(\mathbb{F}_2)$, $(S_2\wr S_h)\cap A_{2h}$ for even $h$, the subgroups in (iii) and (iv) with $k\geq 2$ and the subgroups in (vii), (ix) and (x).
	
	Our other main results find for sufficiently large $n$ all the irreducible induced-multiplicity-free characters of all multiplicity-free groups $G$ apart from those in Theorem~\ref{Theorem main}(ii) and (v). Moreover, we classify certain `elementary irreducible' induced-multiplicity-free characters of groups in (v) and explain why the remaining `non-elementary irreducible' characters are `essentially induced' from irreducible characters of groups in (ii); see the beginning of \S\ref{Sec n=2m} for details. This leaves the classification for groups in (ii) as the only missing piece to the full classification of all irreducible induced-multiplicity-free characters. In fact, it only remains to tell which characters from a greatly reduced list of irreducible characters of $(S_k\times S_l)\cap A_{k+l}$, $A_k\times S_l$ and $A_k\times A_l$, arising from Theorem~\ref{Theorem Stembridge}, are induced multiplicity-free. However, due to intricate combinatorics of `finding conjugate partitions' this question remains open. 
	
	For practical reasons, we only provide references to the respective results for each family of groups (apart from (i) and (ii)) here. Note that (i) is omitted since all the irreducible characters of $S_n$ and $A_n$ are induced-multiplicity-free. This is clear for $S_n$ and well known for $A_n$ as $|S_n:A_n|=2$; see Lemma~\ref{Lemma index two}. For the remaining families see the notation $\rho_N$ introduced after Lemma~\ref{Lemma index two} as well as 
	\begin{enumerate}[label=\textnormal{(\roman*)}]
		\setcounter{enumi}{2}
		\item Theorem~\ref{Theorem direct product k geq 2}(i) and (ii), Proposition~\ref{Prop direct product k eq 1}(i), Corollary~\ref{Cor Sk times Sm sr S2 class}, Corollary~\ref{Cor S1 times Sm wr S2 class};
		\item Theorem~\ref{Theorem direct product k geq 2}(iii) and (iv), Proposition~\ref{Prop direct product k eq 1}(ii) and Corollary~\ref{Cor Sk times S2 wr Sm class};
		\item  Proposition~\ref{Prop MF wreath products}(i) and (ii), Corollary~\ref{Cor Am wr S2 class}, Corollary~\ref{Cor Sm wr S2 cap An class} and Corollary~\ref{Cor Tm,2 class};
		\item Proposition~\ref{Prop MF wreath products}(i) and (iv) and Corollary~\ref{Cor S2 wr Sm class}(i);
		\item Proposition~\ref{Prop MF wreath products}(iii) and Proposition~\ref{Prop Sm wr S3 class}(i);
		\item Proposition~\ref{Prop sk times primitive}, Theorem~\ref{Theorem direct product arbitrary}(iii), Proposition~\ref{Prop Sk times small easy class}(i) and Proposition~\ref{Prop Prop Sk times small hard class};
		\item Theorem~\ref{Theorem direct product arbitrary}(ii), Proposition~\ref{Prop Sk times small easy class}(ii) and (iii);
		\item Proposition~\ref{Prop Am wr S2 class}.
	\end{enumerate}
	
	We also prove several combinatorial results regarding conjugate partitions which may be of independent interest. We now state one of them which the author finds particularly surprising. In the statement, for positive integers $a,b$, we say that a partition $\lambda$ is \textit{$(a,b)$-birectangular} if $\chi^{\lambda}$ is an irreducible constituent of $\left(\chi^{(a^b)}\boxtimes \chi^{(a^b)} \right)\Ind^{S_{2ab}}$ (see Definition~\ref{Definition birectangular} for an alternative definition).
	
	\begin{proposition}\label{Prop conjugate (a,b)-birectangles}
		Let $a>b$ be positive integers and write $d=a-b$. There is a partition $\lambda$ which is both $(a,b)$-birectangular and $(b,a)$-birectangular if and only if $d|a$. Moreover, in such a case there is a unique such $\lambda$ given by $((2b)^{2d}, (2b-2d)^{2d}, \dots, (2d)^{2d})$.
	\end{proposition}
	
	As a consequence of this result the slightly surprising condition $a-b$ divides $a$, or more precisely its negation, appears in classifications of irreducible induced-multiplicity-free characters of several groups; for example, see Corollary~\ref{Cor S1 times Sm wr S2 class}. 
	
	A corollary of our proofs in $\S\ref{Sec S_k times L}$ is the classification of multiplicity-free symmetric functions of the form $s_{\lambda}(s_{\nu}\circ s_{\mu})$ with $|\lambda|\geq 1$ and $|\mu|,|\nu|\geq 2$. One can find the possible choices of $\lambda,\mu$ and $\nu$ in Proposition~\ref{Prop product with sm wr s2} and Proposition~\ref{Prop product with s2 wr sm}, bearing in mind that in the latter partitions $\lambda$ and $(2)$ can be simultaneously replaced with their conjugates.
	
	From Theorem~\ref{Theorem main} we see that the triple wreath product $S_a\wr S_b\wr S_c$ is not multiplicity-free, provided that $a,b,c\geq 2$ and $n=abc\geq 66$. Using the relation between plethysms of Schur functions and characters of wreath products described at the end of \S\ref{Sec background}, it follows that for any partitions $\lambda,\mu$ and $\nu$ of size at least two such that $|\lambda||\mu||\nu|\geq 66$, the symmetric function $s_{\nu}\circ s_{\mu}\circ s_{\lambda}$ is not multiplicity-free. This is analogous to \cite[Remark~3.5(a)]{StembridgeMultiplicity-free01} which states that a product of three non-trivial Schur functions is not multiplicity-free. Note that the condition $|\lambda||\mu||\nu|\geq 66$ is needed as, for instance, $s_{(2)}\circ s_{(1^2)}\circ s_{(2)} =s_{(2)}\circ s_{(3,1)}$ is multiplicity-free by \cite[Theorem~1.1]{BessenrodtBowmanPagetMF22}.  
	
	\subsection{Background}\label{Sec background}
	
	As already mentioned this paper generalises the classification of multiplicity-free permutation characters of symmetric groups established by Wildon \cite[Theorem~2]{WildonMultiplicity-free09} and independently by Godsil and Meagher \cite[\S5]{GodsilMeagherMultiplicity-free10}. Both papers are based on an earlier result by Saxl \cite[p.~340]{SaxlMultiplicity-free81} which narrows down the list of possible multiplicity-free permutation characters of symmetric groups.
	
	The above classification can be stated as the classification of Gelfand pairs $(S_n,G)$; recall that a pair of groups $(H,G)$ is a \textit{Gelfand pair} if $G\leq H$ and the permutation character of $H$ acting on the cosets of $G$ is multiplicity-free. See \cite[Ch. VII]{MacdonaldPolynomials95} for a background on Gelfand pairs and closely related zonal polynomials. 
	
	A stronger notion is given by \textit{strong Gelfand pairs}, defined as pairs of groups $(H,G)$ with $G\leq H$ such that for each irreducible character $\rho$ of $G$ the induced character $\rho\ind^H$ is multiplicity-free. The strong Gelfand pairs $(S_n,G)$ have recently been found by Anderson, Humphries and Nicholson \cite[Theorem~1.2]{AndersonHumphriesNicholsonStrongGelfPairsSn21}. Note that the classifications of Gelfand pairs and strong Gelfand pairs $(S_n,G)$ are immediate from our results as long as one omits groups $G$ from Theorem~\ref{Theorem main}(ii).
	
	In recent years there have been several new results regarding (strong) Gelfand pairs $(M\wr S_h,G)$, where $M$ and $G$ are finite groups. In particular, \cite[Theorems~1.1 and 1.2]{BensonRatcliffGelfarndWreath18} proved that $(M\wr S_h,\Delta M\times S_h)$ (where $\Delta M$ is the diagonal subgroup in $M\wr S_h$) is not a Gelfand pair as long as $M$ is not abelian and $h=h(M)$ is sufficiently large. In the case $M$ is abelian, we obtain a Gelfand pair which has been studied with connections to the parking functions in \cite{AkerCanParkingGelfand12}. According to Tout \cite[Theorem~1.1]{ToutGelfandWreath21}, also $(M\wr S_h, M\wr S_{h-1})$ is a Gelfand pair if and only if $M$ is abelian. Finally, a study of strong Gelfand pairs $(M\wr S_h, G)$ and their classification in the case $M=S_2$ can be found in Can, She and Speyer \cite{CanSheSpeyerStrongGelfPairsFwrSn21}.
	
	Inspired by the above definitions we say that a pair of groups $(H,G)$ with $G\leq H$ is a \textit{weak Gelfand pair} if there is an irreducible character $\rho$ of $G$ such that $\rho\ind^H$ is multiplicity-free. Our Theorem~\ref{Theorem main} then classifies weak Gelfand pairs $(S_n,G)$. Note that the word `weak' was chosen intentionally since every strong Gelfand pair is a Gelfand pair and every Gelfand pair is a weak Gelfand pair.   
	
	Another way of thinking about multiplicity-free characters uses that a character of $G$ is multiplicity-free if and only if the endomorphism ring of the associated $\mathbb{C}G$-module is commutative. Thus our work investigates induced $\mathbb{C}S_n$-modules $V$ with $\End(V)$ commutative. Such modules were used in \cite{GiannelliWildonFoulkesandDecomposition15} to prove a result on the decomposition matrices of the symmetric groups.
	
	A crucial role in this paper is played by the symmetric functions which provide combinatorial tools for computing characters of symmetric groups. The key symmetric functions for our purposes are the Schur functions $s_{\lambda}$ indexed by partitions $\lambda$ and defined, for instance, in \cite[\S7.10]{StanleyEnumerativeII99}. We also use two associative binary operations of symmetric functions: the natural multiplication and the plethysm product, defined, for instance, in \cite[p.~447]{StanleyEnumerativeII99} and \cite[\S I.8]{MacdonaldPolynomials95}.
	
	Products of Schur functions correspond to induced characters from Young subgroups; more precisely, the decompositions of $\left( \chi^{\mu}\boxtimes\chi^{\nu}\right)\ind^{S_{k+l}} $ (where $k=|\mu|, l=|\nu|$) into irreducible characters and of $s_{\mu} s_{\nu}$ into Schur functions are given by the same multiset of partitions. In the same way the induced character $\charwr{\chi^{\mu}}{\chi^{\nu}}{h}\ind^{S_{mh}}$ of the imprimitive wreath product $S_{m}\wr S_{h}$ (where $m=|\mu|, h=|\nu|$) corresponds to the plethysm $s_{\nu}\circ s_{\mu}$. While a combinatorial rule for decomposing products of Schur functions is known (see \S\ref{Sec LR} and, in particular, Theorem~\ref{Theorem LR rule}), the decompositions of even the simplest plethysms, such as $s_{(h)}\circ s_{(m)}$, is yet to be described; see \cite[Problem~9]{StanleyPositivity00}. The connection between plethysms and induced characters from wreath products plays an important role (not only) in this paper, which allows one to combine algebraic and combinatorial methods in studying either of the two objects.
	
	Many of our results are stated using multiplicity-free symmetric functions and two of the already proved results we build on also concern multiplicity-free symmetric functions. They are the classification of the multiplicity-free products of Schur functions found by Stembridge \cite[Theorem~3.1]{StembridgeMultiplicity-free01} and the classification of the multiplicity-free plethysms of Schur functions established by Bessenrodt, Bowman and Paget \cite[Theorem~1.1]{BessenrodtBowmanPagetMF22}. When classifying irreducible induced-multiplicity-free characters of more complicated groups, such as those in (ix), we often use these two results as the initial step followed by a novel combinatorial argument.
	
	\subsection{Outline}
	
	In \S\S\ref{Sec Par}--\ref{Sec plethysms} we recall definitions and relevant results regarding characters of symmetric groups and symmetric functions. In \S\ref{Sec primitive} we establish lower bounds on the size of a multiplicity-free subgroup of $S_n$ and in turn show that $S_n$ and $A_n$ are the only primitive multiplicity-free subgroups of $S_n$ for $n\geq 13$. The lower bounds are also used in \S\ref{Sec subgroups} when looking at particular subgroups of $S_k\times S_l$ and $S_m\wr S_2$.  
	
	The aim of \S\ref{Sec S_k times L} is to classify the multiplicity-free subgroups $G$ of the form $S_k\times L$ with $k\geq 2$, for $L<S_l$ not equal to $A_l$, as well as their irreducible induced-multiplicity-free characters.
	
	In \S\ref{Sec consequences} we build on the classification from \S\ref{Sec S_k times L} and consider other forms of the group $G$. We end the section by narrowing the list of all subgroups of symmetric groups down to a list of possible multiplicity-free subgroups not so different to the one in Theorem~\ref{Theorem main}.
	
	The goal of \S\S\ref{Sec SIS I}--\ref{Sec SIS II} is then to go through most of the groups on our narrowed list and seek their irreducible induced-multiplicity-free characters. This often involves using combinatorial results, such as Proposition~\ref{Prop conjugate (a,b)-birectangles}, to investigate when a sum of two particular multiplicity-free characters remains multiplicity-free. We have devoted \S\ref{Sec rectangles} to these results.
	
	Finally, Theorem~\ref{Theorem main} is proved in \S\ref{Sec main thm}.  
	
	\section{Preliminaries}
	
	\subsection{Partitions and tableaux}\label{Sec Par}
	A \textit{partition} is a non-increasing sequence $\lambda=(\lambda_1,\lambda_2,\dots,\lambda_t)$ of positive integers, called the \textit{parts} of $\lambda$. We refer to the sum $\sum_{i=1}^{t}\lambda_i$ as the \textit{size} of $\lambda$ and denote it by $|\lambda|$. If $n$ is the size of a partition $\lambda$, we say that $\lambda$ is a \textit{partition of $n$} and write $\lambda\vdash n$. The number of parts of a partition $\lambda$ is called the \textit{length} of $\lambda$ and is denoted by $\ell(\lambda)$. The \textit{Young diagram} of $\lambda$ is the set $Y_{\lambda}=\left\lbrace (i,j)\in\mathbb{N}^2 : i\leq \ell(\lambda), j\leq \lambda_i \right\rbrace $.We refer to the elements of a Young diagram as \textit{boxes}. Finally, the \textit{conjugate partition} of $\lambda$, denoted by $\lambda'$, is the partition such that $Y_{\lambda'}=\left\lbrace (i,j) : (j,i)\in Y_{\lambda} \right\rbrace $. We write $\lambda^{\prime r}$ for the partition obtained from $\lambda$ by conjugating $r$ times. That is, $\lambda^{\prime r}$ is $\lambda$ if $r$ is even and $\lambda'$ if $r$ is odd.
	
	We write $a^b$ for $b$ parts of size $a$ and, depending on the situation, we may attach a finite or an infinite number of zeros at the end of a partition. For example, $(3,3,3,1,1)$ and $(3^3,1^2,0)$ denote the same partition of length $5$.
	
	For two partitions $\lambda$ and $\mu$ we define $\lambda+\mu$ to be the partition $(\lambda_1+\mu_1,\lambda_2+\mu_2,\dots)$ using the convention on tailing zeros just stated. We also define $\lambda\sqcup\mu$ to be the conjugate of $\lambda'+\mu'$. Thus the multiset of the parts of $\lambda\sqcup\mu$ is the union of the multisets of the parts of $\lambda$ and $\mu$. Another pair of partitions we use are $M(\lambda,\mu)$ and $m(\lambda,\mu)$ defined by $M(\lambda,\mu)_i=\max(\lambda_i,\mu_i)$ and $m(\lambda,\mu)_i=\min(\lambda_i,\mu_i)$. Letting $|\lambda-\mu|=\sum_{i\geq 1}|\lambda_i-\mu_i|$, we obtain \[|M(\lambda,\mu)|=(|\lambda| + |\mu| + |\lambda-\mu|)/2\]
	and
	\[|m(\lambda,\mu)|=(|\lambda| + |\mu| - |\lambda-\mu|)/2.\]
	
	For two partitions $\lambda$ and $\mu$ we write $\mu\subseteq\lambda$ if $Y_{\mu}\subseteq Y_{\lambda}$. We call a pair of partitions $(\lambda,\mu)$ with $\mu\subseteq \lambda$ a \textit{skew partition} and denote it by $\lambda/\mu$. We write $Y_{\lambda/\mu}=Y_{\lambda}\setminus Y_{\mu}$ for its Young diagram.
	
	One can see that partitions $\lambda,\mu$ and $\nu$ satisfy $\mu,\nu\subseteq \lambda$ if and only if $M(\mu,\nu)\subseteq \lambda$. Similarly, the condition $\lambda\subseteq\mu,\nu$ is equivalent to $\lambda\subseteq m(\mu,\nu)$. We conclude the following.
	
	\begin{lemma}\label{Lemma contains two partitions}
		Let $\mu$ and $\nu$ be partitions and $n$ a non-negative integer.
		\begin{enumerate}[label=\textnormal{(\roman*)}]
			\item There is a partition $\lambda\vdash n$ such that $\mu,\nu\subseteq \lambda$ if and only if $|\mu-\nu|\leq 2n-|\mu|-|\nu|$.
			\item There is a partition $\lambda\vdash n$ such that $\lambda\subseteq\mu,\nu$ if and only if $|\mu-\nu|\leq |\mu|+|\nu|-2n$.
		\end{enumerate}
	\end{lemma}
	
	A \textit{$\lambda/\mu$-tableau} is a map $t:Y_{\lambda/\mu} \to \mathbb{N}$. A \textit{$\lambda$-tableau} is defined analogously. When the underlying (skew) partition is either implicit from the context or irrelevant we simply say tableau. The \textit{weight} $w(t)$ of a tableau $t$ is a sequence $(t_1,t_2,\dots)$, where $t_i$ is the number of boxes labelled $i$ in $t$, where by the label of a box $b$ we mean $t(b)$, so that $t_i:=|t^{-1}(i)|$. In this paper Young diagrams and tableaux are drawn using the `English' convention; see Figure~\ref{Figure tableau} for an example.
	
	A tableau is \textit{semistandard} if its entries are non-decreasing rightwards along its rows and increasing downwards along its columns. In Figure~\ref{Figure tableau} the middle tableau is semistandard, while the rightmost one is not.
	
	\begin{figure}[h!]
		$\young(\ \ \ \ast\ast,\ \ \ast,\ \ast\ast,\ast\ast,\ast) \qquad  \young(12445,345,556,78,8) \qquad \young(:::11,::1,:23,22,3)$
		\caption{The first diagram is the Young diagram of the partition $\lambda=(5,3^2,2,1)$. The starred boxes form the Young diagram of $\lambda/\mu$ where $\mu=(3,2,1)$. The second diagram is a $\lambda$-tableau with weight $(1,1,1,3,4,1,1,2)$. The final diagram is a $\lambda/\mu$-tableau with weight $(3,3,2)$.}
		\label{Figure tableau}
	\end{figure}
	
	Throughout, most of the appearing partitions are of particularly nice forms summarised in the following definition. We use an adaptation of the terminology from \cite{BessenrodtBowmanPagetMF22}.
	
	\begin{definition}\label{Defn nice partitions}
		We say that a partition $\lambda$ is
		\begin{itemize}
			\item the \textit{empty} partition denoted by $\o$ if $|\lambda|=0$;
			\item a \textit{row} partition if $\ell(\lambda)=1$;
			\item a \textit{column} partition if $\lambda_1=1$;
			\item a \textit{linear} partition if it is a row or a column partition;
			\item a \textit{square} partition if it is of the form $(a^a)$ for some $a>0$;
			\item a \textit{rectangular} partition if it is of the form $(a^b)$ for some $a,b>0$;
			\item a \textit{properly rectangular} partition if it is rectangular but not linear;
			\item a \textit{$2$-rectangular} partition if it is rectangular and either $\lambda_1=2$ or $\ell(\lambda)=2$;
			\item an \textit{almost rectangular} partition if it is non-empty, not rectangular but it becomes a rectangular partition after increasing or decreasing one of its parts by $1$;
			\item a \textit{row-near rectangular} partition if it is not rectangular but becomes rectangular after omitting the final part;
			\item a \textit{column-near rectangular} partition if $\lambda'$ is row-near rectangular;
			\item a \textit{near rectangular} partition if it is row-near rectangular or column-near rectangular;
			\item a \textit{hook} if it is of the form $(a+1,1^b)$ for some $a,b>0$ (compared to the usual conventions, this definition excludes linear partitions);
			\item a \textit{fat hook} if it is of the form $(a^b,c^d)$ for some $a,b,c,d>0$ with $a>c$;
			\item a \textit{row-unbalanced fat hook} if it is of the form $(a,c^d)$ for some $a,c,d>0$ with $a>c$;
			\item a \textit{column-unbalanced fat hook} if $\lambda'$ is a row-unbalanced fat hook. 
		\end{itemize}
	\end{definition}
	
	\subsection{Characters of symmetric groups and symmetric functions}\label{Sec chi and s}
	The irreducible characters of the symmetric group $S_n$ are indexed by partitions $\lambda$ of size $n$ and denoted by $\chi^{\lambda}$; see \cite[\S6]{JamesSymmetric78}. The two linear characters of $S_n$, namely the trivial character and the sign, are given by $\chi^{(n)}$ and $\chi^{(1^n)}$, respectively.
	
	We leave most of the needed results regarding characters of symmetric groups to \S\S\ref{Sec LR} and \ref{Sec plethysms}.  Now we only recall the identity $\sgn\times\chi^{\lambda}=\chi^{\lambda'}$, which holds for all partitions $\lambda$.
	
	Let us introduce a vital graded ring in the representation theory of symmetric groups called $\CF$. We write $\CF^n$ for the free abelian group of virtual characters of $S_n$, freely generated by the irreducible characters of $S_n$. As a graded abelian group $\CF=\bigoplus_{n\in\mathbb{N}_0}\CF^n$. It is equipped with a multiplication $\cdot$ given by the bilinear extension of the identity $\chi^{\mu}\cdot\chi^{\nu} = \left(\chi^{\mu}\boxtimes\chi^{\nu} \right)\Ind^{S_{k+l}}$, where $k=|\mu|, l=|\nu|$ and $\boxtimes$ denotes the outer tensor product.
	
	Let $\Lambda$ be the graded ring of the symmetric functions over integers defined in \cite[\S7.1]{StanleyEnumerativeII99}. An integral basis of $\Lambda$ is given by \textit{Schur functions} $s_{\lambda}$ indexed by partitions $\lambda$; see \cite[\S7.10]{StanleyEnumerativeII99}. The degree of $s_{\lambda}$ is $|\lambda|$ and we denote by $\left\langle \cdot,\cdot\right\rangle $ the unique bilinear form on $\Lambda$ with respect to which Schur functions form an orthonormal basis. 
	
	The following key result links the representation theory of symmetric groups and symmetric functions.
	
	\begin{theorem}[Frobenius characteristic map]\label{Theorem frobenius}
		The map $\xi:\CF \to \Lambda$ defined by $\chi^{\lambda}\mapsto s_{\lambda}$ is an isomorphism of graded rings.
	\end{theorem} 
	
	In light of Theorem~\ref{Theorem frobenius}, we will freely move between characters of symmetric groups and symmetric functions. To make this transition easier, write $\Lambda^{+}$ for the subset of the non-zero homogeneous elements $x$ of $\Lambda$ such that for any partition $\lambda$ we have $\left\langle x,s_{\lambda} \right\rangle \geq 0$. Thus $\Lambda^{+}$ is the image of $\xi$ restricted to the genuine characters. We say $x\in\Lambda^{+}$ is \textit{multiplicity-free} if for all partitions $\lambda$ the inequality $\left\langle x,s_{\lambda} \right\rangle \leq 1$ holds, or equivalently if $\xi^{-1}(x)$ is multiplicity-free. On a similar note, we say that $s_\lambda$ is a \textit{constituent} of $x\in\Lambda^{+}$ if $\left\langle x,s_{\lambda} \right\rangle \geq 1$.
	
	We also introduce the linear involution $\omega:\Lambda\to\Lambda$ given by the linear extension of $\omega(s_{\lambda})=s_{\lambda'}$ for all partitions $\lambda$. This involution corresponds via $\xi$ to tensoring with a sign, and consequently $\omega$ is a ring homomorphism.
	
	Note that one could carry most of the work in this paper with characters only. However, working with characters often requires a hard-to-read notation, while symmetric functions come with a simpler and more aesthetic notation; thus we prioritise them whenever possible. 
	
	\subsection{Products of Schur functions}\label{Sec LR}
	
	For three partitions $\lambda, \mu$ and $\nu$, the \textit{Littlewood--Richardson coefficient} $c(\mu,\nu;\lambda)$ is defined as $\left\langle s_{\mu}s_{\nu},s_{\lambda}\right\rangle $. We can immediately see that this coefficient is non-zero only if $|\mu|+|\nu|=|\lambda|$. Two easy properties one should keep in mind are $c(\mu,\nu;\lambda)=c(\nu,\mu;\lambda)$ and $c(\mu,\nu;\lambda)=c(\mu',\nu';\lambda')$ following after applying $\omega$ to $s_{\mu}s_{\nu}$.
	
	The coefficients $c(\mu,\nu;\lambda)$ can be computed combinatorially using the Littlewood--Richardson rule. We use the following two definitions to state the rule in the form we use it later.
	
	\begin{definition}\label{Definition reading word}
		A \textit{reading word} $r(t)$ of a tableau $t$ is a word of positive integers obtained by reading the columns of $t$ downwards, starting with the rightmost column and moving left.
	\end{definition}
	
	\begin{remark}\label{Remark reading word}
		In the literature various ways of reading the entries of tableaux (and domino tableaux, see Definition~\ref{Definition reading domino}) appear. For instance, James \cite{JamesSymmetric78} reads each row leftwards, starting with the top row and moving down, while Carr\'{e}--Leclerc \cite{CarreSplitting95} read our reading word in reverse order.
	\end{remark}
	
	\begin{example}\label{Example reading word}
		The tableaux in Figure~\ref{Figure LR rule} have reading words $123121$, $112132$ and $112231$, respectively.
	\end{example}
	
	\begin{figure}[h!]
		$\young(::::1,::12,:123) \qquad  \young(::11,:12,23) \qquad \young(::11,:22,13)$
		\caption{Tableaux for Examples~\ref{Example reading word} and \ref{Example LR rule}.}
		\label{Figure LR rule}
	\end{figure}
	
	\begin{definition}\label{Definition latticed word}
		Let $w$ be a word of positive integers. We say that $w$ is a \textit{latticed word} if for any positive integer $i$, any prefix of $w$ contains at least as many occurrences of $i$ as of $i+1$.
	\end{definition}
	
	\begin{example}\label{Example latticed word}
		The words $1122313$ and $121132434$ are latticed. On the other hand, the word $1123132$ is not latticed since in the first $6$ letters there are two occurrences of $3$ but only one of $2$.
	\end{example}
	
	\begin{theorem}[Littlewood--Richardson rule]\label{Theorem LR rule}
		The Littlewood--Richardson coefficient $c(\mu,\nu;\lambda)$ is zero unless $\mu\subseteq\lambda$ in which case it equals the number of the semistandard $\lambda/\mu$-tableaux with weight $\nu$ and a latticed reading word.
	\end{theorem}
	
	\begin{example}\label{Example LR rule}
		Let $\mu=(4,2,1)$ and $\nu=(3,2,1)$. If $\lambda=(5,4^2)$ and $\bar{\lambda}=(4^2,3,2)$, we compute that $c(\mu,\nu;\lambda)=1$ and $c(\mu,\nu;\bar{\lambda})=2$ using the $\lambda/\mu$-tableau and the $\bar{\lambda}/\mu$-tableaux from Figure~\ref{Figure LR rule}.
	\end{example}
	
	Let us mention several easy consequences of the Littlewood--Richardson rule we use. If $c(\mu,\nu;\lambda)\neq 0$, then $\lambda_1\leq \mu_1+\nu_1$ and $\ell(\lambda)\leq \ell(\mu)+\ell(\nu)$. Moreover, in such a case $\lambda/\mu$ has at least $\ell(\nu)$ rows. The next consequence plays an important role in this paper.
	
	\begin{lemma}\label{Lemma rotate linear}
		Let $\mu$ and $\nu$ be partitions of equal size and $x$ and $y$ be elements of $\Lambda^{+}$ of degree $l$. If $s_{\mu}x$ and $s_{\nu}y$ are multiplicity-free and $|\mu-\nu|>2l$, then $s_{\mu}x +s_{\nu}y$ is also multiplicity-free.
	\end{lemma}
	
	\begin{proof}
		Write $k=|\mu|=|\nu|$. If the statement fails to hold, there is $\lambda\vdash k+l$ such that $s_{\lambda}$ is a common constituent of $s_{\mu}x$ and $s_{\nu}y$. From the Littlewood--Richardson rule $\mu,\nu\subseteq\lambda$ and by Lemma~\ref{Lemma contains two partitions}(i) we have $|\mu-\nu|\leq 2(k+l)-2k=2l$, a contradiction.
	\end{proof}
	
	We often use this lemma with non-square rectangular partitions $\lambda$ and $\lambda'$ in which case the next lemma provides the required hypothesis.  
	
	\begin{lemma}\label{Lemma difference of rectangles}
		Suppose that $\lambda$ is a non-square rectangular partition and $l$ a positive integer. If $|\lambda|>l^2 +l$, then $|\lambda-\lambda'|>2l$.
	\end{lemma}
	
	\begin{proof}
		Write $\lambda=(a^b)$ and without loss of generality let $a>b$. Then $|\lambda-\lambda'|$ equals $2b(a-b)$; thus we need to show that $b(a-b)>l$. If this fails to hold, then $b\leq l$ and $|\lambda|=ab\leq ab+l^2-b^2 = b(a-b)+l^2\leq l^2+l$, a contradiction.
	\end{proof}
	
	The Littlewood--Richardson rule specialises to simpler Young's and Pieri's rules when $\nu$ is a linear partition. We define a \textit{horizontal strip} to be a skew partition with at most one box in each column and a \textit{vertical strip} to be a skew partition with at most one box in each row.
	
	\begin{theorem}[Young's rule]\label{Theorem Y rule}
		The coefficient $c(\mu,(k);\lambda)$ is zero unless $\lambda/\mu$ is a horizontal strip of size $k$, in which case $c(\mu,(k);\lambda)=1$.
	\end{theorem} 
	
	\begin{theorem}[Pieri's rule]\label{Theorem P rule}
		The coefficient $c(\mu,(1^k);\lambda)$ is zero unless $\lambda/\mu$ is a vertical strip of size $k$, in which case $c(\mu,(1^k);\lambda)=1$.
	\end{theorem}
	
	We give an example of an application of Pieri's rule in the following lemma.
	
	\begin{lemma}\label{Lemma Pieri application}
		Let $x$ be a multiplicity-free element of $\Lambda^{+}$ and $k$ a non-negative integer. Then $s_{(1^k)}x$ is multiplicity-free if and only if for any constituents $s_{\mu}$ and $s_{\nu}$ of $x$ with $\mu\neq \nu$ one of the following holds:
		\begin{enumerate}[label=\textnormal{(\roman*)}]
			\item $|\mu-\nu|>2k$,
			\item there is a positive integer $i$ such that $|\mu_i-\nu_i|\geq 2$. 
		\end{enumerate}
	\end{lemma}
	
	\begin{proof}
		By Pieri's rule, $s_{(1^k)}s_{\mu}$ is multiplicity-free for any partition $\mu$; thus $s_{(1^k)}x$ is multiplicity-free if and only if for any constituents $s_{\mu}$ and $s_{\nu}$ of $x$ with $\mu\neq \nu$, the symmetric functions $s_{(1^k)}s_{\mu}$ and $s_{(1^k)}s_{\nu}$ do not share any constituents. Hence it is sufficient to show that for any partitions $\mu\neq \nu$, the symmetric functions $s_{(1^k)}s_{\mu}$ and $s_{(1^k)}s_{\nu}$ share a constituent, say $s_{\lambda}$, if and only if (i) and (ii) fail to hold. 
		
		For the `only if' direction we see that (i) fails to hold from Lemma~\ref{Lemma rotate linear} applied with $x=y=s_{(1^k)}$. Moreover, $1\geq\lambda_i- \nu_i\geq \mu_i-\nu_i$ for any $i\geq 1$ by Pieri's rule. Swapping $\mu$ and $\nu$ shows that (ii) fails to hold.
		
		For the `if' direction let $l=k+|\mu|-|M(\mu,\nu)|=k-|\mu-\nu|/2$. Note $l\geq 0$ since (i) is false. Thus $\lambda=M(\mu,\nu)\sqcup (1^l)$ is a well-defined partition of size $|\mu|+k$ and $\mu,\nu\subseteq\lambda$. As (ii) is false we conclude that $\lambda/\mu$ and $\lambda/\nu$ are vertical strips, as needed.
	\end{proof}
	
	Young's and Pieri's rules coincide when $k=1$ and in such a case one obtains the induction branching rule. While the rule is implicit from Young's and Pieri's rules we state it here using characters to accompany the less obvious reduction branching rule.
	
	\begin{theorem}[Branching rules]\label{Theorem branching rules}
		Let $n$ be a non-negative integer.
		\begin{enumerate}[label=\textnormal{(\roman*)}]
			\item If $\mu$ is a partition of $n$, then
			\[\chi^{\mu}\ind^{S_{n+1}}=\sum_{\substack{\lambda\vdash n+1\\ \mu\subseteq\lambda}}\chi^{\lambda}.\]
			\item If $\lambda$ is a partition of $n+1$, then \[\chi^{\lambda}\res_{S_{n}}=\sum_{\substack{\mu\vdash n\\ \mu\subseteq\lambda}}\chi^{\mu}.\]
		\end{enumerate}		
	\end{theorem}
	
	One immediately obtains a simple corollary.
	
	\begin{corollary}\label{Cor branching rules}
		Let $m\leq n$ be non-negative integers and $\mu$ and $\lambda$ partitions of $m$ and $n$, respectively.
		\begin{enumerate}[label=\textnormal{(\roman*)}]
			\item $\chi^{\lambda}$ is an irreducible constituent of $\chi^{\mu}\ind^{S_n}$ if and only if $\mu\subseteq \lambda$.
			\item $\chi^{\mu}$ is an irreducible constituent of $\chi^{\lambda}\res_{S_m}$ if and only if $\mu\subseteq \lambda$.
		\end{enumerate}
	\end{corollary}
	
	\begin{proof}
		This follows from the branching rules by induction on $n-m$.
	\end{proof}
	
	Another particularly simple case of the Littlewood--Richardson rule is when $\mu$ and $\nu$ are rectangular partitions. For the scope of this paper it is sufficient to consider the cases $\mu=\nu$ and $\mu=\nu'$. To describe the irreducible constituents of $s_{\mu}s_{\nu}$ in these cases we introduce the following terminology. 
	
	\begin{definition}\label{Definition birectangular}
		Let $a$ and $b$ be positive integers and write $m=\min(a,b)$ and $M=\max(a,b)$. A partition $\lambda$ of $2ab$ is
		\begin{enumerate}[label=\textnormal{(\roman*)}]
			\item \textit{$(a,b)$-birectangular} if for all $1\leq i \leq 2b$ we have $\lambda_i + \lambda_{2b+1-i}=2a$;
			\item \textit{$\left\lbrace a,b \right\rbrace$-birectangular} if for all $1\leq i\leq m$ we have $\lambda_i + \lambda_{a+b+1-i}=a+b$ and for $m+1\leq i\leq M$ we have $\lambda_i=m$.
		\end{enumerate} 
	\end{definition}
	
	\begin{remark}\label{Remark birectangular}
		Every $(a,b)$-birectangular partition has length at most $2b$. This is since the sum of the first $2b$ parts of such a partition is $2ab$. Similarly, every $\left\lbrace a,b \right\rbrace$-birectangular partition has length at most $a+b$. 
	\end{remark}
	
	\begin{proposition}\label{Prop LR rectangles}
		Let $\mu=(a^b)$.
		\begin{enumerate}[label=\textnormal{(\roman*)}]
			\item $c(\mu,\mu;\lambda)$ is zero unless $\lambda$ is an $(a,b)$-birectangular partition, in which case $c(\mu,\mu;\lambda)=1$.
			\item $c(\mu,\mu';\lambda)$ is zero unless $\lambda$ is an $\left\lbrace a,b\right\rbrace$-birectangular partition, in which case $c(\mu,\mu';\lambda)=1$.
		\end{enumerate} 
	\end{proposition}
	
	\begin{proof}
		See \cite[Lemma~3.3(a)]{StanleyPlaneSymmetries86} for (i) (and (ii) when $a=b$). After observing that for $a\neq b$ any $\left\lbrace a,b\right\rbrace$-birectangular partition $\lambda$ satisfies $\lambda_{\min(a,b)}=a+b-\lambda_{\max(a,b)+1}\geq a+b-\lambda_{\max(a,b)}= \max(a,b)$ we obtain (ii) from \cite[Theorem~2.4]{OkadaRectangularProducts98} applied with $s=n=\max(a,b)$ and $t=m=\min(a,b)$.
	\end{proof}
	
	For later use, we also mention a result concerning the coefficients $c(\mu,\mu;\lambda)$ when $\mu$ is a hook.
	
	\begin{lemma}\label{Lemma LR hooks}
		Let $a$ and $b$ be positive integers and $\mu=(a+1,1^b)$. Then $c(\mu,\mu;\lambda)$ is zero unless $\lambda_1+\lambda_2\in\left\lbrace 2a+2,2a+3,2a+4 \right\rbrace$. In the case that $\lambda_1 + \lambda_2=2a+4$, then
		\begin{align*}
		c(\mu,\mu;\lambda)=\begin{cases}
		1 &\text{if } \lambda_2\geq 2\geq \lambda_3,\\
		0 &\text{otherwise}.\\
		\end{cases}
		\end{align*}
		If instead $\lambda_1 + \lambda_2=2a+2$, then
		\begin{align*}
		c(\mu,\mu;\lambda)=\begin{cases}
		1 &\text{if } \lambda_2\geq 2\geq \lambda_3,\\
		1 &\text{if } \lambda=(2a+1,1^{2b+1}),\\
		0 &\text{otherwise}.\\
		\end{cases}
		\end{align*} 
	\end{lemma}
	
	\begin{proof}
		This follows from \cite[(3.5)]{BessenrodtBowmanPagetMF22}.
	\end{proof}
	
	\subsubsection*{Multiplicity-free products of Schur functions}
	
	The following classification is due to Stembridge \cite[Theorem~3.1]{StembridgeMultiplicity-free01} and corresponds via the isomorphism in Theorem~\ref{Theorem frobenius} to a classification of the irreducible induced-multiplicity-free characters of $S_k\times S_l$. We use the terms for particular partitions from Definition~\ref{Defn nice partitions} in the statement.
	
	\begin{theorem}\label{Theorem Stembridge}
		Let $\mu$ and $\nu$ be non-empty partitions. The product $s_{\mu}s_{\nu}$ is multiplicity-free if and only if one of the following holds:
		\begin{enumerate}[label=\textnormal{(\roman*)}]
			\item $\mu$ or $\nu$ is linear,
			\item $\mu$ is $2$-rectangular and $\nu$ is a fat hook (or vice versa),
			\item $\mu$ is rectangular and $\nu$ is near rectangular (or vice versa),
			\item $\mu$ and $\nu$ are rectangular. 
		\end{enumerate} 
	\end{theorem}
	
	While Theorem~\ref{Theorem Stembridge} can be applied only to products of two symmetric functions, corresponding to Young subgroups with two direct factors, the analogous classification for other Young subgroups is even simpler and implies the following result.
	
	\begin{lemma}\label{Lemma sanity lemma}
		Let $G\leq S_n$ have at least three orbits. Then $G$ is not multiplicity-free.
	\end{lemma}
	
	\begin{proof}
		Since $G$ is embedded in a Young subgroup of the form $S_k\times S_l\times S_m$, without loss of generality assume that $G=S_k\times S_l\times S_m$. The result then follows from \cite[Remark~3.5(a)]{StembridgeMultiplicity-free01}.
	\end{proof}
	
	Let us also state the key tool in Stembridge.
	
	\begin{lemma}\label{Lemma Stembridge observation}
		Let $\lambda, \mu$ and $\nu$ be partitions and $r$ a non-negative integer. Then we have:
		\begin{enumerate}[label=\textnormal{(\roman*)}]
			\item $c(\mu+(1^r), \nu;\lambda+(1^r))\geq c(\mu,\nu;\lambda)$,
			\item $c(\mu\sqcup(r), \nu;\lambda\sqcup(r))\geq c(\mu,\nu;\lambda)$.
		\end{enumerate}
	\end{lemma}
	
	\begin{proof}
		See \cite[Lemma~3.2]{StembridgeMultiplicity-free01}.
	\end{proof}
	
	Inductively, we obtain the same inequalities when $(1^r)$ and $(r)$ are replaced with an arbitrary partition. This motivates the following definition.
	
	\begin{definition}\label{Definition multifunction}
		Let $f=(f_1,f_2,\dots,f_l)$ be a sequence of functions of the form $\lambda\mapsto \lambda + \alpha$ or $\lambda\mapsto\lambda\sqcup\alpha$ from the set of partitions to itself. We refer to $f$ as a \textit{multifunction} and write $f(\lambda)=f_l\left(  f_{l-1}\dots\left(  f_1 \left( \lambda\right) \right) \dots\right) $. Moreover, for two partitions $\lambda$ and $\mu$ we write $\mu\to\lambda$ if there is a multifunction $f$ such that $f(\mu)=\lambda$.
	\end{definition}
	
	\begin{corollary}\label{Corollary Stembridge order}
		Let $\lambda,\mu$ and $\nu$ be partitions and $f$ a multifunction. Then $c(f(\mu), \nu;f(\lambda))\geq c(\mu,\nu;\lambda)$.
	\end{corollary}
	
	\begin{proof}
		Iterate Lemma~\ref{Lemma Stembridge observation}.
	\end{proof}
	
	\begin{example}\label{Example multiorder}
		Let us look at two examples of partitions $\lambda$ and $\mu$ such that $\mu\to\lambda$.
		\begin{enumerate}[label=\textnormal{(\roman*)}]
			\item Suppose that $\mu$ is a rectangular partition. Then for any partition $\lambda$ such that $\mu\subseteq\lambda$ we have $\mu\to\lambda$. Indeed, writing $l=\ell(\mu)$ we have $\lambda=(\mu+\alpha)\sqcup \beta$ where $\alpha=(\lambda_1-\mu_1,\lambda_2-\mu_2,\dots,\lambda_l-\mu_l)$ and $\beta=(\lambda_{l+1},\lambda_{l+2},\dots)$. For instance, if $\mu=(3^2)$ and $\lambda=(4,3,1)$, then $\lambda=(\mu+(1))\sqcup(1)$ and Corollary~\ref{Corollary Stembridge order} shows that $c((4,3,1),(2,1);(5,3,2,1))\geq c((3^2),(2,1);(4,3,2))$. In fact, the left-hand side is $2$, while the right-hand side is $1$.
			
			\item If $\mu=(2,1)$ and $\lambda$ is a hook or an almost rectangular partition, then $\mu\to\lambda$. In the first case $\lambda=(a+1,1^b)$ and we can write $\lambda=(\mu+(a-1))\sqcup (1^{b-1})$. In the second case either $\lambda=((a+1)^{b},a)$ and we can write $\lambda = (\mu+((a-1)^2))\sqcup ((a+1)^{b-1})$ or, up to conjugation, $\lambda=(a+1,a^b)$ and we can write $\lambda=(\mu+((a-1)^2))\sqcup(a^{b-1})$. For instance, $(3,2^2)=((2,1)+(1^2))\sqcup(2)$, and consequently $c((3,2^2),(3^2);(5,4,2^2))\geq c((2,1),(3^2);(4,3,2))$. In fact, both sides equal $1$.  
		\end{enumerate}
	\end{example}
	
	For a later reference let us mention the `injectivity' of multifunctions.
	
	\begin{lemma}\label{Lemma injectivity}
		Let $\mu$ and $\nu$ be partitions and let $f$ be a multifunction such that $f(\mu)=f(\nu)$. Then $\mu=\nu$.
	\end{lemma}
	
	\begin{proof}
		For any partition $\alpha$ the functions $\lambda\mapsto\lambda+\alpha$ and $\lambda\mapsto\lambda\sqcup\alpha$ are injective.
	\end{proof}
	
	\subsection{Plethysms and characters of imprimitive wreath products}\label{Sec plethysms}
	
	For subgroups $M\leq S_m$ and $H\leq S_h$ let $M\wr H$ be the wreath product embedded imprimitively in $S_{mh}$. To describe elements, multiplication and characters of $M\wr H$ we use the conventions from \cite[\S4]{JamesKerberSymmetric81}. In particular, we denote the elements of $M\wr H$ as $(\List{g}{h};\sigma)$ with $(\List{g}{h})\in M^h$ and $\sigma\in H$, and the multiplication in $M\wr H$ is given by $(\List{g}{h};\sigma)(\List{g'}{h};\tau)=(g_1g'_{\sigma^{-1}(1)}, g_2g'_{\sigma^{-1}(2)},\dots, g_hg'_{\sigma^{-1}(h)};\sigma\tau)$. In the case $H=S_h$, the irreducible characters are constructed as follows.  
	
	\begin{definition}\label{Definition elementary irr reps}
		Given characters $\rho$ and $\kappa$ of $M$ and $S_h$, respectively, we define a character $\charwrnb{\rho}{\kappa}{h}$ of $M\wr S_h$ by \[(\List{g}{h};\sigma)\mapsto\kappa(\sigma)\prod_C \rho(g_C),\]
		where the product is over the cycles $C$ of $\sigma$ and for cycle $C=(i_1 \; i_2\; \dots \; i_t)$, $g_C$ denotes $g_{i_t}g_{i_{t-1}}\dots g_{i_1}$.
		If $\rho$ and $\kappa$ are irreducible, we call this character an \textit{elementary irreducible character}.   
	\end{definition}
	
	\begin{theorem}\label{Theorem characters of wreath products}
		Let $\List{\rho}{t}$ be the irreducible characters of $M$. There is a bijection between the set of $t$-tuples of partitions $(\lambda^{[1]},\lambda^{[2]},\dots,\lambda^{[t]})$ of total size $h$ and the irreducible characters of $M\wr S_h$ defined as follows: a $t$-tuple $(\lambda^{[1]},\lambda^{[2]},\dots,\lambda^{[t]})$ of partitions with sizes $(h_1,h_2,\dots,h_t)$ is sent to
		\[
		\left( \left( \charwrnb{\rho_1}{\chi^{\lambda^{[1]}}}{h_1}\right) \boxtimes\left( \charwrnb{\rho_2}{\chi^{\lambda^{[2]}}}{h_2}\right) \boxtimes\dots\boxtimes \left( \charwrnb{\rho_t}{\chi^{\lambda^{[t]}}}{h_t}\right)\right)  \Ind^{M\wr S_h}.
		\]
	\end{theorem}
	
	\begin{proof}
		See \cite[Theorem~4.4.3]{JamesKerberSymmetric81}.
	\end{proof}
	
	We use the isomorphism $\xi$ from Theorem~\ref{Theorem frobenius} to define plethysms in $\Lambda^{+}$; for two characters $\rho$ and $\kappa$ of $S_m$ and $S_h$, respectively, we define the \textit{plethysm} $\xi(\kappa)\circ \xi(\rho)$ as $\xi\left(\left( \charwrnb{\rho}{\kappa}{h}\right) \Ind^{S_{mh}} \right)$. In fact, the operation $\circ$ can be extended to a binary operation on $\Lambda$ (see \cite[p. 447]{StanleyEnumerativeII99}). However, for our purposes our definition is sufficient.
	
	Given partitions $\lambda,\mu$ and $\nu$, the \textit{plethysm coefficient} $p(\mu, \nu;\lambda)$ is defined as $\left\langle s_{\nu}\circ s_{\mu}, s_{\lambda} \right\rangle $. It is zero unless $|\lambda|=|\mu||\nu|$. From the definition if $m=|\mu|$ and $h=|\nu|$, then $\left(\charwr{\chi^{\mu}}{\chi^{\nu}}{h} \right)\Ind^{S_{mh}} = \sum_{\lambda\vdash mh}p(\mu,\nu;\lambda)\chi^{\lambda}$. Note that in our notation the order of `$\mu$' and `$\nu$' in the plethysm coefficient is consistent with the elementary irreducible characters, but, unavoidably, goes the opposite way to the plethysms of Schur functions. 
	
	Recall the notation $\lambda^{\prime r}$ from \S\ref{Sec Par} for the $r$-fold conjugate of the partition $\lambda$. The following are basic properties of plethysms we need.
	
	\begin{lemma}\label{Lemma plethysms and char properties}
		Let $\lambda, \mu, \bar{\mu}$ and $\nu$ be partitions with $|\mu|=|\bar{\mu}|$. Writing $m=|\mu|$, we have:
		\begin{enumerate}[label=\textnormal{(\roman*)}]
			\item $\omega(s_{\nu}\circ s_{\mu}) = s_{\nu^{\prime m}}\circ s_{\mu'}$,
			\item $p(\mu, \nu;\lambda)=p(\mu', \nu^{\prime m};\lambda')$,
			\item $s_{\mu}^2 = s_{(2)}\circ s_{\mu} + s_{(1^2)}\circ s_{\mu}$,
			\item if $|\nu|=2$, then $s_{\nu}\circ (s_{\mu} + s_{\bar{\mu}}) = s_{\nu}\circ s_{\mu} + s_{\mu}s_{\bar{\mu}} + s_{\nu}\circ s_{\bar{\mu}}$.
		\end{enumerate}
	\end{lemma}
	
	\begin{proof}
		The identities follow from routine manipulations with symmetric functions; see \cite[Exercises~I.8.1 and I.8.2]{MacdonaldPolynomials95}.
	\end{proof}
	
	We also frequently use the following formulae in combination with later results; see Example~\ref{Example index two}.
	
	\begin{lemma}\label{Lemma tensoring plethysms}
		Let $m$ and $h$ be two positive integers with $h\geq 2$. If $\alpha,\beta$ and $\gamma$ are characters of $S_m$, $\delta$ is a character of $S_{h-1}$ and $\epsilon$ and $\zeta$ are characters of $S_h$, then
		\[\left( \charwrnb{\alpha}{\epsilon}{h}\right) \times\left( \charwrnb{\beta}{\zeta}{h}\right)  = \charwrdb{\alpha\times\beta}{\epsilon\times\zeta}{h}\]
		and
		\begin{align*}
		&\left(\alpha\boxtimes\left( \charwrnb{\beta}{\delta}{h-1}\right)  \right)\Ind^{S_m\wr S_h}\times \left(\charwrnb{\gamma}{\epsilon}{h} \right)=\\ 
		&\left( \left( \alpha\times\gamma\right) \boxtimes\left( \charwrdb{\beta\times \gamma}{\delta\times \epsilon\res_{S_{h-1}}}{h-1}\right) \right) \Ind^{S_m\wr S_h}.
		\end{align*}
	\end{lemma}
	
	\begin{proof}
		The first formula is obvious. The second one follows from routine manipulations with characters.
	\end{proof}
	
	While a combinatorial rule for computing an arbitrary plethysm coefficient is missing, there are particular plethysms that are well-understood. To describe some of them we say a partition $\lambda$ is \textit{even} if all parts of $\lambda$ are even and write $\mathcal{E}(2h)$ for the set of even partitions of $2h$.
	
	Given a non-empty partition $\alpha$ with pairwise distinct parts, the \textit{shift symmetric partition} $ss[\alpha]$ is a partition of size $2|\alpha|$ such that for all $i\leq \ell(\alpha)$ we have $ss[\alpha]_i=\alpha_i + i, ss[\alpha]'_i=\alpha_i+i-1$ and implicitly $ss[\alpha]_{\ell(\alpha)+1}\leq\ell(\alpha)$. We write $\mathcal{S}(2h)$ for the set of shift symmetric partitions of $2h$.
	
	\begin{example}\label{Ex shift}
		The shift symmetric partitions $ss[(h)], ss[(h-1,1)]$ and $ss[(h-2,2)]$ are $(h+1,1^{h-1}), (h,3,1^{h-3})$ and $(h-1,4,2,1^{h-5})$, respectively (with $h\geq 1, h\geq 3$ and $h\geq 5$, respectively). 
	\end{example}
	
	The following are well-known rules for computing particular plethysms; see \cite[Exercise~I.8.6]{MacdonaldPolynomials95}.
	
	\begin{lemma}\label{Lemma even and shift rules}
		For all positive integers $h$ we have:
		\begin{enumerate}[label=\textnormal{(\roman*)}]
			\item $s_{(h)}\circ s_{(2)} = \sum_{\lambda\in\mathcal{E}(2h)}s_{\lambda}$,
			\item $s_{(h)}\circ s_{(1^2)} = \sum_{\lambda\in\mathcal{E}(2h)}s_{\lambda'}$,
			\item $s_{(1^h)}\circ s_{(2)} = \sum_{\lambda\in\mathcal{S}(2h)}s_{\lambda}$,
			\item $s_{(1^h)}\circ s_{(1^2)} = \sum_{\lambda\in\mathcal{S}(2h)}s_{\lambda'}$.
		\end{enumerate}
	\end{lemma}
	
	\subsubsection*{Domino rule}
	
	Lemma~\ref{Lemma even and shift rules} covers all plethysms $s_{\nu}\circ s_{\mu}$ with $|\mu|=2$ and $\nu$ linear. We now discuss a result due to Carr\'{e}--Leclerc \cite{CarreSplitting95} which applies for $|\nu|=2$ (and arbitrary $\mu$). By a \textit{horizontal domino} we mean a set of two horizontally adjacent boxes in a Young diagram, that is a set of the form $\left\lbrace (i,j), (i,j+1) \right\rbrace $. Similarly, we refer to sets of two vertically adjacent boxes in a Young diagram as \textit{vertical dominoes}. Define a \textit{domino $\mu$-diagram} to be a set of dominoes which partition the Young diagram of $(2\mu_1,2\mu_1,2\mu_2,2\mu_2,\dots,2\mu_{\ell(\mu)},2\mu_{\ell(\mu)})$. It is a fact that the number of horizontal dominoes of any domino $\mu$-diagram is even. If after dividing this number by $2$ we end up with an even number, we call such a diagram \textit{even}, otherwise we call it \textit{odd}.
	
	A \textit{domino $\mu$-tableau} is a map $T:\Delta\to \mathbb{N}$ where $\Delta$ is a domino $\mu$-diagram. We may omit $\mu$- when $\mu$ is implicit or irrelevant. An example of two domino $\mu$-tableaux with $\mu=(2,2,1)$ can be seen in Figure~\ref{Figure domino}. After extending the definition of even and odd domino diagrams to domino tableaux we see that the first domino tableau from Figure~\ref{Figure domino} is odd, while the second is even.
	
	\begin{figure}[h!]
		\begin{tikzpicture}[x=0.5cm, y=0.5cm]
		\node[h2domino] at (1,0.5){$4$};
		\node[h2domino] at (1,1.5){$3$};
		\node[h2domino] at (3,2.5){$4$};
		\node[h2domino] at (3,3.5){$3$};
		\node[h2domino] at (3,4.5){$2$};
		\node[h2domino] at (3,5.5){$1$};
		\node[v2domino] at (0.5,3){$2$};
		\node[v2domino] at (1.5,3){$2$};
		\node[v2domino] at (0.5,5){$1$};
		\node[v2domino] at (1.5,5){$1$};
		\end{tikzpicture}
		\qquad
		\begin{tikzpicture}[x=0.5cm, y=0.5cm]
		\node[v2domino] at (0.5,1){$5$};
		\node[v2domino] at (1.5,1){$5$};
		\node[v2domino] at (0.5,3){$3$};
		\node[v2domino] at (3.5,3){$6$};
		\node[v2domino] at (2.5,5){$2$};
		\node[v2domino] at (3.5,5){$2$};
		\node[h2domino] at (2,2.5){$4$};
		\node[h2domino] at (2,3.5){$3$};
		\node[h2domino] at (1,4.5){$3$};
		\node[h2domino] at (1,5.5){$1$};
		\end{tikzpicture}
		\caption{Two domino $\mu$-tableaux with $\mu=(2,2,1)$. There are $6$  horizontal dominoes in the first domino diagram and $4$ horizontal dominoes in the other one. After dividing by $2$ we get numbers $3$ and $2$, respectively. Thus the first domino tableau is odd and the second is even.}
		\label{Figure domino}
	\end{figure}
	
	We naturally extend the definition of a weight to domino tableaux; a domino tableau $T$ has \textit{weight} $W(T)=(T_1,T_2,\dots)$ where $T_i$ is the number of dominoes labelled $i$ in $T$. Similarly, we say that a domino tableau is \textit{semistandard} if for any two dominoes sharing a vertical edge, their entries are non-decreasing rightwards and for any two dominoes sharing a horizontal edge, their entries are increasing downwards.
	
	According to the definition, the first domino tableau from Figure~\ref{Figure domino} is semistandard with weight $(3,3,2,2)$, while the second one has weight equal to $(1,2,3,1,2,1)$ and it is not semistandard. We can make the second one into a semistandard domino tableau by replacing $3$ in the second row with $2$. Finally, there is also an analogue of a reading word (see Remark~\ref{Remark reading word} for a comment on our reading order). 
	
	\begin{definition}\label{Definition reading domino}
		The \textit{reading word} $R(T)$ of a domino tableau $T$ is the word of positive integers obtained by first moving numbers inside horizontal dominoes to their left boxes and then reading columns of the partition $(2\mu_1,2\mu_1,2\mu_2,2\mu_2,\dots,2\mu_{\ell(\mu)},2\mu_{\ell(\mu)})$ downwards, starting with the rightmost column and moving left.    
	\end{definition}
	
	\begin{example}\label{Example reading domino}
		The reading words of the domino tableaux from Figure~\ref{Figure domino} are $1234121234$ and $2623451335$.
	\end{example}
	
	\begin{example}\label{Example top filling}
		Given a partition $\mu$ and a sequence of non-negative integers $\varphi=(\varphi_1,\varphi_2,\dots)$ such that $\varphi_i\leq \mu_i$ for all $i$, we construct a domino $\mu$-diagram $\Delta_{\varphi}$ by firstly partitioning the Young diagram of the partition $(2\mu_1,2\mu_1,2\mu_2,2\mu_2,\dots,2\mu_{\ell(\mu)},2\mu_{\ell(\mu)})$ into sets of $2\times 2$ boxes and then halving each set such that in rows $2i-1$ and $2i$ the first $\varphi_i$ sets are halved vertically into two vertical dominoes and the rest are halved horizontally into two horizontal dominoes. For each domino $d$ of $\Delta_{\varphi}$ we let $T_{\varphi}(d)$ be  one plus the number of dominoes `above' $d$. Formally, if $d$ contains a box $(i,j)$, then $T_{\varphi}(d)$ counts the number of dominoes which contain a box $(i',j)$ with $i'\leq i$. See the first domino tableau in Figure~\ref{Figure domino} for this construction with $\mu=(2,2,1)$ and $\varphi=(1,1,0)$.
		
		From the definition of $T_{\varphi}$ we see it is semistandard and it is even if and only if $|\mu|-\sum_{i\geq 1}\varphi_i$ is even. If $\varphi$ is a partition, then $R(T_{\varphi})$ is a concatenation of words of the form $12\dots i12\dots i\dots j$ (obtained by reading a particular pair of columns). Thus it is latticed. 
	\end{example}
	
	\begin{theorem}[Domino rule]\label{Theorem domino rule}
		The plethysm coefficient $p(\mu,(2);\lambda)$ equals the number of the semistandard even domino $\mu$-tableaux $T$ with weight $\lambda$ and a latticed reading word. Replacing the word `even' with `odd' we obtain the coefficient $p(\mu,(1^2);\lambda)$.  
	\end{theorem}
	
	\begin{proof}
		After replacing reading words with their reverses and changing `prefix' to `suffix' in Definition~\ref{Definition latticed word} the theorem becomes \cite[Corollary~5.5]{CarreSplitting95}.
	\end{proof}
	
	\begin{example}\label{Example domino rule}
		Let $\mu=(2,1)$. Using the semistandard domino tableaux with a latticed reading word from Figure~\ref{Figure (2,1) example} we conclude that $s_{(2)}\circ s_{(2,1)}=s_{(4,2)}+s_{(3,2,1)} + s_{(3,1^3)} + s_{(2^3)}$ and $s_{(1^2)}\circ s_{(2,1)}=s_{(4,1^2)}+s_{(3^2)} + s_{(3,2,1)} + s_{(2^2,1^2)}$.
	\end{example}
	
	\begin{figure}[h!]
		\begin{tikzpicture}[x=0.5cm, y=0.5cm]
		\node[v2domino] at (0.5,3){$1$};
		\node[v2domino] at (1.5,3){$1$};
		\node[v2domino] at (2.5,3){$1$};
		\node[v2domino] at (3.5,3){$1$};
		\node[v2domino] at (0.5,1){$2$};
		\node[v2domino] at (1.5,1){$2$};
		\end{tikzpicture}
		\qquad
		\begin{tikzpicture}[x=0.5cm, y=0.5cm]
		\node[h2domino] at (1,1.5){$2$};
		\node[h2domino] at (1,0.5){$3$};
		\node[h2domino] at (3,3.5){$1$};
		\node[h2domino] at (3,2.5){$2$};
		\node[v2domino] at (0.5,3){$1$};
		\node[v2domino] at (1.5,3){$1$};
		\end{tikzpicture}
		\qquad
		\begin{tikzpicture}[x=0.5cm, y=0.5cm]
		\node[h2domino] at (1,1.5){$3$};
		\node[h2domino] at (1,0.5){$4$};
		\node[h2domino] at (3,3.5){$1$};
		\node[h2domino] at (3,2.5){$2$};
		\node[v2domino] at (0.5,3){$1$};
		\node[v2domino] at (1.5,3){$1$};
		\end{tikzpicture}
		\qquad
		\begin{tikzpicture}[x=0.5cm, y=0.5cm]
		\node[h2domino] at (1,3.5){$1$};
		\node[h2domino] at (1,2.5){$2$};
		\node[h2domino] at (3,3.5){$1$};
		\node[h2domino] at (3,2.5){$2$};
		\node[v2domino] at (0.5,1){$3$};
		\node[v2domino] at (1.5,1){$3$};
		\end{tikzpicture}
		\vskip 5pt
		\begin{tikzpicture}[x=0.5cm, y=0.5cm]
		\node[v2domino] at (0.5,3){$1$};
		\node[v2domino] at (1.5,3){$1$};
		\node[v2domino] at (2.5,3){$1$};
		\node[v2domino] at (3.5,3){$1$};
		\node[h2domino] at (1,0.5){$3$};
		\node[h2domino] at (1,1.5){$2$};
		\end{tikzpicture}
		\qquad
		\begin{tikzpicture}[x=0.5cm, y=0.5cm]
		\node[v2domino] at (1.5,1){$2$};
		\node[v2domino] at (0.5,1){$2$};
		\node[h2domino] at (3,3.5){$1$};
		\node[h2domino] at (3,2.5){$2$};
		\node[v2domino] at (0.5,3){$1$};
		\node[v2domino] at (1.5,3){$1$};
		\end{tikzpicture}
		\qquad
		\begin{tikzpicture}[x=0.5cm, y=0.5cm]
		\node[v2domino] at (1.5,1){$3$};
		\node[v2domino] at (0.5,1){$2$};
		\node[h2domino] at (3,3.5){$1$};
		\node[h2domino] at (3,2.5){$2$};
		\node[v2domino] at (0.5,3){$1$};
		\node[v2domino] at (1.5,3){$1$};
		\end{tikzpicture}
		\qquad
		\begin{tikzpicture}[x=0.5cm, y=0.5cm]
		\node[h2domino] at (1,3.5){$1$};
		\node[h2domino] at (1,2.5){$2$};
		\node[h2domino] at (3,3.5){$1$};
		\node[h2domino] at (3,2.5){$2$};
		\node[h2domino] at (1,0.5){$4$};
		\node[h2domino] at (1,1.5){$3$};
		\end{tikzpicture}
		\caption{The semistandard domino $\mu$-tableaux with $\mu=(2,1)$ and a latticed reading word from Example~\ref{Example domino rule}. The first line contains the even domino tableaux, while the second contains the odd domino tableaux.}
		\label{Figure (2,1) example}
	\end{figure} 
	
	Similarly to the Littlewood--Richardson rule, we obtain a simpler version in the case our underlying partition $\mu$ is rectangular.
	
	\begin{proposition}\label{Prop domino rectangles}
		Let $a$ and $b$ be positive integers and let $\mu=(a^b)$. The plethysm coefficient $p(\mu,(2);\lambda)$ is zero unless $\lambda$ is $(a,b)$-birectangular and $\lambda_1+\lambda_2+\dots+\lambda_b$ is even, in which case $p(\mu,(2);\lambda)=1$. We obtain the analogous result for $p(\mu,(1^2);\lambda)$ after replacing `even' with `odd'. 
	\end{proposition}
	
	\begin{proof}
		From Lemma~\ref{Lemma plethysms and char properties}(iii), we have $c(\mu,\mu;\lambda) = p(\mu,(2);\lambda) + p(\mu,(1^2);\lambda)$. By Proposition~\ref{Prop LR rectangles}(i), the left-hand side is one if $\lambda$ is $(a,b)$-birectangular and zero otherwise. Thus it is sufficient to show that for $\lambda$ an $(a,b)$-birectangular partition $p(\mu,(2);\lambda)\geq 1$ if $\lambda_1+\lambda_2+\dots+\lambda_b$ is even and $p(\mu,(1^2);\lambda)\geq 1$ otherwise.
		
		By Theorem~\ref{Theorem domino rule}, we need to present an appropriate semistandard domino $\mu$-tableaux. Given an $(a,b)$-birectangular partition $\lambda$, let $\varphi=(\lambda_1-a,\lambda_2-a,\dots, \lambda_b-a)$. Since $a\geq\lambda_1-a\geq \lambda_2-a\geq\dots\geq\lambda_b-a\geq 0$, we see that $\varphi$ is a partition and $\varphi\subseteq(a^b)$. Thus the domino $\mu$-tableau $T_{\varphi}$ from Example~\ref{Example top filling} is a well-defined semistandard domino tableau with a latticed reading word. An example of $T_{\varphi}$ is in Figure~\ref{Figure rectangle domino}.
		
		\begin{figure}[h!]
			\begin{tikzpicture}[x=0.5cm, y=0.5cm]
			\node[v2domino] at (0.5,3){$1$};
			\node[v2domino] at (1.5,3){$1$};
			\node[v2domino] at (2.5,3){$1$};
			\node[v2domino] at (3.5,3){$1$};
			\node[v2domino] at (0.5,1){$2$};
			\node[v2domino] at (1.5,1){$2$};
			\node[h2domino] at (5,3.5){$1$};
			\node[h2domino] at (5,2.5){$2$};
			\node[h2domino] at (3,1.5){$2$};
			\node[h2domino] at (3,0.5){$3$};
			\node[h2domino] at (5,1.5){$3$};
			\node[h2domino] at (5,0.5){$4$};
			\end{tikzpicture}
			\caption{A construction of $T_{\varphi}$ for $a=3, b=2$ and $\lambda=(5,4,2,1)$ (and implicitly $\varphi=(2,1)$).}
			\label{Figure rectangle domino}
		\end{figure}
		
		For $1\leq i\leq b$ observe that each column of $((2a)^{2b})$ either intersects one vertical domino with $i$ and no domino with $2b+1-i$ \emph{or} one horizontal domino with $i$ and one horizontal domino with $2b+1-i$. Consequently, given our choice of $\varphi$, there are $2a-\lambda_i$ horizontal dominoes with $i$ and $2a-\lambda_i$ horizontal dominoes with $2b+1-i$. Thus there are $\lambda_i$ dominoes with $i$ and $2a-\lambda_i=\lambda_{2b+1-i}$ dominoes with $2b+1-i$ (using that $\lambda$ is $(a,b)$-birectangular). Since the entries of $T_{\varphi}$ are at most $2b$ and $\ell(\lambda)\leq 2b$, we conclude that $W(T_{\varphi})=\lambda$. Finally, as mentioned in Example~\ref{Example top filling}, $T_{\varphi}$ is even if and only if $|\mu|-|\varphi|=2ab-(\lambda_1+\lambda_2+\dots+\lambda_b)$ is even, or equivalently, if and only if $\lambda_1+\lambda_2+\dots+\lambda_b$ is even, as required. 
	\end{proof}
	
	\subsubsection*{Multiplicity-free plethysms of Schur functions}
	
	The following classification obtained by Bessenrodt, Bowman and Paget  \cite[Theorem~1.1.]{BessenrodtBowmanPagetMF22} yields, using Theorem~\ref{Theorem frobenius}, the elementary irreducible induced-multiplicity-free characters of $S_m\wr S_h$. We omit the cases $m=1$ and $h=1$ when $S_1\wr S_h=S_h$ and $S_m\wr S_1=S_m$.
	
	\begin{theorem}\label{Theorem MF plethysms}
		Let $n=mh$ with $m,h\geq 2$. For partitions $\mu\vdash m$ and $\nu\vdash h$ the plethysm $s_{\nu}\circ s_{\mu}$ is multiplicity-free if and only if one of the following happens:
		\begin{enumerate}[label=\textnormal{(\roman*)}]
			\item $h=2$ and $\mu$ is rectangular, almost rectangular or a hook, 
			\item $m=2$ and $\nu$ is linear,
			\item $n=mh\leq 18$ and the pair $(\mu, \nu)$ belongs to an explicitly defined family of $49$ members.
		\end{enumerate} 
	\end{theorem} 
	
	A simple consequence of Theorem~\ref{Theorem MF plethysms} is as follows.
	
	\begin{lemma}\label{Lemma find two}
		If $\mu$ is an almost rectangular partition and $\lambda$ is any partition such that $c(\mu,\mu;\lambda) \geq 2$, then $p(\mu,(2);\lambda) = p(\mu,(1^2);\lambda) = 1$.
	\end{lemma}
	
	\begin{proof}
		By Theorem~\ref{Theorem MF plethysms}, we know that $p(\mu,(2);\lambda), p(\mu,(1^2);\lambda) \leq 1$. Using Lemma~\ref{Lemma plethysms and char properties}(iii), we obtain $2\leq c(\mu,\mu;\lambda)=p(\mu,(2);\lambda) + p(\mu,(1^2);\lambda)\leq 2$ forcing equalities everywhere.
	\end{proof}
	
	We end the subsection with an analogue of Lemma~\ref{Lemma Stembridge observation} for plethysms.
	
	\begin{theorem}\label{Theorem plethysms observation}
		Let $r$ be a non-negative integer and let $\lambda,\mu$ and $\nu$ be partitions. Writing $|\nu|=h$, we have:
		\begin{enumerate}[label=\textnormal{(\roman*)}]
			\item $p(\mu+(1^r), \nu;\lambda+(h^r))\geq p(\mu,\nu;\lambda)$,
			\item $p(\mu\sqcup(r), \nu^{\prime r};\lambda\sqcup(r^h))\geq p(\mu,\nu;\lambda)$.
		\end{enumerate}
	\end{theorem}
	
	\begin{proof}
		See \cite[Theorem~1.2]{deBoeckPagetWildonPlethysms21} for (i). By Lemma~\ref{Lemma plethysms and char properties}(ii) we obtain (ii).
	\end{proof}
	
	After iterating the result we can replace $(1^r)$ in (i) with any partition $\alpha$ as long as $(h^r)$ is replaced with $h\alpha=(h\alpha_1,h\alpha_2,\dots)$. Similarly, in the second part $(r)$ can be replaced with any partition $\alpha$ as long as $(r^h)$ is replaced with $\alpha^h=(\alpha_1^h,\alpha_2^h,\dots)$ and $\nu^{\prime r}$ by $\nu^{\prime |\alpha|}$. This motivates the following definitions.
	
	\begin{definition}\label{Definition l-multifunction}
		Let $f=(f_1,f_2,\dots,f_t)$ be a multifunction.
		\begin{enumerate}[label=\textnormal{(\roman*)}]
			\item For a non-negative integer $h$ we define $f^h$ to be the multifunction $(f_1^h,f_2^h,\dots,f_t^h)$ where $f_i^h$ stands for the $h$-fold composition of $f_i$.
			\item For a partition $\nu$ we define $\nu^{\prime f}$ as $\nu^{\prime r}$, where $r=|f'(\o)|$ and $f'$ is the multifunction obtained from $f$ by omitting the functions $f_i$ of the form $\lambda\to\lambda+\alpha$, keeping those of the form $\lambda\to\lambda\sqcup\alpha$.
		\end{enumerate}		
	\end{definition}
	
	\begin{corollary}\label{Corollary plethysms observation}
		Let $\lambda,\mu$ and $\nu$ be partitions and $f$ a multifunction. If $|\nu|=h$, then $p(f(\mu),\nu^{\prime f};f^h(\lambda))\geq p(\mu,\nu;\lambda)$.
	\end{corollary}
	
	\begin{proof}
		Iterate Theorem~\ref{Theorem plethysms observation}.
	\end{proof} 
	
	\subsection{Primitive subgroups and the number of involutions}\label{Sec primitive}
	
	A simple way to prove that a subgroup $G$ of a symmetric group is not multiplicity-free is by observing that it is sufficiently small, and thus all the induced characters of $G$ are too large. To obtain a concrete bound, let $a_n$ denote the number of involutions of $S_n$. We use the formula $a_n=\sum_{\lambda\vdash n}\deg(\chi^{\lambda})$, which is the Frobenius--Schur count of involutions, see for instance \cite[Corollary~7.13.9]{StanleyEnumerativeII99} for a proof.
	
	\begin{lemma}\label{Lemma order bound - inv}
		If $G\leq S_n$ is multiplicity-free, then $|G|\geq n!/a_n$.
	\end{lemma}
	
	\begin{proof}
		Let $\rho$ be an induced-multiplicity-free character of $G$. Then the degree of $\rho\ind^{S_n}$ is at most $a_n$, that is, $n!\deg(\rho)/|G|\leq a_n$. Rearranging this inequality and using $\deg(\rho)\geq 1$ yield the result.  
	\end{proof}
	
	To apply this result we bound $a_n$. For our purposes we establish only two straightforward bounds, which follow from a well-known recursion formula, avoiding a detailed analysis of the sequence $(a_n)$, such as in \cite[pp.~62--64]{KnuthArtofCPVol398}. 
	
	\begin{lemma}\label{Lemma involutions}
		Let $n$ be a non-negative integer.
		\begin{enumerate}[label=\textnormal{(\roman*)}]
			\item If $n\geq 2$, then $a_n = a_{n-1} + (n-1)a_{n-2}$.
			\item If $n$ is odd, then $a_n \leq n a_{n-1}$.
			\item If $n\geq 7$, then $a_n< n!/\left(2\lfloor n/2 \rfloor ! \right) $.
			\item If $n\geq 11$, then $a_n< n!/(2^{n-1})$. 
		\end{enumerate}
	\end{lemma}
	
	\begin{proof}
		In the first two parts we say `remaining permutation' to mean a permutation restricted to the set obvious from the context. In the last two parts the base cases of inductions on $n$ were verified using {\sc Magma} \cite{BosmaCannonPlayoustMagma97}.  
		\begin{enumerate}[label=\textnormal{(\roman*)}]
			\item An involution of $S_n$ either fixes $n$ and the remaining permutation is an involution of $S_{n-1}$ or it swaps $n$ with some $i\leq n-1$ and the remaining permutation is an involution of $S_{n-2}$.
			\item If $n$ is odd, then every involution in $S_n$ has a fixed point and the remaining permutation is an involution of $S_{n-1}$.
			\item We induct on $n$. The values $n= 7,8$ are checked directly. For even $n$ such that $n\geq 10$ by the induction hypothesis with $n-1$ and $n-2$ and (i)
			\begin{align*}
			a_n &= a_{n-1} + (n-1)a_{n-2} < \frac{(n-1)!}{2\left( \frac{n}{2} -1\right) !} +\frac{(n-1)!}{2\left( \frac{n}{2} -1\right) !}\\ &= \frac{2(n-1)!}{2\left( \frac{n}{2} -1\right) !} = \frac{n!}{2\lfloor \frac{n}{2} \rfloor !}. 
			\end{align*}
			Now if $n\geq 9$ is odd, then by (ii) and the induction hypothesis $a_n\leq na_{n-1}< n!/(2\lfloor (n-1)/2 \rfloor !) = n!/(2\lfloor n/2 \rfloor !)$.
			\item The inequality can be checked using direct calculations for $n=11,12$. The result now follows by induction on $n$ since by (i), we get for $n\geq 13$
			\[
			a_n = a_{n-1} + (n-1)a_{n-2} < \frac{3}{2^{n-2}}(n-1)! < \frac{n!}{2^{n-1}}.\qedhere
			\]
		\end{enumerate}
	\end{proof}
	
	\begin{corollary}\label{Cor Lemma order bound - exp}
		If $G\leq S_n$ is multiplicity-free and $n\geq 7$, then $|G|> 2\lfloor n/2\rfloor !$. Moreover if $n\geq 11$, then $|G|> 2^{n-1}$.
	\end{corollary}
	
	\begin{proof}
		Combine Lemma~\ref{Lemma order bound - inv} and Lemma~\ref{Lemma involutions}(iii) and (iv).
	\end{proof}
	
	The second lower bound is particularly useful when $G$ is a primitive subgroup of $S_n$. This is because of a result by Mar\'{o}ti \cite[Corollary~1.4]{MarotiPrimitive02}, which we present in a shorter form, omitting to list the groups in the exceptional family, which we call $\mathcal{F}$.
	
	\begin{theorem}\label{Theorem primitive groups small}
		Let $G$ be a proper primitive subgroup of $S_n$ not equal to $A_n$. Then the order of $G$ is less or equal to $2^{n-1}$ \emph{or} $n\leq 24$ and $G$ belongs to a small family $\mathcal{F}$ with $24$ members.
	\end{theorem}
	
	\begin{remark}\label{Remark primitive bound}
		The theorem uses the classification of finite simple groups and improves an earlier upper bound $4^n$ by Praeger and Saxl \cite{PraegerSaxlPrimitiveOrder80} used in \cite{SaxlMultiplicity-free81}.
	\end{remark}
	
	\begin{corollary}\label{Corollary primitive groups}
		The symmetric group $S_n$ does not have a proper multiplicity-free primitive subgroup $G$ different from $A_n$ unless $n\leq 12$.
	\end{corollary}
	
	\begin{proof}
		Using direct computations, it is easily verified that for any group $G$ from $\mathcal{F}$ in Theorem~\ref{Theorem primitive groups small} with $n\geq 13$ we have $|G|< n!/a_n$. This has been verified using {\sc Magma}. The result then follows from Corollary~\ref{Cor Lemma order bound - exp}, Theorem~\ref{Theorem primitive groups small} and Lemma~\ref{Lemma order bound - inv}. 
	\end{proof}
	
	\begin{remark}\label{Remark primitive groups}
		If $n\leq 12$, we can narrow down the list of potential proper multiplicity-free primitive subgroups $G\neq A_n$ using the inequality $|G|\geq n!/a_n$ from Lemma~\ref{Lemma order bound - inv}. We obtain the following list (verified by {\sc Magma}):
		\begin{enumerate}[label=\textnormal{(\roman*)}]
			\item $M_{11}$ with its action on $12$ points and $M_{12}$ for $n=12$,
			\item $M_{11}$ for $n=11$,
			\item $S_6$ acting primitively on the cosets of $S_3\wr S_2$, $\mathrm{PGL}_2(\mathbb{F}_9), M_{10}$ and $\mathrm{P}\Gamma\mathrm{L}_2(\mathbb{F}_9)$ for $n=10$,
			\item $\mathrm{A}\Gamma\mathrm{L}_1(\mathbb{F}_9), \mathrm{ASL}_2(\mathbb{F}_3), \mathrm{AGL}_2(\mathbb{F}_3), \mathrm{PSL}_2(\mathbb{F}_8)$ and $\mathrm{P}\Gamma\mathrm{L}_2(\mathbb{F}_8)$ for $n=9$,
			\item $\mathrm{AGL}_1(\mathbb{F}_8), \mathrm{A}\Gamma\mathrm{L}_1(\mathbb{F}_8), \mathrm{ASL}_3(\mathbb{F}_2), \mathrm{PSL}_2(\mathbb{F}_7)$ and $\mathrm{PGL}_2(\mathbb{F}_7)$ for $n=8$,
			\item $\mathrm{AGL}_1(\mathbb{F}_7)$ and $\mathrm{PSL}_3(\mathbb{F}_2)$ for $n=7$,
			\item $\mathrm{PSL}_2(\mathbb{F}_5)$ and $\mathrm{PGL}_2(\mathbb{F}_5)$ for $n=6$,
			\item $C_5, D_{10}$ and $\mathrm{AGL}_1(\mathbb{F}_5)$ for $n=5$.
		\end{enumerate}  
	\end{remark} 
	
	\subsection{Subgroups of direct and wreath products}\label{Sec subgroups}
	
	Let us start this section by recalling standard results about subgroups of index $2$ and their characters.
	
	\begin{lemma}\label{Lemma index two subgps}
		Let $G$ be a finite group. There is a bijection between the non-trivial real linear characters of $G$ and the index $2$ subgroups of $G$ given by $\eta\mapsto \ker \eta$.
	\end{lemma}
	
	\begin{lemma}\label{Lemma index two}
		Let $G$ be a finite group and $\rho$ an irreducible character of $G$. Fix a non-trivial real linear character $\eta$ of $G$ and let $N=\ker \eta$. 
		\begin{enumerate}[label=\textnormal{(\roman*)}]
			\item If $\rho\neq \rho\times\eta$, then $\rho\res_N$ is irreducible.
			\item If $\rho= \rho\times\eta$, then $\rho\res_N$ is a sum of two distinct irreducible characters.
			\item The characters arising from (i) and (ii) applied to all irreducible characters of $G$, up to tensoring with $\eta$, form the full list of irreducible characters of $N$ (without repetitions),
			\item Let $\kappa$ be an irreducible constituent of $\rho\res_N$. If $\rho\neq \rho\times\eta$, then $\kappa\ind^G=\rho + \rho\times\eta$, otherwise $\kappa\ind^G=\rho$.
		\end{enumerate}
	\end{lemma}
	
	\begin{proof}[Proof of Lemma~\ref{Lemma index two subgps} and Lemma~\ref{Lemma index two}]
		See the discussion about index two subgroups in \cite[\S20]{JamesLiebeckRepTheory01} for Lemma~\ref{Lemma index two subgps} and Lemma~\ref{Lemma index two}(i)--(iii). Frobenius reciprocity then establishes (iv).
	\end{proof}
	
	For groups $N\leq G$ with $|G:N|=2$ and an irreducible character $\rho$ of $G$ we write $\rho_N$ for an arbitrary irreducible constituent of $\rho\res_N$. Using Lemma~\ref{Lemma index two} $\rho_N$ is not uniquely defined if $\rho$ is as in (ii). However, $\rho_N\ind^G$ is well-defined according to (iv) which is sufficient for our purposes. Let us emphasise that in this paper $\rho_N$ is \emph{not} the restriction $\rho\res_N$. With our new notation introduced, we have the following useful corollary.
	
	\begin{corollary}\label{Cor index two}
		Suppose that $G\leq S_n$. Let $\eta$ be a non-trivial real linear character of $G$ and let $N=\ker \eta$.
		\begin{enumerate}[label=\textnormal{(\roman*)}]
			\item If $\rho$ is an irreducible character of $G$ such that $\rho_N$ is an induced-multiplicity-free character, then so are $\rho$ and $\rho\times\eta$.
			\item If $N$ is multiplicity-free, then there is an irreducible character $\rho$ of $G$ such that $\rho$ and $\rho\times\eta$ are induced-multiplicity-free.
		\end{enumerate}
	\end{corollary}
	
	\begin{proof}
		Lemma~\ref{Lemma index two}(iv) implies (i), which in turn implies (ii). 
	\end{proof}
	
	\begin{example}\label{Example index two}
		We can use formulae from Lemma~\ref{Lemma tensoring plethysms} to apply the last two statements in practice.
		\begin{enumerate}[label=\textnormal{(\roman*)}]
			\item Let $\eta=\sgn\boxtimes\left(\charwrnb{\mathbbm{1}}{\sgn}{5} \right) $ be a character of $G=S_3\times S_2\wr S_5$. Write $N=\ker \eta$ and $\rho=\chi^{(1^3)}\boxtimes\left( \charwr{\chi^{(2)}}{\chi^{(5)}}{5}\right) $. Corollary~\ref{Cor index two} says that if $\rho_N$ is induced-multiplicity-free, so are $\rho$ and $\rho\times\eta = \chi^{(3)}\boxtimes\left( \charwr{\chi^{(2)}}{\chi^{(1^5)}}{5}\right) $.
			\item Let $\eta=\charwrnb{\sgn}{\mathbbm{1}}{3}$ be a character of $G=S_{10}\wr S_3$. Then $N=\ker\eta=\left( S_{10}\wr S_3\right) \cap A_{30}$ and for $\rho=\left( \chi^{(6,3,1)}\boxtimes \left(\charwr{\chi^{(4^2,2)}}{\chi^{(2)}}{2} \right)\right)\Ind^G$, using Lemma~\ref{Lemma index two}(iv), we compute $\rho_N \ind^G$ to be
			\[\left( \chi^{(6,3,1)}\boxtimes \left(\charwr{\chi^{(4^2,2)}}{\chi^{(2)}}{2} \right)\right)\Ind^G + \left( \chi^{(3,2^2,1^3)}\boxtimes \left(\charwr{\chi^{(3^2,2^2)}}{\chi^{(2)}}{2} \right)\right) \Ind^G.\] 
		\end{enumerate}
	\end{example}
	
	We now consider subgroups of direct products. Let $\List{G}{t}$ be groups. For $1\leq i\leq t$ we denote by $\pi_i:G_1\times G_2\times \dots\times G_t\to G_i$ the natural projection. A subgroup $G$ of $G_1\times G_2\times \dots\times G_t$ is called a \textit{subdirect subgroup} if all restrictions $\pi_i|_{G}:G\to G_i$ are surjective. Note that for any $G\leq G_1\times G_2\times \dots\times G_t$ there are unique $H_1\leq G_1, H_2\leq G_2,\dots, H_t\leq G_t$ such that $G$ is a subdirect subgroup of $H_1\times H_2\times \dots\times H_t$. Namely, we take $H_i=\pi_i(G)$. For $t=2$ subdirect subgroups are easily characterised.
	
	\begin{proposition}\label{Prop subdirect subgps}
		Let $K$ and $L$ be finite groups and $G$ a subdirect subgroup of $K\times L$. There exist a group $H$ and surjections $\alpha: K\to H$ and $\beta:L\to H$ such that $G = \left\lbrace (x,y)\in K\times L : \alpha(x) = \beta(y) \right\rbrace $. Moreover, $|G||H|=|K||L|$.
	\end{proposition}
	
	\begin{proof}
		Write $\pi_K$ and $\pi_L$ for the projections from $K\times L$ to $K$ and $L$, respectively. Consider the kernel $M$ of $\pi_L|_G$ and let $N=\pi_K(M)$. Since $\pi_K|_G$ is surjective, $N$ is a normal subgroup of $K$. Let $H=K/N$ and $\alpha: K\to H$ be the surjective quotient map. We define $\beta:L\to H$ to be the composition of the isomorphism $L\to G/M$ arising from $\pi_L|_G$ and the surjection $G/M\to K/N=H$ arising from $\pi_K|_G$. Alternatively, $\beta(y)=\left\lbrace x\in K: (x,y)\in G \right\rbrace$ for $y\in L$. Thus $(x,y)\in K\times L$ lies in $G$ if and only if $xN=\beta(y)$ which can be written as $\alpha(x)=\beta(y)$, as required. The equality $|G||H|=|K||L|$ follows as for each $y\in L$ there are $|N|=|K|/|H|$ elements $x\in K$ such that $(x,y)\in G$.
	\end{proof}
	
	We deduce two important consequences of Proposition~\ref{Prop subdirect subgps} by taking either one or both of $K$ and $L$ to be a symmetric or an alternating group, and using the facts that alternating groups do not have any subgroups of index $2$ and apart from $A_4$ are simple.
	
	\begin{corollary}\label{Corollary subdirect subgps}
		Let $k\geq 5$ be an integer and let $L$ be a finite group.
		\begin{enumerate}[label=\textnormal{(\roman*)}]
			\item If $G$ is a proper subdirect subgroup of $S_k\times L$, then either $G$ has index $2$ in $S_k\times L$ or $G\cong L$ and $L$ surjects onto $S_k$.
			\item If $G$ is a proper subdirect subgroup of $A_k\times L$, then $G\cong L$ and $L$ surjects onto $A_k$.
		\end{enumerate} 
	\end{corollary}
	
	\begin{proof}
		Let $G$ be a proper subdirect product of $S_k\times L$. Since $k\geq 5$, there are $3$ normal subgroups of $S_n$, namely $S_n,A_n$ and the trivial subgroup. This leaves us with three choices of $\alpha:S_k\to H$ in Proposition~\ref{Prop subdirect subgps} applied to $G$. Either $\alpha$ is the trivial map and then $G=S_k\times L$, a contradiction. Or $\alpha$ is the sign character and the `moreover' part of Proposition~\ref{Prop subdirect subgps} shows that $G$ has index $2$ in $S_k\times L$. Or finally, $\alpha$ is an isomorphism. Then the projection $G\to L$ is also an isomorphism and $\beta:L\to S_k$ is the desired surjection. This establishes (i). The argument for (ii) is analogous.
	\end{proof}
	
	\begin{corollary}\label{Corollary MF subdirect subgpa}
		Let $k$ and $l$ be positive integers such that $k+l\geq 9$.
		\begin{enumerate}[label=\textnormal{(\roman*)}]
			\item If $G$ is a proper multiplicity-free subdirect subgroup of $S_k\times S_l$, then $G=\left(S_k\times S_l \right) \cap A_{k+l}$.
			\item The groups $A_k\times S_l$ and $ A_k\times A_l$ have no proper multiplicity-free subdirect subgroup.
		\end{enumerate}
	\end{corollary}
	
	\begin{proof}
		Allowing $S_k\times A_l$ in (ii), without loss of generality assume that $k\geq l$, meaning that $k\geq 5$. Let $K\times L$ be $S_k\times A_l$ or one of the three direct products in the statement and suppose that $G$ is its proper multiplicity-free subdirect subgroup. By Corollary~\ref{Corollary subdirect subgps}, either $K=S_k$ and $G$ has index $2$ in $K\times L$ or $G\cong L$. In the latter case $|G|\leq\lfloor n/2\rfloor!$, where $n=k+l$, contradicting Corollary~\ref{Cor Lemma order bound - exp}. In the former case it is easy to see, using, for instance, Lemma~\ref{Lemma index two subgps}, that $S_k\times S_l$ has at most one subdirect index $2$ subgroup, namely $\left(S_k\times S_l \right) \cap A_{k+l}$, while $S_k\times A_l$ has none.
	\end{proof} 
	
	Let us now move to subgroups of wreath products. For the wreath product $M\wr S_h$ we have the short exact sequence
	\[1\to \underbrace{M\times M\times \dots\times M}_{h \text{ times}}\to M\wr S_h \xrightarrow{\alpha} S_h\to 1. \]
	Given a subgroup $G\leq M\wr S_h$ we obtain the short exact sequence
	\begin{equation}\label{Eq ses}
	1\to G_B\to G \to \alpha(G)\to 1,
	\end{equation}
	where $G_B=G\cap \left( M\times M\times\dots\times M\right) $. Let us define $\List{G}{h}$ to be the images of $G_B$ under the $h$ projections to $M$.
	Therefore $G_B$ is a subdirect subgroup of $G_1\times G_2\times\dots\times G_h$.
	
	\begin{lemma}\label{Lemma conjugation lemma}
		Let $m$ and $h$ be positive integers. For a transitive subgroup $G$ of $S_m\wr S_h\leq S_{mh}$ the groups $\List{G}{h}$ are conjugate in $S_m$.
	\end{lemma}
	
	\begin{proof}
		Since $G$ is transitive, for any $i,j\leq h$ there is $g=(\List{g}{h}; \sigma)\in G$ such that $\sigma(i) =j$. Since $G_B=\prescript{g}{}{G_B}$, we get $G_j = \prescript{g_j}{}{G_i}$, as required.
	\end{proof}
	
	We can adapt this proof for $h=2$ to develop a concept of `subdirect products' for $S_m\wr S_2$.
	
	\begin{lemma}\label{Lemma wreath subgps}
		Let $m$ be a positive integer. If $G\leq S_m\wr S_2\leq S_{2m}$ is a transitive subgroup, then, up to conjugation in $S_m$, there is a unique $M\leq S_m$ such that, up to conjugation in $S_m\wr S_2$, we have $G\leq M\wr S_2$ and $G_B$ is a subdirect subgroup of $M\times M$.
	\end{lemma}
	
	\begin{proof}
		Since every conjugate of $G_B$ in $S_m\wr S_2$ is a subdirect subgroup of a unique direct product, namely the corresponding conjugate of $G_1\times G_2,$ we see that $M$ must be, up to conjugation in $S_m$, equal to $G_1$. We verify that this is a valid choice of $M$.
		
		Write $\tau\in S_2$ for the transposition. By transitivity, there is an element $g=(g_1,g_2;\tau)\in G$ and $G=G_B\cup gG_B$. Consequently, $g^2\in G_B$, and hence $g_1g_2\in G_1$. As in the proof of Lemma~\ref{Lemma conjugation lemma}, we have $\prescript{g_1}{}{G_2}=G_1$. Letting $h=(g_1,e;\tau)\in S_m\wr S_2$, where $e$ denotes the identity of $S_m$, we conclude that $\prescript{h}{}{G_B}$ is a subdirect subgroup of $\prescript{g_1}{}{G_2}\times G_1=G_1\times G_1$ and $\prescript{h}{}{G}=\prescript{h}{}{G_B}\cup \prescript{h}{}{g}\:\prescript{h}{}{G_B}=\prescript{h}{}{G_B}\cup (g_1g_2,e;\tau)\prescript{h}{}{G_B}\leq G_1\wr S_2$.
	\end{proof}
	
	In the setting of Lemma~\ref{Lemma wreath subgps}, we say that $G$ is a \textit{subwreath subgroup} of the group $M\wr S_2$. With this notation we establish an analogue of Corollary~\ref{Corollary MF subdirect subgpa} for wreath products $S_m\wr S_2$ and $A_m\wr S_2$. In the statement we use the notation $T_{m,h}$ introduced in Definition~\ref{Defn important subgroups} for $\ker\left(\charwrnb{\sgn}{\psi}{h} \right)\leq S_m\wr S_h$, where $\psi =\sgn$ if $m$ is even and $\psi = \mathbbm{1}$ otherwise. We also need the group $\ker\left(\charwrnb{\sgn}{(\psi\times \sgn)}{h} \right)$, which is just $(S_m\wr S_h)\cap A_{mh}$.
	
	\begin{corollary}\label{Cor subwreath lemma}
		Suppose that $m\geq 5$.
		\begin{enumerate}[label=\textnormal{(\roman*)}]
			\item If $G$ is a proper transitive multiplicity-free subwreath product of the group $S_m\wr S_2\leq S_{2m}$, then $G$ is either $\left( S_m\wr S_2\right) \cap A_{2m}$ or $T_{m,2}$.
			\item There is no proper transitive multiplicity-free subwreath product of the group $A_m\wr S_2\leq S_{2m}$.
		\end{enumerate}
	\end{corollary}
	
	\begin{proof}
		Suppose that $G$ is a proper transitive multiplicity-free subwreath product of $S_m\wr S_2$ or $A_m\wr S_2$. Since $G_B$ is a proper subdirect subgroup of either $S_m\times S_m$ or $A_m \times A_m$, we use Corollary~\ref{Corollary subdirect subgps} to deduce that $G_B$ is either an index $2$ subdirect subgroup of $S_m\times S_m $ or it has size $m!$ or $m!/2$. In the former case $G$ has index $2$ in $S_m\wr S_2$. Using Lemma~\ref{Lemma index two subgps}, we recover $\left( S_m\wr S_2\right) \cap A_{2m}, T_{m,2}$ and $S_m\times S_m$, however, only the first two are transitive. In the latter cases $|G|\leq 2m!$, which cannot happen according to Corollary~\ref{Cor Lemma order bound - exp} with $n=2m\geq 10$. 
	\end{proof}
	
\section{Induced-multiplicity-free characters of products of a symmetric and a transitive group}\label{Sec S_k times L}

We start this section by examining induced-multiplicity-free characters of $S_k\times S_m\wr S_h$ which have the form $\chi^{\lambda}\boxtimes \left( \charwr{\chi^{\mu}}{\chi^{\nu}}{h}\right) $, or equivalently, multiplicity-free symmetric functions $s_{\lambda}(s_{\nu}\circ s_{\mu})$. Clearly, if $s_{\lambda}(s_{\nu}\circ s_{\mu})$ is multiplicity-free, so is the plethysm $s_{\nu}\circ s_{\mu}$. Hence the choices of $\mu$ and $\nu$ are restricted by Theorem~\ref{Theorem MF plethysms}. We cover the three cases from Theorem~\ref{Theorem MF plethysms} individually. The notation $\mu\to\lambda$ introduced in Definition~\ref{Definition multifunction} is used throughout the section.  

\begin{proposition}\label{Prop sk times sporadic plethysm}
	Let $(\mu, \nu)$ be a pair of partitions belonging to the family of $49$ members in Theorem~\ref{Theorem MF plethysms}(iii). Then there is no non-empty partition $\lambda$ such that $s_{\lambda}\left(s_{\nu}\circ s_{\mu} \right) $ is multiplicity-free.
\end{proposition}

\begin{proof}
	Any non-empty partition $\lambda$ satisfies $(1)\to\lambda$. Hence, using Corollary~\ref{Corollary Stembridge order}, we only need to show that $s_{(1)}\left(s_{\nu}\circ s_{\mu} \right)$ is not multiplicity-free. This has been verified using {\sc Magma}.
\end{proof}

For the next case we use the following observation.

\begin{lemma}\label{Lemma product with sm wr 2s}
	Let $\alpha,\beta,\lambda$ and $\mu$ be partitions such that $\alpha\to\lambda$ and $\beta\to\mu$. If $s_{\alpha}\left(s_{\nu}\circ s_{\beta} \right) $ is not multiplicity-free for $\nu\in\left\lbrace (2),(1^2)\right\rbrace $, then also $s_{\lambda}\left(s_{\nu}\circ s_{\mu} \right) $ is not multiplicity-free for $\nu\in\left\lbrace (2),(1^2)\right\rbrace $.
\end{lemma}

\begin{proof}
	Pick a partition $\delta_{\nu}$ such that $\left\langle s_{\alpha}\left( s_{\nu}\circ s_{\beta}\right),s_{\delta_{\nu}}  \right\rangle \geq 2$ and multifunctions $f$ and $g$ such that $f(\alpha)=\lambda$ and $g(\beta)=\mu$. Using Corollary~\ref{Corollary Stembridge order} (twice), Lemma~\ref{Lemma injectivity} and Corollary~\ref{Corollary plethysms observation} we obtain
	\begin{align*}
	2&\leq\left\langle s_{\alpha}\left( s_{\nu}\circ s_{\beta}\right),s_{\delta_{\nu}}  \right\rangle =  \sum_{\gamma} p(\beta,\nu;\gamma) c(\alpha,\gamma;\delta_{\nu})\\
	&\leq \sum_{\gamma} p(\mu,\nu^{\prime g};g^2(\gamma)) c(\alpha,\gamma;\delta_{\nu})\leq \sum_{\gamma} p(\mu,\nu^{\prime g};g^2(\gamma)) c(\alpha,g^2(\gamma);g^2(\delta_{\nu}))\\
	&\leq \sum_{\gamma} p(\mu,\nu^{\prime g};g^2(\gamma))c(\lambda,g^2(\gamma);f(g^2(\delta_{\nu})))\\
	&\leq \sum_{\bar{\gamma}} p(\mu,\nu^{\prime g};\bar{\gamma})c(\lambda,\bar{\gamma};f(g^2(\delta_{\nu})))=\left\langle s_{\lambda}\left( s_{\nu^{\prime g}}\circ s_{\mu}\right),s_{f(g^2(\delta_{\nu}))}  \right\rangle.
	\end{align*}
	Since $(2)^{\prime g}$ and $(1^2)^{\prime g}$ are in some order $(2)$ and $(1^2)$, we have proved the statement.
\end{proof}

\begin{proposition}\label{Prop product with sm wr s2}
	Let $k$ and $m$ be positive integers with $m\geq 2$ and $\lambda,\mu$ and $\nu$ partitions of $k,m$ and $2$, respectively. The symmetric function $s_{\lambda}\left(s_{\nu}\circ s_{\mu} \right)$ is multiplicity-free if and only if $\lambda=(1^k)$ and $\mu=(m)$, \emph{or} $\lambda=(k)$ and $\mu=(1^m)$, \emph{or} $k=1$ and $\mu$ is rectangular, \emph{or} $m=2$, $\lambda$ is rectangular and $\nu=(1^2)$.
\end{proposition}

\begin{proof}
	Using Theorem~\ref{Theorem MF plethysms} and Proposition~\ref{Prop sk times sporadic plethysm}, if $s_{\lambda}\left(s_{\nu}\circ s_{\mu} \right)$ is multiplicity-free, then $\mu$ is rectangular, almost rectangular or a hook. We prove the `only if' direction using Lemma~\ref{Lemma product with sm wr 2s} and Example~\ref{Example multiorder} which allow us to use particular $\lambda$ and $\mu$ (which are given by $\alpha$, respectively, $\beta$ in Lemma~\ref{Lemma product with sm wr 2s}) to show that $s_{\lambda}\left(s_{\nu}\circ s_{\mu} \right)$ is not multiplicity-free for more general $\lambda$ and $\mu$. We need to show that $s_{\lambda}\left(s_{\nu}\circ s_{\mu} \right)$ is not multiplicity-free when: 
	
	\begin{enumerate}[label=\textnormal{(\roman*)}]
		\item $\mu$ is almost rectangular or a hook. It is sufficient to take $\lambda=(1)$ and $\mu=(2,1)$. From Example~\ref{Example domino rule} and the inductive branching rule we compute $\left\langle s_{(1)}\left( s_{(2)}\circ s_{(2,1)}\right),s_{(4,2,1)}  \right\rangle=\left\langle s_{(1)}\left( s_{(1^2)}\circ s_{(2,1)}\right),s_{(4,2,1)}  \right\rangle= 2$, as desired.
		
		\item $\mu$ is properly rectangular and $k\geq 2$. It is sufficient to let $\lambda=(2)$ and $\mu=(2^2)$ as applying $\omega$ gives us $\lambda=(1^2)$. Direct calculations using Proposition~\ref{Prop domino rectangles} give $\left\langle s_{(2)}\left( s_{(2)}\circ s_{(2^2)}\right),s_{(4^2,2)}  \right\rangle= \left\langle s_{(2)}\left( s_{(1^2)}\circ s_{(2^2)}\right),s_{(4,3,2,1)}  \right\rangle = 2$.
		
		\item Up to $\omega$, the partition $\mu$ is a row partition with $m\geq 3$ and $\lambda$ is \emph{not} a column partition. It is sufficient to take $\lambda=(2)$ and $\mu=(3)$, for which we get $\left\langle s_{(2)}\left( s_{(2)}\circ s_{(3)}\right),s_{(6,2)}  \right\rangle=\left\langle s_{(2)}\left( s_{(1^2)}\circ s_{(3)}\right),s_{(5,3)}  \right\rangle = 2$.
		
		\item $\nu=(2)$ and, up to $\omega$, $\mu=(2)$ and $\lambda$ is \emph{not} a column partition. Using Lemma~\ref{Lemma Stembridge observation} it is sufficient to let $\lambda=(2)$. Since $\left\langle s_{(2)}\left( s_{(2)}\circ s_{(2)}\right),s_{(4,2)}  \right\rangle = 2$, we are done.
		
		\item The partition $\nu=(1^2)$, $m=2$ and $\lambda$ is not rectangular. Up to $\omega$, $s_{\nu}\circ s_{\mu}=s_{(1^2)}\circ s_{(2)}=s_{(3,1)}$, and thus $s_{\lambda}\left(s_{\nu}\circ s_{\mu} \right)$ is not multiplicity-free by Theorem~\ref{Theorem Stembridge}. 
	\end{enumerate}  
	
	Let us now move to the `if' direction. We need to check that $s_{\lambda}\left(s_{\nu}\circ s_{\mu} \right)$ is multiplicity-free when:
	
	\begin{enumerate}[label=\textnormal{(\roman*)}]
		\item $\lambda=(1^k)$ and $\mu=(m)$. We fix $\nu\vdash 2$ and verify that Lemma~\ref{Lemma Pieri application}(ii) applies. Theorem~\ref{Theorem MF plethysms}(i) shows that $s_{\nu}\circ s_{(m)}$ is multiplicity-free. If $\gamma\neq\delta$ are partitions such that $p((m),\nu;\gamma)= p((m),\nu;\delta)=1$, then by Proposition~\ref{Prop domino rectangles} (and Remark~\ref{Remark birectangular}) the parts $\gamma_1$ and $\delta_1$ are distinct but their parities agree. Thus $|\gamma_1- \delta_1|\geq 2$, as needed. The case $\lambda=(k)$ and $\mu=(1^m)$ is then obtained by applying $\omega$. 
		
		\item $k=1$ and $\mu$ is rectangular, say $(a^b)$. We fix $\nu\vdash 2$. Since $s_{\nu}\circ s_{(a^b)}$ is multiplicity-free, we can check Lemma~\ref{Lemma Pieri application}(i) with $k=1$. Pick partitions $\gamma\neq\delta$ with $p((a^b),\nu;\gamma) = p((a^b),\nu;\delta)=1$. By Proposition~\ref{Prop domino rectangles} (and Remark~\ref{Remark birectangular}) we have $\ell(\gamma),\ell(\delta)\leq 2b$ and for all $1\leq i\leq b$ also $|\gamma_i-\delta_i| = |\gamma_{2b+1-i} - \delta_{2b+1-i}|$. Consequently, $\sum_{i\leq b} |\gamma_i-\delta_i| = \sum_{i> b} |\gamma_i-\delta_i|$. Proposition~\ref{Prop domino rectangles} implies that these sums are even, and thus at least $2$ since $\gamma\neq \delta$. This gives $|\gamma-\delta|\geq 4$, as needed.
		
		\item $m=2$, $\lambda$ is rectangular and $\nu=(1^2)$. We can use Theorem~\ref{Theorem Stembridge}(iii) and the computation from (v) of the `only if' part. \qedhere  
	\end{enumerate} 
\end{proof}

Similarly, we obtain results for symmetric functions of the form $s_{\lambda}\left(s_{\nu}\circ s_{\mu} \right)$ with $\mu\vdash 2$. We assume $\mu=(2)$ as one can apply $\omega$ to get to $\mu=(1^2)$. Note that the case (i) in the proof of Proposition~\ref{Prop product with s2 wr sm} also follows from \cite[Lemma~2]{InglisRichardsonSaxlModel90}.

\begin{proposition}\label{Prop product with s2 wr sm}
	Let $k$ and $h$ be positive integers with $h\geq 3$ and $\lambda$ and $\nu$ partitions of $k$ and $h$, respectively. The symmetric function $s_{\lambda}\left(s_{\nu}\circ s_{(2)} \right)$ is multiplicity-free if and only $\lambda=(1^k)$ and \emph{either} $\nu=(h)$ \emph{or} $k=1$ and $\nu=(1^h)$ \emph{or} $\nu=(1^3)$ \emph{or} $\nu=(1^4)$.
\end{proposition}

\begin{proof}
	We start with the `only if' direction. Theorem~\ref{Theorem MF plethysms} and Proposition~\ref{Prop sk times sporadic plethysm} imply that if $s_{\lambda}\left(s_{\nu}\circ s_{(2)} \right)$ is multiplicity-free, then $\nu$ is linear.
	We claim that then also $\lambda=(1^k)$. Otherwise, $(2)\to\lambda$ and by Corollary~\ref{Corollary Stembridge order} we only need to show that $s_{\lambda}\left(s_{\nu}\circ s_{(2)} \right)$ is not multiplicity-free for $\lambda=(2)$. Since $(2h)$ and $(2h-2,2)$ are even and $(h+1,1^{h-1})$ and $(h,3,1^{h-3})$ are shift symmetric, we conclude using Lemma~\ref{Lemma even and shift rules} that $\left\langle s_{(2)}\left( s_{(h)}\circ s_{(2)}\right),s_{(2h,2)}  \right\rangle\geq 2$ and $\left\langle s_{(2)}\left( s_{(1^h)}\circ s_{(2)}\right),s_{(h+1,3,1^{h-2})}  \right\rangle \geq 2$, as needed.
	
	It remains to show that $s_{\lambda}\left(s_{\nu}\circ s_{(2)} \right)$ is not multiplicity-free when $\lambda=(1^k)$ and $\nu=(1^h)$ with $k\geq 2$ and $h\geq 5$. As before, it is sufficient to take $k=2$. Since $h\geq 5$, the partitions $(h-1,4,2,1^{h-5})$ and $(h,3,1^{h-3})$ are shift symmetric; thus $\left\langle s_{(1^2)}\left( s_{(1^h)}\circ s_{(2)}\right),s_{(h,4,2,1^{h-4})}  \right\rangle \geq 2$.
	
	For the `if' direction we need to verify that $s_{\lambda}\left(s_{\nu}\circ s_{(2)} \right)$ is multiplicity-free for $\lambda=(1^k)$ and four choices of $k$ and $\nu$.
	
	\begin{enumerate}[label=\textnormal{(\roman*)}]
		\item For $\nu=(h)$, by Lemma~\ref{Lemma even and shift rules}(i), the plethysm $s_{(h)}\circ s_{(2)}$ is multiplicity-free and its constituents are labelled by even partitions of $2h$. For two such partitions $\gamma\neq\delta$ there is $i\geq 1$ such that $|\gamma_i-\delta_i|\geq 2$ and Lemma~\ref{Lemma Pieri application}(ii) applies.
		
		\item For $k=1$ and $\nu=(1^h)$ we apply Lemma~\ref{Lemma Pieri application}(i). By Lemma~\ref{Lemma even and shift rules}(iii), $s_{(1^h)}\circ s_{(2)}$ is multiplicity-free, and if $\gamma\neq\delta$ are two partitions labelling constituents of $s_{(1^h)}\circ s_{(2)}$, there are distinct-parts partitions $\alpha\neq\beta$ of size $h$ such that $\gamma=ss[\alpha]$ and $\delta=ss[\beta]$. Without loss of generality $l=\ell(\alpha)\leq\ell(\beta)$. Let $1\leq C=\sum_{i\leq l}|\alpha_i-\beta_i|$. From the definition of shift symmetric partitions, the same equality holds when $\alpha$ and $\beta$ are replaced with $\gamma$ and $\delta$ or with $\gamma'$ and $\delta'$. From the same definition also $\gamma_l,\delta_l\geq l$; thus $|\delta-\gamma|\geq \sum_{i\leq l} |\delta_i-\gamma_i| + \sum_{i\leq l}|\delta'_i-\gamma'_i|=2C$. Hence we are done if $C>1$.
		
		If $C=1$, we have $\ell(\beta)=l+1$ and $\beta_{l+1}=1$. Then $|\delta-\gamma|\geq \sum_{i\leq l}|\delta_i-\gamma_i| + |\delta_{l+1}-\gamma_{l+1}| =1 +|\delta_{l+1}-\gamma_{l+1}|\geq 3$, as required, since from the definition $\delta_{l+1}-\gamma_{l+1}=l+2-\gamma_{l+1}\geq 2$.
		
		\item For $\nu=(1^3)$ we compute $s_{(1^3)}\circ s_{(2)}=s_{(4,1^2)}+s_{(3,3)}$ and apply Lemma~\ref{Lemma Pieri application}(ii).
		
		\item For $\nu=(1^4)$ we compute $s_{(1^4)}\circ s_{(2)}=s_{(5,1^3)}+s_{(4,3,1)}$ and apply Lemma~\ref{Lemma Pieri application}(ii). \qedhere
	\end{enumerate}	
\end{proof}

To obtain the final ingredient for our main theorem of this section we replace wreath products with primitive subgroups.

\begin{proposition}\label{Prop sk times primitive}
	Let $k$ and $l$ be positive integers with $k+l\geq 18$ and let $\lambda$ be a partition of $k$. Suppose that $L$ is a proper primitive subgroup of $S_l$ not equal to $A_l$ and $\rho$ is an irreducible character of $L$. The character $\chi^{\lambda}\boxtimes\rho$ of $S_k\times L$ is induced-multiplicity-free if and only if one of the following happens:
	\begin{enumerate}[label=\textnormal{(\roman*)}]
		\item $L=\mathrm{P}\Gamma\mathrm{L}_2(\mathbb{F}_8)$, $\rho$ is the inflated character of any of the three linear characters of the quotient $\Gal(\mathbb{F}_8/\mathbb{F}_2)\cong C_3$ of $\mathrm{P}\Gamma\mathrm{L}_2(\mathbb{F}_8)$ and $\lambda$ is linear,
		\item $L=\mathrm{ASL}_3(\mathbb{F}_2)$, $\rho$ is the inflated character of any of the two conjugate Weil characters of degree $3$ of $\mathrm{SL}_3(\mathbb{F}_2)$ and $\lambda$ is linear,
		\item $L=\mathrm{PGL}_2(\mathbb{F}_5)$, $\rho$ is the trivial character and $\lambda$ is rectangular, column-near rectangular or a column-unbalanced fat hook, \emph{or} $\rho$ is the restriction of the sign of $S_6$ and $\lambda$ is rectangular, row-near rectangular or a row-unbalanced fat hook,
		\item $L=\mathrm{AGL}_1(\mathbb{F}_5)$, $\rho$ is one of the two linear real characters of $L$ and $\lambda$ is rectangular, \emph{or} $\rho$ belongs to the pair of complex conjugate linear characters of $L$ and $\lambda$ is linear or $2$-rectangular.
	\end{enumerate}
\end{proposition}

\begin{proof}
	Most of the work is in proving the `only if' direction. We start by narrowing down the list of potential multiplicity-free groups $S_k\times L$ (with $L$ as in the statement) and their irreducible induced-multiplicity-free characters. Throughout we use Corollary~\ref{Corollary Stembridge order} and Example~\ref{Example multiorder}(i) to show that certain characters are not induced-multiplicity-free.
	
	If $\chi^{\lambda}\boxtimes\rho$ is induced-multiplicity-free, then so  is $\rho$ and Corollary~\ref{Corollary primitive groups} shows $l\leq 12$ and Remark~\ref{Remark primitive groups} lists potential groups $L$. Given $\rho$, observe that if $t$ is an integer such that $(t-1)^2 + l\leq 17$ and $\chi^{(t)}\boxtimes \rho$ and $\chi^{(1^t)}\boxtimes \rho$ are not induced-multiplicity-free, then neither is $\chi^{\lambda}\boxtimes\rho$ for any partition $\lambda$. Indeed, since $n\geq 18$ (where $n=k+l$) we have $|\lambda|>(t-1)^2$, and thus $\lambda_1\geq t$ or $\ell(\lambda)\geq t$, meaning that $(t)\to \lambda$ or $(1^t)\to\lambda$, as needed.
	
	Using this observation, one can computationally eliminate all choices of $L$ and $\rho$ apart from those in (i)--(iv), together with $L=M_{12}$ and $\rho$ the trivial character of $M_{12}$ (this has been done using {\sc Magma}). In the last case we have
	\[ \rho\ind^{S_{12}} = \chi^{(12)} + \chi^{(6^2)} + \chi^{(6,2^3)} + \chi^{(4^3)} + \chi^{(4^2,1^4)} + \chi^{(3^4)} + \chi^{(2^6)} + \chi^{(1^{12})}. \]
	Writing $x_M$ for the corresponding symmetric function of degree $12$, one computes that $\left\langle s_{(6)}x_M,s_{(12,6)}\right\rangle, \left\langle s_{(1^6)}x_M,s_{(2^6,1^6)}\right\rangle, \left\langle  s_{(2^2)}x_M,s_{(6,4^2,2)}\right\rangle \geq 2$ using the constituents $s_{(12)}$ and $s_{(6^2)}$, $s_{(1^{12})}$ and $s_{(2^6)}$, and finally $s_{(6,2^3)}$ and $s_{(4^3)}$ of $x_M$. Hence $(2^2),(6),(1^6)\not\to\lambda$, meaning that $\lambda$ is linear and $k\leq 5$. This contradicts $n\geq 18$. Now it only remains to find for which $\lambda$ the character $\chi^{\lambda}\boxtimes\rho$ is induced-multiplicity-free, where $\rho$ (and $L$) are from (i)--(iv). 
	
	\begin{enumerate}[label=\textnormal{(\roman*)}]
		\item For any of the three choices of $\rho$ computations show there is a non-rectangular partition which labels a constituent of $\rho\ind^{S_9}$. Thus $\lambda$ must be rectangular by Theorem~\ref{Theorem Stembridge}. Direct computations with $\lambda=(2,2)$ (verified by {\sc Magma}) then force $\lambda$ to be linear. For $\lambda=(1^k)$ and any of the three possible $\rho$, direct calculations and Lemma~\ref{Lemma Pieri application}(ii) show that we obtain an induced-multiplicity-free character. The case $\lambda=(k)$ follows by applying $\omega$.
		\item Since the induced character $\rho\ind^{S_8}$ is in both cases $\chi^{(4,2,1^2)}$, we can apply Theorem~\ref{Theorem Stembridge}.
		\item Using $\omega$ if necessary, assume $\rho$ is the trivial character, in which case $\rho\ind^{S_6}=\chi^{(6)} + \chi^{(2^3)}$. By Theorem~\ref{Theorem Stembridge}, $s_{\lambda}s_{(6)}$ and $s_{\lambda}s_{(2^3)}$ are multiplicity-free if and only if $\lambda$ is rectangular or a fat hook. Fix such $\lambda$ and suppose there is a common constituent $s_{\mu}$.
		
		If $\lambda$ is rectangular, then, finding $\mu$ from the Littlewood--Richardson rule, the skew partition $\mu/\lambda$ has at most two rows as $s_{\mu}$ is a constituent of $s_{\lambda}s_{(6)}$, and has at least three rows as $s_{\mu}$ is a constituent of $s_{\lambda}s_{(2^3)}$. Hence no such $\mu$ exists and $\lambda$ can be rectangular. If $\lambda$ is a fat hook $(a^b,c^d)$ and $a-c\geq 2$ and $c\geq 2$, then we can take $\mu=(a+2,a^{b-1},c+2,c^{d-1},2)$, which is a constituent of $s_{\lambda}s_{(6)}$ by Young's rule and of $s_{\lambda}s_{(2^3)}$ using the $\mu/\lambda$-tableau with reading word $112233$. Finally, for $a-c=1$ or $c=1$, that is for $\lambda$ column-near rectangular or a column-unbalanced fat hook, by Young's rule, $\mu/\lambda$ has at most three rows one of which has length at most $1$ since $s_{\mu}$ is a constituent of $s_{\lambda}s_{(6)}$. One can easily see that there is no such semistandard $\mu/\lambda$-tableau with weight $(2^3)$ and a latticed reading word, establishing that $s_{\mu}$ cannot be simultaneously a constituent of $s_{\lambda}s_{(2^3)}$. Thus (iii) is established. 
		
		\item When $\rho$ is one of the complex linear characters, then $\rho\ind^{S_5}=\chi^{(3,1^2)}$ and Theorem~\ref{Theorem Stembridge} yields the result. For the other two characters, working up to $\omega$, we have $\rho\ind^{S_5}=\chi^{(5)}+\chi^{(2^2,1)}$. The latter summand, together with Theorem~\ref{Theorem Stembridge}, forces $\lambda$ to be rectangular. Such $\lambda$ give induced-multiplicity-free characters by the argument used for rectangular $\lambda$ in (iii). \qedhere  
	\end{enumerate}     
\end{proof}

We can now combine all the previous results and obtain the main theorem of this section.

\begin{theorem}\label{Theorem direct product k geq 2}
	Suppose that $k$ and $l$ are positive integers such that $k\geq 2$ and $k+l\geq 18$. Let $L < S_l$ be a subgroup not equal to $A_l$. For a partition $\lambda$ of $k$ and an irreducible character $\rho$ of $L$ the character $\chi^{\lambda}\boxtimes\rho$ is induced-multiplicity-free if and only if it belongs to the following list (where $\mu,\nu\vdash 2$ and $m\geq 2$ and $h\geq 3$ are arbitrary unless specified differently):
	\begin{enumerate}[label=\textnormal{(\roman*)}]
		\item $L=S_m\wr S_2$, $\rho = \charwr{\chi^{(m)}}{\chi^{\nu}}{2}$ and $\lambda=(1^k)$, \emph{or} $\rho= \charwr{\chi^{(1^m)}}{\chi^{\nu}}{2}$ and $\lambda=(k)$,
		\item $L=S_2\wr S_2$, $\rho = \charwr{\chi^{\mu}}{\chi^{(1^2)}}{2}$ and $\lambda$ is rectangular,
		\item $L=S_2\wr S_h$, $\rho= \charwr{\chi^{(2)}}{\chi^{(h)}}{h}$ and $\lambda=(1^k)$, \emph{or } $\rho = \charwr{\chi^{(1^2)}}{\chi^{(h)}}{h}$ and $\lambda=(k)$,
		\item $L=S_2\wr S_h$ with $h\in\left\lbrace 3,4\right\rbrace $, $\rho =\charwr{\chi^{(2)}}{\chi^{(1^h)}}{h}$ and $\lambda=(1^k)$, \emph{or} $\rho = \charwr{\chi^{(1^2)}}{\chi^{(1^h)}}{h}$ and $\lambda=(k)$,
		\item $L$ is a primitive subgroup and $L,\rho$ and $\lambda$ belong to the list in Proposition~\ref{Prop sk times primitive}.
	\end{enumerate}
\end{theorem}

\begin{proof}
	Suppose $S_k\times L$ is multiplicity-free, then by Lemma~\ref{Lemma sanity lemma} the group $L$ must be a transitive subgroup. If it is primitive, we use Proposition~\ref{Prop sk times primitive}. If it is imprimitive transitive, there are $m,h\geq 2$ such that $L\leq S_m\wr S_h$. Therefore $S_k\times S_m\wr S_h$ has an irreducible induced-multiplicity-free character $\chi^{\lambda}\boxtimes\rho'$. By Lemma~\ref{Lemma sanity lemma}, we know $\rho'$ is an elementary irreducible character of $S_m\wr S_h$ defined in Definition~\ref{Definition elementary irr reps}. Combining Theorem~\ref{Theorem MF plethysms} and Proposition~\ref{Prop sk times sporadic plethysm}, we conclude that $m=2$ or $h=2$. In the case $L=S_m\wr S_h$ we obtain the characters from Proposition~\ref{Prop product with sm wr s2} and Proposition~\ref{Prop product with s2 wr sm} (with the aid of $\omega$ in the latter), yielding (i)--(iv).
	
	It remains to show that $L$ cannot be a proper subgroup of $S_m\wr S_h$. Since all the characters $\rho$ of the wreath products in (i)--(iv) have degree $1$, it is sufficient to show that for any $\lambda\vdash k$ all the induced-multiplicity-free characters of $S_k\times S_m\wr S_h$ of the form $\chi^{\lambda}\boxtimes \rho'$ have $\rho'$ irreducible. This is clear if $m=2$ and $h\geq 5$ from (iii). From (iii) and (iv), for $h=3$ and $h=4$ we only need to take, up to $\omega$, $\lambda=(1^k)$ and, in the language of symmetric functions, we need to show that $s_{(1^k)}(s_{(h)}\circ s_{(2)})$ and $s_{(1^k)}(s_{(1^h)}\circ s_{(2)})$ share a constituent. For $k=1$ and $h=3$ we take $s_{(4,2,1)}$ and for $k=2$ and $h=4$ an appropriate constituent is $s_{(5,2^2,1)}$. We are done by Corollary~\ref{Corollary Stembridge order}.
	
	We are left with $h=2$. Note that $s_{\lambda}(s_{(2)}\circ s_{\mu} + s_{(1^2)}\circ s_{\mu})$ is never multiplicity-free since, by Lemma~\ref{Lemma plethysms and char properties}(iii), it equals $s_{\lambda}s_{\mu}^2$ and Lemma~\ref{Lemma sanity lemma} applies. Based on (i), we therefore need $m=h=2$. Using (i) and (ii) we are left, up to $\omega$, with two options: either $\lambda$ is rectangular and we need to show that $s_{\lambda}(s_{(1^2)}\circ s_{(2)})$ and $s_{\lambda}(s_{(1^2)}\circ s_{(1^2)})$ share a constituent or $\lambda$ is a column partition and we need to show that $s_{\lambda}(s_{(2)}\circ s_{(2)})$ and $s_{\lambda}(s_{(1^2)}\circ s_{(1^2)})$ share a constituent. In both cases we can take $\lambda=(1)$ and use $s_{(3,1^2)}$, respectively, $s_{(2^2,1)}$ as the common constituent. 
\end{proof}

While we assume $k\geq 2$ in Theorem~\ref{Theorem direct product k geq 2}, the first paragraph of its proof can be adapted for $k=1$, as we now show.

\begin{proposition}\label{Prop direct product k eq 1}
	Suppose that $l\geq 13$ is a positive integer and let $n=l+1$. Let $L < S_l$ be a subgroup not equal to $A_l$ and embed $L$ in $S_n$ by the natural inclusion of $S_{n-1}$ in $S_n$. If $L$ is multiplicity-free, then $l$ is even \emph{and} $L\leq S_m\wr S_2$ or $L\leq S_2\wr S_h$. Moreover,
	\begin{enumerate}[label=\textnormal{(\roman*)}]
		\item if $L=S_m\wr S_2$, then the irreducible induced-multiplicity-free characters of $L$ are $\charwr{\chi^{\mu}}{\chi^{\nu}}{2}$ with $\mu$ rectangular and $\nu\vdash 2$;
		\item if $L=S_2\wr S_h$, then the irreducible induced-multiplicity-free characters of $L$ are $\charwr{\chi^{\mu}}{\chi^{\nu}}{h}$ with $\mu$ and $\nu$ linear.
	\end{enumerate}
\end{proposition}

\begin{proof}
	The subgroup $L\leq S_l$ cannot be non-transitive or primitive using Lemma~\ref{Lemma sanity lemma} and Corollary~\ref{Corollary primitive groups}. Thus $L\leq S_m\wr S_h$ for $m,h\geq 2$. Since any irreducible induced-multiplicity-free character of $S_m\wr S_h$, embedded in $S_{mh+1}$, is elementary irreducible by Lemma~\ref{Lemma sanity lemma}, the rest of the statement follows from Theorem~\ref{Theorem MF plethysms}, Proposition~\ref{Prop sk times sporadic plethysm}, Proposition~\ref{Prop product with sm wr s2} and Proposition~\ref{Prop product with s2 wr sm}.
\end{proof}

\section{Consequences of Theorem~\ref{Theorem direct product k geq 2}}\label{Sec consequences}

Given Theorem~\ref{Theorem direct product k geq 2}, we can deduce a more general result where $S_k$ is replaced with an arbitrary subgroup of $S_k$. To prove this generalised classification, we use the following technical lemma to find the irreducible induced-multiplicity-free characters of $A_k\times\mathrm{PGL}_2(\mathbb{F}_5)$. In the lemma and the rest of the paper we write $x_P$ for the symmetric function $s_{(6)}+s_{(2^3)}$ corresponding to the trivial character of $\mathrm{PGL}_2(\mathbb{F}_5)$ induced to $S_6$.

\begin{lemma}\label{Lemma ak times pgl(2,5)}
	Let $x_P=s_{(6)}+s_{(2^3)}$ and let $\lambda$ be a hook or an almost rectangular partition of size at least $58$ which is not self-conjugate. The symmetric function $x_P(s_{\lambda} + s_{\lambda'})$ is \emph{not} multiplicity-free if and only if $\lambda$ is (up to conjugation) of the form $(a+1,a^{a-1})$ or $(b+2,1^b)$. 
\end{lemma}

\begin{proof}
	Let us start with an almost rectangular $\lambda$. Let $\mu$ be a rectangular partition obtained from $\lambda$ by increasing or decreasing one of its parts by one. Hence $|\mu|\geq 57$. If $\mu$ is not a square partition, then by Lemma~\ref{Lemma difference of rectangles} applied with $l=7$, we have $|\mu-\mu'|>14 $; thus $|\lambda-\lambda'|>12$. Then Lemma~\ref{Lemma rotate linear} applied with $l=6$, $x=y=x_P$ and partitions $\lambda$ and $\lambda'$ shows that $x_P(s_{\lambda} + s_{\lambda'})$ is multiplicity-free. In the case that $\mu$ is a square partition, then, up to conjugation, $\lambda$ has the form $(a+1,a^{a-1})$ and we claim that $x_Ps_{\lambda}$ and $x_Ps_{\lambda'}$ share a constituent $s_{\nu}$ with $\nu=(a+2,a+1,a^{a-2},2^2)$. Indeed, this follows from the Littlewood--Richardson rule stated in Theorem~\ref{Theorem LR rule} using the $\nu/\lambda$-tableau with reading word $112323$ and the $\nu/\lambda'$-tableau with reading word $112233$. Thus we obtain $(a+1,a^{a-1})$ in the statement.
	
	Now, let $\lambda=(a+1,1^b)$ with $a>b$ and write $d=a-b$. If $d=1$, we see that $s_{\nu}$ with $\nu=(b+3,3,2,1^b)$ is a common constituent of $x_Ps_{\lambda}$ and $x_Ps_{\lambda'}$ using the $\nu/\lambda$-tableau with reading word $121323$ and the $\nu/\lambda'$-tableau with reading word $112233$, yielding $\lambda=(b+2,1^b)$ in the statement. If $d\geq 4$, then all constituents $s_{\nu}$ of $x_Ps_{\lambda}$ satisfy $\ell(\nu)\leq \ell(\lambda)+3<\ell(\lambda')$; thus $s_{\nu}$ cannot be a constituent of $x_Ps_{\lambda'}$. It remains to show that the symmetric function is multiplicity-free when $d=2$ or $d=3$. We give full details for the case $d=2$; the case $d=3$ is very similar.
	
	If $s_{\nu}$ is a common constituent of $x_Ps_{\lambda}$ and $x_Ps_{\lambda'}$, then $\ell(\nu)\geq \ell(\lambda')=b+3$. Thus $c(\lambda,(2^3);\nu)=1$. If also $c(\lambda',(2^3);\nu)=1$, then $\nu_1\leq \lambda'_1+2=\lambda_1$, but clearly $\lambda_1\leq \nu_1$ which forces an equality. Consequently, $c(\lambda,(2^3);\nu)=1$ forces $\nu'_2\geq 4$ and $c(\lambda',(2^3);\nu)=1$ forces $\nu'_2\leq 3$, a contradiction. Therefore $c(\lambda,(2^3);\nu)=1$ and $c(\lambda',(6);\nu)=1$. By Young's rule we conclude that $\nu=(b+3+i,1+j,1^{b+1})$ with $i+j=4$ or $\nu=(b+3+i,1+j,1^{b+2})$ with $i+j=3$. One easily checks that $c(\lambda,(2^3);\nu)=0$ in all cases, a contradiction.    
\end{proof}

We are now ready to generalise Theorem~\ref{Theorem direct product k geq 2}. When talking about characters of the alternating group $A_k$ we use the notation $\chi^{\lambda}_{A_k}$ introduced after Lemma~\ref{Lemma index two} for an arbitrary constituent of the character $\chi^{\lambda}\res_{A_k}$. 

\begin{theorem}\label{Theorem direct product arbitrary}
	Let $k,l\geq 2$ be such that $k+l\geq 64$ and $K\leq S_k$ and $L\leq S_l$ be such that $L\neq S_l, A_l$. An irreducible character $\kappa\boxtimes\rho$ of $K\times L$ is induced-multiplicity-free if and only if $K,L,\kappa$ and $\rho$ are given by the following list:
	\begin{enumerate}[label=\textnormal{(\roman*)}]
		\item $K=S_k$ and $L$ and the characters are given by Theorem~\ref{Theorem direct product k geq 2},
		\item $K=A_k$, $L=S_2\wr S_2$ and $\kappa\boxtimes\rho=\chi^{\lambda}_{A_k}\boxtimes\left( \charwr{\chi^{\mu}}{\chi^{(1^2)}}{2}\right)$ with $\lambda$ rectangular and $\mu\vdash 2$,
		\item $K=A_k$, $\kappa=\chi^{\lambda}_{A_k}$ and $L, \rho$ and $\lambda$ are given by Proposition~\ref{Prop sk times primitive} with an exception when $L=\mathrm{PGL}_2(\mathbb{F}_5)$, $\rho$ is the restriction of the sign or the trivial character of $S_6$ and $\lambda$ is, up to conjugation, of the form $(a+1,a^{a-1})$ or $(b+2,1^b)$. 
	\end{enumerate}  
\end{theorem}

\begin{proof}
	Suppose that $\kappa\boxtimes\rho$ is induced-multiplicity-free. Then so are $\kappa\ind^{S_k}\boxtimes\rho$ and $\kappa\boxtimes\rho\ind^{S_l}$, and thus our choices for $K,L,\kappa$ and $\rho$ are narrowed down by Theorem~\ref{Theorem direct product k geq 2}. Firstly, suppose that $L=S_m\wr S_h$ with precisely one of $m$ or $h$ equal to two. In Theorem~\ref{Theorem direct product k geq 2}(i), (iii) and (iv) we see that for each character $\rho$ of $L$, there is only one choice of $\lambda$ which also forces $\lambda$ to be linear. Thus $\kappa\ind^{S_k}$ must have degree $1$, which means $K=S_k$.
	
	In the remaining cases $L$ is $S_2\wr S_2$ or a primitive subgroup of $S_l$ (distinct from $S_l$ and $A_l$). From the previous paragraph (with the roles of $K$ and $L$ swapped), $K$ must be of the same form as $L$ (with $K=S_k$ and $K=A_k$ allowed). This forces $K=S_k$ or $K=A_k$, as otherwise $k,l\leq 9$, and thus $k+l\leq 18$, a contradiction. Thus it remains to obtain (ii) and (iii) by examining groups $A_k\times L$ with $L$ as above. Write $\kappa=\chi^{\lambda}_{A_k}$. From Lemma~\ref{Lemma index two}(iv), the character $\chi^{\lambda}_{A_k}\boxtimes \rho$ induced to $S_k\times L$ equals $\chi^{\lambda}\boxtimes \rho + \chi^{\lambda'}\boxtimes \rho$ if $\lambda\neq\lambda'$ and $\chi^{\lambda}\boxtimes \rho$ otherwise.
	
	We start with $L=S_2\wr S_2$. By Corollary~\ref{Cor index two}(i) and Theorem~\ref{Theorem direct product k geq 2}, we see that $\lambda$ must be rectangular and $\rho$ must equal $\charwr{\chi^{\mu}}{\chi^{(1^2)}}{2}$ with $\mu\vdash 2$ for $\chi^{\lambda}_{A_k}\boxtimes \rho$ to be induced-multiplicity-free. If $\lambda$ is square, then clearly the character is induced-multiplicity-free. Otherwise, we can apply Lemma~\ref{Lemma rotate linear} with $\lambda$ and $\lambda'$ and Lemma~\ref{Lemma difference of rectangles} (both with $l=4$, which applies as $|\lambda|\geq 60>20$) to obtain that $\left( \chi^{\lambda}_{A_k}\boxtimes \rho\right)\ind^{S_k\times L}=\chi^{\lambda}\boxtimes \rho + \chi^{\lambda'}\boxtimes \rho$ is induced-multiplicity-free.
	
	In the case that $L$ is primitive and $\lambda$ is a properly rectangular partition, we use the same argument. If $\lambda$ is linear, we can apply Lemma~\ref{Lemma rotate linear} directly. Thus it only remains to examine the case where $L=\mathrm{PGL}_2(\mathbb{F}_5)$ and, up to tensoring with the restricted sign of $S_{k+6}$, $\rho$ is the trivial character of $L$ and $\lambda$ is either column-near rectangular or a column-unbalanced fat hook. The same must hold for $\lambda'$ which forces $\lambda$ to be a hook or an almost rectangular partition. The result then follows from Lemma~\ref{Lemma ak times pgl(2,5)}.    
\end{proof}

An important observation to make in Theorem~\ref{Theorem direct product arbitrary} is that $K$ must be $S_k$ or $A_k$. In other words, for $n\geq 64$ we cannot have a multiplicity-free subgroup $K\times L\leq S_k\times S_l\leq S_n$ with $K\neq S_k,A_k$ and $L\neq S_l,A_l$.

Next on the list, we can classify all irreducible induced-multiplicity-free characters of $S_m\wr S_h$.

\begin{proposition}\label{Prop MF wreath products}
	Let $m,h\geq 2$ be such that $mh\geq 64$. The irreducible induced-multiplicity-free characters $\rho$ of $S_m\wr S_h$ are given by the following list:
	\begin{enumerate}[label=\textnormal{(\roman*)}]
		\item the elementary irreducible characters in Theorem~\ref{Theorem MF plethysms},
		\item $h=2$ and $\rho=\left( \chi^{\mu}\boxtimes \chi^{\nu}\right) \Ind^{S_m\wr S_2}$ such that $\mu$ and $\nu$ are partitions of $m$ listed in Theorem~\ref{Theorem Stembridge},
		\item $h=3$ and $\rho=\left(\chi^{\lambda}\boxtimes \left( \charwr{\chi^{\mu}}{\chi^{\nu}}{2}\right)  \right)\Ind^{S_m\wr S_3} $ such that $\lambda$ and $\mu$ are in some order $(m)$ and $(1^m)$ and $\nu\vdash 2$,
		\item $m=2$ and $\rho=\left(\chi^{\lambda}\boxtimes\left( \charwr{\chi^{\mu}}{\chi^{(h-1)}}{h-1} \right) \right) \Ind^{S_2\wr S_h} $ with $\lambda$ and $\mu$ in some order equal to $(2)$ and $(1^2)$.
	\end{enumerate}
\end{proposition}

\begin{proof}
	Let $(\lambda^{[1]},\lambda^{[2]},\dots,\lambda^{[t]})$ be the sequence of partitions of total size $h$ corresponding to $\rho$ via Theorem~\ref{Theorem characters of wreath products}. By Lemma~\ref{Lemma sanity lemma} at most two of the partitions $\lambda^{[i]}$ are non-empty. If only one is non-empty, then $\rho$ is an elementary irreducible character and we are in (i). Now assume that precisely two $\lambda^{[i]}$ are non-zero and let $h_1$ and $h_2$ be their sizes. Then $S_m\wr S_{h_1} \times S_m\wr S_{h_2} $ is multiplicity-free.
	
	Either $h_1=h_2=1$ and we are in (ii). Or without loss of generality $h_2>1$ and Theorem~\ref{Theorem direct product arbitrary} with $L=S_m\wr S_{h_2}$ applies and forces $h_1=1$. From Theorem~\ref{Theorem direct product k geq 2}(i) and (iii) we obtain (iii) and (iv), respectively. No contribution comes from Theorem~\ref{Theorem direct product k geq 2}(ii) and (iv) as in these cases $mh\leq 10$. 
\end{proof}

We can also apply Theorem~\ref{Theorem direct product arbitrary} to narrow down the list of potential multiplicity-free subgroups inside $S_m\wr S_2$.

\begin{proposition}\label{Prop Sm wr S2 subgroups}
	Let $m\geq 32$. If $G\ < S_m\wr S_2\leq S_{2m}$ is a transitive multiplicity-free subgroup, then it is $A_m\wr S_2,\left( S_m\wr S_2\right) \cap A_{2m}$ or $T_{m,2}$.
\end{proposition}

\begin{proof}
	Firstly assume that $G=M\wr S_2$ for some transitive and (necessarily) multiplicity-free subgroup $M$ of $S_m$. We show that $M$ must be primitive, and thus $S_m$ or $A_m$ by Corollary~\ref{Corollary primitive groups}. Suppose for contradiction that $M$ is imprimitive, and by passing to a larger group if necessary, write $M=S_{m'}\wr S_{h'}$ with $m',h'\geq 2$. Since $M\times M$ is not multiplicity-free by Theorem~\ref{Theorem direct product arbitrary}, there is an elementary irreducible induced-multiplicity-free character of $M\wr S_2$, say $\charwrnb{\rho}{\chi^{\nu}}{2}$ with $\nu\vdash 2$.
	
	We claim that $\rho$ must be an elementary irreducible character of $M$. Otherwise for some positive integers $h_1$ and $h_2$ the group $\left( S_{m'}\wr S_{h_1} \times S_{m'}\wr S_{h_2} \right) \wr S_2$ is multiplicity-free. But this group lies inside $\left(S_{m'}\wr S_{h_1} \right)\wr S_2 \times \left(S_{m'}\wr S_{h_2} \right)\wr S_2$ which is not multiplicity-free by Theorem~\ref{Theorem direct product arbitrary}. Hence we can write $\rho=\charwr{\chi^{\lambda}}{\chi^{\mu}}{h'}$.
	
	Now, Theorem~\ref{Theorem MF plethysms} implies that $m'=2$ or $h'=2$. We eliminate $h'=2$ immediately since $\left( S_{m'}\wr S_{2}\right) \wr S_2 \cong S_{m'}\wr \left( S_{2} \wr S_2\right) \leq S_{m'}\wr S_4$ and $S_{m'}\wr S_4$ is not multiplicity-free by Proposition~\ref{Prop MF wreath products}.
	
	For $m'=2$, using Theorem~\ref{Theorem MF plethysms} and $\omega$, we need to show that $s_{\nu}\circ\left(s_{(h')}\circ s_{(2)} \right) $ and $s_{\nu}\circ\left(s_{(1^{h^{\prime}})}\circ s_{(2)} \right) $ are not multiplicity-free for $\nu\vdash 2$. By Lemma~\ref{Lemma even and shift rules}, we get constituents $s_{(2h'-2,2)}$ of $s_{(h')}\circ s_{(2)}$ and $s_{(h',3,1^{h^{\prime}-3})}$ of $s_{(1^{h^{\prime}})}\circ s_{(2)}$. Since $h'\geq 16\geq 4$, neither of $(2h'-2,2)$ and $(h',3,1^{h^{\prime}-3})$ is rectangular, almost rectangular or a hook; thus $s_{\nu}\circ\left(s_{(h')}\circ s_{(2)} \right) $ and $s_{\nu}\circ\left(s_{(1^{h^{\prime}})}\circ s_{(2)} \right) $ are not multiplicity-free by Theorem~\ref{Theorem MF plethysms}. Hence $M$ is primitive.
	
	For a general transitive subgroup $G$, using Lemma~\ref{Lemma wreath subgps}, we know that $G$ is a subwreath subgroup of $M\wr S_2$ for some $M\leq S_m$. We have shown above that $M$ is $S_m$ or $A_m$; thus the result follows from Corollary~\ref{Cor subwreath lemma}.
\end{proof}

To show that other imprimitive subgroups are not multiplicity-free we would like to have similar results for other wreath products $S_m\wr S_h$. In particular, according to Proposition~\ref{Prop MF wreath products}, we are interested in the cases when $m=2$ or $h=3$. Such results are part of Proposition~\ref{Prop possible MF subgroups} and use the following lemmas, the first two of which are needed for the case $m=2$.

\begin{lemma}\label{Lemma transitive size}
	Let $n\geq 8$. If $G < S_n$ is a transitive subgroup distinct from $A_n$, then $|S_n:G|>2n+4$.
\end{lemma}

\begin{proof}
	Theorem~\ref{Theorem primitive groups small} proves the result for all primitive $G$ apart from the $20$ members of $\mathcal{F}$ with $n\geq 8$ for which one uses direct calculations (which have been verified by {\sc Magma}). If $G$ is imprimitive, without loss of generality let $G=S_m\wr S_h$ with $m,h\geq 2$. Then we have $|S_n:G|= (mh)!/(h!(m!)^h)$. If $m>2$, we can bound this below by $(mh-1)(mh-2)/2$, while for $m=2$, we get a lower bound $(mh-1)(mh-3)$. Since $(n-1)(n-3)\geq (n-1)(n-2)/2>2n+4$ for $n\geq 8$ we are done.
\end{proof}

\begin{lemma}\label{Lemma submodules in char two}
	Suppose that $h\geq 5$. The (non-zero) submodules of the natural permutation $\mathbb{F}_2S_h$-module $\mathbb{F}_2^h$ have dimensions $1,h-1$ and $h$. Moreover, this remains true when $S_h$ is replaced with $A_h$. 
\end{lemma}

\begin{proof}
	By \cite[Example~5.1]{JamesSymmetric78}, for even $h$ the module $\mathbb{F}_2^h$ is uniserial with factors $\mathbb{F}_2, D^{(h-1,1)}$ and $\mathbb{F}_2$ from top to bottom. On the other hand for odd $h$ we have the decomposition $\mathbb{F}_2^h\cong \mathbb{F}_2 \oplus D^{(h-1,1)}$. The result follows as $D^{(h-1,1)}$ has dimension $h-2$, respectively, $h-1$ for even, respectively, odd $h$. The result remains true when we restrict to $A_h$ since $D^{(h-1,1)}$ (and clearly $\mathbb{F}_2$) remain irreducible when restricted to $A_h$ by \cite[Theorem~1.1]{BensonSpinModules88}.
\end{proof}

We combine the last two lemmas using an argument from \cite[p. 342]{SaxlMultiplicity-free81}.

\begin{lemma}\label{Lemma small index of S2 wr Sm}
	Suppose that $h\geq 8$. Then the only proper transitive subgroups of $S_2\wr S_h$ of index at most $2h+4$ are $S_2\wr A_h, \left(S_2\wr S_h \right)\cap A_{2h}, T_{2,h}$ and $T_{2,h}\cap A_{2h}$.
\end{lemma}

\begin{proof}
	Let $G$ be such a subgroup. Recall the notation and the short exact sequence \eqref{Eq ses}
	\begin{equation*}
	1\to G_B\to G \to \alpha(G)\to 1.
	\end{equation*}
	By our assumptions, $\alpha(G)$ must be a transitive subgroup of $S_h$ of index at most $2h+4$. By Lemma~\ref{Lemma transitive size} either $\alpha(G)=S_h$ or $\alpha(G)=A_h$.
	
	Write $P$ for the abelian group $\underbrace{S_2\times S_2\times\dots\times S_2}_{h \text{ times}}\leq S_2\wr S_h$. The group $S_h$ acts on $P$ by place permutations which makes $P$ into an $\mathbb{F}_2S_h$-module isomorphic to the natural permutation $\mathbb{F}_2S_h$-module $\mathbb{F}_2^h$. By restricting from $S_h$ to $\alpha(G)$, we see that $G_B$ is an $\mathbb{F}_2\alpha(G)$-submodule of $P\res_{\alpha(G)}\cong \mathbb{F}_2^h\res_{\alpha(G)}$.
	
	Using Lemma~\ref{Lemma submodules in char two}, we conclude that $G_B$ has dimension $0,1,h-1$ or $h$. Since, as groups, $|P:G_B|\leq 2h+4$, the only possible dimensions are $h-1$ and $h$, and in turn $|P:G_B|\leq 2$. Consequently, $G$ has index at most $2$ in either $S_2\wr S_h$ or $S_2\wr A_h$. Lemma~\ref{Lemma index two subgps} yields our four desired groups. 
\end{proof}

Finally, we move to $S_m\wr S_3$.

\begin{lemma}\label{Lemma small index of Sm wr S3}
	Let $m\geq 5$. There are no transitive subgroups of $S_m\wr S_3$ of index $3$ or $6$.
\end{lemma}

\begin{proof}
	Suppose for contradiction that $G$ is such a subgroup and recall the short exact sequence \eqref{Eq ses}
	\begin{equation*}
	1\to G_B\to G \to \alpha(G)\to 1.
	\end{equation*}
	By Lemma~\ref{Lemma conjugation lemma}, we know that $G_B$ is a subdirect subgroup of $G_1\times G_2\times G_3$ with $G_1,G_2$ and $G_3$ conjugate in $S_m$. In particular $|G_1|=|G_2|=|G_3|$. Consequently, since $|S_m\times S_m\times S_m:G_B|$ divides $6$, we conclude that $G_1=G_2=G_3=S_m$. This can be restated as: $G_B$ is a subdirect subgroup of $S_m\times H$ with $H$ a subdirect subgroup of $S_m\times S_m$. Two applications of Corollary~\ref{Corollary subdirect subgps}(i) imply that $G_B$ has index $1,2$ or $4$ in $S_m\times S_m\times S_m$; thus $\alpha(G)$ has index $3$ or $6$ in $S_3$, which cannot happen as it must be transitive.
\end{proof}

We end the section by combining all previous results into a small list of potential multiplicity-free subgroups. In the statement we divide these groups into several families (i)--(ix) which are organised similarly as in Theorem~\ref{Theorem main} with the exception that the groups from Theorem~\ref{Theorem main}(iii), (iv) with $k=1$ and (x) now belong to families (v) and (vi) embedded in $S_{2m+1}$ and $S_{2h+1}$, respectively. As in Theorem~\ref{Theorem main}, the family (ix) consists of subgroups of index $2$ in $S_k\times S_2\wr S_h$ for some $h\in\left\lbrace 2,3,4 \right\rbrace $ and the families (i)--(viii) are of the form $H, H_1,H_2,\dots,H_t$ where all $H_i$ are subgroups of $H$ of index $2$ or $4$. Moreover, for all these $H$ we have already found all their irreducible induced-multiplicity-free characters. 

\begin{proposition}\label{Prop possible MF subgroups}
	Let $n\geq 65$. Suppose that $G\leq S_n$ is a multiplicity-free subgroup. Then $G$ belongs to the list:
	\begin{enumerate}[label=\textnormal{(\roman*)}]
		\item $S_n$ and $A_n$,
		\item $S_k\times S_l, (S_k\times S_l)\cap A_{k+l}, A_k\times S_l$ and $A_k\times A_l$,
		\item $S_k\times S_m\wr S_2, \left( S_k\times S_m\wr S_2\right) \cap A_{k+2m}$ and $T_{k,m,2}$,
		\item $S_k\times S_2\wr S_h$ and $\left( S_k\times S_2\wr S_h\right)\cap A_{k+2h} $,
		\item  $S_m\wr S_2, \left( S_m\wr S_2\right) \cap A_{2m}, A_m\wr S_2$ and $T_{m,2}$ embedded in $S_{2m}$ or $S_{2m+1}$,
		\item $S_2\wr S_h, S_2\wr A_h, (S_2\wr S_h)\cap A_{2h}, T_{2,h}$ and $T_{2,h}\cap A_{2h}$ embedded in $S_{2h}$ or $S_{2h+1}$,
		\item $S_m\wr S_3, \left( S_m\wr S_3\right) \cap A_{3m}, S_m\wr A_3$ and $T_{m,3}$,
		\item $S_k\times L$,$A_k\times L$ and $(S_k\times L)\cap A_n$ where $L$ is one of $\mathrm{P}\Gamma\mathrm{L}_2(\mathbb{F}_8), \mathrm{ASL}_3(\mathbb{F}_2)$, $\mathrm{PGL}_2(\mathbb{F}_5)$ and $\mathrm{AGL}_1(\mathbb{F}_5)$,
		\item $A_k\times S_2\wr S_2, N_k$ and $T_{k,2,h}$ with $h\in \left\lbrace 3,4 \right\rbrace$,
	\end{enumerate}
	where $l\geq 1$, $k,m\geq 2$ and $h\geq 3$.
\end{proposition}

\begin{proof}
	In the proof $m\geq 2$, $h\geq 3$ and $k,l\geq 1$.
	According to Corollary~\ref{Corollary primitive groups}, group $G$ either belongs to (i) or is non-transitive or imprimitive. If it is non-transitive, write $K\times L\leq S_k\times S_l$ (with $n=k+l$) for a multiplicity-free group such that $G$ is a subdirect subgroup of it. If $K$ is $S_k$ or $A_k$ and $L$ is $S_l$ or $A_l$, Corollary~\ref{Corollary MF subdirect subgpa} implies (ii) (since (ii) does not change if $k=1$ is allowed). Now suppose this is not the case.
	
	If $k=1$, implicitly $L\neq S_l,A_l$, and by Proposition~\ref{Prop direct product k eq 1}, the group $G$ must be a subgroup of $S_m\wr S_2$ or $S_2\wr S_h$ embedded in $S_{2m+1}$, respectively, $S_{2h+1}$. By  Lemma~\ref{Lemma sanity lemma}, it must be transitive in $S_{2m}$, respectively, $S_{2h}$. We can use Proposition~\ref{Prop Sm wr S2 subgroups} for $S_m\wr S_2$ to get groups in (v) embedded in $S_{2m+1}$. For $G\leq S_2\wr S_h$, according to Proposition~\ref{Prop direct product k eq 1}(ii), there are only $4$ irreducible induced-multiplicity-free characters of $S_2\wr S_h$ and all of them have degree $1$. Hence we only need to consider subgroups of $S_2\wr S_h$ of index at most $4$. By Lemma~\ref{Lemma small index of S2 wr Sm} we get groups in (vi) embedded in $S_{2h+1}$.
	
	Now assume that $k,l\geq 2$. Theorem~\ref{Theorem direct product arbitrary} applies and lists all the possible choices of $K\times L$. In particular, $K$ is $S_k$ or $A_k$. We claim that $|K\times L:G|\leq 2$ for $K=S_k$ and $G=K\times L$ for $K=A_k$. If $L=S_m\wr S_2$ with $m\geq 3$, we have $K=S_k$. We conclude $|K\times L:G|\leq 2$ from the fact that the induced-multiplicity-free characters of $K\times L$ have degree at most $2$. This is because by Theorem~\ref{Theorem direct product k geq 2}(i) there are only four such characters which are irreducible, all have degree $1$ and we cannot add more than two of them together without creating a non-induced-multiplicity-free character since the sum of the induced characters corresponding to $s_{\lambda}(s_{(2)}\circ s_{\mu})$ and $s_{\lambda}(s_{(1^2)}\circ s_{\mu})$ corresponds to $s_{\lambda}s_{\mu}^2$, which is not multiplicity-free by Lemma~\ref{Lemma sanity lemma}. For $L=S_2\wr S_h$ with $h\geq 5$, we reach the same conclusion since, by Theorem~\ref{Theorem direct product k geq 2}(iii), there are only two irreducible induced-multiplicity-free characters of $K\times L$ and they are both of degree $1$. In the remaining cases $L\leq S_l$ with $l\leq 9$. Thus $L$ cannot surject to $S_k$ or $A_k$ and Lemma~\ref{Corollary subdirect subgps} proves the claim.
	
	We can now use Lemma~\ref{Lemma index two subgps} (and character tables of $\mathrm{P}\Gamma\mathrm{L}_2(\mathbb{F}_8), \mathrm{ASL}_3(\mathbb{F}_2)$, $\mathrm{PGL}_2(\mathbb{F}_5)$ and $\mathrm{AGL}_1(\mathbb{F}_5)$) to obtain index $2$ subdirect subgroups of $S_k\times L$. Note, the corresponding real linear character $\eta=\eta_1\boxtimes\eta_2$ must be such that $\eta_1$ and $\eta_2$ are non-trivial as $G$ is a subdirect subgroup of $S_k\times L$. Using Corollary~\ref{Cor index two}(ii) (and characters from Theorem~\ref{Theorem direct product k geq 2}(i)--(iv)), we eliminate $\eta =\sgn\boxtimes \left( \charwrnb{\mathbbm{1}}{\sgn}{2}\right)$ for $m\geq 3$ (with $L=S_m\wr S_2$), $\eta =\sgn\boxtimes \left( \charwrnb{\sgn}{\sgn}{h}\right) $ for $h\geq 5$ and $\eta =\sgn\boxtimes \left( \charwrnb{\mathbbm{1}}{\sgn}{h}\right) $ for $h\geq 3$ (both with $L=S_2\wr S_h$). This leaves us with (iii), (iv), (viii) and (ix).
	
	It remains to consider $G$ transitive but imprimitive. Using Proposition~\ref{Prop MF wreath products}, $G$ is a subgroup of $S_m\wr S_2$, $S_2\wr S_h$ or  $S_m\wr S_3$. If $G\leq S_m\wr S_2$, we apply Proposition~\ref{Prop Sm wr S2 subgroups} to get groups in (v) embedded in $S_{2m}$. Looking at the degrees of irreducible characters of $S_2\wr S_h$ in Proposition~\ref{Prop MF wreath products}, if $G\leq S_2\wr S_h$, then $|S_2\wr S_h: G|\leq 2h +4$. Lemma~\ref{Lemma small index of S2 wr Sm} then gives groups in (vi) embedded in $S_{2h}$. Finally, for $G\leq S_m\wr S_3$ observe that any induced-multiplicity-free character of $S_m\wr S_3$ has degree $3$ or $6$. Indeed, there are four of them which are irreducible, all have degree $3$ and we cannot add more than two of them together without creating a non-induced-multiplicity-free character since the sum of the induced characters corresponding to $s_{\lambda}(s_{(2)}\circ s_{\mu})$ and $s_{\lambda}(s_{(1^2)}\circ s_{\mu})$ corresponds to $s_{\lambda}s_{\mu}^2$, which is not multiplicity-free by Lemma~\ref{Lemma sanity lemma}. Therefore $G$ has index $1,2,3$ or $6$ in $S_m\wr S_3$. Lemma~\ref{Lemma small index of Sm wr S3} reduces these options to $1$ and $2$, which yields (vii) according to Lemma~\ref{Lemma index two subgps}.  
\end{proof}
	
\section{Small index subgroups}\label{Sec SIS I}
The subsections of this section are all apart from the final one titled by the first groups $H$ in the families in Proposition~\ref{Prop possible MF subgroups}(iii)--(vii) with (v) excluded. Each subsection starts by recalling the irreducible induced-multiplicity-free characters of $H$ as well as the other groups in the corresponding family. The goal is then to find all the irreducible induced-multiplicity-free characters of these other groups. The final (rather technical) subsection then focuses on families (viii) and (ix).

One should keep in mind Lemma~\ref{Lemma tensoring plethysms}, Lemma~\ref{Lemma index two} and Corollary~\ref{Cor index two} since computations as in Example~\ref{Example index two} are commonly used. Unless specified otherwise, $k,m\geq 2$ and $h\geq 3$ throughout. 

\subsection{$S_k\times S_m\wr S_2$}\label{Section k,m,2}

We consider the groups $G_+=\ker\left( \sgn\boxtimes \left( \charwrnb{\sgn}{\mathbbm{1}}{2}\right) \right) $ and $G_-=\ker\left(  \sgn\boxtimes \left( \charwrnb{\sgn}{\sgn}{2}\right)\right)  $ which equal $\left( S_k\times S_m\wr S_2\right) \cap A_{k+2m}$ and $T_{k,m,2}$ in some order: if $m$ is even, $G_+=\left( S_k\times S_m\wr S_2\right) \cap A_{k+2m}$ and $G_-=T_{k,m,2}$, while for odd $m$ it is the other way around.

By Theorem~\ref{Theorem direct product k geq 2}, for $n\geq 18$ the irreducible induced-multiplicity-free characters of $S_k\times S_m\wr S_2$ are $\chi^{(1^k)}\boxtimes \left( \charwr{\chi^{(m)}}{\chi^{\nu}}{2}\right) $ and $\chi^{(k)}\boxtimes \left( \charwr{\chi^{(1^m)}}{\chi^{\nu}}{2}\right) $ with $\nu\vdash 2$, and in the case $m=2$ also $\chi^{\lambda}\boxtimes \left( \charwr{\chi^{\mu}}{\chi^{(1^2)}}{2}\right) $ with $\mu\vdash 2$ and $\lambda\vdash k$ rectangular. Therefore the irreducible characters of $G_{\pm}$ which can be induced-multiplicity-free are $\left( \chi^{(1^k)}\boxtimes \left( \charwr{\chi^{(m)}}{\chi^{\nu}}{2}\right) \right)_{G_{\pm}}$ with $\nu\vdash 2$, and for $m=2$ also $\left( \chi^{\lambda}\boxtimes \left( \charwr{\chi^{(2)}}{\chi^{(1^2)}}{2}\right)\right) _{G_+}$ with $\lambda$ rectangular.

The first character is induced-multiplicity-free if and only if the symmetric function $s_{(1^k)}(s_{\nu}\circ s_{(m)})+s_{(k)}(s_{\bar{\nu}}\circ s_{(1^m)})$, with $\bar{\nu}=\nu$ for $G_+$ and $\bar{\nu}=\nu'$ for $G_-$, is multiplicity-free. Similarly for the second character, we need to find when $s_{\lambda}\left(s_{(1^2)}\circ s_{(2)} \right) + s_{\lambda'}\left(s_{(1^2)}\circ s_{(1^2)} \right)$ is multiplicity-free.

\begin{lemma}\label{Lemma hook and a box}
	Let $k,m\geq 2$ and $\nu,\bar{\nu}\vdash 2$. The symmetric function $s_{(1^k)}(s_{\nu}\circ s_{(m)})+s_{(k)}(s_{\bar{\nu}}\circ s_{(1^m)})$ is \emph{not} multiplicity-free if and only if $k,\nu,\bar{\nu}$ lie in one of these four cases:
	\begin{enumerate}[label=\textnormal{(\roman*)}]
		\item $\nu=(2), \bar{\nu}=(2)^{\prime m}$ and $2m-4\leq k\leq 2m+1$,
		\item $\nu=(2), \bar{\nu}=(1^2)^{\prime m}$ and $2m-3\leq k\leq 2m$,
		\item $\nu=(1^2), \bar{\nu}=(2)^{\prime m}$ and $2m-3\leq k\leq 2m$,
		\item $\nu=(1^2), \bar{\nu}=(1^2)^{\prime m}$ and $2m-3\leq k\leq 2m$.
	\end{enumerate}
\end{lemma}

\begin{proof}
	Both summands are multiplicity-free by Proposition~\ref{Prop product with sm wr s2}. Any partition $\lambda$ such that $p((m),\nu;\lambda)=1$ for some $\nu\vdash 2$ is $(m,1)$-birectangular by Proposition~\ref{Prop domino rectangles}, and thus it satisfies $\ell(\lambda)\leq 2$ by Remark~\ref{Remark birectangular}. Thus, by Pieri's rule, if $s_{\lambda}$ is a constituent of $s_{(1^k)}(s_{\nu}\circ s_{(m)})$ for some $\nu\vdash 2$, then $\lambda=(a,b,1^c)$ with $a + b\in\left\lbrace 2m, 2m+1, 2m+2\right\rbrace $. Applying $\omega$ shows that whenever $s_{\lambda}$ is a constituent of $s_{(k)}(s_{\bar{\nu}}\circ s_{(1^m)})$ for some $\bar{\nu}\vdash 2$, then $\lambda'$ is of the form $(a,b,1^c)$ with $a + b\in\left\lbrace 2m, 2m+1, 2m+2\right\rbrace $. Thus any common constituent $s_{\lambda}$ must be labelled by a hook $(u+1,1^v)$ with $u,v\in \left\lbrace 2m-2,2m-1,2m \right\rbrace $ or by a partition of the form $(u+1,2,1^{v-1})$ with $u,v\in \left\lbrace 2m-3,2m-2,2m-1 \right\rbrace $.
	
	Now, for each of these choices of $\lambda$ we compute for which $k,\nu$ and $\bar{\nu}$ the symmetric functions $s_{(1^k)}(s_{\nu}\circ s_{(m)})$ and $s_{(k)}(s_{\bar{\nu}}\circ s_{(1^m)})$ share a constituent $s_{\lambda}$. Note that the latter has $s_{\lambda}$ as a constituent if and only if $s_{(1^k)}(s_{\bar{\nu}^{\prime m}}\circ s_{(m)})$ has $s_{\lambda'}$ as a constituent. We demonstrate the computations for $\lambda=(u+1,1^v)$ with $u=2m-1$ and $v=2m-2$.
	
	We have $k=|\lambda|-2m=2m-2$. Since $\lambda_1+\lambda_2=2m+1$, we can obtain $\lambda$ by adding a vertical strip of size $k$ to two possible $(m,1)$-birectangular partitions, namely $(2m)$ and $(2m-1,1)$. This means that $\nu$ can be arbitrary by Proposition~\ref{Prop domino rectangles}. On the other hand $\lambda'_1 + \lambda'_2=2m$; hence $\lambda'$ can be obtained by adding a vertical strip of size $k$ to only one $(m,1)$-birectangular partition which is $(2m-1,1)$. By Proposition~\ref{Prop domino rectangles}, we have $\bar{\nu}^{\prime m}=(1^2)$, that is $\bar{\nu}=(1^2)$ if $m$ is even and $\bar{{\nu}}=(2)$ if $m$ is odd.
	
	For even $m$ the values of $k,\nu$ and $\bar{\nu}$ are presented in Table~\ref{Table hook} for $\lambda=(u+1,1^v)$ and in Table~\ref{Table hook and a box} for $\lambda=(u+1,2,1^{v-1})$. One obtains the values for odd $m$ by conjugating $\bar{\nu}$. The result then follows.\qedhere
	
	\begin{table}[h!]
		\centering
		\begin{tabular}{c|ccc}
			\specialrule{0.1em}{0em}{0em}
			\diagbox[width=1.7cm,  height=0.7cm]{$u$}{$v$}&$2m-2$&$2m-1$&$2m$\\
			\hline
			$2m-2$&$2m-3,(1^2),(1^2)$&$2m-2,(1^2),\text{any}$&$2m-1,(1^2),(2)$\\
			$2m-1$&$2m-2,\text{any}, (1^2)$&$2m-1, \text{any},\text{any}$&$2m,\text{any},(2)$\\
			$2m$&$2m-1,(2), (1^2)$&$2m,(2), \text{any}$&$2m+1,(2),(2)$\\
			\specialrule{0.1em}{0em}{0em}
		\end{tabular}
		\vspace{4pt}
		\caption{For even $m$ and $\lambda=(u+1,1^v)$, these are the possible values of $k,\nu$ and $\bar{\nu}$ for which $s_{\lambda}$ is a common constituent of $s_{(1^k)}(s_{\nu}\circ s_{(m)})$ and $s_{(k)}(s_{\bar{\nu}}\circ s_{(1^m)})$.}
		\label{Table hook}
	\end{table}
	
	\begin{table}[h!]
		\centering
		\begin{tabular}{c|ccc}
			\specialrule{0.1em}{0em}{0em}
			\diagbox[width=1.7cm,  height=0.7cm]{$u$}{$v$}&$2m-3$&$2m-2$&$2m-1$\\
			\hline
			$2m-3$&$2m-4,(2),(2)$&$2m-3,(2),\text{any}$&$2m-2,(2),(1^2)$\\
			$2m-2$&$2m-3,\text{any},(2)$&$2m-2, \text{any},\text{any}$&$2m-1,\text{any},(1^2)$\\
			$2m-1$&$2m-2,(1^2),(2)$&$2m-1,(1^2), \text{any}$&$2m,(1^2),(1^2)$\\
			\specialrule{0.1em}{0em}{0em}
		\end{tabular}
		\vspace{4pt}
		\caption{For even $m$ and $\lambda=(u+1,2,1^{v-1})$, these are the possible values of $k,\nu$ and $\bar{\nu}$ for which $s_{\lambda}$ is a common constituent of $s_{(1^k)}(s_{\nu}\circ s_{(m)})$ and $s_{(k)}(s_{\bar{\nu}}\circ s_{(1^m)})$.}
		\label{Table hook and a box}
	\end{table}
\end{proof}

\begin{lemma}\label{Lemma m=2 index 2}
	Let $\lambda$ be a rectangular partition of size at least $21$. The symmetric function $s_{\lambda}\left(s_{(1^2)}\circ s_{(2)} \right) + s_{\lambda'}\left(s_{(1^2)}\circ s_{(1^2)} \right)$ is multiplicity-free if and only $\lambda$ is non-square.
\end{lemma}

\begin{proof}
	Since both summands are multiplicity-free by Proposition~\ref{Prop product with sm wr s2}, if $\lambda$ is a non-square rectangular partition, we let $l=4$ and use Lemma~\ref{Lemma rotate linear} with $\mu=\lambda, \nu=\lambda'$ and $x$ and $y$ equal to our two plethysms together with Lemma~\ref{Lemma difference of rectangles} to get that our sum is multiplicity-free. On the other hand if $\lambda$ is a square, say $(a^a)$, then it follows from the Littlewood--Richardson rule stated in Theorem~\ref{Theorem LR rule} that $s_{(a+2,a+1,a^{a-2},1)}$ is a constituent of both $s_{\lambda}\left(s_{(1^2)}\circ s_{(2)} \right)=s_{(a^a)}s_{(3,1)}$ and $s_{\lambda'}\left(s_{(1^2)}\circ s_{(1^2)} \right)=s_{(a^a)}s_{(2,1^2)}$ (the reading words of the required $(a+2,a+1,a^{a-2},1)/(a^a)$-tableaux are $1121$, respectively, $1123$).
\end{proof}

\begin{corollary}\label{Cor Sk times Sm sr S2 class}
	Let $k,m\geq 2$ be such that $k+2m\geq 25$. 
	\begin{enumerate}[label=\textnormal{(\roman*)}]
		\item The irreducible induced-multiplicity-free characters of the group $G=( S_k\times S_m\wr S_2)\cap A_{k+2m}$ are $\left( \chi^{(1^k)}\boxtimes \left( \charwr{\chi^{(m)}}{\chi^{(2)}}{2}\right)\right) _{G}$, provided $k\notin\left\lbrace2m-4,2m-3,\dots,2m+1 \right\rbrace $, $\left( \chi^{(1^k)}\boxtimes \left( \charwr{\chi^{(m)}}{\chi^{(1^2)}}{2}\right)\right) _{G}$, provided $k\notin\left\lbrace 2m-3,2m-2,2m-1,2m \right\rbrace $, and for $m=2$ also the characters $\left( \chi^{\lambda}\boxtimes \left( \charwr{\chi^{(2)}}{\chi^{(1^2)}}{2}\right)\right) _{G}$ with $\lambda$ a non-square rectangular partition.
		\item The irreducible induced-multiplicity-free characters of $G=T_{k,m,2}$ are $\left( \chi^{(1^k)}\boxtimes \left( \charwr{\chi^{(m)}}{\chi^{\nu}}{2}\right)\right) _{G}$ with $\nu\vdash 2$ arbitrary, provided $k\notin\left\lbrace 2m-3,2m-2,2m-1,2m \right\rbrace $. Otherwise, there is no such character. 
	\end{enumerate}
\end{corollary}

\begin{proof}
	We obtain the characters in (i) from Lemma~\ref{Lemma hook and a box}(i) and (iv), that is the parts of Lemma~\ref{Lemma hook and a box} when $\nu$ and $\bar{\nu}$ agree precisely when $m$ is even, and Lemma~\ref{Lemma m=2 index 2}. The characters in (ii) then come from Lemma~\ref{Lemma hook and a box}(ii) and (iii).
\end{proof}

\subsection{$S_m\wr S_3$}

We now move to groups $\left( S_m\wr S_3\right) \cap A_{3m}=\ker \left( \charwrnb{\sgn}{\sgn}{3}\right) $, $T_{m,3}=\ker \left(\charwrnb{\sgn}{\mathbbm{1}}{3}\right) $ and $S_m\wr A_3=\ker \left( \charwrnb{\mathbbm{1}}{\sgn}{3}\right) $. For $n\geq 64$, according to Proposition~\ref{Prop MF wreath products}, the irreducible induced-multiplicity-free characters of $S_m\wr S_3$ are $ \left(\chi^{\lambda}\boxtimes \left( \charwr{\chi^{\mu}}{\chi^{\nu}}{2}\right)  \right)\Ind^{S_m\wr S_3}$ such that $\lambda$ and $\mu$ are in some order $(m)$ and $(1^m)$ and $\nu\vdash 2$. Thus we are left with two potential choices of irreducible induced-multiplicity-free characters for each of our three groups.

\begin{proposition}\label{Prop Sm wr S3 class}
	Let $m\geq 22$.
	\begin{enumerate}[label=\textnormal{(\roman*)}]
		\item Let $G$ be $\left( S_m\wr S_3\right) \cap A_{3m}$ or $T_{m,3}$. The irreducible induced-multiplicity-free characters of $G$ are $\left(\left(\chi^{(1^m)}\boxtimes \left( \charwr{\chi^{(m)}}{\chi^{\nu}}{2}\right)  \right)\Ind^{S_m\wr S_3}\right) _G $ with $\nu\vdash 2$.
		\item The group $S_m\wr A_3$ is not multiplicity-free.
	\end{enumerate}
\end{proposition}

\begin{proof}
	As the characters in (i) are the only possible choices of irreducible induced-multiplicity-free characters, it remains to show that for $\nu,\bar{\nu}\vdash 2$ the symmetric function $s_{(1^m)}(s_{\nu}\circ s_{(m)})+s_{(m)}(s_{\bar{\nu}}\circ s_{(1^m)})$ is multiplicity-free. This follows from Lemma~\ref{Lemma hook and a box} with $k=m<2m-4$.
	
	To obtain (ii) we need to show that for partitions $\lambda$ and $\mu$ given by $(m)$ and $(1^m)$ in some order, the symmetric function $s_{\lambda}(s_{(2)}\circ s_{\mu}) + s_{\lambda}(s_{(1^2)}\circ s_{\mu}) $ is not multiplicity-free. This is clear since it equals $s_{\lambda}s_{\mu}^2$ by Lemma~\ref{Lemma plethysms and char properties}(iii), and thus we can apply Lemma~\ref{Lemma sanity lemma}. 
\end{proof}

\subsection{$S_k\times S_2\wr S_h$}

This time we examine the characters of $( S_k\times S_2\wr S_h) \cap A_{k+2h} = \ker \left( \sgn\boxtimes \left( \charwrnb{\sgn}{\mathbbm{1}}{h}\right) \right) $. Assuming $n\geq 18$, according to Theorem~\ref{Theorem direct product k geq 2}, up to tensoring with the sign, the irreducible induced-multiplicity-free characters of $S_k\times S_2\wr S_h$ are $\chi^{(1^k)}\boxtimes \left( \charwr{\chi^{(2)}}{\chi^{(h)}}{h}\right) $ and for $h\in\left\lbrace 3,4 \right\rbrace$ also $\chi^{(1^k)}\boxtimes \left( \charwr{\chi^{(2)}}{\chi^{(1^h)}}{h}\right) $. In the next two lemmas, using symmetric functions, we investigate whether restrictions of these characters are induced-multiplicity-free.

\begin{lemma}\label{Lemma even and strip}
	Let $k$ and $h$ be positive integers. The symmetric function $s_{(1^k)}(s_{(h)}\circ s_{(2)})+s_{(k)}(s_{(h)}\circ s_{(1^2)})$ is multiplicity-free if and only if $k\geq 2h+2$.
\end{lemma}

\begin{proof}
	Both summands are multiplicity-free by Proposition~\ref{Prop product with s2 wr sm}. Suppose that $s_{\lambda}$ is a common constituent of both summands. Since it is a constituent of the first one, we have $(1^k)\subseteq \lambda$, that is $\ell(\lambda)\geq k$. On the other hand, partitions indexing constituents of $s_{(h)}\circ s_{(1^2)}$ have length at most $2h$; thus $\ell(\lambda)\leq 2h+1$. This gives $k\leq 2h+1$, establishing the `if' direction.
	
	Now suppose that $k\leq 2h+1$ and let $r=2h-k\geq -1$. By Lemma~\ref{Lemma even and shift rules}, to prove the `only if' direction we want $\lambda$ such that there are partitions $\mu$ and $\nu$ such that $\mu$ and $\nu'$ are even and $\lambda$ is obtained by adding a vertical strip of size $k$ to $\mu$ and a horizontal strip of size $k$ to $\nu$. The desired partitions $\lambda, \mu$ and $\nu$, depending on $r$ modulo $4$, are presented in Table~\ref{Table even and strip}. \qedhere
	
	\begin{table}[h!]
		\centering
		\begin{tabular}{cc}
			\toprule
			$r$&$\lambda,\mu,\nu$\\
			\midrule
			&$(h+(k-1)/2, h-(k-1)/2,1^k)$,\\
			$r\equiv -1$ (mod $4$), $r\neq -1$&$(h+(k-1)/2, h-(k-1)/2)$,\\
			&$(h-(k-1)/2, h-(k-1)/2,1^{k-1})$\\
			\midrule
			&$(2h+1,1^{2h})$,\\
			$r=-1$&$(2h)$,\\
			&$(1^{2h})$\\
			\midrule
			&$(h+k/2,h-k/2,1^k)$,\\
			$r\equiv 0$ (mod $4$), $r\neq 0$&$(h+k/2,h-k/2)$,\\
			&$(h-k/2,h-k/2,1^k)$\\
			\midrule
			&$(2h,1^{2h})$,\\
			$r=0$&$(2h)$,\\
			&$(1^{2h})$\\
			\midrule
			&$(h+(k-3)/2, h-(k+1)/2,2,1^{k})$,\\
			$r\equiv 1$ (mod $4$), $r\neq 1$&$(h+(k-3)/2, h-(k+1)/2,2)$,\\
			&$(h-(k+1)/2, h-(k+1)/2,1^{k+1})$\\
			\midrule
			&$(2h, 1^{2h-1})$,\\
			$r=1$&$(2h)$,\\
			&$(1^{2h})$\\
			\midrule
			&$(h+k/2-1, h-k/2,2,1^{k-1})$,\\
			$r\equiv 2$ (mod $4$), $r\neq 2$&$(h+k/2-1, h-k/2-1,2)$,\\
			&$(h-k/2,h-k/2,1^k)$\\
			\midrule
			&$(2h-1,2,1^{2h-3})$,\\
			$r=2$&$(2h-2,2)$,\\
			&$(2^2,1^{2h-4})$\\
			\bottomrule
		\end{tabular}
		\vspace{4pt}
		\caption{Table of partitions $\lambda,\mu$ and $\nu$ such that $\lambda$ can be obtained by adding a vertical strip of size $k$ to $\mu$ and a horizontal strip of size $k$ to $\nu$ and, moreover, $\mu$ and $\nu'$ are even partitions of $2h$. The number $r$ equals $2h-k\geq -1$.}
		\label{Table even and strip}
	\end{table}
\end{proof}

\begin{lemma}\label{Lemma shift small}
	Suppose that $h\in\left\lbrace 3,4 \right\rbrace $ and $k\geq 10$. The symmetric function $s_{(1^k)}(s_{(1^h)}\circ s_{(2)})+s_{(k)}(s_{(1^h)}\circ s_{(1^2)})$ is multiplicity-free.
\end{lemma}

\begin{proof}
	As both summands are multiplicity-free by Proposition~\ref{Prop product with s2 wr sm}, we can apply Lemma~\ref{Lemma rotate linear} with $x$ and $y$ equal to the relevant plethysms, $l=2h\in\left\lbrace 6,8\right\rbrace $, $\lambda=(1^k)$ and $\mu=(k)$ (meaning that $|\lambda-\mu|=2k-2>2l$).
\end{proof}

\begin{corollary}\label{Cor Sk times S2 wr Sm class}
	Let $k\geq 2$ and $h\geq 3$ be such that $k+2h\geq 18$. The irreducible induced-multiplicity-free characters of $G=(S_k\times S_2\wr S_h)\cap A_{k+2h}$ are $\left(\chi^{(1^k)}\boxtimes \left( \charwr{\chi^{(2)}}{\chi^{(h)}}{h}\right)  \right)_G $, provided that we have $k\geq 2h+2$, and for $h\in\left\lbrace 3,4 \right\rbrace$ also $\left(\chi^{(1^k)}\boxtimes \left( \charwr{\chi^{(2)}}{\chi^{(1^h)}}{h}\right)  \right)_G $.
\end{corollary}

\begin{proof}
	Combine Lemma~\ref{Lemma even and strip} and Lemma~\ref{Lemma shift small}.
\end{proof}

\subsection{$S_2\wr S_h$}\label{Sec S2 wr Sh}

We consider groups $S_2\wr A_h=\ker \left( \charwrnb{\mathbbm{1}}{\sgn}{h}\right), (S_2\wr S_h)\cap A_{2h}=\ker \left( \charwrnb{\sgn}{\mathbbm{1}}{h}\right), T_{2,h}=\ker \left( \charwrnb{\sgn}{\sgn}{h}\right) $ and $T_{2,h}\cap A_{2h}$. By Proposition~\ref{Prop MF wreath products}, for $n\geq 64$, the irreducible induced-multiplicity-free characters of $S_2\wr S_h$ are $\charwr{\chi^{\mu}}{\chi^{\nu}}{h}$ with $\mu$ and $\nu$ linear, and $\left( \chi^{\lambda}\boxtimes \left( \charwr{\chi^{\mu}}{\chi^{(h-1)}}{h-1}\right)\right)  \Ind^{S_2\wr S_h}$ with $\lambda$ and $\mu$ equal to $(2)$ and $(1^2)$ in some order. It is clear that if $G$ is one of the three index two subgroups above, then $\left( \left( \chi^{\lambda}\boxtimes \left( \charwr{\chi^{\mu}}{\chi^{(h-1)}}{h-1}\right)\right)  \Ind^{S_2\wr S_h}\right)_G$ can be induced-multiplicity-free only if $G=(S_2\wr S_h)\cap A_{2h}$. We now apply Lemma~\ref{Lemma even and shift rules} to obtain the following five crucial lemmas. The first three appear with a different proof in \cite[\S2.4]{WildonMultiplicity-free09} and \cite[Proposition~2.12, Theorem~2.16 and Theorem~2.17]{GodsilMeagherMultiplicity-free10}.

\begin{lemma}\label{Lemma double even}
	Suppose $h$ is a positive integer. The symmetric function $s_{(h)}\circ s_{(2)} + s_{(h)}\circ s_{(1^2)}$ is multiplicity-free if and only if $h$ is odd.
\end{lemma}

\begin{proof}
	If $h$ is even, both of the summands have a constituent labelled by the partition $(h^2)$ since it is even and so is its conjugate. Conversely, any common constituent would have to be labelled by a partition of the form $(2\lambda_1,2\lambda_1, 2\lambda_2,2\lambda_2,\dots)$ which has size divisible by $4$. Thus for it to exist we need $4\mid 2h$, that is $h$ must be even. As both summands are multiplicity-free, this finishes the proof.
\end{proof}

\begin{lemma}\label{Lemma even and shift}
	Let $h\geq 25$. The symmetric functions $s_{(h)}\circ s_{(2)} + s_{(1^h)}\circ s_{(2)}$ and $s_{(h)}\circ s_{(1^2)} + s_{(1^h)}\circ s_{(1^2)}$ are \emph{not} multiplicity-free.
\end{lemma}

\begin{proof}
	We only consider the first symmetric function as using $\omega$ recovers the second one. Suppose that $\beta=(\beta_1,\beta_2,\dots,\beta_t)$ is a partition with pairwise distinct odd parts. Then $\alpha=(\beta_1,\beta_1-1,\beta_2,\beta_2-1,\dots,\beta_t-1)$ has also pairwise distinct parts and we can define $\lambda=ss[\alpha]$. From the definition $\lambda_i$ is even for all $i\leq \ell(\alpha)$ and $\lambda_{\ell(\alpha)+1}\leq \ell(\alpha)\leq 2t$. Thus we conclude that $\lambda$ is even once we show that $\lambda'_{2j-1}=\lambda'_{2j}$ for any $j\leq t$. This is clear from the definition except in the case when $\beta_t=1$ and $j=t$. In such a case $\lambda'_{2t-1}=2t-1$ and $\lambda_{2t-1}=2t$, forcing $\lambda'_{2t}=2t-1$, as needed.
	
	Now, it is sufficient to show that for given $h\geq 25$ there is such a partition $\beta$ for which $|\alpha|=2|\beta|-t$ equals $h$ since then $s_{\lambda}$ appears in our symmetric function with multiplicity $2$. To obtain such $\beta$, take $t\in\left\lbrace 1,2,3,4\right\rbrace$ such that $h\equiv t$ (mod $4$), and let $\beta$ be $((h+1)/2), (h/2,1), ((h-5)/2,3,1)$ or $(h/2-7,5,3,1)$ according to $t$. Then $|\beta|=(h+t)/2$ and $|\alpha|=h$.
\end{proof}

The next lemma is analogous and we omit repeated details.

\begin{lemma}\label{Lemma conjugate even and shift}
	Let $h\geq 33$. The symmetric functions $s_{(h)}\circ s_{(1^2)} + s_{(1^h)}\circ s_{(2)}$ and $s_{(h)}\circ s_{(2)} + s_{(1^h)}\circ s_{(1^2)}$ are \emph{not} multiplicity-free.
\end{lemma}

\begin{proof}
	Applying $\omega$ if necessary, we consider only the first symmetric function. Let $\beta=(\beta_1,\beta_2,\dots,\beta_t)$ be a partition with pairwise distinct even parts. Let $\alpha=(\beta_1,\beta_1-1,\beta_2,\beta_2-1,\dots,\beta_t-1)$ and $\lambda=ss[\alpha]$. We see that $\lambda'$ is even, and therefore we are done as long as we can take $\beta$ for which $|\alpha|=h$. Letting $t\in\left\lbrace 1,2,3,4\right\rbrace$ to be such that $h\equiv -t$ (mod $4$), we can take $\beta$ to be one of $((h+1)/2), (h/2-1,2), ((h-9)/2,4,2)$ or $(h/2-10,6,4,2)$. 
\end{proof}

\begin{lemma}\label{Lemma double shift}
	Suppose $h$ is a positive integer. The symmetric function $s_{(1^h)}\circ s_{(2)} + s_{(1^h)}\circ s_{(1^2)}$ is multiplicity-free. 
\end{lemma}

\begin{proof}
	As both summands are multiplicity-free we need to check that there is no shift symmetric partition $\lambda$ such that $\lambda'$ is also shift symmetric. This is clear since such $\lambda$ would satisfy $\lambda_1-\lambda'_1=1$ and $\lambda'_1-\lambda_1=1$, a contradiction.
\end{proof}

\begin{lemma}\label{Lemma even and two}
	Let $h\geq 2$. The symmetric function $s_{(1^2)}(s_{(h-1)}\circ s_{(2)}) + s_{(2)}(s_{(h-1)}\circ s_{(1^2)})$ is \emph{not} multiplicity-free.
\end{lemma}

\begin{proof}
	Apply Lemma~\ref{Lemma even and strip} with $k=2$ and $h$ replaced with $h-1$.
\end{proof}

\begin{corollary}\label{Cor S2 wr Sm class}
	Let $h\geq 33$.
	\begin{enumerate}[label=\textnormal{(\roman*)}]
		\item The irreducible induced-multiplicity-free characters of $G=(S_2\wr S_h)\cap A_{2h}$ are $\left( \charwr{\chi^{(2)}}{\chi^{(1^h)}}{h}\right)_G$ and for odd $h$ also $\left( \charwr{\chi^{(2)}}{\chi^{(h)}}{h}\right)_G$.
		\item The groups $S_2\wr A_h, T_{2,h}$ and $T_{2,h}\cap A_{2h}$ are not multiplicity-free.
	\end{enumerate}
\end{corollary}

\begin{proof}
	We obtain (i) from Lemma~\ref{Lemma double even}, Lemma~\ref{Lemma double shift} and Lemma~\ref{Lemma even and two}. The groups $S_2\wr A_h$ and $T_{2,h}$ are not multiplicity-free because of Lemma~\ref{Lemma even and shift} and Lemma~\ref{Lemma conjugate even and shift}, respectively. Finally, $T_{2,h}\cap A_{2h}$ is contained in $T_{2,h}$; thus it is not multiplicity-free.
\end{proof}

\subsection{$S_2\wr S_h$ embedded in $S_{2h+1}$}

We now consider the same subgroups of $S_2\wr S_h$ as in \S\ref{Sec S2 wr Sh} but this time embedded in $S_{2h+1}$. Clearly, if $\rho$ is an induced-multiplicity-free character of one of these subgroups $G$, then $\rho\ind^{S_{2h}}$ is multiplicity-free. Thus, by Corollary~\ref{Cor S2 wr Sm class}, the only possible multiplicity-free subgroup is $G=(S_2\wr S_h)\cap A_{2h}$ and the only possible choices of an induced-multiplicity-free character of $G$ are $\left( \charwr{\chi^{(2)}}{\chi^{\nu}}{h}\right)_G$ with $\nu$ linear.

As in the previous section, we use Lemma~\ref{Lemma even and shift rules}. The next lemma appears with a different proof in \cite[\S2.4]{WildonMultiplicity-free09} and \cite[Proposition~3.6]{GodsilMeagherMultiplicity-free10}.

\begin{lemma}\label{Lemma 1+double even}
	Let $h$ be a positive integer. The symmetric function $s_{(1)}(s_{(h)}\circ s_{(2)} + s_{(h)}\circ s_{(1^2)})$ is \emph{not} multiplicity-free.
\end{lemma}

\begin{proof}
	Apply Lemma~\ref{Lemma even and strip} with $k=1$.
\end{proof}

\begin{lemma}\label{Lemma 1+double shift}
	Let $h$ be a positive integer. The symmetric function $s_{(1)}(s_{(1^h)}\circ s_{(2)} + s_{(1^h)}\circ s_{(1^2)})$ is not multiplicity-free.
\end{lemma}

\begin{proof}
	The Schur function labelled by $(h+1,1^h)$ appears with multiplicity two by the induction branching rule, since it can be obtained by adding a box to $ss[(h)] = (h+1,1^{h-1})$ as well as to its conjugate $(h,1^h)$.
\end{proof}

\begin{corollary}\label{Cor S1 times S2 wr Sm class}
	Let $h\geq 33$. The groups $S_2\wr A_h, (S_2\wr S_h)\cap A_{2h}, T_{2,h}$ and $T_{2,h}\cap A_{2h}$ embedded in $S_{2h+1}$ are not multiplicity-free.
\end{corollary}

\begin{proof}
	This follows from Corollary~\ref{Cor S2 wr Sm class}, Lemma~\ref{Lemma 1+double even} and Lemma~\ref{Lemma 1+double shift}.
\end{proof}

\subsection{Sporadic cases}\label{Section sporadic}

We now consider groups $(S_k\times \mathrm{AGL}_1(\mathbb{F}_5))\cap A_{k+5}$ and $(S_k\times \mathrm{PGL}_2(\mathbb{F}_5))\cap A_{k+6}$, which are kernels of the sign characters restricted to $S_k\times \mathrm{AGL}_1(\mathbb{F}_5)$ and $S_k\times \mathrm{PGL}_2(\mathbb{F}_5)$, respectively, as well as the groups $N_k=\ker\left( \sgn\boxtimes \left( \charwrnb{\mathbbm{1}}{\sgn}{2}\right)\right) \leq S_k\times S_2\wr S_2$ and $T_{k,2,h}=\ker\left( \sgn\boxtimes \left( \charwrnb{\sgn}{\sgn}{h}\right)\right)\leq S_k\times S_2\wr S_h$ for $h\in \left\lbrace 3,4 \right\rbrace$. These groups together with groups from Proposition~\ref{Prop sk times primitive} and Theorem~\ref{Theorem direct product arbitrary}(ii) and (iii) cover both families in Proposition~\ref{Prop possible MF subgroups}(viii) and (ix) since $(S_k\times \mathrm{P}\Gamma\mathrm{L}_2(\mathbb{F}_8))\cap A_{k+9}=A_k\times \mathrm{P}\Gamma\mathrm{L}_2(\mathbb{F}_8)$ and $(S_k\times \mathrm{ASL}_3(\mathbb{F}_2))\cap A_{k+8} = A_k\times \mathrm{ASL}_3(\mathbb{F}_2)$.

Write $\psi_R$, respectively, $\psi_C$ for the trivial character, respectively, one of the linear complex conjugate characters of $\mathrm{AGL}_1(\mathbb{F}_5)$. Note that when we tensor $\psi_R$, respectively, $\psi_C$ with the restriction of the sign of $S_5$, we obtain the other real, respectively, complex linear character. For $n\geq 18$, using Theorem~\ref{Theorem direct product k geq 2} (and Proposition~\ref{Prop sk times primitive}), the possible irreducible induced-multiplicity-free characters of $(S_k\times \mathrm{AGL}_1(\mathbb{F}_5))\cap A_{k+5}, N_k$ and $T_{k,2,h}$ with $h\in \left\lbrace 3,4\right\rbrace $ are those listed in the following proposition once the word `non-square' is removed.

\begin{proposition}\label{Prop Sk times small easy class}
	Let $n\geq 36$. The irreducible induced-multiplicity-free characters of
	\begin{enumerate}[label=\textnormal{(\roman*)}]
		\item $G=(S_k\times \mathrm{AGL}_1(\mathbb{F}_5))\cap A_{k+5}$ are $(\chi^{\lambda}\boxtimes \psi_R)_{G}$ with $\lambda$ non-square rectangular and $(\chi^{\lambda}\boxtimes \psi_C)_{G}$ with $\lambda$ linear or $2$-rectangular;
		\item $G=N_k$ are $\left( \chi^{\lambda}\boxtimes \left( \charwr{\chi^{\mu}}{\chi^{(1^2)}}{2}\right) \right) _{G}$ such that $\lambda=(k)$ and $\mu=(2)$ \emph{or} $\lambda=(1^k)$ and $\mu=(1^2)$;
		\item $G=T_{k,2,h}$ with $h\in\left\lbrace 3,4\right\rbrace $ are $\left( \chi^{(1^k)}\boxtimes \left( \charwr{\chi^{(2)}}{\chi^{\nu}}{h}\right) \right) _{G}$ with $\nu$ linear.
	\end{enumerate}
\end{proposition}

\begin{proof}
	As mentioned we need to show that the character $(\chi^{\lambda}\boxtimes \psi_R)_{G}$ with $G=(S_k\times \mathrm{AGL}_1(\mathbb{F}_5))\cap A_{k+5}$ and $\lambda$ square is not induced-multiplicity-free, while all the characters in the statement are induced-multiplicity-free. Suppose first that $\lambda=(a^a)$. The character $\psi_R\ind^{S_5}$ decomposes as $\chi^{(5)} + \chi^{(2^2,1)}$. Thus $\left( (\chi^{\lambda}\boxtimes \psi_R)_{G}\right) \Ind^{S_n}$ corresponds to $(s_{(5)}+s_{(2^2,1)})s_{\lambda}+(s_{(1^5)}+s_{(3,2)})s_{\lambda'}$. Since both $s_{\lambda}s_{(2^2,1)}$ and $s_{\lambda'}s_{(3,2)}$ have a constituent labelled by $(a+2,a+1,a^{a-2},2)$, using the $(a+2,a+1,a^{a-2},2)/(a^a)$-tableaux with reading words $11232$ and $11221$, respectively, the induced character is not multiplicity-free, as required.
	
	The remaining characters induced to $S_n$ also correspond to symmetric functions of the form $xs_{\lambda}+ys_{\lambda'}$ with both summands multiplicity-free and $x$ and $y$ being of homogeneous degree $4,5,6$ or $8$. If $\lambda$ is linear, we obtain that the sum is multiplicity-free directly from Lemma~\ref{Lemma rotate linear}, while for $\lambda$ non-square rectangular (appearing only in (i)) we firstly use Lemma~\ref{Lemma difference of rectangles} with $l=5$ and then Lemma~\ref{Lemma rotate linear} to deduce the sum is multiplicity-free.
\end{proof}

We are left with the group $G=(S_k\times \mathrm{PGL}_2(\mathbb{F}_5))\cap A_{k+6}$. In the remainder of the section write $\mathbbm{1}$ for the trivial character of $\mathrm{PGL}_2(\mathbb{F}_5)$. Proposition~\ref{Prop sk times primitive}(iii) shows that for $n\geq 18$ the potential irreducible induced-multiplicity-free characters of $G$ are $(\chi^{\lambda}\boxtimes \mathbbm{1})_{G}$ with $\lambda$ rectangular, column-near rectangular or a column-unbalanced fat hook. Inspired by the above proof, we start by establishing results about column-near rectangular partitions and column-unbalanced fat hooks analogous to Lemma~\ref{Lemma difference of rectangles}.  

\begin{lemma}\label{Lemma unbalanced fat hooks}
	Let $\lambda=(a^b,1^d)$ be a column-unbalanced fat hook of size at least $58$. If $|\lambda-\lambda'|\leq 12$, then either $a=b$ and $d\leq 6$ \emph{or} $b=1$ and $|d+1-a|\leq 6$.
\end{lemma}

\begin{proof}
	Suppose firstly that $b\geq a$, and let $\mu=(a^b)$. Then $12\geq|\lambda-\lambda'|=2d+|\mu-\mu'|$ and clearly $d\leq 6$. Lemma~\ref{Lemma difference of rectangles} shows that either $|\lambda|=|\mu|+d\leq (6-d)^2 + 6\leq 31$, which is impossible, or $\mu$ is a square partition, as required.
	
	Now suppose that $b<a$. If $b=1$, we get the required condition $|d+1-a|\leq 6$; thus assume that $b\geq 2$. Then $12\geq|\lambda-\lambda'|\geq 2\sum_{i=2}^{b}|\lambda_i-\lambda'_i|\geq 2(b-1)(a-b)$. We conclude that $a\leq 8$ with equality if and only if $b=2$ and $d=6$ \emph{or} $b=7$ and $d=1$, which is impossible as $|\lambda|\leq 57$ in both cases. Thus $a\leq 7$, and consequently $d=|\lambda|-ab\geq 58-42\geq 16$. This is again impossible as $12\geq|\lambda-\lambda'|\geq 2|\lambda_1-\lambda'_1|=2|d+b-a|\geq 2(16+1-7)=20$. 
\end{proof}

\begin{lemma}\label{Lemma near rectangles}
	Let $\lambda=((a+1)^b,a^d)$ be a column-near rectangular partition of size at least $56$. If $|\lambda-\lambda'|\leq 12$, then either $a=b+d$ and $b\leq 6$ or $a+1=b+d$ and $d\leq 7$.
\end{lemma}

\begin{proof}
	Let $c=b+d$, and suppose firstly that $c<a$. If $\mu=(a^c)$, we obtain $12\geq|\lambda-\lambda'|=2b+|\mu-\mu'|$; thus $b\leq 6$. By Lemma~\ref{Lemma difference of rectangles} we also conclude that $|\lambda|=|\mu|+b\leq (6-b)^2+6\leq 31$, a contradiction.  If $c=a$, then $|\lambda-\lambda'|=2b$ which recovers $b\leq 6$. Similarly, for $c=a+1$ we obtain $d\leq 7$.
	
	Finally, suppose that $c>a+1$. Then $12\geq|\lambda-\lambda'|\geq 2\sum_{i=1}^{a} |\lambda_i-\lambda'_i| \geq 2(c-a-1)a$. Thus $c\leq 8$, and consequently $|\lambda|< c(a+1)\leq 56$, a contradiction.
\end{proof}

The next lemma concerns partitions $\lambda$ in Lemma~\ref{Lemma unbalanced fat hooks} and Lemma~\ref{Lemma near rectangles} for which $|\lambda-\lambda'|\leq 12$, meaning that Lemma~\ref{Lemma rotate linear} with $l=6$ does not apply to them. We recall the symmetric function $x_P=s_{(6)} +s_{(2^3)}$ which corresponds to $\mathbbm{1}\ind^{S_6}$, and write $x_P(\lambda)$ for $x_Ps_{\lambda}+\omega(x_P)s_{\lambda'}$.

\begin{lemma}\label{Lemma technical proof}
	Let $a$ be a positive integer and $\lambda$ be a partition of size at least $11$.
	\begin{enumerate}[label=\textnormal{(\roman*)}]
		\item If $\lambda=(a^a,1^i)$ with $1\leq i\leq 6$, then $x_P(\lambda)$ is multiplicity-free if and only if $i=3$.
		\item If $\lambda=(a,1^{a+i-1})$ with $-6\leq i\leq 6$ (and $a>1-i$), then $x_P(\lambda)$ is multiplicity-free if and only if $i\leq -2$.
		\item If $\lambda=((a+1)^i,a^{a-i})$ with $1\leq i\leq 6$ (and $a>i$), then $x_P(\lambda)$ is multiplicity-free if and only if $i=3$.
		\item If $\lambda=((a+1)^{a+1-i},a^i)$ with $1\leq i\leq 7$ (and $a\geq i$), then $x_P(\lambda)$ is multiplicity-free if and only if $i\geq 4$.
	\end{enumerate} 
\end{lemma}

\begin{proof}
	Note that $|\lambda|\geq 11$ forces $a\geq 3$ in all four cases. For each $\lambda$ such that $x_P(\lambda)$ is claimed not to be multiplicity-free we present a partition $\mu$ which labels constituents of $x_Ps_{\lambda}$ and $\omega(x_P)s_{\lambda'}$ in Table~\ref{Table technical}, proving that $x_P(\lambda)$ is not multiplicity-free.
	
	\begin{table}[h!]
		\centering
		\begin{tabular}{ccccc}
			\toprule
			$\lambda$&$i$&$\mu$&$r_1$&$r_2$\\
			\midrule
			&$1$&$((a+2)^2,a^{a-2},2,1)$&$121233$&$121212$\\
			&$2$&$((a+2)^3,a^{a-3},1^2)$&$123123$&$121212$\\
			$(a^a,1^i)$&$4$&$(a+5,a^{a-1},1^5)$&$111111$&$123456$\\
			&$5$&$(a+6,a^{a-1},1^5)$&$111111$&$123456$\\
			&$6$&$(a+6,a^{a-1},1^6)$&$111111$&$123456$\\
			\midrule
			&$-1$&$(a+1,3,2,1^{a-2})$&$121323$&$112122$\\
			&$0$&$(a+1,3^2,1^{a-2})$&$123123$&$112122$\\
			&$1$&$(a+2,3,2,1^{a-1})$&$112233$&$121212$\\
			$(a,1^{a+j-1})$&$2$&$(a+2,3^2,1^{a-1})$&$112323$&$121212$\\
			&$3$&$(a+4,2,1^{a+2})$&$111111$&$123456$\\
			&$4$&$(a+5,1^{a+4})$&$111111$&$123456$\\
			&$5$&$(a+5,2,1^{a+3})$&$111111$&$123456$\\
			&$6$&$(a+6,1^{a+5})$&$111111$&$123456$\\
			\midrule
			&$1$&$((a+2)^2,a^{a-2},2,1)$&$121323$&$121212$\\
			&$2$&$(a+3,a+1,a^{a-2},2^2)$&$112323$&$111222$\\
			$((a+1)^i,a^{a-i})$&$4$&$((a+1)^5,a^{a-5},5)$&$111111$&$123456$\\
			&$5$&$((a+1)^6,a^{a-6},5)$&$111111$&$123456$\\
			&$6$&$((a+1)^6,a^{a-6},6)$&$111111$&$123456$\\
			\midrule
			&$1$&$(a+3,a+2,(a+1)^{a-2},a,2,1)$&$112323$&$112212$\\
			$((a+1)^{a+1-i},a^i)$&$2$&$((a+2)^2,(a+1)^{a-1},2)$&$121233$&$121122$\\
			&$3$&$((a+2)^2,(a+1)^{a-2},a,2)$&$121233$&$121122$\\
			\bottomrule
		\end{tabular}
		\vspace{4pt}
		\caption{Given a partition $\lambda$, the table shows a partition $\mu$ such that $s_{\mu}$ is a constituent of $x_Ps_{\lambda}$ and $\omega(x_P)s_{\lambda'}$. It also presents the reading words $r_1$ and $r_2$ of a $\mu/\lambda$-tableau, respectively, a $\mu/\lambda'$-tableau required by Theorem~\ref{Theorem LR rule} to show that $s_{\mu}$ is a constituent of $x_Ps_{\lambda}$, respectively, $\omega(x_P)s_{\lambda'}$.}
		\label{Table technical}
	\end{table}
	
	We now prove that if $\lambda=(a^a,1^3)$, then $x_P(\lambda)$ is multiplicity-free. Suppose for contradiction that there is a partition $\mu$ of size $a^2+9$ such that $s_{\mu}$ is constituent of the multiplicity-free symmetric functions $s_{\lambda}(s_{(6)}+s_{(2^3)})$ and $s_{\lambda'}(s_{(1^6)}+s_{(3^2)})$. From $\lambda'_1=a+3$ we obtain $\mu_1\geq a+3$. Therefore $s_{\mu}$ cannot be a constituent of $s_{\lambda}s_{(2^3)}$; thus it is a constituent of $s_{\lambda}s_{(6)}$. Hence it is of the form $(a+3+u,a^{(a-1)},1+v,1^{2+w})$ with $u+v+w=3$ and $w\leq 1$. Since the same must be true for $\mu'$, we conclude that $v=0$ and $u\leq 1$, a contradiction. For the remaining $\lambda$ the argument that $x_P(\lambda)$ is multiplicity-free is similar.
\end{proof}

We now establish the desired classification for $(S_k\times \mathrm{PGL}_2(\mathbb{F}_5))\cap A_{k+6}$. Unfortunately, due to the many cases in Lemma~\ref{Lemma technical proof}, it does not have a particularly unified form.

\begin{proposition}\label{Prop Prop Sk times small hard class}
	Let $k\geq 58$. The irreducible induced-multiplicity-free characters of $G=(S_k\times \mathrm{PGL}_2(\mathbb{F}_5))\cap A_{k+6}$ are $\left( \chi^{\lambda}\boxtimes \mathbbm{1}\right) _G$ with $\lambda$ non-square rectangular, column-near rectangular or a column-unbalanced fat hook with the exceptions given by $\lambda=(a^a,1^i),((a+1)^i,a^{a-i})$ with $i\in \left\lbrace 1,2,4,5,6 \right\rbrace $, $\lambda=(a,1^{a+j-1})$ with $-1\leq j\leq 6$ and $\lambda=((a+1)^{a+1-l},a^l)$ with $1\leq l\leq 3$ (for some positive integer $a$).
\end{proposition}

\begin{proof}
	From Proposition~\ref{Prop sk times primitive}(iii) we know that $\lambda$ must be rectangular, column-near rectangular or a column-unbalanced fat hook. Recall that the character $\left( \left( \chi^{\lambda}\boxtimes \mathbbm{1}\right) _G\right) \Ind^{S_k+6}$ corresponds to $x_P(\lambda)=x_Ps_{\lambda}+\omega(x_P)s_{\lambda'}$. If $\lambda$ is non-square rectangular, we can apply Lemma~\ref{Lemma difference of rectangles} and Lemma~\ref{Lemma rotate linear}. If $\lambda$ is a square partition, say $(a^a)$, we claim that $s_{\mu}$ is a common constituent of $x_Ps_{\lambda}$ and $\omega(x_P)s_{\lambda'}$ for $\mu=(a+2,a+1,a^{a-2},2,1)$. Indeed, this follows from the Littlewood--Richardson rule stated in Theorem~\ref{Theorem LR rule} using the $\mu/\lambda$-tableaux with reading words $112323$ and $112212$.
	
	Now suppose that $\lambda=(a^b,1^d)$ is a column-unbalanced fat hook. We can apply Lemma~\ref{Lemma unbalanced fat hooks} and Lemma~\ref{Lemma rotate linear} to conclude that either $a=b$ and $d\leq 6$ \emph{or} $b=1$ and $|d+1-a|\leq 6$. We can use Lemma~\ref{Lemma technical proof}(i) and (ii) to get two of our exceptional families (namely the first and third). Finally, for $\lambda$ column-near rectangular we proceed similarly using Lemma~\ref{Lemma near rectangles} instead of Lemma~\ref{Lemma unbalanced fat hooks}. The remaining exceptions then arise from Lemma~\ref{Lemma technical proof}(iii) and (iv). 
\end{proof}  

\section{Rectangular and almost rectangular partitions}\label{Sec rectangles}

The main objects of this section are $(a,b)$-birectangular partitions, that is the partitions $\lambda$ of size $2ab$ such that for all $1\leq i\leq 2b$ we have $\lambda_i+ \lambda_{2b+1-i}=2a$ as defined in Definition~\ref{Definition birectangular}(i). Recall from Remark~\ref{Remark birectangular} that any $(a,b)$-birectangular partition has length at most $2b$. In fact, the condition $|\lambda|=2ab$ in the definition of $(a,b)$-birectangular partitions can be replaced with $\ell(\lambda)\leq 2b$. Using Proposition~\ref{Prop LR rectangles}(i), one can alternatively define $(a,b)$-birectangular partitions as labels of constituents of $s_{(a^b)}^2$, and consequently we see that the conjugates of $(a,b)$-birectangular partitions are $(b,a)$-birectangular. With this in mind, let us prove the combinatorial result from the introduction, which we restate below for the reader's convenience.

\setcounter{section}{1}\setcounter{theorem}{2}
\begin{proposition}
	Let $a>b$ be positive integers and write $d=a-b$. There is a partition $\lambda$ which is both $(a,b)$-birectangular and $(b,a)$-birectangular if and only if $d|a$. Moreover, in such a case there is a unique such $\lambda$ given by $((2b)^{2d}, (2b-2d)^{2d}, \dots, (2d)^{2d})$.
\end{proposition}

\begin{proof}
	Suppose that such a partition $\lambda$ exists. We claim that for all $i\leq 2b$ we have $\lambda_{i+2d} = \lambda_i -2d$. Indeed, we have $\lambda_i + \lambda_{2b+1-i}=2a$, and since $2b+1-i\leq 2a$ and $\lambda$ is $(b,a)$-birectangular, we also have $\lambda_{2b+1-i} + \lambda_{2a-2b+i} =2b$. Subtracting these equalities yields the claimed equality.
	
	Let us now observe that $\lambda_i=2b$ for $1\leq i\leq 2d$. Indeed, since $\lambda$ is $(b,a)$-birectangular $\lambda_i = 2b-\lambda_{2a+1-i}$. Moreover $\ell(\lambda)\leq 2b< 2a+1-i$ as $\lambda$ is $(a,b)$-birectangular; thus $\lambda_{2a+1-i}=0$, and in turn $\lambda_i=2b$, as needed. Now, to show the `only if' part pick an integer $q$ such that $2b<2qd+1< 2a$. Using $\lambda_1=2b$, $\ell(\lambda)\leq 2b$ and $\lambda_{i+2d} = \lambda_i -2d$ with $i=1,1+2d,\dots, 1+2(q-1)d\leq 2b$ yields $0=\lambda_{2qd+1} = \lambda_1 - 2qd=2b-2qd$. Thus $d|b$, or equivalently, $d|b+d$ which is $d|a$.
	
	In the case $d|a$, we still know that $\lambda_i=2b$ for $1\leq i\leq 2d$ and $\lambda_{i+2d} = \lambda_i -2d$ for all $i\leq 2b$. Thus a simple induction on a non-negative integer $j\leq b/d=a/d -1$ shows that for all $1\leq i\leq 2d$ we have $\lambda_{i+2jd}=2b-2jd$. Omitting the final $2b$ zeros, we conclude that $\lambda=((2b)^{2d}, (2b-2d)^{2d}, \dots, (2d)^{2d})$. It is easy to see that this partition is $(a,b)$-birectangular and self-conjugate; thus it satisfies the required constraints. 
\end{proof}
\setcounter{section}{6}\setcounter{theorem}{0}

For the next result let us also recall $\left\lbrace a,b\right\rbrace $-birectangular partitions from Definition~\ref{Definition birectangular}(ii), which label constituents of $s_{(a^b)}s_{(b^a)}$. For $a>b$ they are the partitions $\lambda$ of size $2ab$ such that for all $1\leq i\leq b$ we have $\lambda_i + \lambda_{a+b+1-i}=a+b$ and for $b+1\leq i\leq a$ we have $\lambda_i=b$. 

\begin{proposition}\label{Prop conjugate set ab-birectangles}
	Let $a>b$ be positive integers and write $d=a-b$. There is a partition $\lambda$ of $2ab$ which is both $(a,b)$-birectangular and $\left\lbrace a,b\right\rbrace $-birectangular if and only if $d|a$. Moreover, in such a case there is a unique such $\lambda$ given by $((a+b)^d, (a+b-d)^d, \dots,(b+3d)^d, (b+2d)^d, b^d, (b-d)^d,\dots, d^d)$.
\end{proposition}

\begin{proof}
	Suppose that such $\lambda$  exists. For $1\leq i\leq b-d$ or $b+1\leq i\leq 2b$ we have $\lambda_i + \lambda_{2b+1-i}=2a$ and $\lambda_{2b+1-i} + \lambda_{a-b+i}=a+b$. In turn $\lambda_{i+d} = \lambda_i -d$. Similarly to the proof of Proposition~\ref{Prop conjugate (a,b)-birectangles}, this equality combined with $\lambda_{b+1}=b$ (which holds for $\left\lbrace a,b\right\rbrace $-birectangular partitions) and $\ell(\lambda)\leq 2b$ gives $d\mid b$, or equivalently, $d\mid a$.
	
	Since, in fact, $\lambda_i=b$ for all $b+1\leq i\leq a$, for $d\mid a$ we conclude that the final $b$ parts $(\lambda_{b+1},\lambda_{b+2},\dots,\lambda_{2b})$ equal $(b^d, (b-d)^d,\dots, d^d)$. Using the equalities $\lambda_i=a+b-\lambda_{a+b+1-i}$ for $1\leq i\leq b$ we conclude that $\lambda=((a+b)^d, (a+b-d)^d, \dots,(b+3d)^d, (b+2d)^d, b^d, (b-d)^d,\dots, d^d)$. It is easily seen that this partition is $(a,b)$-birectangular and $\left\lbrace a,b\right\rbrace $-birectangular, finishing the proof.
\end{proof}

We can adapt the last two proofs to obtain a stronger version of the main part of Proposition~\ref{Prop conjugate (a,b)-birectangles}. 

\begin{proposition}\label{Prop conjugate birectangles and boxes}
	Let $t$ be a non-negative integer. Suppose that $a>b$ are positive integers and write $d=a-b$ and $0\leq r\leq d-1$ for the remainder of $a$ modulo $d$. There is an $(a,b)$-birectangular partition $\mu$ and a $(b,a)$-rectangular partition $\nu$ such that $|\mu-\nu|\leq 2t$ if and only if $t\geq 2r(d-r)$.
\end{proposition}

\begin{proof}
	Suppose that $\mu$ and $\nu$ are such partitions. For any positive integer $i$ define $\varrho_i=\mu_i-\nu_i$. Extend $\mu=(\mu_1,\mu_2,\dots)$ to a sequence $(\mu_{1-2d}, \mu_{2-2d}, \dots)$ by letting $\mu_i=2a$ for $1-2d\leq i\leq 0$. Since $\ell(\mu)\leq 2b$, the equality $\mu_i + \mu_{2b+1-i}=2a$ holds for all $1-2d\leq i\leq 2a$. We claim that for any $1-2d\leq i\leq 2b$ we also have 
	\begin{equation}\label{Eq 2d move}
	\mu_i = \mu_{i+2d} + 2d - (\varrho_{2b+1 - i} + \varrho_{i+2d}).
	\end{equation}
	Indeed, using the inequalities $1\leq 2b+1-i\leq 2a$ and $1\leq i+2d$, we can write $\mu_i = \left( \mu_i + \mu_{2b+1-i}\right)  - (\mu_{2b+1-i} -\nu_{2b+1-i}) - (\nu_{2b+1-i} + \nu_{i+2d}) - (\mu_{i+2d} - \nu_{i+2d}) + \mu_{i+2d} = 2a - \varrho_{2b+1 - i} - 2b - \varrho_{i+2d} + \mu_{i+2d}$, which is the required equation.
	
	Write $a=qd+r$ for some integer $q$ (hence $b=(q-1)d + r$). For $1-2d\leq j\leq 2r-2d$ we have $j+2qd\leq 2b<j+2(q+1)d$, and therefore we can add the equalities (\ref{Eq 2d move}) with $i=j, j+2d, \dots, j+2qd$ to get $\mu_j = \mu_{j+2(q+1)d} + 2(q+1)d - \sum_{e\in \mathcal{S}_j} \varrho_e= 2(q+1)d - \sum_{e\in \mathcal{S}_j} \varrho_e$, where $\mathcal{S}_j$ denotes the multiset $\left\lbrace 2b+1-i, i+2d : i=j, j+2d, \dots, j+2qd\right\rbrace $ of positive integers. Since $j$ is non-positive we have $\mu_j=2a$; thus $\sum_{e\in \mathcal{S}_j} \varrho_e=2d-2r$. 
	
	For $2r-2d<j\leq 0$ we have $j+2(q-1)d\leq 2b<j+2qd$, and in a similar manner we obtain $2a=\mu_j = \mu_{j+2qd} + 2qd - \sum_{e\in \mathcal{S}_j} \varrho_e=2qd - \sum_{e\in \mathcal{S}_j} \varrho_e$, where  $\mathcal{S}_j$ is the multiset $\left\lbrace 2b+1-i,i+2d : i=j, j+2d, \dots, j+2(q-1)d\right\rbrace $ of positive integers. Rearranging this equation yields $\sum_{e\in \mathcal{S}_j} \varrho_e=-2r$.
	
	Note the elements `$2b+1-i$' of sets $S_j$ with $1-2d\leq j\leq 0$ are pairwise distinct and the same holds for the elements `$i+2d$'. Thus any positive integer has multiplicity at most two in the union of all the $S_j$ over $1-2d\leq j\leq 0$. Using the triangle inequality, we can therefore write
	\begin{align*}
	4t &\geq 2|\mu-\nu|=2\sum_{e\geq 1} |\varrho_e| \geq\sum_{j=1-2d}^{0} \sum_{e\in \mathcal{S}_j}|\varrho_e|\\ &\geq\sum_{j=1-2d}^{2r-2d} \left| \sum_{e\in \mathcal{S}_j}\varrho_e\right|  + \sum_{j=1+2r-2d}^{0} \left| \sum_{e\in \mathcal{S}_j}\varrho_e\right| \\
	&=2r\times 2(d-r) + 2(d-r)\times 2r\\
	&= 8r(d-r),
	\end{align*}
	which gives the desired condition $t\geq 2r(d-r)$.
	
	This condition is indeed sufficient as the partition $\mu = (((2q-1)d + r)^{2r}, ((2q-2)d + r)^{2(d-r)}, ((2q-3)d + r)^{2r}, ((2q-4)d + r)^{2(d-r)},\dots, (d+r)^{2r})$ (of length $2rq +2(d-r)(q-1)=2b$) is $(a,b)$-rectangular, the partition $\nu$ obtained from $\mu$ by replacing its first $2d$ parts with $((2q-2)d+2r)^{2d}$ is $(b,a)$-rectangular and $|\mu-\nu| = \sum_{i= 1}^{2r}(\mu_i-\nu_i) + \sum_{i= 2r+1}^{2d}(\nu_i-\mu_i)=4r(d-r)$.
\end{proof}

We rephrase this into a more practical result.

\begin{corollary}\label{Cor rectangles}
	Let $t$ be a non-negative integer. Suppose that $a>b$ are positive integers. Write $d=a-b$ and $0\leq r\leq d-1$ for the remainder of $a$ modulo $d$.
	\begin{enumerate}[label=\textnormal{(\roman*)}]
		\item There is a partition $\lambda$ of $2ab+t$ such that for some $(a,b)$-birectangular partition $\mu$ and some $(b,a)$-birectangular partition $\nu$ we have $\mu, \nu\subseteq \lambda$ if and only if $t\geq 2r(d-r)$.
		\item Suppose further that $t\leq 2ab$. There is a partition $\lambda$ of $2ab-t$ such that for some $(a,b)$-birectangular partition $\mu$ and some $(b,a)$-birectangular partition $\nu$ we have $\lambda\subseteq \mu, \nu$ if and only if $t\geq 2r(d-r)$.
	\end{enumerate}
\end{corollary}

\begin{proof}
	Use Lemma~\ref{Lemma contains two partitions} and Proposition~\ref{Prop conjugate birectangles and boxes}.
\end{proof}

We end this section by proving a result similar to Lemma~\ref{Lemma Stembridge observation}, however, this time we focus solely on almost rectangular partitions. In the proof we need to distinguish between almost rectangular partitions of the form $(a^{b-1},a-1)$ and $(a^b,1)$, and almost rectangular partitions of the form $(a+1,a^{b-1})$. To avoid repeating the same argument, we omit some of the details for the latter form.  

\begin{lemma}\label{Lemma constructions of partitions}
	Let $a,b,a'$ and $b'$ be positive integers such that $a<a'$ and $b+1<b'$. Suppose that $\mu$ and $\nu$ are almost rectangular partitions which belong to the list:
	\begin{enumerate}[label=\textnormal{(\roman*)}]
		\item $\mu=(a^{b-1},a-1)$ and $\nu=((a')^{b'-1},a'-1)$,
		\item $\mu=(a^b,1)$ and $\nu=((a')^{b'},1)$,
		\item $\mu=(a+1,a^{b-1})$ and $\nu=(a'+1,(a')^{b'-1})$.
	\end{enumerate}
	If $\alpha$ is a partition with $\alpha_1\leq 2a$, then there is a partition $\beta$ such that $c(\mu,\mu;\alpha)\leq c(\nu,\nu;\beta)$. Moreover, if $a'-a=b'-b$, $a+a'=2b'$ and $\alpha$ is self-conjugate, then $\beta$ can be chosen to be self-conjugate.
\end{lemma}

\begin{proof}
	We can assume that $c(\mu,\mu;\alpha)\neq 0$. Thus $\ell(\alpha)\leq 2\ell(\mu)\leq 2b+2\leq b+b'$. Let $\gamma=\alpha\sqcup ( (2a)^{b'-b})=((2a)^{b'-b}, \alpha_1,\alpha_2,\dots)$ and $\beta = \gamma+((a'-a)^{2b'})=((a+a')^{b'-b}, \alpha_1+a'-a,\alpha_2+a'-a,\dots,\alpha_{b+b'}+a'-a)$. We show that this $\beta$ satisfies the required conditions. Firstly, note the moreover part follows immediately from the definition of $\beta$. To prove $c(\mu,\mu;\alpha)\leq c(\nu,\nu;\beta)$, given a semistandard $\alpha/\mu$-tableau $t$ with weight $\mu$ and a latticed reading word we need to construct semistandard $\beta/\nu$-tableau $t''$ with weight $\nu$ and a latticed reading word in such a way that $t\mapsto t''$ is injective.
	
	We firstly consider the cases (i) and (ii). An example of the construction of $t''$ (and of an intermediate tableau $t'$) is given in Figure~\ref{Figure constructions}.
	
	\begin{figure}[h!]
		\begin{tikzpicture}[x=0.5cm, y=0.5cm]
		\begin{pgfonlayer}{nodelayer}
		\node [style=none] (0) at (-11, 7) {};
		\node [style=none] (1) at (-9, 7) {};
		\node [style=none] (2) at (-9, 5) {};
		\node [style=none] (3) at (-10, 5) {};
		\node [style=none] (4) at (-10, 4) {};
		\node [style=none] (5) at (-12, 4) {};
		\node [style=none] (6) at (-12, 5) {};
		\node [style=none] (7) at (-11, 5) {};
		\node [style=none] (8) at (-12, 3) {};
		\node [style=none] (9) at (-13, 3) {};
		\node [style=none] (10) at (-13, 2) {};
		\node [style=none] (11) at (-14, 2) {};
		\node [style=none] (14) at (-15, 4) {};
		\node [style=none] (15) at (-8, 2) {};
		\node [style=none] (18) at (-7, 0) {};
		\node [style=none] (19) at (-6, 0) {};
		\node [style=none] (20) at (-6, 1) {};
		\node [style=none] (21) at (-5, 1) {};
		\node [style=none] (22) at (-5, 2) {};
		\node [style=none] (23) at (-5, 3) {};
		\node [style=none] (24) at (-3, 2) {};
		\node [style=none] (25) at (-4, 3) {};
		\node [style=none] (26) at (-2, 5) {};
		\node [style=none] (27) at (-2, 3) {};
		\node [style=none] (28) at (-4, 7) {};
		\node [style=none] (29) at (0, 7) {};
		\node [style=none] (30) at (0, 5) {};
		\node [style=none] (31) at (-3, 3) {};
		\node [style=none] (32) at (3, 2) {};
		\node [style=none] (33) at (3, -1) {};
		\node [style=none] (35) at (4, 0) {};
		\node [style=none] (36) at (5, 0) {};
		\node [style=none] (37) at (5, 1) {};
		\node [style=none] (38) at (6, 1) {};
		\node [style=none] (39) at (6, 2) {};
		\node [style=none] (40) at (6, 3) {};
		\node [style=none] (41) at (8, 2) {};
		\node [style=none] (42) at (7, 3) {};
		\node [style=none] (43) at (9, 5) {};
		\node [style=none] (44) at (9, 3) {};
		\node [style=none] (45) at (7, 7) {};
		\node [style=none] (46) at (11, 7) {};
		\node [style=none] (47) at (11, 5) {};
		\node [style=none] (48) at (8, 3) {};
		\node [style=none] (49) at (1, 2) {};
		\node [style=none] (50) at (1, -3) {};
		\node [style=none] (51) at (3, -3) {};
		\node [style=none] (52) at (-10, 7) {};
		\node [style=none] (53) at (-9, 6) {};
		\node [style=none] (54) at (-11, 6) {};
		\node [style=none] (55) at (-11, 4) {};
		\node [style=none] (56) at (-13, 4) {};
		\node [style=none] (57) at (-14, 4) {};
		\node [style=none] (58) at (-15, 3) {};
		\node [style=none] (59) at (-15, 2) {};
		\node [style=none] (60) at (-7, 2) {};
		\node [style=none] (61) at (-6, 2) {};
		\node [style=none] (62) at (-8, 1) {};
		\node [style=none] (63) at (-8, 0) {};
		\node [style=none] (64) at (-4, 2) {};
		\node [style=none] (65) at (-4, 4) {};
		\node [style=none] (66) at (-2, 4) {};
		\node [style=none] (67) at (-4, 5) {};
		\node [style=none] (68) at (-4, 6) {};
		\node [style=none] (69) at (-3, 7) {};
		\node [style=none] (70) at (-2, 7) {};
		\node [style=none] (71) at (-1, 7) {};
		\node [style=none] (72) at (0, 6) {};
		\node [style=none] (73) at (-1, 5) {};
		\node [style=none] (74) at (8, 7) {};
		\node [style=none] (75) at (9, 7) {};
		\node [style=none] (76) at (10, 7) {};
		\node [style=none] (77) at (11, 6) {};
		\node [style=none] (78) at (10, 5) {};
		\node [style=none] (79) at (9, 4) {};
		\node [style=none] (80) at (7, 2) {};
		\node [style=none] (81) at (3, -2) {};
		\node [style=none] (82) at (2, -3) {};
		\node [style=none] (83) at (1, -2) {};
		\node [style=none] (84) at (1, -1) {};
		\node [style=none] (85) at (1, 0) {};
		\node [style=none] (86) at (1, 1) {};
		\node [style=none] (87) at (2, 2) {};
		\node [style=none] (88) at (4, 2) {};
		\node [style=none] (89) at (5, 2) {};
		\node [style=none] (90) at (7, 4) {};
		\node [style=none] (91) at (7, 5) {};
		\node [style=none] (92) at (7, 6) {};
		\node [style=none] (93) at (-10.5, 6.5) {$1$};
		\node [style=none] (94) at (-9.5, 6.5) {$1$};
		\node [style=none] (95) at (-9.5, 5.5) {$2$};
		\node [style=none] (96) at (-10.5, 5.5) {$2$};
		\node [style=none] (97) at (-10.5, 4.5) {$3$};
		\node [style=none] (98) at (-11.5, 4.5) {$1$};
		\node [style=none] (99) at (-14.5, 2.5) {$2$};
		\node [style=none] (101) at (-12.5, 3.5) {$3$};
		\node [style=none] (102) at (-14.5, 3.5) {$1$};
		\node [style=none] (103) at (-13.5, 3.5) {$2$};
		\node [style=none] (105) at (-3.5, 6.5) {$1$};
		\node [style=none] (106) at (-2.5, 6.5) {$1$};
		\node [style=none] (107) at (-1.5, 6.5) {$1$};
		\node [style=none] (108) at (-0.5, 6.5) {$1$};
		\node [style=none] (109) at (-0.5, 5.5) {$2$};
		\node [style=none] (110) at (-1.5, 5.5) {$2$};
		\node [style=none] (111) at (-2.5, 5.5) {$2$};
		\node [style=none] (112) at (-3.5, 5.5) {$2$};
		\node [style=none] (113) at (-2.5, 4.5) {$3$};
		\node [style=none] (114) at (-3.5, 4.5) {$3$};
		\node [style=none] (115) at (-2.5, 3.5) {$4$};
		\node [style=none] (116) at (-3.5, 3.5) {$4$};
		\node [style=none] (117) at (-3.5, 2.5) {$5$};
		\node [style=none] (118) at (-4.5, 2.5) {$3$};
		\node [style=none] (119) at (-5.5, 1.5) {$5$};
		\node [style=none] (120) at (-6.5, 1.5) {$4$};
		\node [style=none] (121) at (-6.5, 0.5) {$5$};
		\node [style=none] (122) at (-7.5, 1.5) {$3$};
		\node [style=none] (123) at (-7.5, 0.5) {$4$};
		\node [style=none] (125) at (7.5, 6.5) {$1$};
		\node [style=none] (126) at (8.5, 6.5) {$1$};
		\node [style=none] (127) at (9.5, 6.5) {$1$};
		\node [style=none] (128) at (10.5, 6.5) {$1$};
		\node [style=none] (129) at (10.5, 5.5) {$2$};
		\node [style=none] (130) at (9.5, 5.5) {$2$};
		\node [style=none] (131) at (8.5, 5.5) {$2$};
		\node [style=none] (132) at (7.5, 5.5) {$2$};
		\node [style=none] (133) at (7.5, 4.5) {$3$};
		\node [style=none] (134) at (8.5, 4.5) {$3$};
		\node [style=none] (135) at (8.5, 3.5) {$4$};
		\node [style=none] (136) at (7.5, 3.5) {$4$};
		\node [style=none] (137) at (7.5, 2.5) {$5$};
		\node [style=none] (138) at (6.5, 2.5) {$3$};
		\node [style=none] (139) at (5.5, 1.5) {$5$};
		\node [style=none] (140) at (4.5, 1.5) {$4$};
		\node [style=none] (141) at (4.5, 0.5) {$5$};
		\node [style=none] (142) at (3.5, 1.5) {$3$};
		\node [style=none] (143) at (3.5, 0.5) {$4$};
		\node [style=none] (145) at (2.5, 1.5) {$1$};
		\node [style=none] (146) at (1.5, 1.5) {$1$};
		\node [style=none] (147) at (2.5, 0.5) {$2$};
		\node [style=none] (148) at (1.5, 0.5) {$2$};
		\node [style=none] (149) at (2.5, -0.5) {$3$};
		\node [style=none] (150) at (1.5, -0.5) {$3$};
		\node [style=none] (151) at (2.5, -1.5) {$4$};
		\node [style=none] (152) at (1.5, -1.5) {$4$};
		\node [style=none] (153) at (2.5, -2.5) {$5$};
		\node [style=none] (154) at (1.5, -2.5) {$5$};
		\node [style=none] (155) at (-8.5, 3.5) {$\mapsto$};
		\node [style=none] (156) at (0.5, 3.5) {$\mapsto$};
		\node [style=none] (157) at (3, 0) {};
		\node [style=none] (158) at (-13.5, 2.5) {$3$};
		\node [style=none] (159) at (-9, -4) {};
		\node [style=none] (160) at (-11, -4) {};
		\node [style=none] (161) at (-9, -6) {};
		\node [style=none] (162) at (-10, -6) {};
		\node [style=none] (163) at (-10, -8) {};
		\node [style=none] (164) at (-11, -7) {};
		\node [style=none] (165) at (-14, -7) {};
		\node [style=none] (166) at (-13, -8) {};
		\node [style=none] (167) at (-14, -8) {};
		\node [style=none] (168) at (-13, -10) {};
		\node [style=none] (169) at (-15, -10) {};
		\node [style=none] (170) at (-15, -8) {};
		\node [style=none] (171) at (-10, -4) {};
		\node [style=none] (172) at (-9, -5) {};
		\node [style=none] (173) at (-11, -5) {};
		\node [style=none] (174) at (-11, -6) {};
		\node [style=none] (175) at (-10, -7) {};
		\node [style=none] (176) at (-11, -8) {};
		\node [style=none] (177) at (-12, -8) {};
		\node [style=none] (178) at (-12, -7) {};
		\node [style=none] (179) at (-13, -7) {};
		\node [style=none] (180) at (-13, -9) {};
		\node [style=none] (181) at (-14, -10) {};
		\node [style=none] (182) at (-15, -9) {};
		\node [style=none] (183) at (-10.5, -4.5) {$1$};
		\node [style=none] (184) at (-9.5, -4.5) {$1$};
		\node [style=none] (185) at (-10.5, -5.5) {$2$};
		\node [style=none] (186) at (-9.5, -5.5) {$2$};
		\node [style=none] (187) at (-10.5, -6.5) {$3$};
		\node [style=none] (188) at (-10.5, -7.5) {$4$};
		\node [style=none] (189) at (-11.5, -7.5) {$3$};
		\node [style=none] (190) at (-12.5, -7.5) {$1$};
		\node [style=none] (191) at (-13.5, -7.5) {$1$};
		\node [style=none] (192) at (-13.5, -8.5) {$2$};
		\node [style=none] (193) at (-13.5, -9.5) {$3$};
		\node [style=none] (194) at (-14.5, -8.5) {$2$};
		\node [style=none] (195) at (-14.5, -9.5) {$3$};
		\node [style=none] (196) at (-8.5, -7.5) {$\mapsto$};
		\node [style=none] (198) at (-4, -6) {};
		\node [style=none] (199) at (-2, -8) {};
		\node [style=none] (200) at (-3, -8) {};
		\node [style=none] (201) at (-3, -10) {};
		\node [style=none] (202) at (-4, -9) {};
		\node [style=none] (203) at (-7, -9) {};
		\node [style=none] (204) at (-6, -10) {};
		\node [style=none] (205) at (-7, -10) {};
		\node [style=none] (206) at (-6, -12) {};
		\node [style=none] (207) at (-8, -12) {};
		\node [style=none] (208) at (-8, -10) {};
		\node [style=none] (210) at (-2, -7) {};
		\node [style=none] (211) at (-4, -7) {};
		\node [style=none] (212) at (-4, -8) {};
		\node [style=none] (213) at (-3, -9) {};
		\node [style=none] (214) at (-4, -10) {};
		\node [style=none] (215) at (-5, -10) {};
		\node [style=none] (216) at (-5, -9) {};
		\node [style=none] (217) at (-6, -9) {};
		\node [style=none] (218) at (-6, -11) {};
		\node [style=none] (219) at (-7, -12) {};
		\node [style=none] (220) at (-8, -11) {};
		\node [style=none] (221) at (-3.5, -6.5) {$3$};
		\node [style=none] (222) at (-2.5, -6.5) {$3$};
		\node [style=none] (223) at (-3.5, -7.5) {$4$};
		\node [style=none] (224) at (-2.5, -7.5) {$4$};
		\node [style=none] (225) at (-3.5, -8.5) {$5$};
		\node [style=none] (226) at (-3.5, -9.5) {$6$};
		\node [style=none] (227) at (-4.5, -9.5) {$5$};
		\node [style=none] (228) at (-5.5, -9.5) {$3$};
		\node [style=none] (229) at (-6.5, -9.5) {$3$};
		\node [style=none] (230) at (-6.5, -10.5) {$4$};
		\node [style=none] (231) at (-6.5, -11.5) {$5$};
		\node [style=none] (232) at (-7.5, -10.5) {$4$};
		\node [style=none] (233) at (-7.5, -11.5) {$5$};
		\node [style=none] (234) at (-2, -6) {};
		\node [style=none] (235) at (-4, -4) {};
		\node [style=none] (236) at (0, -4) {};
		\node [style=none] (237) at (0, -6) {};
		\node [style=none] (238) at (-3, -4) {};
		\node [style=none] (239) at (-2, -4) {};
		\node [style=none] (240) at (-1, -4) {};
		\node [style=none] (241) at (0, -5) {};
		\node [style=none] (242) at (-1, -6) {};
		\node [style=none] (243) at (-4, -5) {};
		\node [style=none] (244) at (-3.5, -4.5) {$1$};
		\node [style=none] (245) at (-2.5, -4.5) {$1$};
		\node [style=none] (246) at (-1.5, -4.5) {$1$};
		\node [style=none] (247) at (-0.5, -4.5) {$1$};
		\node [style=none] (248) at (-0.5, -5.5) {$2$};
		\node [style=none] (249) at (-1.5, -5.5) {$2$};
		\node [style=none] (250) at (-2.5, -5.5) {$2$};
		\node [style=none] (251) at (-3.5, -5.5) {$2$};
		\node [style=none] (252) at (0.5, -7.5) {$\mapsto$};
		\node [style=none] (253) at (7, -6) {};
		\node [style=none] (254) at (9, -8) {};
		\node [style=none] (255) at (8, -8) {};
		\node [style=none] (256) at (8, -10) {};
		\node [style=none] (257) at (7, -9) {};
		\node [style=none] (258) at (4, -9) {};
		\node [style=none] (259) at (5, -10) {};
		\node [style=none] (261) at (5, -12) {};
		\node [style=none] (262) at (3, -12) {};
		\node [style=none] (264) at (9, -7) {};
		\node [style=none] (265) at (7, -7) {};
		\node [style=none] (266) at (7, -8) {};
		\node [style=none] (267) at (8, -9) {};
		\node [style=none] (268) at (7, -10) {};
		\node [style=none] (269) at (6, -10) {};
		\node [style=none] (270) at (6, -9) {};
		\node [style=none] (271) at (5, -9) {};
		\node [style=none] (272) at (5, -11) {};
		\node [style=none] (273) at (4, -12) {};
		\node [style=none] (275) at (7.5, -6.5) {$3$};
		\node [style=none] (276) at (8.5, -6.5) {$3$};
		\node [style=none] (277) at (7.5, -7.5) {$4$};
		\node [style=none] (278) at (8.5, -7.5) {$4$};
		\node [style=none] (279) at (7.5, -8.5) {$5$};
		\node [style=none] (280) at (7.5, -9.5) {$6$};
		\node [style=none] (281) at (6.5, -9.5) {$5$};
		\node [style=none] (282) at (5.5, -9.5) {$3$};
		\node [style=none] (283) at (4.5, -9.5) {$3$};
		\node [style=none] (284) at (4.5, -10.5) {$4$};
		\node [style=none] (285) at (4.5, -11.5) {$5$};
		\node [style=none] (286) at (3.5, -10.5) {$4$};
		\node [style=none] (287) at (3.5, -11.5) {$5$};
		\node [style=none] (288) at (9, -6) {};
		\node [style=none] (289) at (7, -4) {};
		\node [style=none] (290) at (11, -4) {};
		\node [style=none] (291) at (11, -6) {};
		\node [style=none] (292) at (8, -4) {};
		\node [style=none] (293) at (9, -4) {};
		\node [style=none] (294) at (10, -4) {};
		\node [style=none] (295) at (11, -5) {};
		\node [style=none] (296) at (10, -6) {};
		\node [style=none] (297) at (7, -5) {};
		\node [style=none] (298) at (7.5, -4.5) {$1$};
		\node [style=none] (299) at (8.5, -4.5) {$1$};
		\node [style=none] (300) at (9.5, -4.5) {$1$};
		\node [style=none] (301) at (10.5, -4.5) {$1$};
		\node [style=none] (302) at (10.5, -5.5) {$2$};
		\node [style=none] (303) at (9.5, -5.5) {$2$};
		\node [style=none] (304) at (8.5, -5.5) {$2$};
		\node [style=none] (305) at (7.5, -5.5) {$2$};
		\node [style=none] (306) at (2, -9) {};
		\node [style=none] (307) at (1, -10) {};
		\node [style=none] (308) at (3, -14) {};
		\node [style=none] (309) at (1, -14) {};
		\node [style=none] (310) at (2, -10) {};
		\node [style=none] (311) at (3, -9) {};
		\node [style=none] (312) at (1, -11) {};
		\node [style=none] (313) at (1, -12) {};
		\node [style=none] (314) at (1, -13) {};
		\node [style=none] (315) at (2, -14) {};
		\node [style=none] (316) at (3, -13) {};
		\node [style=none] (317) at (2.5, -12.5) {$4$};
		\node [style=none] (318) at (2.5, -13.5) {$5$};
		\node [style=none] (319) at (1.5, -13.5) {$5$};
		\node [style=none] (320) at (1.5, -12.5) {$4$};
		\node [style=none] (321) at (1.5, -11.5) {$3$};
		\node [style=none] (322) at (2.5, -11.5) {$3$};
		\node [style=none] (323) at (2.5, -10.5) {$2$};
		\node [style=none] (324) at (1.5, -10.5) {$2$};
		\node [style=none] (325) at (2.5, -9.5) {$1$};
		\node [style=none] (326) at (3.5, -9.5) {$1$};
		\end{pgfonlayer}
		\begin{pgfonlayer}{edgelayer}
		\draw [style=Border edge] (0.center) to (1.center);
		\draw [style=Border edge] (1.center) to (2.center);
		\draw [style=Border edge] (2.center) to (3.center);
		\draw [style=Border edge] (3.center) to (4.center);
		\draw [style=Border edge] (4.center) to (5.center);
		\draw [style=Border edge] (5.center) to (8.center);
		\draw [style=Border edge] (8.center) to (9.center);
		\draw [style=Border edge] (9.center) to (10.center);
		\draw [style=Border edge] (14.center) to (5.center);
		\draw [style=Border edge] (5.center) to (6.center);
		\draw [style=Border edge] (6.center) to (7.center);
		\draw [style=Border edge] (7.center) to (0.center);
		\draw [style=Border edge] (28.center) to (29.center);
		\draw [style=Border edge] (29.center) to (30.center);
		\draw [style=Border edge] (30.center) to (26.center);
		\draw [style=Border edge] (26.center) to (27.center);
		\draw [style=Border edge] (27.center) to (31.center);
		\draw [style=Border edge] (31.center) to (24.center);
		\draw [style=Border edge] (24.center) to (22.center);
		\draw [style=Border edge] (22.center) to (21.center);
		\draw [style=Border edge] (21.center) to (20.center);
		\draw [style=Border edge] (20.center) to (19.center);
		\draw [style=Border edge] (15.center) to (22.center);
		\draw [style=Border edge] (22.center) to (23.center);
		\draw [style=Border edge] (23.center) to (25.center);
		\draw [style=Border edge] (25.center) to (28.center);
		\draw [style=Border edge] (45.center) to (46.center);
		\draw [style=Border edge] (46.center) to (47.center);
		\draw [style=Border edge] (47.center) to (43.center);
		\draw [style=Border edge] (43.center) to (44.center);
		\draw [style=Border edge] (44.center) to (48.center);
		\draw [style=Border edge] (48.center) to (41.center);
		\draw [style=Border edge] (41.center) to (39.center);
		\draw [style=Border edge] (39.center) to (38.center);
		\draw [style=Border edge] (38.center) to (37.center);
		\draw [style=Border edge] (37.center) to (36.center);
		\draw [style=Border edge] (39.center) to (40.center);
		\draw [style=Border edge] (40.center) to (42.center);
		\draw [style=Border edge] (42.center) to (45.center);
		\draw [style=Border edge] (39.center) to (49.center);
		\draw [style=Border edge] (49.center) to (50.center);
		\draw [style=Border edge] (50.center) to (51.center);
		\draw (52.center) to (3.center);
		\draw (54.center) to (53.center);
		\draw (7.center) to (3.center);
		\draw (7.center) to (55.center);
		\draw (56.center) to (9.center);
		\draw (57.center) to (11.center);
		\draw (58.center) to (9.center);
		\draw (71.center) to (73.center);
		\draw (70.center) to (26.center);
		\draw (69.center) to (31.center);
		\draw (25.center) to (64.center);
		\draw (61.center) to (20.center);
		\draw (60.center) to (18.center);
		\draw (62.center) to (20.center);
		\draw (25.center) to (31.center);
		\draw (65.center) to (66.center);
		\draw (67.center) to (26.center);
		\draw (68.center) to (72.center);
		\draw (76.center) to (78.center);
		\draw (75.center) to (43.center);
		\draw (74.center) to (48.center);
		\draw (42.center) to (80.center);
		\draw (89.center) to (37.center);
		\draw (88.center) to (35.center);
		\draw (87.center) to (82.center);
		\draw (83.center) to (81.center);
		\draw (84.center) to (33.center);
		\draw (86.center) to (37.center);
		\draw (42.center) to (48.center);
		\draw (90.center) to (79.center);
		\draw (91.center) to (43.center);
		\draw (92.center) to (77.center);
		\draw [style=Border edge] (10.center) to (59.center);
		\draw [style=Border edge] (59.center) to (14.center);
		\draw [style=Border edge] (19.center) to (63.center);
		\draw [style=Border edge] (63.center) to (15.center);
		\draw [style=Border edge] (36.center) to (157.center);
		\draw [style=Border edge] (157.center) to (51.center);
		\draw (32.center) to (157.center);
		\draw (85.center) to (157.center);
		\draw [style=Border edge] (160.center) to (159.center);
		\draw [style=Border edge] (159.center) to (161.center);
		\draw [style=Border edge] (161.center) to (162.center);
		\draw [style=Border edge] (162.center) to (163.center);
		\draw [style=Border edge] (163.center) to (166.center);
		\draw [style=Border edge] (166.center) to (168.center);
		\draw [style=Border edge] (168.center) to (169.center);
		\draw [style=Border edge] (169.center) to (170.center);
		\draw [style=Border edge] (170.center) to (167.center);
		\draw [style=Border edge] (167.center) to (165.center);
		\draw [style=Border edge] (165.center) to (164.center);
		\draw [style=Border edge] (164.center) to (160.center);
		\draw (171.center) to (162.center);
		\draw (173.center) to (172.center);
		\draw (174.center) to (162.center);
		\draw (164.center) to (175.center);
		\draw (164.center) to (176.center);
		\draw (178.center) to (177.center);
		\draw (179.center) to (166.center);
		\draw (167.center) to (166.center);
		\draw (167.center) to (181.center);
		\draw (182.center) to (180.center);
		\draw [style=Border edge] (199.center) to (200.center);
		\draw [style=Border edge] (200.center) to (201.center);
		\draw [style=Border edge] (201.center) to (204.center);
		\draw [style=Border edge] (204.center) to (206.center);
		\draw [style=Border edge] (206.center) to (207.center);
		\draw [style=Border edge] (207.center) to (208.center);
		\draw [style=Border edge] (208.center) to (205.center);
		\draw [style=Border edge] (205.center) to (203.center);
		\draw [style=Border edge] (203.center) to (202.center);
		\draw (211.center) to (210.center);
		\draw (212.center) to (200.center);
		\draw (202.center) to (213.center);
		\draw (202.center) to (214.center);
		\draw (216.center) to (215.center);
		\draw (217.center) to (204.center);
		\draw (205.center) to (204.center);
		\draw (205.center) to (219.center);
		\draw (220.center) to (218.center);
		\draw [style=Border edge] (202.center) to (235.center);
		\draw [style=Border edge] (235.center) to (236.center);
		\draw [style=Border edge] (236.center) to (237.center);
		\draw [style=Border edge] (237.center) to (234.center);
		\draw [style=Border edge] (234.center) to (199.center);
		\draw (238.center) to (200.center);
		\draw (239.center) to (234.center);
		\draw (240.center) to (242.center);
		\draw (243.center) to (241.center);
		\draw (198.center) to (234.center);
		\draw [style=Border edge] (254.center) to (255.center);
		\draw [style=Border edge] (255.center) to (256.center);
		\draw [style=Border edge] (256.center) to (259.center);
		\draw [style=Border edge] (259.center) to (261.center);
		\draw (265.center) to (264.center);
		\draw (266.center) to (255.center);
		\draw (257.center) to (267.center);
		\draw (257.center) to (268.center);
		\draw (270.center) to (269.center);
		\draw (271.center) to (259.center);
		\draw [style=Border edge] (257.center) to (289.center);
		\draw [style=Border edge] (289.center) to (290.center);
		\draw [style=Border edge] (290.center) to (291.center);
		\draw [style=Border edge] (291.center) to (288.center);
		\draw [style=Border edge] (288.center) to (254.center);
		\draw (292.center) to (255.center);
		\draw (293.center) to (288.center);
		\draw (294.center) to (296.center);
		\draw (297.center) to (295.center);
		\draw (253.center) to (288.center);
		\draw [style=Border edge] (257.center) to (306.center);
		\draw [style=Border edge] (306.center) to (310.center);
		\draw [style=Border edge] (310.center) to (307.center);
		\draw [style=Border edge] (307.center) to (309.center);
		\draw [style=Border edge] (309.center) to (308.center);
		\draw [style=Border edge] (308.center) to (262.center);
		\draw [style=Border edge] (262.center) to (261.center);
		\draw (311.center) to (262.center);
		\draw (310.center) to (315.center);
		\draw (314.center) to (316.center);
		\draw (313.center) to (262.center);
		\draw (312.center) to (272.center);
		\draw (258.center) to (273.center);
		\draw (310.center) to (259.center);
		\end{pgfonlayer}
		\end{tikzpicture}
		\caption{Two examples of the transformation $t\mapsto t'\mapsto t''$ with $a=4,b=3,a'=6$ and $b'=5$. On the first line $\mu=(4^2,3)$ with $\alpha=(6^2,5,3,2)$, while on the other line $\mu=(4^3,1)$ and $\alpha=(6^2,5^2,2^2)$.}
		\label{Figure constructions}
	\end{figure}
	
	Let $\lambda=\mu\sqcup (a^{b'-b})$ which is either $(a^{b'-1},a-1)$ or $(a^{b'},1)$. Let $t'$ be a $\gamma/\lambda$-tableau obtained from $t$ by increasing all its entries by $b'-b$, moving all boxes downwards by $b'-b$ places and adding $b'-b$ boxes filled with entries $1,2,\dots,b'-b$ from top to bottom on top of each of the columns $a+1,a+2,\dots, 2a$. Clearly, the weight of $t'$ is $\lambda$ and $t'$ is semistandard as $t$ is semistandard. Its reading word $r(t')$ is obtained from $r(t)$ by increasing all its letters by $b'-b$ and then inserting $a$ copies of consecutive letters $1,2,\dots, b'-b$. Since $t$ is semistandard, for any $i\leq a$ there are at most $i-1$ copies of $b'-b+1$ before the $i$th appearance of $b'-b$ in $r(t')$. Therefore $r(t')$ is latticed.
	
	We obtain $\beta/\nu$-tableau $t''$ by moving all boxes of $t'$ rightwards by $a'-a$ places and for each $1\leq i\leq b'$ adding $a'-a$ boxes with $i$ on the left of row $b'+i$ so that the underlying skew partition of $t''$ is $\beta/\nu$. The weight of $t''$ is then $\nu$. If $\mu=(a^{b-1},a-1)$, the tableau $t''$ is clearly semistandard. This is also the case for $\mu=(a^b,1)$ since, from the construction of $t'$, the entry in the box $(b'+2,1)$ of $t'$, if it exists, is at least $b'-b+1\geq 2$. Finally, the reading word $r(t'')$ is obtained from $r(t')$ by adding letters $1,2,\dots,b'$ in this order $(a'-a)$-times, and as $r(t')$ is latticed so is $r(t'')$. The maps $t\mapsto t'$ and $t'\mapsto t''$ are clearly injective, establishing the result.
	
	For (iii) the construction is similar with a slight modification in the map $t\to t'$. See Figure~\ref{Figure construction} for an example.

	\begin{figure}[h!]
		\begin{tikzpicture}[x=0.5cm, y=0.5cm]
		\begin{pgfonlayer}{nodelayer}
		\node [style=none] (0) at (-12, 8) {};
		\node [style=none] (1) at (-10, 8) {};
		\node [style=none] (2) at (-10, 7) {};
		\node [style=none] (3) at (-11, 7) {};
		\node [style=none] (4) at (-11, 6) {};
		\node [style=none] (5) at (-12, 6) {};
		\node [style=none] (6) at (-13, 7) {};
		\node [style=none] (7) at (-13, 6) {};
		\node [style=none] (8) at (-13, 5) {};
		\node [style=none] (9) at (-12, 5) {};
		\node [style=none] (10) at (-13, 4) {};
		\node [style=none] (11) at (-15, 4) {};
		\node [style=none] (12) at (-17, 5) {};
		\node [style=none] (13) at (-17, 2) {};
		\node [style=none] (14) at (-15, 2) {};
		\node [style=none] (15) at (-12, 7) {};
		\node [style=none] (16) at (-11, 8) {};
		\node [style=none] (17) at (-14, 5) {};
		\node [style=none] (18) at (-14, 4) {};
		\node [style=none] (19) at (-15, 5) {};
		\node [style=none] (20) at (-16, 5) {};
		\node [style=none] (21) at (-17, 4) {};
		\node [style=none] (22) at (-17, 3) {};
		\node [style=none] (23) at (-16, 2) {};
		\node [style=none] (24) at (-15, 3) {};
		\node [style=none] (25) at (-10.5, 7.5) {$1$};
		\node [style=none] (26) at (-11.5, 7.5) {$1$};
		\node [style=none] (27) at (-11.5, 6.5) {$2$};
		\node [style=none] (28) at (-12.5, 6.5) {$2$};
		\node [style=none] (29) at (-12.5, 5.5) {$3$};
		\node [style=none] (30) at (-13.5, 4.5) {$3$};
		\node [style=none] (31) at (-14.5, 4.5) {$1$};
		\node [style=none] (32) at (-15.5, 4.5) {$1$};
		\node [style=none] (33) at (-16.5, 4.5) {$1$};
		\node [style=none] (34) at (-16.5, 3.5) {$2$};
		\node [style=none] (35) at (-15.5, 3.5) {$2$};
		\node [style=none] (36) at (-15.5, 2.5) {$3$};
		\node [style=none] (37) at (-16.5, 2.5) {$3$};
		\node [style=none] (38) at (-9.5, 4.5) {$\mapsto$};
		\node [style=none] (40) at (-2, 6) {};
		\node [style=none] (41) at (-2, 5) {};
		\node [style=none] (42) at (-3, 5) {};
		\node [style=none] (43) at (-3, 4) {};
		\node [style=none] (44) at (-4, 4) {};
		\node [style=none] (45) at (-5, 5) {};
		\node [style=none] (46) at (-5, 4) {};
		\node [style=none] (47) at (-5, 3) {};
		\node [style=none] (48) at (-4, 3) {};
		\node [style=none] (49) at (-5, 2) {};
		\node [style=none] (50) at (-7, 2) {};
		\node [style=none] (51) at (-9, 3) {};
		\node [style=none] (52) at (-9, 0) {};
		\node [style=none] (53) at (-7, 0) {};
		\node [style=none] (56) at (-6, 3) {};
		\node [style=none] (57) at (-6, 2) {};
		\node [style=none] (58) at (-7, 3) {};
		\node [style=none] (59) at (-8, 3) {};
		\node [style=none] (60) at (-9, 2) {};
		\node [style=none] (61) at (-9, 1) {};
		\node [style=none] (62) at (-8, 0) {};
		\node [style=none] (63) at (-7, 1) {};
		\node [style=none] (64) at (-2.5, 5.5) {$3$};
		\node [style=none] (65) at (-3.5, 5.5) {$3$};
		\node [style=none] (66) at (-3.5, 4.5) {$4$};
		\node [style=none] (67) at (-4.5, 4.5) {$4$};
		\node [style=none] (68) at (-4.5, 3.5) {$5$};
		\node [style=none] (69) at (-5.5, 2.5) {$5$};
		\node [style=none] (70) at (-6.5, 2.5) {$3$};
		\node [style=none] (71) at (-7.5, 2.5) {$3$};
		\node [style=none] (72) at (-8.5, 2.5) {$1$};
		\node [style=none] (73) at (-8.5, 1.5) {$4$};
		\node [style=none] (74) at (-7.5, 1.5) {$4$};
		\node [style=none] (75) at (-7.5, 0.5) {$5$};
		\node [style=none] (76) at (-8.5, 0.5) {$5$};
		\node [style=none] (77) at (-4, 8) {};
		\node [style=none] (80) at (-4, 7) {};
		\node [style=none] (81) at (-3, 8) {};
		\node [style=none] (82) at (-2, 8) {};
		\node [style=none] (85) at (-1, 6) {};
		\node [style=none] (86) at (-4.5, 5.5) {$2$};
		\node [style=none] (87) at (-4.5, 6.5) {$1$};
		\node [style=none] (88) at (-1.5, 7.5) {$1$};
		\node [style=none] (89) at (-1.5, 6.5) {$2$};
		\node [style=none] (90) at (-2.5, 6.5) {$2$};
		\node [style=none] (91) at (-2.5, 7.5) {$1$};
		\node [style=none] (92) at (-3.5, 7.5) {$1$};
		\node [style=none] (93) at (-3.5, 6.5) {$2$};
		\node [style=none] (94) at (-0.5, 4.5) {$\mapsto$};
		\node [style=none] (95) at (-1, 6) {};
		\node [style=none] (96) at (-1, 8) {};
		\node [style=none] (98) at (9, 6) {};
		\node [style=none] (99) at (9, 5) {};
		\node [style=none] (100) at (8, 5) {};
		\node [style=none] (101) at (8, 4) {};
		\node [style=none] (102) at (7, 4) {};
		\node [style=none] (103) at (6, 5) {};
		\node [style=none] (104) at (6, 4) {};
		\node [style=none] (105) at (6, 3) {};
		\node [style=none] (106) at (7, 3) {};
		\node [style=none] (107) at (6, 2) {};
		\node [style=none] (108) at (4, 2) {};
		\node [style=none] (109) at (2, 3) {};
		\node [style=none] (110) at (2, 0) {};
		\node [style=none] (111) at (4, 0) {};
		\node [style=none] (113) at (5, 3) {};
		\node [style=none] (114) at (5, 2) {};
		\node [style=none] (115) at (4, 3) {};
		\node [style=none] (116) at (3, 3) {};
		\node [style=none] (119) at (3, 0) {};
		\node [style=none] (120) at (4, 1) {};
		\node [style=none] (121) at (8.5, 5.5) {$3$};
		\node [style=none] (122) at (7.5, 5.5) {$3$};
		\node [style=none] (123) at (7.5, 4.5) {$4$};
		\node [style=none] (124) at (6.5, 4.5) {$4$};
		\node [style=none] (125) at (6.5, 3.5) {$5$};
		\node [style=none] (126) at (5.5, 2.5) {$5$};
		\node [style=none] (127) at (4.5, 2.5) {$3$};
		\node [style=none] (128) at (3.5, 2.5) {$3$};
		\node [style=none] (129) at (2.5, 2.5) {$1$};
		\node [style=none] (130) at (2.5, 1.5) {$4$};
		\node [style=none] (131) at (3.5, 1.5) {$4$};
		\node [style=none] (132) at (3.5, 0.5) {$5$};
		\node [style=none] (133) at (2.5, 0.5) {$5$};
		\node [style=none] (134) at (7, 8) {};
		\node [style=none] (137) at (7, 7) {};
		\node [style=none] (138) at (8, 8) {};
		\node [style=none] (139) at (9, 8) {};
		\node [style=none] (141) at (10, 6) {};
		\node [style=none] (142) at (6.5, 5.5) {$2$};
		\node [style=none] (143) at (6.5, 6.5) {$1$};
		\node [style=none] (144) at (9.5, 7.5) {$1$};
		\node [style=none] (145) at (9.5, 6.5) {$2$};
		\node [style=none] (146) at (8.5, 6.5) {$2$};
		\node [style=none] (147) at (8.5, 7.5) {$1$};
		\node [style=none] (148) at (7.5, 7.5) {$1$};
		\node [style=none] (149) at (7.5, 6.5) {$2$};
		\node [style=none] (150) at (10, 6) {};
		\node [style=none] (151) at (10, 8) {};
		\node [style=none] (152) at (1, 3) {};
		\node [style=none] (153) at (0, 3) {};
		\node [style=none] (154) at (0, 2) {};
		\node [style=none] (155) at (0, 1) {};
		\node [style=none] (156) at (0, 0) {};
		\node [style=none] (157) at (0, -1) {};
		\node [style=none] (158) at (0, -2) {};
		\node [style=none] (159) at (1, -2) {};
		\node [style=none] (160) at (2, -2) {};
		\node [style=none] (161) at (2, -1) {};
		\node [style=none] (162) at (0.5, 2.5) {$1$};
		\node [style=none] (163) at (1.5, 2.5) {$1$};
		\node [style=none] (164) at (1.5, 1.5) {$2$};
		\node [style=none] (165) at (0.5, 1.5) {$2$};
		\node [style=none] (166) at (0.5, 0.5) {$3$};
		\node [style=none] (167) at (1.5, 0.5) {$3$};
		\node [style=none] (168) at (1.5, -0.5) {$4$};
		\node [style=none] (169) at (0.5, -0.5) {$4$};
		\node [style=none] (170) at (0.5, -1.5) {$5$};
		\node [style=none] (171) at (1.5, -1.5) {$5$};
		\node [style=none] (172) at (-5, 6) {};
		\node [style=none] (173) at (-5, 7) {};
		\node [style=none] (174) at (-1, 7) {};
		\node [style=none] (175) at (6, 6) {};
		\node [style=none] (176) at (6, 7) {};
		\node [style=none] (177) at (10, 7) {};
		\end{pgfonlayer}
		\begin{pgfonlayer}{edgelayer}
		\draw [style=Border edge] (0.center) to (1.center);
		\draw [style=Border edge] (1.center) to (2.center);
		\draw [style=Border edge] (2.center) to (3.center);
		\draw [style=Border edge] (3.center) to (4.center);
		\draw [style=Border edge] (4.center) to (5.center);
		\draw [style=Border edge] (5.center) to (9.center);
		\draw [style=Border edge] (9.center) to (8.center);
		\draw [style=Border edge] (8.center) to (10.center);
		\draw [style=Border edge] (10.center) to (11.center);
		\draw [style=Border edge] (11.center) to (14.center);
		\draw [style=Border edge] (14.center) to (13.center);
		\draw [style=Border edge] (13.center) to (12.center);
		\draw [style=Border edge] (12.center) to (8.center);
		\draw [style=Border edge] (8.center) to (6.center);
		\draw [style=Border edge] (6.center) to (15.center);
		\draw [style=Border edge] (15.center) to (0.center);
		\draw (16.center) to (3.center);
		\draw (15.center) to (5.center);
		\draw (17.center) to (18.center);
		\draw (19.center) to (11.center);
		\draw (20.center) to (23.center);
		\draw (22.center) to (24.center);
		\draw (21.center) to (11.center);
		\draw (7.center) to (5.center);
		\draw (15.center) to (3.center);
		\draw [style=Border edge] (41.center) to (42.center);
		\draw [style=Border edge] (42.center) to (43.center);
		\draw [style=Border edge] (43.center) to (44.center);
		\draw [style=Border edge] (44.center) to (48.center);
		\draw [style=Border edge] (48.center) to (47.center);
		\draw [style=Border edge] (47.center) to (49.center);
		\draw [style=Border edge] (49.center) to (50.center);
		\draw [style=Border edge] (50.center) to (53.center);
		\draw [style=Border edge] (53.center) to (52.center);
		\draw [style=Border edge] (52.center) to (51.center);
		\draw [style=Border edge] (51.center) to (47.center);
		\draw (56.center) to (57.center);
		\draw (58.center) to (50.center);
		\draw (59.center) to (62.center);
		\draw (61.center) to (63.center);
		\draw (60.center) to (50.center);
		\draw (46.center) to (44.center);
		\draw [style=Border edge] (41.center) to (40.center);
		\draw (82.center) to (40.center);
		\draw (81.center) to (42.center);
		\draw [style=Border edge] (99.center) to (100.center);
		\draw [style=Border edge] (100.center) to (101.center);
		\draw [style=Border edge] (101.center) to (102.center);
		\draw [style=Border edge] (102.center) to (106.center);
		\draw [style=Border edge] (106.center) to (105.center);
		\draw [style=Border edge] (105.center) to (107.center);
		\draw [style=Border edge] (107.center) to (108.center);
		\draw [style=Border edge] (108.center) to (111.center);
		\draw [style=Border edge] (111.center) to (110.center);
		\draw (113.center) to (114.center);
		\draw (115.center) to (108.center);
		\draw (116.center) to (119.center);
		\draw (104.center) to (102.center);
		\draw [style=Border edge] (99.center) to (98.center);
		\draw (139.center) to (98.center);
		\draw (138.center) to (100.center);
		\draw [style=Border edge] (110.center) to (160.center);
		\draw [style=Border edge] (160.center) to (158.center);
		\draw [style=Border edge] (158.center) to (153.center);
		\draw [style=Border edge] (153.center) to (105.center);
		\draw (109.center) to (110.center);
		\draw (152.center) to (159.center);
		\draw (157.center) to (161.center);
		\draw (156.center) to (110.center);
		\draw (155.center) to (120.center);
		\draw (154.center) to (108.center);
		\draw [style=Border edge] (47.center) to (173.center);
		\draw [style=Border edge] (173.center) to (80.center);
		\draw [style=Border edge] (80.center) to (77.center);
		\draw [style=Border edge] (77.center) to (96.center);
		\draw [style=Border edge] (96.center) to (95.center);
		\draw [style=Border edge] (95.center) to (40.center);
		\draw (80.center) to (44.center);
		\draw (45.center) to (42.center);
		\draw (172.center) to (40.center);
		\draw (80.center) to (174.center);
		\draw [style=Border edge] (105.center) to (176.center);
		\draw [style=Border edge] (176.center) to (137.center);
		\draw [style=Border edge] (137.center) to (134.center);
		\draw [style=Border edge] (134.center) to (151.center);
		\draw [style=Border edge] (151.center) to (150.center);
		\draw [style=Border edge] (150.center) to (98.center);
		\draw (137.center) to (177.center);
		\draw (175.center) to (98.center);
		\draw (103.center) to (100.center);
		\draw (137.center) to (102.center);
		\end{pgfonlayer}
		\end{tikzpicture}
		\caption{The transformation $t\mapsto t'\mapsto t''$ with $a=4,b=3,a'=6$ and $b'=5$. In this example $\mu=(5,4^2)$ and $\alpha=(7,6,5,4,2^2)$.}
		\label{Figure construction}
	\end{figure}
	
	Observe firstly that the box $(b+1,1)$ of $t$ contains $1$. Indeed, since $\alpha_1\leq 2a$ and $t$ is semistandard, there is $1$ in some of the first $a$ columns of $t$. Using that $t$ is semistandard again, we conclude that the box $(b+1,1)$ contains $1$. We now modify the earlier construction $t\mapsto t'$ by non-increasing this $1$. That is, for $\lambda=\mu\sqcup (a^{b'-b}) = (a+1,a^{b'-1})$ define $\gamma/\lambda$-tableau $t'$ obtained from $t$ by increasing all its entries apart from $1$ in the box $(b+1,1)$ by $b'-b$, moving all boxes downwards by $b'-b$ places and adding $b'-b$ boxes filled with entries $1,2,\dots,b'-b$ from top to bottom on top of each of the columns $a+1,a+2,\dots, 2a$.
	
	Clearly, the weight of $t'$ is $\lambda$ and $t'$ is semistandard. Its reading word is latticed as long as its final letters from the first column do not violate this condition. But this cannot be the case as they are in increasing order. The construction of $t''$ from $t'$ is then the same as before.   
\end{proof}

\section{Small index subgroups of $S_m\wr S_2$ embedded in $S_{2m}$ and $S_{2m+1}$}\label{Sec SIS II}

In this section $m\geq 2$. We investigate irreducible induced-multiplicity-free characters of subgroups of $S_m\wr S_2$ in Proposition~\ref{Prop possible MF subgroups}(v), that is of $\left( S_m\wr S_2\right) \cap A_{2m}, A_m\wr S_2$ and $T_{m,2}$ embedded in $S_{2m}$ and $S_{2m+1}$. We reuse the notation $G_+=\ker\left( \charwrnb{\sgn}{\mathbbm{1}}{2}\right) $ and $G_-=\ker\left( \charwrnb{\sgn}{\sgn}{2}\right) $ from \S\ref{Section k,m,2} adapted to $k=0$. Thus $G_+=\left( S_m\wr S_2\right) \cap A_{2m}$ and $G_-=T_{m,2}$ for even $m$, while for odd $m$ we need to swap the indices of $G$. To understand and manipulate the irreducible characters of $A_m\wr S_2$ we use Theorems~\ref{Theorem characters of wreath products}, Lemma~\ref{Lemma plethysms and char properties}(iv) and the readily checked identity $\left( \charwrnb{\rho}{\chi^{\nu}}{2}\right)\Ind^{S_m\wr S_2} = \charwr{\rho\ind^{S_m}}{\chi^{\nu}}{2}$ for a character $\rho$ of $A_m$ and $\nu\vdash 2$.

As in \S\ref{Sec SIS I}, one should bear in mind the results used in Example~\ref{Example index two}.

\subsection{Groups embedded in $S_{2m+1}$}

For $n\geq 14$, by Proposition~\ref{Prop direct product k eq 1}(i), the irreducible induced-multiplicity-free characters of $S_m\wr S_2$ are $\charwr{\chi^{\mu}}{\chi^{\nu}}{2}$ with $\mu$ rectangular and $\nu\vdash 2$. We start with the group $A_m\wr S_2$.

\begin{proposition}\label{Prop Am wr S2 class}
	Let $m\geq 7$. The irreducible induced-multiplicity-free characters of $A_m\wr S_2$ embedded in $S_{2m+1}$ are $\charwr{\chi^{\mu}_{A_m}}{\chi^{\nu}}{2}$ with $\mu$ a square partition and $\nu\vdash 2$.
\end{proposition}

\begin{proof}
	An irreducible induced-multiplicity-free character of $A_m\wr S_2$ must be elementary by Lemma~\ref{Lemma sanity lemma}, that is of the form $\charwr{\chi^{\mu}_{A_m}}{\chi^{\nu}}{2}$. If $\mu$ is self-conjugate, then $\left( \charwr{\chi^{\mu}_{A_m}}{\chi^{\nu}}{2}\right)\Ind^{S_m\wr S_2} = \charwr{\chi^{\mu}}{\chi^{\nu}}{2}$. By Proposition~\ref{Prop direct product k eq 1}(i), this is induced-multiplicity-free if and only if $\mu$ is a (self-conjugate) rectangular partition, that is a square partition.  Otherwise, by Lemma~\ref{Lemma plethysms and char properties}(iv), the character $\left( \charwr{\chi^{\mu}_{A_m}}{\chi^{\nu}}{2}\right)\Ind^{S_m\wr S_2}$ has $\left( \chi^{\mu}\boxtimes\chi^{\mu'}\right)\Ind^{S_m\wr S_2}$ as a constituent, which is not induced-multiplicity-free by Lemma~\ref{Lemma sanity lemma}, establishing the result.
\end{proof}

Moving to $G_{\pm}$ the potential irreducible induced-multiplicity-free characters are $\left( \charwr{\chi^{\mu}}{\chi^{\nu}}{2}\right)_{G_{\pm}}$ with $\mu$ rectangular and $\nu\vdash 2$. After inducing this character to $S_{2m+1}$, if $\mu$ is non-square, we obtain the character corresponding to $s_{(1)}\left( s_{\nu}\circ s_{\mu} + s_{\bar{\nu}}\circ s_{\mu'}\right) $ with $\bar{{\nu}}=\nu$ for $G_+$ and $\bar{{\nu}}=\nu'$ for $G_-$. If $\mu$ is square, we get $s_{(1)}\left( s_{\nu}\circ s_{\mu}\right) $ for $G_+$ and $s_{(1)}\left( s_{\nu}\circ s_{\mu} + s_{\nu'}\circ s_{\mu}\right)=s_{(1)}s_{\mu}^2$ for $G_-$. Thus the first is multiplicity-free, while the second is not by Lemma~\ref{Lemma sanity lemma}. For non-square $\mu$ we use the following lemma, which is the first one among our surprising results where a divisibility condition controls whether a certain symmetric function is multiplicity-free. 

\begin{lemma}\label{Lemma rectangle and one box}
	Let $\mu=(a^b)$ be a non-square rectangular partition and $\nu,\bar{\nu}\vdash 2$. The symmetric function $s_{(1)}\left( s_{\nu}\circ s_{\mu} + s_{\bar{\nu}}\circ s_{\mu'}\right) $ is multiplicity-free if and only if $a-b\nmid a$.
\end{lemma}

\begin{proof}
	Without loss of generality assume $a>b$. By Proposition~\ref{Prop product with sm wr s2}, both $s_{(1)}\left( s_{\nu}\circ s_{\mu}\right) $ and $s_{(1)}\left( s_{\bar{\nu}}\circ s_{\mu'}\right) $ are multiplicity-free. If they share a constituent, so do $s_{(1)}s_{\mu}^2$ and $s_{(1)}s_{\mu'}^2$, and Corollary~\ref{Cor rectangles}(i) with $t=1$ shows that $1\geq 2r(d-r)$ where $d=a-b$ and $0\leq r<d$ is the remainder of $a$ modulo $d$. In turn $r=0$, or equivalently, $d\mid a$, as needed.
	
	It remains to show that if $d\mid a$, then the summands share a constituent. Firstly, suppose that  $a\neq 2b$. Recall the self-conjugate $(a,b)$-birectangular partition $\bar{\lambda}=((2b)^{2d}, (2b-2d)^{2d}, \dots, (2d)^{2d})$ from Proposition~\ref{Prop conjugate (a,b)-birectangles}. Let $\lambda$ be a partition obtained from $\bar{\lambda}$ by adding $1$ to the $(2d+1)$st part, that is $\lambda=((2b)^{2d},2b-2d+1,(2b-2d)^{2d-1},(2b-4d)^{2d},\dots, (2d)^{2d})$. The partition obtained from $\lambda$ by decreasing its $2(b-d)$th part, equal to $4d$, by $1$ is also $(a,b)$-birectangular since $\bar{\lambda}$ is. Hence $\left\langle s_{(1)}s_{\mu}^2,s_{\lambda} \right\rangle\geq 2 $. Analogously, one can show that $\left\langle s_{(1)}s_{\mu'}^2,s_{\lambda} \right\rangle=\left\langle s_{(1)}s_{\mu}^2,s_{\lambda'} \right\rangle \geq 2 $.
	
	Since $s_{(1)}s_{\mu}^2=s_{(1)}\left( s_{(2)}\circ s_{\mu} + s_{(1^2)}\circ s_{\mu}\right) $ has multiplicity-free summands, our $s_{\lambda}$ must be a constituent of both of the summands. We can repeat this for $\mu'$ in place of $\mu$ to deduce that both $s_{(1)}\left( s_{\nu}\circ s_{\mu}\right) $ and $s_{(1)}\left( s_{\bar{\nu}}\circ s_{\mu'}\right) $ have $s_{\lambda}$ as a constituent, as required.
	
	For $a=2b$, using Proposition~\ref{Prop domino rectangles}, we see that $s_{(a^a)}$ is a constituent of $s_{(2)}\circ s_{\mu}$ and $s_{(a+1,a^{a-2},a-1)}$ is a constituent of $s_{(1^2)}\circ s_{\mu}$. Hence $\lambda=(a+1,a^{a-1})$ labels a constituent of $s_{(1)}\left( s_{\nu}\circ s_{\mu}\right) $ for any $\nu\vdash 2$ and $\lambda'$ labels a constituent of $s_{(1)}\left( s_{(2)} \circ s_{\mu}\right) $. Moreover, the self-conjugate partition $\bar{\lambda}=(a+1,a^{a-2},a-1,1)$ labels a constituent of $s_{(1)}\left( s_{(1^2)} \circ s_{\mu}\right) $. Using Lemma~\ref{Lemma plethysms and char properties}(ii) with even $|\mu|$, we conclude that $s_{(1)}\left( s_{\nu}\circ s_{\mu}\right) $ and $s_{(1)}\left( s_{\bar{\nu}}\circ s_{\mu'}\right) $ share a constituent labelled by a partition given by Table~\ref{Table 1x2 rectangles}.
\end{proof}

\begin{table}[h!]
	\centering
	\begin{tabular}{c|cc}
		\specialrule{0.1em}{0em}{0em}
		\diagbox[width=1cm,  height=0.7cm]{$\nu$}{$\bar{\nu}$}&$(2)$&$(1^2)$\\
		\hline
		$(2)$&$\lambda$&$\lambda'$\\
		$(1^2)$&$\lambda$&$\bar{\lambda}$\\
		\specialrule{0.1em}{0em}{0em}
	\end{tabular}
	\vspace{4pt}
	\caption{Partitions labelling a common constituent of $s_{(1)}\left( s_{\nu}\circ s_{\mu}\right) $ and $s_{(1)}\left( s_{\bar{\nu}}\circ s_{\mu'}\right)$ where $\nu,\bar{\nu}\vdash 2$, $\mu=(a^{a/2})$, $\lambda=(a+1,a^{a-1})$ and $\bar{\lambda}=(a+1,a^{a-2},a-1,1)$.} 
	\label{Table 1x2 rectangles}
\end{table}

\begin{corollary}\label{Cor S1 times Sm wr S2 class}
	Let $m\geq 7$. The irreducible induced-multiplicity-free characters of
	\begin{enumerate}[label=\textnormal{(\roman*)}]
		\item $G=\left( S_m\wr S_2\right) \cap A_{2m}$ embedded in $S_{2m+1}$ are $\left( \charwr{\chi^{\mu}}{\chi^{\nu}}{2}\right)_G$ with $\nu\vdash 2$ and $\mu=(a^b)$ such that $a-b\nmid a$, with $a=b$ allowed only for even $m$;
		\item $G=T_{m,2}$ embedded in $S_{2m+1}$ are $\left( \charwr{\chi^{\mu}}{\chi^{\nu}}{2}\right)_G$ with $\nu\vdash 2$ and $\mu=(a^b)$ such that $a-b\nmid a$, with $a=b$ allowed only for odd $m$.
	\end{enumerate}
\end{corollary}

\begin{proof}
	For $a=b$ we use the earlier discussion about square partitions. For $a\neq b$ we apply Lemma~\ref{Lemma rectangle and one box}.
\end{proof}

\subsection{Groups embedded in $S_{2m}$}\label{Sec n=2m}

Let $G$ be $\left( S_m\wr S_2\right) \cap A_{2m}$ or $T_{m,2}$. By an \textit{elementary irreducible character} of $G$ we mean an irreducible character of the form $\rho_G$ where $\rho$ is an elementary irreducible character of $S_m\wr S_2$. In this subsection we focus only on the elementary irreducible characters of $\left( S_m\wr S_2\right) \cap A_{2m}, T_{m,2}$ and $A_m\wr S_2$, which is a reasonable restriction as the non-elementary irreducible induced-multiplicity-free characters can be easily deduced from the irreducible induced-multiplicity-free characters of $S_m\times S_m$, $A_m\times A_m$ and $\left( S_m\wr S_2\right) \cap A_{2m}$.

Indeed, this is clear for $A_m\wr S_2$, and if $G$ is $\left( S_m\wr S_2\right) \cap A_{2m}$ or $T_{m,2}$, Mackey's Theorem states that for any distinct $\lambda, \mu\vdash m$ we have
\[
\left(\chi^{\lambda}\boxtimes \chi^{\mu} \right)\Ind^{S_m\wr S_2}\Res_G = \left(\chi^{\lambda}\boxtimes \chi^{\mu} \right)\Res_{(S_m\times S_m)\cap A_{2m}}\Ind^G.  
\]
Moreover, using Lemma~\ref{Lemma index two}(i) and Lemma~\ref{Lemma tensoring plethysms}, we see that unless $\lambda, \mu\in\left\lbrace \lambda', \mu'\right\rbrace $, the left-hand side is irreducible. Thus
\[
\left( \left(\chi^{\lambda}\boxtimes \chi^{\mu} \right)\Ind^{S_m\wr S_2}\right) _G = \left(\chi^{\lambda}\boxtimes \chi^{\mu} \right)_{(S_m\times S_m)\cap A_{2m}}\Ind^G.  
\] 
And if $\lambda, \mu\in\left\lbrace \lambda', \mu'\right\rbrace $, Lemma~\ref{Lemma index two}(iv) and Lemma~\ref{Lemma tensoring plethysms} show
\[
\left( \left(\chi^{\lambda}\boxtimes \chi^{\mu} \right)\Ind^{S_m\wr S_2}\right) _G\Ind^{S_m\wr S_2} = \left(\chi^{\lambda}\boxtimes \chi^{\mu} \right)\Ind^{S_m\wr S_2},
\]
and one can use Theorem~\ref{Theorem Stembridge} in this case.

For $n\geq 20$ the elementary induced-irreducible characters of $S_m\wr S_2$ are by Theorem~\ref{Theorem MF plethysms} given by $\charwr{\chi^{\mu}}{\chi^{\nu}}{2}$ with $\nu\vdash 2$ and $\mu$ rectangular, almost rectangular or a hook. Starting with the group $A_m\wr S_2$, we need to consider characters $\charwr{\chi^{\mu}_{A_m}}{\chi^{\nu}}{2}$ with $\mu$ and $\nu$ as above. It is clear that $\charwr{\chi^{\mu}_{A_m}}{\chi^{\nu}}{2}$ is multiplicity-free if $\mu$ is self-conjugate.

Otherwise, after applying the induction to $S_{2m}$, this induced character corresponds to $s_{\nu}\circ s_{\mu} + s_{\mu}s_{\mu'} + s_{\nu}\circ s_{\mu'}$ by Lemma~\ref{Lemma plethysms and char properties}(iv). Thus if this is multiplicity-free, we need $\mu$ to be rectangular by Theorem~\ref{Theorem Stembridge}. In such case we can use the following lemmas and the observation that $s_{\nu}\circ s_{\mu} + s_{\mu}s_{\mu'} + s_{\nu}\circ s_{\mu'}$ is multiplicity-free if and only if the sums of any two of the summands are multiplicity-free.

\begin{lemma}\label{Lemma rotate rectangles}
	Let $\mu=(a^b)$ be a non-square rectangular partition and $\nu,\bar{\nu}\vdash 2$. The symmetric function $s_{\nu}\circ s_{\mu} + s_{\bar{\nu}}\circ s_{\mu'}$ is \emph{not} multiplicity-free if and only if $\nu=\bar{\nu}=(2)$ and $a-b\mid a$.
\end{lemma}

\begin{proof}
	Without loss of generality suppose $a>b$. Both summands are multiplicity-free by Theorem~\ref{Theorem MF plethysms}(i). If they share a constituent $s_{\lambda}$, we know that $\lambda$ is $(a,b)$-birectangular and $(b,a)$-birectangular by Proposition~\ref{Prop domino rectangles}. By Lemma~\ref{Prop conjugate (a,b)-birectangles}, in such a case $a-b\mid a$ and $\lambda=((2b)^{2d}, (2b-2d)^{2d}, \dots, (2d)^{2d})$, where $d=a-b$. Since $\lambda$ is an even partition, Proposition~\ref{Prop domino rectangles} implies that $s_{\lambda}$ is a constituent of $s_{\nu}\circ s_{\mu}$, respectively, $s_{\bar{\nu}}\circ s_{\mu'}$ if and only if $\nu=(2)$, respectively, $\bar{\nu}=(2)$. The result follows.
\end{proof}

\begin{lemma}\label{Lemma rotate rectangle}
	Let $\mu=(a^b)$ be a non-square rectangular partition and $\nu\vdash 2$. The symmetric function $s_{\nu}\circ s_{\mu} + s_{\mu} s_{\mu'}$ is \emph{not} multiplicity-free if and only if $a-b\mid a$ and
	\begin{align*}
	\nu=\begin{cases}
	(2) & \text{ if } 4\mid ab,\\
	(1^2) & \text{ otherwise.}
	\end{cases}
	\end{align*}
\end{lemma}

\begin{proof}
	Assume firstly that $a>b$. The summands are multiplicity-free by Theorem~\ref{Theorem MF plethysms}(i) and Theorem~\ref{Theorem Stembridge}(iv). If they have a common constituent $s_{\lambda}$, by Proposition~\ref{Prop LR rectangles}(ii) and Proposition~\ref{Prop domino rectangles} the partition $\lambda$ must be $(a,b)$-birectangular and $\left\lbrace a,b\right\rbrace $-birectangular. Proposition~\ref{Prop conjugate set ab-birectangles} then shows that $a-b\mid a$ and $\lambda=((a+b)^d, (a+b-d)^d, \dots, (b+3d)^d, (b+2d)^d, b^d, (b-d)^d,\dots, d^d)$ with $d=a-b$. Since $\lambda_1+\lambda_2+\dots + \lambda_b=(a+b)d +(a+b-d)d +\dots + (b+2d)d=b(a+2b+2d)/2=3ab/2$, Proposition~\ref{Prop domino rectangles} shows that $s_{\lambda}$ is a constituent of $s_{\nu}\circ s_{\mu}$ if and only if $\nu$ is as in the statement, as required. If $b>a$, we apply $\omega$ according to Lemma~\ref{Lemma plethysms and char properties}(i) and observe that $\nu^{\prime ab}$, defined in \S\ref{Sec Par}, equals $\nu$ if $a-b\mid a$.
\end{proof}

\begin{corollary}\label{Cor Am wr S2 class}
	Let $m\geq 10$. The elementary irreducible induced-multiplicity-free characters of $A_m\wr S_2$ are of the form $\charwr{\chi^{\mu}_{A_m}}{\chi^{\nu}}{2}$ where $\nu\vdash 2$ and $\mu$ is of the form $(a+1,1^a), (a^{a-1},a-1)$ or $(a^b)$ with $a-b\nmid a$ and if $4\mid m$, the choice $\nu=(1^2)$ and $\mu=(a^b)$ with $a-b\mid a$ is also allowed.
\end{corollary}

\begin{proof}
	If $\mu$ is self-conjugate we obtain the first two choices of $\mu$ and the third choice with $a=b$. Otherwise, we apply Lemma~\ref{Lemma rotate rectangles} with $\bar{\nu}=\nu$ and Lemma~\ref{Lemma rotate rectangle} twice with $\mu$ and $\mu'$ as the rectangular partitions to obtain the remaining choices.
\end{proof}

We now move to $G_{\pm}$, or alternatively $\left( S_m\wr S_2\right) \cap A_{2m}$ and $T_{m,2}$. Recall that for $\mu$ not self-conjugate, the character $\left( \charwr{\chi^{\mu}}{\chi^{\nu}}{2}\right)_{G_{\pm}}$ is induced-multiplicity-free if and only if $s_{\nu}\circ s_{\mu} + s_{\bar{\nu}}\circ s_{\mu'}$ is multiplicity-free, where $\bar{\nu}=\nu$ for $G_+$ and $\bar{\nu}=\nu'$ for $G_-$. If $\mu$ is self-conjugate, we need to replace the symmetric function for $G_+$ with $s_{\nu}\circ s_{\mu}$, while for $G_-$ the symmetric function simplifies to $s_{\mu}^2$. We start by examining these symmetric functions for $\mu$ a hook.

\begin{lemma}\label{Lemma hooks}
	Let $a< b$ be positive integers, $\mu=(a+1,1^b)$ and $\nu,\bar{\nu}\vdash 2$. The symmetric function $s_{\nu}\circ s_{\mu} + s_{\bar{\nu}}\circ s_{\mu'}$ is \emph{not} multiplicity-free if and only if $b=a+1$ and $\nu=\bar{\nu}$.
\end{lemma}

\begin{proof}
	The summands are multiplicity-free by Theorem~\ref{Theorem MF plethysms}(i). Suppose that $s_{\lambda}$ is a common constituent of our summands. Then, by Lemma~\ref{Lemma plethysms and char properties}(iii), also $s_{\mu}^2$ and $s_{\mu'}^2$ share this constituent. By Lemma~\ref{Lemma LR hooks} we know that $\lambda_1+\lambda_2\in\left\lbrace 2a+2,2a+3,2a+4 \right\rbrace $ and $\lambda_1+\lambda_2\in\left\lbrace 2b+2,2b+3,2b+4 \right\rbrace $. Since $2a+4\leq 2b+2$, this forces $\lambda_1+\lambda_2=2a+4=2b+2$; thus $b=a+1$. Using Lemma~\ref{Lemma LR hooks} further, we see that $\lambda_2\geq 2\geq \lambda_3$; thus if we let $c=\lambda_1$ and $d=\ell(\lambda)$, we can write $\lambda=(c,2a+4-c,2^{2a+2-d},1^{2d-2a-4})$ and notice that $a+2\leq c,d\leq 2a+2$. Now, for any choice of $a+2\leq c,d\leq 2a+2$, the corresponding $\lambda$ satisfies $c(\mu,\mu;\lambda)=c(\mu',\mu';\lambda)=1$; thus it remains to determine for which $\nu, \bar{\nu}\vdash 2$ our summands $s_{\nu}\circ s_{\mu}$ and $s_{\bar{\nu}}\circ s_{\mu'}$ have $s_{\lambda}$ as a constituent.
	
	We show that $\nu=\bar{\nu}=(2)^{\prime (a+c+d)}$, which finishes the proof as $(2)^{\prime (a+c+d)}$ takes both values $(2)$ and $(1^2)$ as $c$ and $d$ vary. Consider the partition $\varphi=(c-a-1,1^{2a-d+3})\subseteq \mu$ and the semistandard domino $\mu$-tableau $T_{\varphi}$ from Example~\ref{Example top filling} displayed in Figure~\ref{Fig hooks}. We know its reading word is latticed and it is easily seen that its weight is $\lambda$. Finally, it is even if and only if $|\mu|-|\varphi|=a+d-c$ is even. Thus $p(\mu,\nu;\lambda)>0$ if and only if $\nu=(2)^{\prime (a+d-c)}=(2)^{\prime (a+c+d)}$.
	
	\begin{figure}[h]
		\begin{tikzpicture}[x=0.5cm, y=0.5cm]
		\begin{pgfonlayer}{nodelayer}
		\node [style=none] (0) at (-13, 0) {};
		\node [style=none] (1) at (-11, 0) {};
		\node [style=none] (2) at (-13, 13) {};
		\node [style=none] (3) at (-11, 11) {};
		\node [style=none] (4) at (-2, 13) {};
		\node [style=none] (5) at (-2, 11) {};
		\node [style=vdomino] (6) at (-12.5, 12) {$\scriptstyle 1$};
		\node [style=vdomino] (7) at (-11.5, 12) {$\scriptstyle 1$};
		\node [style=vdomino] (8) at (-9, 12) {$\scriptstyle 1$};
		\node [style=vdomino] (9) at (-8, 12) {$\scriptstyle 1$};
		\node [style=hdomino] (10) at (-6.5, 12.5) {$\scriptstyle 1$};
		\node [style=hdomino] (11) at (-6.5, 11.5) {$\scriptstyle 2$};
		\node [style=hdomino] (12) at (-3, 12.5) {$\scriptstyle 1$};
		\node [style=hdomino] (13) at (-3, 11.5) {$\scriptstyle 2$};
		\node [style=vdomino] (14) at (-12.5, 6.5) {$\scriptstyle e+4$};
		\node [style=vdomino] (15) at (-11.5, 6.5) {$\scriptstyle e+4$};
		\node [style=hdomino] (16) at (-12, 5) {$\scriptstyle e+5$};
		\node [style=hdomino] (17) at (-12, 4) {$\scriptstyle e+6$};
		\node [style=hdomino] (18) at (-12, 1.5) {$\scriptstyle d-1$};
		\node [style=hdomino] (19) at (-12, 0.5) {$\scriptstyle d$};
		\node [style=none] (20) at (-10.25, 12) {$\dots$};
		\node [style=none] (21) at (-4.75, 12.5) {$\dots$};
		\node [style=none] (22) at (-4.75, 11.5) {$\dots$};
		\node [style=none] (23) at (-5.5, 12) {};
		\node [style=none] (24) at (-4, 12) {};
		\node [style=none] (25) at (-12, 9) {};
		\node [style=none] (26) at (-12, 7.5) {};
		\node [style=none] (27) at (-12.5, 8.5) {$\vdots$};
		\node [style=none] (28) at (-11.5, 8.5) {$\vdots$};
		\node [style=none] (29) at (-12, 3) {$\vdots$};
		\node [style=vdomino] (30) at (-12.5, 10) {$\scriptstyle 2$};
		\node [style=vdomino] (31) at (-11.5, 10) {$\scriptstyle 2$};
		\node [style=none] (32) at (-13, 13.5) {};
		\node [style=none] (33) at (-7.5, 13.5) {};
		\node [style=none] (34) at (-2, 13.5) {};
		\node [style=none] (35) at (-10.25, 14) {$2c-2a-2$};
		\node [style=none] (36) at (-4.75, 14) {$2(2a+2-c)$};
		\node [style=none] (37) at (-13.5, 13) {};
		\node [style=none] (38) at (-13.5, 5.5) {};
		\node [style=none] (39) at (-13.5, 0) {};
		\node [style=Rotated none] (40) at (-14, 9.25) {$2(2a+4-d)$};
		\node [style=Rotated none] (41) at (-14, 2.75) {$2d-2a-4$};
		\end{pgfonlayer}
		\begin{pgfonlayer}{edgelayer}
		\draw [style=Border edge] (2.center) to (4.center);
		\draw [style=Border edge] (4.center) to (5.center);
		\draw [style=Border edge] (5.center) to (3.center);
		\draw [style=Border edge] (3.center) to (1.center);
		\draw [style=Border edge] (1.center) to (0.center);
		\draw [style=Border edge] (0.center) to (2.center);
		\draw (23.center) to (24.center);
		\draw (25.center) to (26.center);
		\draw [style=measuredots] (32.center) to (33.center);
		\draw [style=measuredots] (33.center) to (34.center);
		\draw [style=measuredots] (37.center) to (38.center);
		\draw [style=measuredots] (38.center) to (39.center);
		\end{pgfonlayer}
		\end{tikzpicture}
		\caption{The domino $\mu$-tableau $T_{\varphi}$ with $\mu=(a+1,1^{a+1}),\varphi=(c-a-1,1^{2a-d+3}) $ and $e=2a-d$. We see $2c-2a-2+2a+2-c=c$ dominoes with $1$ and $2a+2-c+2=2a+4-c$ dominoes with $2$. The weight of $T_{\varphi}$ is thus $\lambda=(c,2a+4-c,2^{2a+2-d},1^{2d-2a-4})$.}
		\label{Fig hooks}
	\end{figure}
	
	By switching the roles of $c$ and $d$ we obtain the partition $\lambda'$. Since $|\mu|=2a+2$ is even, we conclude that $p(\mu',\bar{\nu};\lambda)=p(\mu,\bar{\nu};\lambda')>0$ if and only if $\bar{\nu}=(2)^{\prime (a+c-d)}=(2)^{\prime (a+c+d)}$, as required.     
\end{proof} 

The series of the next three lemmas deals with almost rectangular partitions $\mu$.

\begin{lemma}\label{Lemma square and a box}
	Let $a\geq 2$ and $\mu=(a+1,a^{a-1})$. For any partitions $\nu,\bar{\nu}\vdash 2$ the symmetric function $s_{\nu}\circ s_{\mu} + s_{\bar{\nu}}\circ s_{\mu'}$ is \emph{not} multiplicity-free.
\end{lemma}

\begin{proof}
	Let $\lambda=(2a,(2a-1)^{a-1},2,1^{a-1})$. According to Lemma~\ref{Lemma find two}, it is sufficient to show that $c(\mu,\mu;\lambda)$ and $c(\mu',\mu';\lambda)$ are at least $2$. See Figure~\ref{Figure square and a box} for the required semistandard tableaux for the Littlewood--Richardson rule as stated in Theorem~\ref{Theorem LR rule}.
\end{proof}

\begin{figure}[h]
	\begin{tikzpicture}[x=0.5cm, y=0.5cm]
	\begin{pgfonlayer}{nodelayer}
	\node [style=none] (0) at (-11, 23) {};
	\node [style=none] (1) at (-5, 23) {};
	\node [style=none] (2) at (-5, 22) {};
	\node [style=none] (3) at (-6, 22) {};
	\node [style=none] (4) at (-6, 18) {};
	\node [style=none] (5) at (-11, 18) {};
	\node [style=none] (6) at (-1, 23) {};
	\node [style=none] (7) at (-1, 22) {};
	\node [style=none] (8) at (-2, 22) {};
	\node [style=none] (9) at (-2, 18) {};
	\node [style=none] (10) at (-9, 17) {};
	\node [style=none] (11) at (-9, 18) {};
	\node [style=none] (12) at (-10, 17) {};
	\node [style=none] (13) at (-10, 18) {};
	\node [style=none] (14) at (-11, 17) {};
	\node [style=none] (15) at (-11, 13) {};
	\node [style=none] (16) at (-10, 13) {};
	\node [style=none] (17) at (-11, 14) {};
	\node [style=none] (18) at (-10, 14) {};
	\node [style=none] (19) at (-4, 22) {};
	\node [style=none] (20) at (-4, 23) {};
	\node [style=none] (21) at (-2, 23) {};
	\node [style=none] (22) at (-5, 21) {};
	\node [style=none] (23) at (-6, 21) {};
	\node [style=none] (24) at (-5, 19) {};
	\node [style=none] (25) at (-6, 19) {};
	\node [style=none] (26) at (-5, 18) {};
	\node [style=none] (27) at (-3, 19) {};
	\node [style=none] (28) at (-3, 18) {};
	\node [style=none] (29) at (-2, 19) {};
	\node [style=none] (30) at (-4.5, 22.5) {};
	\node [style=none] (31) at (-4.5, 22.5) {$1$};
	\node [style=none] (32) at (-1.5, 22.5) {};
	\node [style=none] (33) at (-1.5, 22.5) {$1$};
	\node [style=none] (34) at (-2, 23) {};
	\node [style=none] (35) at (-3, 22) {};
	\node [style=none] (36) at (-3, 21) {};
	\node [style=none] (37) at (-2, 21) {};
	\node [style=none] (38) at (-5.5, 18.5) {$a$};
	\node [style=none] (39) at (-2.5, 18.5) {$a$};
	\node [style=none] (40) at (-5.5, 21.5) {$2$};
	\node [style=none] (41) at (-2.5, 21.5) {$2$};
	\node [style=none] (42) at (-3, 22.5) {$\dots$};
	\node [style=none] (44) at (-5.5, 20.25) {$\vdots$};
	\node [style=none] (45) at (-2.5, 20.25) {$\vdots$};
	\node [style=none] (46) at (-4, 21.5) {$\dots$};
	\node [style=none] (47) at (-4, 18.5) {$\dots$};
	\node [style=none] (48) at (-4, 20.25) {$\ddots$};
	\node [style=none] (49) at (-10.5, 17.5) {$1$};
	\node [style=none] (50) at (-9.5, 17.5) {$1$};
	\node [style=none] (51) at (-11, 16) {};
	\node [style=none] (52) at (-10, 16) {};
	\node [style=none] (53) at (-10.5, 16.5) {};
	\node [style=none] (54) at (-10.5, 16.5) {$2$};
	\node [style=none] (55) at (-10.5, 13.5) {$a$};
	\node [style=none] (56) at (-10.5, 15.25) {$\vdots$};
	\node [style=none] (57) at (0, 23) {};
	\node [style=none] (58) at (6, 23) {};
	\node [style=none] (59) at (6, 22) {};
	\node [style=none] (60) at (5, 22) {};
	\node [style=none] (61) at (5, 18) {};
	\node [style=none] (62) at (0, 18) {};
	\node [style=none] (63) at (0, 17) {};
	\node [style=none] (64) at (1, 17) {};
	\node [style=none] (65) at (2, 17) {};
	\node [style=none] (66) at (2, 18) {};
	\node [style=none] (67) at (1, 18) {};
	\node [style=none] (68) at (0, 16) {};
	\node [style=none] (69) at (1, 16) {};
	\node [style=none] (70) at (0, 13) {};
	\node [style=none] (71) at (1, 13) {};
	\node [style=none] (72) at (0, 14) {};
	\node [style=none] (73) at (1, 14) {};
	\node [style=none] (74) at (0.5, 17.5) {$1$};
	\node [style=none] (75) at (1.5, 17.5) {$a$};
	\node [style=none] (76) at (0.5, 16.5) {$2$};
	\node [style=none] (77) at (0.5, 13.5) {$a$};
	\node [style=none] (78) at (0.5, 15.25) {$\vdots$};
	\node [style=none] (79) at (7, 23) {};
	\node [style=none] (80) at (7, 22) {};
	\node [style=none] (81) at (7, 21) {};
	\node [style=none] (82) at (6, 21) {};
	\node [style=none] (83) at (5, 21) {};
	\node [style=none] (85) at (9, 22) {};
	\node [style=none] (86) at (10, 23) {};
	\node [style=none] (87) at (10, 22) {};
	\node [style=none] (90) at (9, 21) {};
	\node [style=none] (91) at (9, 19) {};
	\node [style=none] (93) at (7, 19) {};
	\node [style=none] (94) at (6, 19) {};
	\node [style=none] (95) at (5, 19) {};
	\node [style=none] (96) at (6, 18) {};
	\node [style=none] (97) at (7, 18) {};
	\node [style=none] (99) at (9, 18) {};
	\node [style=none] (100) at (6.5, 22.5) {$1$};
	\node [style=none] (103) at (5.5, 21.5) {$1$};
	\node [style=none] (104) at (6.5, 21.5) {$2$};
	\node [style=none] (107) at (5.5, 20.25) {$\vdots$};
	\node [style=none] (108) at (6.5, 20.25) {$\vdots$};
	\node [style=none] (113) at (6.5, 18.5) {$a$};
	\node [style=none] (114) at (5.5, 18.5) {$a^-$};
	\node [style=none] (115) at (10, 21) {};
	\node [style=none] (116) at (10, 19) {};
	\node [style=none] (117) at (10, 18) {};
	\node [style=none] (118) at (11, 22) {};
	\node [style=none] (119) at (11, 23) {};
	\node [style=none] (120) at (10.5, 22.5) {$1$};
	\node [style=none] (121) at (8.5, 22.5) {$\dots$};
	\node [style=none] (122) at (8, 21.5) {$\dots$};
	\node [style=none] (123) at (8, 20.25) {$\ddots$};
	\node [style=none] (124) at (9.5, 20.25) {$\vdots$};
	\node [style=none] (125) at (9.5, 21.5) {$2$};
	\node [style=none] (126) at (9.5, 18.5) {$a$};
	\node [style=none] (128) at (8, 18.5) {$\dots$};
	\node [style=none] (129) at (-11, 11) {};
	\node [style=none] (130) at (-6, 11) {};
	\node [style=none] (131) at (-6, 6) {};
	\node [style=none] (132) at (-10, 6) {};
	\node [style=none] (133) at (-10, 5) {};
	\node [style=none] (134) at (-11, 5) {};
	\node [style=none] (135) at (-5, 11) {};
	\node [style=none] (136) at (-2, 11) {};
	\node [style=none] (137) at (-1, 11) {};
	\node [style=none] (138) at (-1, 10) {};
	\node [style=none] (139) at (-6, 10) {};
	\node [style=none] (140) at (-2, 6) {};
	\node [style=none] (141) at (-3, 6) {};
	\node [style=none] (142) at (-5, 6) {};
	\node [style=none] (143) at (-2, 7) {};
	\node [style=none] (144) at (-2, 9) {};
	\node [style=none] (145) at (-6, 9) {};
	\node [style=none] (146) at (-6, 7) {};
	\node [style=none] (147) at (-3, 10) {};
	\node [style=none] (148) at (-9, 5) {};
	\node [style=none] (149) at (-9, 6) {};
	\node [style=none] (150) at (-11, 4) {};
	\node [style=none] (151) at (-10, 4) {};
	\node [style=none] (152) at (-10, 3) {};
	\node [style=none] (153) at (-11, 3) {};
	\node [style=none] (154) at (-11, 1) {};
	\node [style=none] (155) at (-10, 1) {};
	\node [style=none] (156) at (-10, 0) {};
	\node [style=none] (157) at (-11, 0) {};
	\node [style=none] (158) at (-5.5, 10.5) {$1$};
	\node [style=none] (159) at (-3.5, 10.5) {$\dots$};
	\node [style=none] (160) at (-1.5, 10.5) {$1$};
	\node [style=none] (161) at (-5.5, 9.5) {$2$};
	\node [style=none] (162) at (-4, 9.5) {$\dots$};
	\node [style=none] (163) at (-2.5, 9.5) {$2$};
	\node [style=none] (164) at (-5.5, 8.25) {$\vdots$};
	\node [style=none] (165) at (-4, 8.25) {$\ddots$};
	\node [style=none] (166) at (-2.5, 8.25) {$\vdots$};
	\node [style=none] (167) at (-5.5, 6.5) {$a$};
	\node [style=none] (168) at (-4, 6.5) {$\dots$};
	\node [style=none] (169) at (-2.5, 6.5) {$a$};
	\node [style=none] (170) at (-9.5, 5.5) {$2$};
	\node [style=none] (171) at (-10.5, 4.5) {$3$};
	\node [style=none] (172) at (-10.5, 3.5) {$4$};
	\node [style=none] (173) at (-10.5, 2.25) {$\vdots$};
	\node [style=none] (174) at (-10.5, 0.5) {$a^+$};
	\node [style=none] (175) at (0, 11) {};
	\node [style=none] (176) at (5, 11) {};
	\node [style=none] (177) at (6, 11) {};
	\node [style=none] (178) at (9, 11) {};
	\node [style=none] (179) at (10, 11) {};
	\node [style=none] (180) at (10, 10) {};
	\node [style=none] (181) at (8, 10) {};
	\node [style=none] (182) at (5, 10) {};
	\node [style=none] (183) at (5, 9) {};
	\node [style=none] (184) at (9, 9) {};
	\node [style=none] (185) at (9, 7) {};
	\node [style=none] (186) at (9, 6) {};
	\node [style=none] (187) at (8, 6) {};
	\node [style=none] (188) at (6, 6) {};
	\node [style=none] (189) at (5, 6) {};
	\node [style=none] (190) at (5, 7) {};
	\node [style=none] (191) at (2, 6) {};
	\node [style=none] (192) at (1, 6) {};
	\node [style=none] (193) at (2, 5) {};
	\node [style=none] (194) at (0, 5) {};
	\node [style=none] (195) at (0, 4) {};
	\node [style=none] (196) at (1, 4) {};
	\node [style=none] (197) at (1, 3) {};
	\node [style=none] (198) at (0, 3) {};
	\node [style=none] (199) at (0, 1) {};
	\node [style=none] (200) at (1, 1) {};
	\node [style=none] (201) at (1, 0) {};
	\node [style=none] (202) at (0, 0) {};
	\node [style=none] (203) at (0.5, 0.5) {$a$};
	\node [style=none] (204) at (0.5, 2.25) {$\vdots$};
	\node [style=none] (205) at (0.5, 3.5) {$3$};
	\node [style=none] (206) at (0.5, 4.5) {$2$};
	\node [style=none] (207) at (1.5, 5.5) {$a^+$};
	\node [style=none] (208) at (5.5, 6.5) {$a$};
	\node [style=none] (209) at (7, 6.5) {$\dots$};
	\node [style=none] (210) at (8.5, 6.5) {$a$};
	\node [style=none] (211) at (8.5, 8.25) {$\vdots$};
	\node [style=none] (212) at (7, 8.25) {$\ddots$};
	\node [style=none] (213) at (5.5, 8.25) {$\vdots$};
	\node [style=none] (214) at (5.5, 9.5) {$2$};
	\node [style=none] (215) at (7, 9.5) {$\dots$};
	\node [style=none] (216) at (8.5, 9.5) {$2$};
	\node [style=none] (217) at (9.5, 10.5) {$1$};
	\node [style=none] (218) at (7.5, 10.5) {$\dots$};
	\node [style=none] (219) at (5.5, 10.5) {$1$};
	\node [style=none] (220) at (1, 5) {};
	\node [style=none] (221) at (9, 10) {};
	\node [style=none] (222) at (-2, 10) {};
	\end{pgfonlayer}
	\begin{pgfonlayer}{edgelayer}
	\draw (13.center) to (12.center);
	\draw (14.center) to (12.center);
	\draw (17.center) to (18.center);
	\draw (20.center) to (19.center);
	\draw (19.center) to (2.center);
	\draw (2.center) to (22.center);
	\draw (22.center) to (23.center);
	\draw (25.center) to (24.center);
	\draw (24.center) to (26.center);
	\draw (19.center) to (8.center);
	\draw (34.center) to (8.center);
	\draw (35.center) to (36.center);
	\draw (36.center) to (37.center);
	\draw (27.center) to (29.center);
	\draw (27.center) to (28.center);
	\draw (22.center) to (24.center);
	\draw (22.center) to (36.center);
	\draw (36.center) to (27.center);
	\draw (27.center) to (24.center);
	\draw (51.center) to (52.center);
	\draw (63.center) to (64.center);
	\draw (64.center) to (67.center);
	\draw (69.center) to (68.center);
	\draw (72.center) to (73.center);
	\draw (86.center) to (87.center);
	\draw (87.center) to (59.center);
	\draw (79.center) to (80.center);
	\draw (83.center) to (90.center);
	\draw (85.center) to (99.center);
	\draw (95.center) to (91.center);
	\draw (93.center) to (97.center);
	\draw (94.center) to (96.center);
	\draw (59.center) to (82.center);
	\draw (82.center) to (94.center);
	\draw (80.center) to (93.center);
	\draw (91.center) to (116.center);
	\draw (90.center) to (115.center);
	\draw (145.center) to (144.center);
	\draw (146.center) to (143.center);
	\draw (135.center) to (142.center);
	\draw (147.center) to (141.center);
	\draw (150.center) to (151.center);
	\draw (153.center) to (152.center);
	\draw (154.center) to (155.center);
	\draw (181.center) to (187.center);
	\draw (177.center) to (188.center);
	\draw (183.center) to (184.center);
	\draw (190.center) to (185.center);
	\draw (195.center) to (196.center);
	\draw (198.center) to (197.center);
	\draw (199.center) to (200.center);
	\draw [style=Border edge] (62.center) to (70.center);
	\draw [style=Border edge] (70.center) to (71.center);
	\draw [style=Border edge] (71.center) to (64.center);
	\draw [style=Border edge] (64.center) to (65.center);
	\draw [style=Border edge] (65.center) to (66.center);
	\draw [style=Border edge] (61.center) to (117.center);
	\draw [style=Border edge] (117.center) to (87.center);
	\draw [style=Border edge] (87.center) to (118.center);
	\draw [style=Border edge] (118.center) to (119.center);
	\draw [style=Border edge] (58.center) to (119.center);
	\draw [style=Grey diagram] (57.center)
	to (58.center)
	to (59.center)
	to (60.center)
	to (61.center)
	to (62.center)
	to cycle;
	\draw [style=Border edge] (1.center) to (6.center);
	\draw [style=Border edge] (6.center) to (7.center);
	\draw [style=Border edge] (7.center) to (8.center);
	\draw [style=Border edge] (8.center) to (9.center);
	\draw [style=Border edge] (9.center) to (4.center);
	\draw [style=Border edge] (11.center) to (10.center);
	\draw [style=Border edge] (10.center) to (12.center);
	\draw [style=Border edge] (12.center) to (16.center);
	\draw [style=Border edge] (16.center) to (15.center);
	\draw [style=Border edge] (15.center) to (5.center);
	\draw [style=Grey diagram] (5.center)
	to (0.center)
	to (1.center)
	to (2.center)
	to (3.center)
	to (4.center)
	to cycle;
	\draw [style=Grey diagram] (189.center)
	to (192.center)
	to (220.center)
	to (194.center)
	to (175.center)
	to (176.center)
	to cycle;
	\draw [style=Border edge] (176.center) to (179.center);
	\draw [style=Border edge] (179.center) to (180.center);
	\draw [style=Border edge] (180.center) to (221.center);
	\draw [style=Border edge] (221.center) to (186.center);
	\draw [style=Border edge] (186.center) to (189.center);
	\draw [style=Border edge] (191.center) to (193.center);
	\draw [style=Border edge] (193.center) to (220.center);
	\draw [style=Border edge] (220.center) to (201.center);
	\draw [style=Border edge] (201.center) to (202.center);
	\draw [style=Border edge] (202.center) to (194.center);
	\draw (182.center) to (221.center);
	\draw (221.center) to (178.center);
	\draw (136.center) to (222.center);
	\draw (222.center) to (139.center);
	\draw [style=Grey diagram] (130.center)
	to (131.center)
	to (132.center)
	to (133.center)
	to (134.center)
	to (129.center)
	to cycle;
	\draw [style=Border edge] (130.center) to (137.center);
	\draw [style=Border edge] (137.center) to (138.center);
	\draw [style=Border edge] (138.center) to (222.center);
	\draw [style=Border edge] (222.center) to (140.center);
	\draw [style=Border edge] (140.center) to (131.center);
	\draw [style=Border edge] (149.center) to (148.center);
	\draw [style=Border edge] (148.center) to (133.center);
	\draw [style=Border edge] (133.center) to (156.center);
	\draw [style=Border edge] (156.center) to (157.center);
	\draw [style=Border edge] (157.center) to (134.center);
	\end{pgfonlayer}
	\end{tikzpicture}
	\caption{Write $a^{\pm}=a\pm 1$. The first line consists of two semistandard $\lambda/\mu$-tableaux with weight $\mu$, where $\lambda=(2a,(2a-1)^{a-1},2,1^{a-1})$ and $\mu=(a+1,a^{a-1})$. The second line consists of two semistandard $\lambda/\mu'$-tableaux with weight $\mu'$. All four tableaux have a latticed reading word, establishing $c(\mu,\mu;\lambda)\geq 2$ and $c(\mu',\mu';\lambda)\geq 2$.}
	\label{Figure square and a box}
\end{figure}

\begin{lemma}\label{Lemma two boxes}
	Let $a$ and $b$ be positive integers and $\nu\vdash 2$.
	\begin{enumerate}[label=\textnormal{(\roman*)}]
		\item Let $\mu$ be one of the partitions $(a+1,a^{b-1})$ and $(a^b,1)$. If $\lambda$ is a partition such that $p(\mu,\nu;\lambda)>0$, then there is an $(a,b)$-birectangular partition $\bar{\lambda}$ such that $\bar{\lambda}\subseteq \lambda$.
		\item Let $\mu=(a^{b-1},a-1)$. If $\lambda$ is a partition such that $p(\mu,\nu;\lambda)>0$, then there is an $(a,b)$-birectangular partition $\bar{\lambda}$ such that $\lambda\subseteq \bar{\lambda}$.  
	\end{enumerate}
\end{lemma}

\begin{proof}
	In both parts $c(\mu,\mu;\lambda)>0$. Let $\bar{\mu}=(a^b)$ and $m=ab$.
	\begin{enumerate}[label=\textnormal{(\roman*)}]
		\item From the branching rules, the character $\left( \left( \chi^{\bar{\mu}}\boxtimes \chi^{\bar{\mu}}\right)\Ind^{S_{2m}} \right) \Ind^{S_{2m+2}}= \left(\chi^{\bar{\mu}}\ind^{S_{m+1}}\boxtimes \chi^{\bar{\mu}}\ind^{S_{m+1}} \right)\Ind^{S_{2m+2}}$ has $\left(\chi^{\mu}\boxtimes \chi^{\mu}\right)\Ind^{S_{2m+2}}$ as a constituent. Thus there is an $(a,b)$-birectangular partition $\bar{\lambda}$ such that $\chi^{\bar{\lambda}}\ind^{S_{2m+2}}$ has $\chi^{\lambda}$ as a constituent. The result follows from Corollary~\ref{Cor branching rules}(i).
		\item By Mackey's theorem, $\left( \left( \chi^{\bar{\mu}}\boxtimes \chi^{\bar{\mu}}\right) \Ind^{S_{2m}}\right) \Res_{S_{2m-2}}$ has a constituent $\left(\chi^{\bar{\mu}}\res_{S_{m-1}}\boxtimes \chi^{\bar{\mu}}\res_{S_{m-1}}\right)\Ind^{S_{2m-2}}=\left(\chi^{\mu}\boxtimes \chi^{\mu}\right)\Ind^{S_{2m-2}}$. Thus there is an $(a,b)$-birectangular partition $\bar{\lambda}$ such that $\chi^{\bar{\lambda}}\res_{S_{2m-2}}$ has $\chi^{\lambda}$ as a constituent. The result follows from Corollary~\ref{Cor branching rules}(ii). \qedhere
	\end{enumerate}
\end{proof}

We are now ready to prove the final lemma concerning almost rectangular $\mu$. In the proof we benefit from introducing three different blocks of boxes to depict some of the tableaux. They are described in Figure~\ref{Figure various boxes}.

\begin{figure}[!h]
	\begin{tikzpicture}[x=0.5cm, y=0.5cm]
	\begin{pgfonlayer}{nodelayer}
	\node [style=none] (10) at (0.5, 0) {};
	\node [style=none] (11) at (0.5, 1.75) {};
	\node [style=none] (12) at (-0.5, 1.75) {};
	\node [style=none] (13) at (-0.5, 1) {};
	\node [style=none] (14) at (-1.25, 1) {};
	\node [style=none] (15) at (-1.25, 0) {};
	\node [style=none] (16) at (-2.25, 0) {};
	\node [style=none] (17) at (-2.25, 1.75) {};
	\node [style=none] (18) at (-4, 1.75) {};
	\node [style=none] (19) at (-4, 0) {};
	\node [style=none] (20) at (-5, 0) {};
	\node [style=none] (21) at (-5.75, 0) {};
	\node [style=none] (22) at (-5, 0.75) {};
	\node [style=none] (23) at (-5.75, 0.75) {};
	\node [style=none] (24) at (-5.375, 0.375) {$\scriptstyle 1$};
	\node [style=none] (25) at (-3.125, 0.875) {$\bar{4}$};
	\node [style=none] (26) at (-0.125, 0.875) {$\bar{1}$};
	\node [style=none] (27) at (-0.875, 0.5) {$\downarrow$};
	\node [style=none] (28) at (-2.625, 0.875) {$\uparrow$};
	\node [style=none] (29) at (0.5, 2.75) {};
	\node [style=none] (30) at (0.5, 4.5) {};
	\node [style=none] (31) at (-0.5, 4.5) {};
	\node [style=none] (32) at (-0.5, 3.75) {};
	\node [style=none] (33) at (-1.25, 3.75) {};
	\node [style=none] (34) at (-1.25, 2.75) {};
	\node [style=none] (35) at (-2.25, 2.75) {};
	\node [style=none] (36) at (-2.25, 4.5) {};
	\node [style=none] (37) at (-4, 4.5) {};
	\node [style=none] (38) at (-4, 2.75) {};
	\node [style=none] (39) at (-5, 2.75) {};
	\node [style=none] (40) at (-5.75, 2.75) {};
	\node [style=none] (41) at (-5, 3.5) {};
	\node [style=none] (42) at (-5.75, 3.5) {};
	\end{pgfonlayer}
	\begin{pgfonlayer}{edgelayer}
	\draw (18.center) to (17.center);
	\draw (17.center) to (16.center);
	\draw (16.center) to (19.center);
	\draw (19.center) to (18.center);
	\draw (23.center) to (22.center);
	\draw (22.center) to (20.center);
	\draw (20.center) to (21.center);
	\draw (21.center) to (23.center);
	\draw (12.center) to (11.center);
	\draw (11.center) to (10.center);
	\draw (10.center) to (15.center);
	\draw (15.center) to (14.center);
	\draw (14.center) to (13.center);
	\draw (13.center) to (12.center);
	\draw (37.center) to (36.center);
	\draw (36.center) to (35.center);
	\draw (35.center) to (38.center);
	\draw (38.center) to (37.center);
	\draw (42.center) to (41.center);
	\draw (41.center) to (39.center);
	\draw (39.center) to (40.center);
	\draw (40.center) to (42.center);
	\draw (31.center) to (30.center);
	\draw (30.center) to (29.center);
	\draw (29.center) to (34.center);
	\draw (34.center) to (33.center);
	\draw (33.center) to (32.center);
	\draw (32.center) to (31.center);
	\end{pgfonlayer}
	\end{tikzpicture}
	\caption{Given a positive integer $d$, we introduce the three blocks on the first line to depict some sets of boxes of a Young diagram. The first one is the usual $1\times 1$ box, the second one is a square of $d\times d$ boxes and the final one is obtained from the second one by removing a box from the respective corner (here from the top left corner). Note that if $d=1$, then the final block is empty. To make the first block into a part of a tableau we need to place a number inside it. For the other two blocks, we place $\bar{v}$ inside them to denote that the boxes in row $i$ of these blocks (with $1\leq i\leq d$) contain $(v-1)d+i$. This can be modified by an upward arrow on the right which increases all entries in the rightmost column by $1$ and a downward arrow on the left which decreases all entries in the leftmost column by $1$. An example of these blocks made into parts of a tableau is on the second row. The entries of any column of the middle block are $3d+1,3d+2,\dots,4d$ downwards except the rightmost column which contains $3d+2,3d+3,\dots, 4d+1$ downwards. Similarly, the first column of the third block contains $1,2,\dots, d-1$ downwards.}
	\label{Figure various boxes}
\end{figure}

\begin{lemma}\label{Lemma almost rectangles}
	Let $a>b$ be positive integers and let $\mu$ be one of the partitions $(a+1,a^{b-1}), (a^b,1)$ and $(a^{b-1},a-1)$ of size at least $37$. Finally, let $\nu,\bar{\nu}\vdash 2$. The symmetric function $s_{\nu}\circ s_{\mu} + s_{\bar{\nu}}\circ s_{\mu'}$ is multiplicity-free if and only if $a-b\nmid a$ and $a>b+2$, \emph{or} $a,b,\mu,\nu$ and $\bar{\nu}$ satisfy one of the following:
	\begin{enumerate}[label=\textnormal{(\roman*)}]
		\item $a=2b$ and $\mu=(a+1,a^{b-1})$,
		\item $a=2b$, $\mu=(a^{b-1},a-1)$ and $\nu\neq \bar{\nu}$,
		\item $2a=3b$, $\mu=(a+1,a^{b-1})$ and $\nu\neq \bar{\nu}$.
	\end{enumerate}
\end{lemma}

\begin{proof}
	We know that $s_{\nu}\circ s_{\mu}$ and $s_{\bar{\nu}}\circ s_{\mu'}$ are multiplicity-free from Theorem~\ref{Theorem MF plethysms}(i). Suppose they share a constituent $s_{\lambda}$. In the case that $\mu$ is $(a+1,a^{b-1})$ or $(a^b,1)$, Lemma~\ref{Lemma two boxes}(i) applied with $a,b,\nu$ and $\mu$, respectively, $b,a,\bar{\nu}$ and $\mu'$ shows that $\lambda$ contains both an $(a,b)$-birectangular partition and a $(b,a)$-birectangular partition. Letting $d=a-b$ and $0\leq r\leq d-1$ be a remainder of $a$ modulo $d$, we can use Corollary~\ref{Cor rectangles}(i) with $t=2$ to obtain $1\geq r(d-r)$. Thus $r=0$, \emph{or} $d=2$ and $r=1$. Equivalently, $a-b\mid a$ or $b$ and $a$ are two consecutive odd numbers. If $\mu$ is $(a^{b-1},a-1)$ we reach the same conclusion by replacing both instances of (i) with (ii). Therefore $s_{\nu}\circ s_{\mu} + s_{\bar{\nu}}\circ s_{\mu'}$ is multiplicity-free when $a-b\nmid a$ and $b$ and $a$ are not two consecutive odd integers, which is easily seen to be equivalent to $a-b\nmid a$ and $a>b+2$.
	
	Now, let us consider the cases when $b$ and $a$ are consecutive odd integers or $a-b\mid a$. To conclude that $s_{\nu}\circ s_{\mu}$ and $s_{\bar{\nu}}\circ s_{\mu'}$ share a constituent independently of $\nu$ and $\bar{\nu}$ it is sufficient to find a self-conjugate partition $\lambda$ such that $c(\mu,\mu;\lambda)\geq 2$. Indeed, then also $c(\mu',\mu';\lambda)\geq 2$ and Lemma~\ref{Lemma find two} shows that $p(\mu,\nu;\lambda)=1$ and $p(\mu',\bar{\nu};\lambda)=1$. We refer to such $\lambda$ as a \textit{good partition} for $\mu$. Note that implicitly $\lambda_1=\ell(\lambda)\leq 2\ell(\mu)\leq 2a$ for any such $\lambda$.
	
	Firstly, consider the case when $a$ and $b$ are of the form $2k+1$ and $2k-1$. Fix one of the three forms of $\mu$ (i), (ii) or (iii) from the statement and write $\mu(k)$ for $\mu$ to emphasise the dependency on $k$. If $\alpha$ is a good partition for $\mu(k)$, we can apply Lemma~\ref{Lemma constructions of partitions} with $a=2k+1, b=2k-1, a'=2k+5, b'=2k+3$ and our $\alpha$ (thus $a'-a=4=b'-b, a+a'=4k+6=2b'$ and $\alpha_1\leq 2a$) to obtain a good partition $\beta$ for $\mu(k+2)$. Therefore, using induction on $k$, we only need good partitions $\lambda$ for $\mu(k)$ for two values of $k\leq 5$ of different parities to obtain good partitions for $\mu(k)$ for all $k\geq 4$ which covers all partitions $\mu(k)$ of size at least $37$. Such partitions $\lambda$ are presented in Table~\ref{Table oddxodd}.
	
	\begin{table}[h!]
		\centering
		\begin{tabular}{ccc}
			\toprule
			Form of $\mu$&$k$&$\lambda$\\
			\midrule
			$(a+1,a^{b-1})$&$4$&$(14^4,11,10^2,9,8,7,5,4^3)$\\
			&$5$&$(18^4,15,14^2,13,11^2,10,8^2,7,5,4^3)$\\
			\midrule
			$(a^b,1)$&$2$&$(7,6^2,5,4,3,1)$\\
			&$3$&$(11,10^2,9,7^2,6,4^2,3,1)$\\
			\midrule
			$(a^{b-1},a-1)$&$4$&$(14^3,13,11,10,9^2,8,6,5,4^2,3)$\\
			&$3$&$(10^3,9,7,6,5,4^2,3)$\\\bottomrule
		\end{tabular}
		\vspace{4pt}
		\caption{A table of self-conjugate partitions $\lambda$ such that $c(\mu,\mu;\lambda)\geq 2$ where $\mu$ is of the form given by the table with $a=2k+1$ and $b=2k-1$. The claims about $\lambda$ have been verified by {\sc Magma}.} 
		\label{Table oddxodd}
	\end{table}
	
	Secondly, consider the case when $d=a-b$ divides $a$, that is $a$ and $b$ are of the form $a=(k+1)d$ and $b=kd$ for some positive integer $k$. Again, fix the form of $\mu$ and write $\mu(k,d)$ for $\mu$. This time we use Lemma~\ref{Lemma constructions of partitions} with $a=(k+1)d, b=kd, a'=(k+3)d, b'=(k+2)d$ and a partition $\alpha$ which is good for $\mu(k,d)$ (thus $a'-a=2d=b'-b, a+a'=2(k+2)d=2b'$ and $\alpha_1\leq 2a$) to obtain a good partition $\beta$ for $\mu(k+2,d)$.
	
	Using the diagrammatic notation from Figure~\ref{Figure various boxes}, we present good partitions $\lambda$ for $\mu(k,d)$ together with tableaux proving $c(\mu,\mu;\lambda)\geq 2$ in Figure~\ref{Figure top} for $\mu(k,d)=(a+1,a^{b-1})$ and $k\in\left\lbrace 3,4\right\rbrace $, Figure~\ref{Figure bottom} for $\mu(k,d)=(a^b,1)$ and $k\in\left\lbrace 1,2\right\rbrace $ and Figure~\ref{Figure without} for $\mu(k,d)=(a^{b-1},a-1)$ and $k\in\left\lbrace 2,3\right\rbrace $. By induction on $k$ we therefore have good partitions for all $\mu(k,d)$ with the exceptions given by $\mu(k,d)=(a+1,a^{b-1})$ with $k\in\left\lbrace 1,2\right\rbrace $ and $\mu(k,d)=(a^{b-1},a-1)$ with $k=1$.
	
	\begin{figure}[!h]
		\begin{tikzpicture}[x=0.5cm, y=0.5cm]
		\begin{pgfonlayer}{nodelayer}
		\node [style=none] (49) at (-2.5, 11) {};
		\node [style=none] (50) at (2.75, 11) {};
		\node [style=none] (51) at (4.5, 11) {};
		\node [style=none] (52) at (6.25, 11) {};
		\node [style=none] (53) at (8, 11) {};
		\node [style=none] (54) at (8, 9.25) {};
		\node [style=none] (55) at (8, 7.5) {};
		\node [style=none] (56) at (6.25, 7.5) {};
		\node [style=none] (57) at (4.5, 7.5) {};
		\node [style=none] (58) at (5.25, 7.5) {};
		\node [style=none] (59) at (5.25, 6.75) {};
		\node [style=none] (60) at (4.5, 6.75) {};
		\node [style=none] (61) at (4.5, 5.75) {};
		\node [style=none] (62) at (4.5, 4) {};
		\node [style=none] (63) at (2.75, 4) {};
		\node [style=none] (64) at (2.75, 5.75) {};
		\node [style=none] (65) at (2.75, 7.5) {};
		\node [style=none] (66) at (2.75, 9.25) {};
		\node [style=none] (67) at (-1.75, 4) {};
		\node [style=none] (68) at (-1.75, 3.25) {};
		\node [style=none] (69) at (-2.5, 3.25) {};
		\node [style=none] (70) at (-0.75, 4) {};
		\node [style=none] (71) at (1, 4) {};
		\node [style=none] (72) at (1.75, 4) {};
		\node [style=none] (73) at (1.75, 3.25) {};
		\node [style=none] (74) at (1, 3.25) {};
		\node [style=none] (75) at (1, 2.25) {};
		\node [style=none] (76) at (1, 0.5) {};
		\node [style=none] (77) at (-0.75, 0.5) {};
		\node [style=none] (78) at (-2.5, 0.5) {};
		\node [style=none] (79) at (-2.5, 2.25) {};
		\node [style=none] (80) at (5.375, 10.125) {$\bar{1}$};
		\node [style=none] (81) at (7.125, 10.125) {$\bar{1}$};
		\node [style=none] (82) at (7.125, 8.375) {$\bar{2}$};
		\node [style=none] (83) at (5.375, 8.375) {$\bar{2}$};
		\node [style=none] (84) at (3.625, 8.375) {$\bar{2}$};
		\node [style=none] (85) at (4.875, 7.125) {$\scriptstyle d_2$};
		\node [style=none] (86) at (3.6, 6.625) {$\bar{3}$};
		\node [style=none] (87) at (3.625, 4.875) {$\bar{4}$};
		\node [style=none] (88) at (1.375, 3.625) {$\scriptstyle d_4$};
		\node [style=none] (89) at (0.125, 3.125) {$\bar{3}$};
		\node [style=none] (90) at (-1.375, 3.125) {$\bar{3}$};
		\node [style=none] (91) at (-1.625, 1.375) {$\bar{4}$};
		\node [style=none] (92) at (0.125, 1.375) {$\bar{4}$};
		\node [style=none] (93) at (3.625, 10.125) {$\bar{1}$};
		\node [style=none] (102) at (9, 11) {};
		\node [style=none] (103) at (14.25, 11) {};
		\node [style=none] (104) at (16, 11) {};
		\node [style=none] (105) at (17.75, 11) {};
		\node [style=none] (106) at (19.5, 11) {};
		\node [style=none] (107) at (19.5, 9.25) {};
		\node [style=none] (108) at (19.5, 7.5) {};
		\node [style=none] (109) at (17.75, 7.5) {};
		\node [style=none] (110) at (16, 7.5) {};
		\node [style=none] (111) at (16.75, 7.5) {};
		\node [style=none] (112) at (16.75, 6.75) {};
		\node [style=none] (113) at (16, 6.75) {};
		\node [style=none] (114) at (16, 5.75) {};
		\node [style=none] (115) at (16, 4) {};
		\node [style=none] (116) at (14.25, 4) {};
		\node [style=none] (117) at (14.25, 5.75) {};
		\node [style=none] (118) at (14.25, 7.5) {};
		\node [style=none] (119) at (14.25, 9.25) {};
		\node [style=none] (120) at (9.75, 4) {};
		\node [style=none] (121) at (9.75, 3.25) {};
		\node [style=none] (122) at (9, 3.25) {};
		\node [style=none] (123) at (10.75, 4) {};
		\node [style=none] (124) at (12.5, 4) {};
		\node [style=none] (125) at (13.25, 4) {};
		\node [style=none] (126) at (13.25, 3.25) {};
		\node [style=none] (127) at (12.5, 3.25) {};
		\node [style=none] (128) at (12.5, 2.25) {};
		\node [style=none] (129) at (12.5, 0.5) {};
		\node [style=none] (130) at (10.75, 0.5) {};
		\node [style=none] (131) at (9, 0.5) {};
		\node [style=none] (132) at (9, 2.25) {};
		\node [style=none] (133) at (16.875, 10.125) {$\bar{1}$};
		\node [style=none] (134) at (18.625, 10.125) {$\bar{1}$};
		\node [style=none] (135) at (18.625, 8.375) {$\bar{2}$};
		\node [style=none] (136) at (16.875, 8.375) {$\bar{2}$};
		\node [style=none] (137) at (15.125, 8.375) {$\bar{2}$};
		\node [style=none] (138) at (16.375, 7.125) {$\scriptstyle d_2$};
		\node [style=none] (139) at (15.1, 6.625) {$\bar{3}$};
		\node [style=none] (140) at (15.125, 4.875) {$\bar{4}$};
		\node [style=none] (141) at (12.875, 3.625) {$\scriptstyle c$};
		\node [style=none] (142) at (11.375, 3.125) {$\bar{3}$};
		\node [style=none] (143) at (10.125, 3.125) {$\bar{3}$};
		\node [style=none] (144) at (9.875, 1.375) {$\bar{4}$};
		\node [style=none] (145) at (11.625, 1.375) {$\bar{4}$};
		\node [style=none] (146) at (15.125, 10.125) {$\bar{1}$};
		\node [style=none] (147) at (12.125, 2.75) {$\uparrow$};
		\node [style=none] (148) at (12.125, 1.375) {$\uparrow$};
		\node [style=none] (149) at (12.125, 3.625) {$\scriptstyle d_2$};
		\node [style=none] (150) at (11.75, 4) {};
		\node [style=none] (151) at (11.75, 3.25) {};
		\end{pgfonlayer}
		\begin{pgfonlayer}{edgelayer}
		\draw (70.center) to (77.center);
		\draw (52.center) to (56.center);
		\draw (66.center) to (54.center);
		\draw (64.center) to (61.center);
		\draw (79.center) to (75.center);
		\draw (65.center) to (58.center);
		\draw (51.center) to (60.center);
		\draw [style=Border edge] (78.center) to (76.center);
		\draw [style=Border edge] (76.center) to (74.center);
		\draw [style=Border edge] (74.center) to (73.center);
		\draw [style=Border edge] (73.center) to (72.center);
		\draw [style=Border edge] (62.center) to (60.center);
		\draw [style=Border edge] (60.center) to (59.center);
		\draw [style=Border edge] (59.center) to (58.center);
		\draw [style=Border edge] (58.center) to (55.center);
		\draw [style=Border edge] (55.center) to (53.center);
		\draw (71.center) to (74.center);
		\draw (123.center) to (130.center);
		\draw (105.center) to (109.center);
		\draw (119.center) to (107.center);
		\draw (117.center) to (114.center);
		\draw (132.center) to (128.center);
		\draw (118.center) to (111.center);
		\draw (104.center) to (113.center);
		\draw [style=Border edge] (131.center) to (129.center);
		\draw [style=Border edge] (129.center) to (127.center);
		\draw [style=Border edge] (127.center) to (126.center);
		\draw [style=Border edge] (126.center) to (125.center);
		\draw [style=Border edge] (115.center) to (113.center);
		\draw [style=Border edge] (113.center) to (112.center);
		\draw [style=Border edge] (112.center) to (111.center);
		\draw [style=Border edge] (111.center) to (108.center);
		\draw [style=Border edge] (108.center) to (106.center);
		\draw (124.center) to (127.center);
		\draw [style=Border edge] (50.center) to (53.center);
		\draw [style=Border edge] (62.center) to (63.center);
		\draw [style=Border edge] (69.center) to (78.center);
		\draw [style=Grey diagram] (63.center)
		to (67.center)
		to (68.center)
		to (69.center)
		to (49.center)
		to (50.center)
		to cycle;
		\draw [style=Border edge] (103.center) to (106.center);
		\draw [style=Border edge] (115.center) to (116.center);
		\draw [style=Border edge] (122.center) to (131.center);
		\draw [style=Grey diagram] (102.center)
		to (103.center)
		to (116.center)
		to (120.center)
		to (121.center)
		to (122.center)
		to cycle;
		\draw (150.center) to (151.center);
		\draw (151.center) to (127.center);
		\end{pgfonlayer}
		\end{tikzpicture}
		
		\vspace*{0.5cm}
		\hspace*{-1.125cm}
		\begin{tikzpicture}[x=0.5cm, y=0.5cm]
		\begin{pgfonlayer}{nodelayer}
		\node [style=none] (88) at (11.9, 2.125) {$\scriptstyle d_5$};
		\node [style=none] (95) at (8, 13) {};
		\node [style=none] (96) at (15, 13) {};
		\node [style=none] (97) at (16.75, 13) {};
		\node [style=none] (98) at (18.5, 13) {};
		\node [style=none] (99) at (20.25, 13) {};
		\node [style=none] (100) at (22, 13) {};
		\node [style=none] (101) at (22, 11.25) {};
		\node [style=none] (102) at (22, 9.5) {};
		\node [style=none] (103) at (20.25, 9.5) {};
		\node [style=none] (104) at (18.5, 9.5) {};
		\node [style=none] (105) at (15, 9.5) {};
		\node [style=none] (106) at (15, 7.75) {};
		\node [style=none] (107) at (18.5, 7.75) {};
		\node [style=none] (108) at (18.5, 6) {};
		\node [style=none] (109) at (16.75, 6) {};
		\node [style=none] (110) at (15, 6) {};
		\node [style=none] (111) at (15, 4.25) {};
		\node [style=none] (112) at (15, 11.25) {};
		\node [style=none] (113) at (13.25, 4.25) {};
		\node [style=none] (114) at (11.5, 4.25) {};
		\node [style=none] (115) at (9.75, 4.25) {};
		\node [style=none] (116) at (8.75, 4.25) {};
		\node [style=none] (117) at (8.75, 3.5) {};
		\node [style=none] (118) at (8, 3.5) {};
		\node [style=none] (119) at (15, 2.5) {};
		\node [style=none] (120) at (13.25, 2.5) {};
		\node [style=none] (121) at (11.5, 2.5) {};
		\node [style=none] (122) at (8, 2.5) {};
		\node [style=none] (123) at (8, 0.75) {};
		\node [style=none] (124) at (11.5, 0.75) {};
		\node [style=none] (125) at (8, -1) {};
		\node [style=none] (126) at (9.75, -1) {};
		\node [style=none] (127) at (11.5, -1) {};
		\node [style=none] (128) at (12.25, 2.5) {};
		\node [style=none] (129) at (11.5, 1.75) {};
		\node [style=none] (130) at (12.25, 1.75) {};
		\node [style=none] (131) at (19.25, 9.5) {};
		\node [style=none] (132) at (18.5, 8.75) {};
		\node [style=none] (133) at (19.25, 8.75) {};
		\node [style=none] (134) at (15.875, 12.1) {$\bar{1}$};
		\node [style=none] (138) at (15.875, 10.375) {$\bar{2}$};
		\node [style=none] (139) at (17.625, 10.375) {$\bar{2}$};
		\node [style=none] (140) at (19.375, 10.375) {$\bar{2}$};
		\node [style=none] (141) at (21.125, 10.375) {$\bar{2}$};
		\node [style=none] (144) at (17.625, 8.625) {$\bar{3}$};
		\node [style=none] (145) at (15.875, 6.875) {$\bar{4}$};
		\node [style=none] (146) at (17.625, 6.875) {$\bar{4}$};
		\node [style=none] (147) at (9.1, 3.375) {$\bar{3}$};
		\node [style=none] (148) at (10.6, 3.375) {$\bar{3}$};
		\node [style=none] (149) at (12.3, 3.375) {$\bar{5}$};
		\node [style=none] (150) at (14.125, 3.375) {$\bar{5}$};
		\node [style=none] (151) at (8.875, 1.6) {$\bar{4}$};
		\node [style=none] (152) at (10.625, 1.625) {$\bar{4}$};
		\node [style=none] (153) at (8.875, -0.125) {$\bar{5}$};
		\node [style=none] (154) at (10.625, -0.125) {$\bar{5}$};
		\node [style=none] (156) at (17.625, 12.125) {$\bar{1}$};
		\node [style=none] (158) at (21.125, 12.125) {$\bar{1}$};
		\node [style=none] (159) at (19.375, 12.125) {$\bar{1}$};
		\node [style=none] (161) at (15.875, 8.625) {$\bar{3}$};
		\node [style=none] (162) at (18.875, 9.125) {$\scriptstyle d_2$};
		\node [style=none] (170) at (26.875, 3) {$\downarrow$};
		\node [style=none] (171) at (26.125, 3) {$\uparrow$};
		\node [style=none] (172) at (26.125, 1.625) {$\uparrow$};
		\node [style=none] (173) at (26.15, -0.125) {$\uparrow$};
		\node [style=none] (174) at (26.875, 3.875) {$\scriptstyle c$};
		\node [style=none] (175) at (26.125, 3.875) {$\scriptstyle d_2$};
		\node [style=none] (176) at (26.9, 2.125) {$\scriptstyle 5d$};
		\node [style=none] (177) at (23, 13) {};
		\node [style=none] (178) at (30, 13) {};
		\node [style=none] (179) at (31.75, 13) {};
		\node [style=none] (180) at (33.5, 13) {};
		\node [style=none] (181) at (35.25, 13) {};
		\node [style=none] (182) at (37, 13) {};
		\node [style=none] (183) at (37, 11.25) {};
		\node [style=none] (184) at (37, 9.5) {};
		\node [style=none] (185) at (35.25, 9.5) {};
		\node [style=none] (186) at (33.5, 9.5) {};
		\node [style=none] (187) at (30, 9.5) {};
		\node [style=none] (188) at (30, 7.75) {};
		\node [style=none] (189) at (33.5, 7.75) {};
		\node [style=none] (190) at (33.5, 6) {};
		\node [style=none] (191) at (31.75, 6) {};
		\node [style=none] (192) at (30, 6) {};
		\node [style=none] (193) at (30, 4.25) {};
		\node [style=none] (194) at (30, 11.25) {};
		\node [style=none] (195) at (28.25, 4.25) {};
		\node [style=none] (196) at (26.5, 4.25) {};
		\node [style=none] (197) at (24.75, 4.25) {};
		\node [style=none] (198) at (23.75, 4.25) {};
		\node [style=none] (199) at (23.75, 3.5) {};
		\node [style=none] (200) at (23, 3.5) {};
		\node [style=none] (201) at (30, 2.5) {};
		\node [style=none] (202) at (28.25, 2.5) {};
		\node [style=none] (203) at (26.5, 2.5) {};
		\node [style=none] (204) at (23, 2.5) {};
		\node [style=none] (205) at (23, 0.75) {};
		\node [style=none] (206) at (26.5, 0.75) {};
		\node [style=none] (207) at (23, -1) {};
		\node [style=none] (208) at (24.75, -1) {};
		\node [style=none] (209) at (26.5, -1) {};
		\node [style=none] (210) at (27.25, 2.5) {};
		\node [style=none] (211) at (26.5, 1.75) {};
		\node [style=none] (212) at (27.25, 1.75) {};
		\node [style=none] (213) at (34.25, 9.5) {};
		\node [style=none] (214) at (33.5, 8.75) {};
		\node [style=none] (215) at (34.25, 8.75) {};
		\node [style=none] (216) at (30.875, 12.1) {$\bar{1}$};
		\node [style=none] (217) at (30.875, 10.375) {$\bar{2}$};
		\node [style=none] (218) at (32.625, 10.375) {$\bar{2}$};
		\node [style=none] (219) at (34.375, 10.375) {$\bar{2}$};
		\node [style=none] (220) at (36.125, 10.375) {$\bar{2}$};
		\node [style=none] (221) at (32.625, 8.625) {$\bar{3}$};
		\node [style=none] (222) at (30.875, 6.875) {$\bar{4}$};
		\node [style=none] (223) at (32.625, 6.875) {$\bar{4}$};
		\node [style=none] (224) at (24.1, 3.375) {$\bar{3}$};
		\node [style=none] (225) at (25.35, 3.375) {$\bar{3}$};
		\node [style=none] (226) at (27.55, 3.375) {$\bar{5}$};
		\node [style=none] (227) at (29.125, 3.375) {$\bar{5}$};
		\node [style=none] (228) at (23.875, 1.6) {$\bar{4}$};
		\node [style=none] (229) at (25.625, 1.625) {$\bar{4}$};
		\node [style=none] (230) at (23.875, -0.125) {$\bar{5}$};
		\node [style=none] (231) at (25.625, -0.125) {$\bar{5}$};
		\node [style=none] (232) at (32.625, 12.125) {$\bar{1}$};
		\node [style=none] (233) at (36.125, 12.125) {$\bar{1}$};
		\node [style=none] (234) at (34.375, 12.125) {$\bar{1}$};
		\node [style=none] (235) at (30.875, 8.625) {$\bar{3}$};
		\node [style=none] (236) at (33.875, 9.125) {$\scriptstyle d_2$};
		\node [style=none] (237) at (25.75, 4.25) {};
		\node [style=none] (238) at (27.25, 4.25) {};
		\node [style=none] (239) at (26.5, 3.5) {};
		\node [style=none] (240) at (27.25, 3.5) {};
		\node [style=none] (241) at (25.75, 3.5) {};
		\end{pgfonlayer}
		\begin{pgfonlayer}{edgelayer}
		\draw (112.center) to (101.center);
		\draw (106.center) to (107.center);
		\draw (123.center) to (124.center);
		\draw (115.center) to (126.center);
		\draw (113.center) to (120.center);
		\draw (97.center) to (109.center);
		\draw (99.center) to (103.center);
		\draw (105.center) to (131.center);
		\draw (98.center) to (132.center);
		\draw (122.center) to (128.center);
		\draw (114.center) to (129.center);
		\draw [style=Border edge] (100.center) to (102.center);
		\draw [style=Border edge] (102.center) to (131.center);
		\draw [style=Border edge] (131.center) to (133.center);
		\draw [style=Border edge] (133.center) to (132.center);
		\draw [style=Border edge] (132.center) to (108.center);
		\draw [style=Border edge] (108.center) to (110.center);
		\draw [style=Border edge] (119.center) to (128.center);
		\draw [style=Border edge] (128.center) to (130.center);
		\draw [style=Border edge] (130.center) to (129.center);
		\draw [style=Border edge] (129.center) to (127.center);
		\draw [style=Border edge] (127.center) to (125.center);
		\draw (194.center) to (183.center);
		\draw (188.center) to (189.center);
		\draw (205.center) to (206.center);
		\draw (197.center) to (208.center);
		\draw (195.center) to (202.center);
		\draw (179.center) to (191.center);
		\draw (181.center) to (185.center);
		\draw (187.center) to (213.center);
		\draw (180.center) to (214.center);
		\draw (204.center) to (210.center);
		\draw (196.center) to (211.center);
		\draw [style=Border edge] (182.center) to (184.center);
		\draw [style=Border edge] (184.center) to (213.center);
		\draw [style=Border edge] (213.center) to (215.center);
		\draw [style=Border edge] (215.center) to (214.center);
		\draw [style=Border edge] (214.center) to (190.center);
		\draw [style=Border edge] (190.center) to (192.center);
		\draw [style=Border edge] (201.center) to (210.center);
		\draw [style=Border edge] (210.center) to (212.center);
		\draw [style=Border edge] (212.center) to (211.center);
		\draw [style=Border edge] (211.center) to (209.center);
		\draw [style=Border edge] (209.center) to (207.center);
		\draw [style=Border edge] (96.center) to (100.center);
		\draw [style=Border edge] (111.center) to (119.center);
		\draw [style=Border edge] (118.center) to (125.center);
		\draw [style=Grey diagram] (111.center)
		to (116.center)
		to (117.center)
		to (118.center)
		to (95.center)
		to (96.center)
		to cycle;
		\draw [style=Border edge] (178.center) to (182.center);
		\draw [style=Border edge] (193.center) to (201.center);
		\draw [style=Border edge] (200.center) to (207.center);
		\draw [style=Grey diagram] (200.center)
		to (177.center)
		to (178.center)
		to (193.center)
		to (198.center)
		to (199.center)
		to cycle;
		\draw (237.center) to (241.center);
		\draw (241.center) to (240.center);
		\draw (240.center) to (238.center);
		\end{pgfonlayer}
		\end{tikzpicture}
		\caption{For $d\in \mathbb{N}$ set $d_j=jd+1$ and $c=2d+2$. The top, respectively, bottom row contains two $\lambda/\mu'$-tableaux where $\lambda$ is a self-conjugate partition and $\mu=(a+1,a^{b-1})$ with $a=(k+1)d, b=kd$ and $k=3$, respectively, $k=4$ using the notation from Figure~\ref{Figure various boxes}. All the tableaux are semistandard with latticed reading words and weights equal to $\mu'$ proving $c(\mu',\mu';\lambda)\geq 2$, and in turn $c(\mu,\mu;\lambda)\geq 2$. Let us check this for the top right tableau. To see the tableau is semistandard, note that the neighbours of the two boxes with $d_2$ contain $2d$ (top), $2d+1$ (left), $2d+3$ (bottom) and $2d+2$ (right). Counting the boxes with $i\leq 2d, i=2d+1,2d+2$ and $i\geq 2d+3$, separately, yields the correct weight. After swapping the two neighbouring boxes with $d_2$ and $c$, the reading word becomes a concatenation of words of the form $i \:i+1 \dots j$ with $i$'s non-decreasing. This is latticed and so is the original reading word as the swap back preserves the latticed condition.}
		\label{Figure top}
	\end{figure} 
	
	\begin{figure}[!h]
		\begin{tikzpicture}[x=0.5cm, y=0.5cm]
		\begin{pgfonlayer}{nodelayer}
		\node [style=none] (0) at (-9, 1.25) {};
		\node [style=none] (1) at (-8.25, 1.25) {};
		\node [style=none] (2) at (-9, 0.5) {};
		\node [style=none] (3) at (-8.25, 0.5) {};
		\node [style=none] (5) at (-9, 2.25) {};
		\node [style=none] (6) at (-8.25, 2.25) {};
		\node [style=none] (7) at (-8.25, 3) {};
		\node [style=none] (8) at (-7.25, 1.25) {};
		\node [style=none] (9) at (-7.25, 3) {};
		\node [style=none] (10) at (-9, 4.75) {};
		\node [style=none] (11) at (-5.5, 4.75) {};
		\node [style=none] (12) at (-5.5, 3) {};
		\node [style=none] (13) at (-5.5, 1.25) {};
		\node [style=none] (14) at (-4.75, 4.75) {};
		\node [style=none] (15) at (-4.75, 4) {};
		\node [style=none] (16) at (-5.5, 4) {};
		\node [style=none] (17) at (-5.125, 4.375) {$\scriptstyle 1$};
		\node [style=none] (18) at (-6.375, 2.125) {$\bar{1}$};
		\node [style=none] (19) at (-7.875, 2.125) {$\bar{1}$};
		\node [style=none] (20) at (-0.625, 2.125) {$\uparrow$};
		\node [style=none] (41) at (-8.625, 0.875) {$\scriptstyle d_1$};
		\node [style=none] (48) at (-3.75, 1.25) {};
		\node [style=none] (49) at (-3, 1.25) {};
		\node [style=none] (50) at (-3.75, 0.5) {};
		\node [style=none] (51) at (-3, 0.5) {};
		\node [style=none] (52) at (-3.75, 2.25) {};
		\node [style=none] (53) at (-3, 2.25) {};
		\node [style=none] (54) at (-3, 3) {};
		\node [style=none] (55) at (-2, 1.25) {};
		\node [style=none] (56) at (-2, 3) {};
		\node [style=none] (57) at (-3.75, 4.75) {};
		\node [style=none] (58) at (-0.25, 4.75) {};
		\node [style=none] (59) at (-0.25, 3) {};
		\node [style=none] (60) at (-0.25, 1.25) {};
		\node [style=none] (61) at (0.5, 4.75) {};
		\node [style=none] (62) at (0.5, 4) {};
		\node [style=none] (63) at (-0.25, 4) {};
		\node [style=none] (64) at (0.125, 4.375) {$\scriptstyle 1$};
		\node [style=none] (65) at (-1.125, 2.125) {$\bar{1}$};
		\node [style=none] (66) at (-2.625, 2.1) {$\bar{1}$};
		\node [style=none] (68) at (-3.375, 0.875) {$\scriptstyle d$};
		\node [style=none] (69) at (-3.375, 1.75) {$\downarrow$};
		\end{pgfonlayer}
		\begin{pgfonlayer}{edgelayer}
		\draw (0.center) to (1.center);
		\draw (8.center) to (9.center);
		\draw [style=Border edge] (14.center) to (15.center);
		\draw [style=Border edge] (15.center) to (16.center);
		\draw [style=Border edge] (3.center) to (2.center);
		\draw [style=Border edge] (1.center) to (3.center);
		\draw (48.center) to (49.center);
		\draw (55.center) to (56.center);
		\draw [style=Border edge] (61.center) to (62.center);
		\draw [style=Border edge] (62.center) to (63.center);
		\draw [style=Border edge] (51.center) to (50.center);
		\draw [style=Border edge] (49.center) to (51.center);
		\draw [style=Border edge] (1.center) to (13.center);
		\draw [style=Border edge] (60.center) to (49.center);
		\draw [style=Border edge] (11.center) to (14.center);
		\draw [style=Border edge] (12.center) to (13.center);
		\draw [style=Border edge] (5.center) to (2.center);
		\draw [style=Grey diagram] (5.center)
		to (10.center)
		to (11.center)
		to (12.center)
		to (7.center)
		to (6.center)
		to cycle;
		\draw [style=Border edge] (58.center) to (61.center);
		\draw [style=Border edge] (59.center) to (60.center);
		\draw [style=Border edge] (52.center) to (50.center);
		\draw [style=Grey diagram] (57.center)
		to (58.center)
		to (59.center)
		to (54.center)
		to (53.center)
		to (52.center)
		to cycle;
		\end{pgfonlayer}
		\end{tikzpicture}
		
		\vspace*{0.5cm}
		\begin{tikzpicture}[x=0.5cm, y=0.5cm]
		\begin{pgfonlayer}{nodelayer}
		\node [style=none] (67) at (2.25, 6.75) {};
		\node [style=none] (68) at (9.25, 6.75) {};
		\node [style=none] (69) at (10, 6.75) {};
		\node [style=none] (70) at (10, 6) {};
		\node [style=none] (71) at (9.25, 6) {};
		\node [style=none] (72) at (9.25, 5) {};
		\node [style=none] (73) at (7.5, 3.25) {};
		\node [style=none] (74) at (5.75, 3.25) {};
		\node [style=none] (75) at (7.5, 5) {};
		\node [style=none] (76) at (7.5, 6.75) {};
		\node [style=none] (77) at (4, 3.25) {};
		\node [style=none] (78) at (9.25, 3.25) {};
		\node [style=none] (79) at (3, 3.25) {};
		\node [style=none] (80) at (3, 2.5) {};
		\node [style=none] (81) at (2.25, 2.5) {};
		\node [style=none] (82) at (2.25, 1.5) {};
		\node [style=none] (83) at (5.75, 1.5) {};
		\node [style=none] (84) at (5.75, -0.25) {};
		\node [style=none] (85) at (4, -0.25) {};
		\node [style=none] (86) at (2.25, -0.25) {};
		\node [style=none] (87) at (3, -0.25) {};
		\node [style=none] (88) at (3, -1) {};
		\node [style=none] (89) at (2.25, -1) {};
		\node [style=none] (90) at (3.375, 2.375) {$\bar{1}$};
		\node [style=none] (91) at (4.875, 2.375) {$\bar{1}$};
		\node [style=none] (92) at (3.125, 0.625) {$\bar{2}$};
		\node [style=none] (93) at (4.875, 0.625) {$\bar{2}$};
		\node [style=none] (94) at (2.625, -0.625) {$\scriptstyle d_2$};
		\node [style=none] (95) at (9.625, 6.375) {$\scriptstyle 1$};
		\node [style=none] (96) at (8.35, 5.875) {$\bar{1}$};
		\node [style=none] (97) at (8.375, 4.125) {$\bar{2}$};
		\node [style=none] (109) at (11, 6.75) {};
		\node [style=none] (110) at (18, 6.75) {};
		\node [style=none] (111) at (18.75, 6.75) {};
		\node [style=none] (112) at (18.75, 6) {};
		\node [style=none] (113) at (18, 6) {};
		\node [style=none] (114) at (18, 5) {};
		\node [style=none] (115) at (16.25, 3.25) {};
		\node [style=none] (116) at (14.5, 3.25) {};
		\node [style=none] (117) at (16.25, 5) {};
		\node [style=none] (118) at (16.25, 6.75) {};
		\node [style=none] (119) at (12.75, 3.25) {};
		\node [style=none] (120) at (18, 3.25) {};
		\node [style=none] (121) at (11.75, 3.25) {};
		\node [style=none] (122) at (11.75, 2.5) {};
		\node [style=none] (123) at (11, 2.5) {};
		\node [style=none] (124) at (11, 1.5) {};
		\node [style=none] (125) at (14.5, 1.5) {};
		\node [style=none] (126) at (14.5, -0.25) {};
		\node [style=none] (127) at (12.75, -0.25) {};
		\node [style=none] (128) at (11, -0.25) {};
		\node [style=none] (129) at (11.75, -0.25) {};
		\node [style=none] (130) at (11.75, -1) {};
		\node [style=none] (131) at (11, -1) {};
		\node [style=none] (132) at (12.125, 2.375) {$\bar{1}$};
		\node [style=none] (133) at (13.625, 2.375) {$\bar{1}$};
		\node [style=none] (134) at (11.875, 0.625) {$\bar{2}$};
		\node [style=none] (135) at (13.625, 0.625) {$\bar{2}$};
		\node [style=none] (136) at (11.375, -0.625) {$\scriptstyle 2d$};
		\node [style=none] (137) at (18.375, 6.375) {$\scriptstyle 1$};
		\node [style=none] (138) at (17.1, 5.875) {$\bar{1}$};
		\node [style=none] (139) at (17.125, 4.125) {$\bar{2}$};
		\node [style=none] (140) at (14.125, 2.375) {$\uparrow$};
		\node [style=none] (141) at (14.125, 0.625) {$\uparrow$};
		\node [style=none] (142) at (11.375, 2) {$\downarrow$};
		\node [style=none] (143) at (11.375, 0.625) {$\downarrow$};
		\end{pgfonlayer}
		\begin{pgfonlayer}{edgelayer}
		\draw (75.center) to (72.center);
		\draw (82.center) to (83.center);
		\draw (77.center) to (85.center);
		\draw (68.center) to (71.center);
		\draw (86.center) to (87.center);
		\draw [style=Border edge] (89.center) to (88.center);
		\draw [style=Border edge] (88.center) to (87.center);
		\draw [style=Border edge] (87.center) to (84.center);
		\draw [style=Border edge] (84.center) to (74.center);
		\draw [style=Border edge] (73.center) to (78.center);
		\draw [style=Border edge] (78.center) to (71.center);
		\draw [style=Border edge] (71.center) to (70.center);
		\draw [style=Border edge] (70.center) to (69.center);
		\draw (117.center) to (114.center);
		\draw (124.center) to (125.center);
		\draw (119.center) to (127.center);
		\draw (110.center) to (113.center);
		\draw (128.center) to (129.center);
		\draw [style=Border edge] (131.center) to (130.center);
		\draw [style=Border edge] (130.center) to (129.center);
		\draw [style=Border edge] (129.center) to (126.center);
		\draw [style=Border edge] (126.center) to (116.center);
		\draw [style=Border edge] (115.center) to (120.center);
		\draw [style=Border edge] (120.center) to (113.center);
		\draw [style=Border edge] (113.center) to (112.center);
		\draw [style=Border edge] (112.center) to (111.center);
		\draw [style=Border edge] (76.center) to (69.center);
		\draw [style=Border edge] (81.center) to (89.center);
		\draw [style=Grey diagram] (81.center)
		to (67.center)
		to (76.center)
		to (73.center)
		to (79.center)
		to (80.center)
		to cycle;
		\draw [style=Border edge] (118.center) to (111.center);
		\draw [style=Border edge] (123.center) to (131.center);
		\draw [style=Grey diagram] (109.center)
		to (118.center)
		to (115.center)
		to (121.center)
		to (122.center)
		to (123.center)
		to cycle;
		\end{pgfonlayer}
		\end{tikzpicture}
		\caption{Let $d\in \mathbb{N}$ and set $d_j=jd+1$. The figure shows two pairs of $\lambda/\mu$-tableaux establishing $c(\mu,\mu;\lambda)\geq 2$ with $\lambda$ self-conjugate and $\mu=\left( a^b,1\right) $ with $a=(k+1)d, b=kd$, where $k=1$ for the top pair and $k=2$ for the bottom pair.}
		\label{Figure bottom}
	\end{figure}
	
	\begin{figure}[!h]
		\begin{tikzpicture}[x=0.5cm, y=0.5cm]
		\begin{pgfonlayer}{nodelayer}
		\node [style=none] (29) at (-9, 7.5) {};
		\node [style=none] (30) at (-3.75, 7.5) {};
		\node [style=none] (31) at (-2, 7.5) {};
		\node [style=none] (32) at (-2, 5.75) {};
		\node [style=none] (33) at (-3.75, 5.75) {};
		\node [style=none] (34) at (-2, 4.75) {};
		\node [style=none] (35) at (-2.75, 4.75) {};
		\node [style=none] (36) at (-2.75, 4) {};
		\node [style=none] (37) at (-3.75, 4) {};
		\node [style=none] (38) at (-3.75, 4.75) {};
		\node [style=none] (39) at (-4.5, 4.75) {};
		\node [style=none] (40) at (-4.5, 4) {};
		\node [style=none] (41) at (-9, 4) {};
		\node [style=none] (42) at (-7.25, 4) {};
		\node [style=none] (43) at (-5.5, 4) {};
		\node [style=none] (44) at (-5.5, 2.25) {};
		\node [style=none] (45) at (-9, 2.25) {};
		\node [style=none] (46) at (-9, 0.5) {};
		\node [style=none] (47) at (-7.25, 0.5) {};
		\node [style=none] (48) at (-6.25, 0.5) {};
		\node [style=none] (49) at (-6.25, 1.25) {};
		\node [style=none] (50) at (-5.5, 1.25) {};
		\node [style=none] (51) at (-2.875, 6.625) {$\bar{1}$};
		\node [style=none] (52) at (-3.125, 4.875) {$\bar{2}$};
		\node [style=none] (53) at (-4.125, 4.375) {$\scriptstyle 2d$};
		\node [style=none] (54) at (-6.375, 3.125) {$\bar{1}$};
		\node [style=none] (55) at (-8.125, 3.125) {$\bar{1}$};
		\node [style=none] (56) at (-8.125, 1.375) {$\bar{2}$};
		\node [style=none] (57) at (-6.625, 1.375) {$\bar{2}$};
		\node [style=none] (65) at (-1, 7.5) {};
		\node [style=none] (66) at (4.25, 7.5) {};
		\node [style=none] (67) at (6, 7.5) {};
		\node [style=none] (68) at (6, 5.75) {};
		\node [style=none] (69) at (4.25, 5.75) {};
		\node [style=none] (70) at (6, 4.75) {};
		\node [style=none] (71) at (5.25, 4.75) {};
		\node [style=none] (72) at (5.25, 4) {};
		\node [style=none] (73) at (4.25, 4) {};
		\node [style=none] (74) at (4.25, 4.75) {};
		\node [style=none] (75) at (3.5, 4.75) {};
		\node [style=none] (76) at (3.5, 4) {};
		\node [style=none] (77) at (-1, 4) {};
		\node [style=none] (78) at (0.75, 4) {};
		\node [style=none] (79) at (2.5, 4) {};
		\node [style=none] (80) at (2.5, 2.25) {};
		\node [style=none] (81) at (-1, 2.25) {};
		\node [style=none] (82) at (-1, 0.5) {};
		\node [style=none] (83) at (0.75, 0.5) {};
		\node [style=none] (84) at (1.75, 0.5) {};
		\node [style=none] (85) at (1.75, 1.25) {};
		\node [style=none] (86) at (2.5, 1.25) {};
		\node [style=none] (87) at (5.125, 6.625) {$\bar{1}$};
		\node [style=none] (88) at (4.875, 4.875) {$\bar{2}$};
		\node [style=none] (89) at (3.875, 4.375) {$\scriptstyle 1$};
		\node [style=none] (90) at (1.625, 3.125) {$\bar{1}$};
		\node [style=none] (91) at (-0.125, 3.125) {$\bar{1}$};
		\node [style=none] (92) at (-0.125, 1.375) {$\bar{2}$};
		\node [style=none] (93) at (1.375, 1.375) {$\bar{2}$};
		\node [style=none] (94) at (2.15, 3.125) {$\uparrow$};
		\node [style=none] (95) at (2.125, 1.75) {$\uparrow$};
		\end{pgfonlayer}
		\begin{pgfonlayer}{edgelayer}
		\draw (33.center) to (32.center);
		\draw (45.center) to (44.center);
		\draw (42.center) to (47.center);
		\draw [style=Border edge] (31.center) to (34.center);
		\draw [style=Border edge] (34.center) to (35.center);
		\draw [style=Border edge] (35.center) to (36.center);
		\draw [style=Border edge] (43.center) to (50.center);
		\draw [style=Border edge] (50.center) to (49.center);
		\draw [style=Border edge] (49.center) to (48.center);
		\draw [style=Border edge] (48.center) to (46.center);
		\draw (38.center) to (37.center);
		\draw (69.center) to (68.center);
		\draw (81.center) to (80.center);
		\draw (78.center) to (83.center);
		\draw [style=Border edge] (67.center) to (70.center);
		\draw [style=Border edge] (70.center) to (71.center);
		\draw [style=Border edge] (71.center) to (72.center);
		\draw [style=Border edge] (79.center) to (86.center);
		\draw [style=Border edge] (86.center) to (85.center);
		\draw [style=Border edge] (85.center) to (84.center);
		\draw [style=Border edge] (84.center) to (82.center);
		\draw (74.center) to (73.center);
		\draw [style=Border edge] (41.center) to (46.center);
		\draw [style=Border edge] (40.center) to (36.center);
		\draw [style=Border edge] (30.center) to (31.center);
		\draw [style=Grey diagram] (40.center)
		to (41.center)
		to (29.center)
		to (30.center)
		to (38.center)
		to (39.center)
		to cycle;
		\draw [style=Border edge] (66.center) to (67.center);
		\draw [style=Border edge] (76.center) to (72.center);
		\draw [style=Border edge] (77.center) to (82.center);
		\draw [style=Grey diagram] (74.center)
		to (75.center)
		to (76.center)
		to (77.center)
		to (65.center)
		to (66.center)
		to cycle;
		\end{pgfonlayer}
		\end{tikzpicture}
		
		\vspace*{0.5cm}
		\begin{tikzpicture}[x=0.5cm, y=0.5cm]
		\begin{pgfonlayer}{nodelayer}
		\node [style=none] (42) at (-4.5, 9.5) {};
		\node [style=none] (43) at (2.5, 9.5) {};
		\node [style=none] (44) at (4.25, 9.5) {};
		\node [style=none] (45) at (6, 9.5) {};
		\node [style=none] (46) at (6, 7.75) {};
		\node [style=none] (47) at (6, 6.75) {};
		\node [style=none] (48) at (5.25, 6.75) {};
		\node [style=none] (49) at (5.25, 6) {};
		\node [style=none] (50) at (4.25, 6) {};
		\node [style=none] (51) at (2.5, 6) {};
		\node [style=none] (52) at (2.5, 7.75) {};
		\node [style=none] (53) at (2.5, 5) {};
		\node [style=none] (54) at (1.75, 5) {};
		\node [style=none] (55) at (1.75, 4.25) {};
		\node [style=none] (56) at (2.5, 4.25) {};
		\node [style=none] (57) at (2.5, 2.5) {};
		\node [style=none] (58) at (0.75, 2.5) {};
		\node [style=none] (59) at (0.75, 4.25) {};
		\node [style=none] (60) at (-1, 4.25) {};
		\node [style=none] (61) at (-2.75, 4.25) {};
		\node [style=none] (62) at (-4.5, 4.25) {};
		\node [style=none] (63) at (-4.5, 2.5) {};
		\node [style=none] (64) at (-4.5, 0.75) {};
		\node [style=none] (65) at (-4.5, -1) {};
		\node [style=none] (66) at (-2.75, -1) {};
		\node [style=none] (67) at (-1.75, -1) {};
		\node [style=none] (68) at (-1.75, -0.25) {};
		\node [style=none] (69) at (-1, -0.25) {};
		\node [style=none] (70) at (-1, 0.75) {};
		\node [style=none] (71) at (3.375, 8.625) {$\bar{1}$};
		\node [style=none] (72) at (5.125, 8.625) {$\bar{1}$};
		\node [style=none] (73) at (4.875, 6.875) {$\bar{2}$};
		\node [style=none] (74) at (3.375, 6.875) {$\bar{2}$};
		\node [style=none] (75) at (2.125, 4.625) {$\scriptstyle 2d$};
		\node [style=none] (76) at (1.625, 3.375) {$\bar{3}$};
		\node [style=none] (77) at (-0.125, 3.375) {$\bar{3}$};
		\node [style=none] (78) at (-1.875, 3.375) {$\bar{1}$};
		\node [style=none] (79) at (-3.625, 3.375) {$\bar{1}$};
		\node [style=none] (80) at (-3.625, 1.625) {$\bar{2}$};
		\node [style=none] (81) at (-3.625, -0.125) {$\bar{3}$};
		\node [style=none] (82) at (-2.125, -0.125) {$\bar{3}$};
		\node [style=none] (83) at (-1.875, 1.625) {$\bar{2}$};
		\node [style=none] (90) at (-1, 2.5) {};
		\node [style=none] (98) at (7, 9.5) {};
		\node [style=none] (99) at (14, 9.5) {};
		\node [style=none] (100) at (15.75, 9.5) {};
		\node [style=none] (101) at (17.5, 9.5) {};
		\node [style=none] (102) at (17.5, 7.75) {};
		\node [style=none] (103) at (17.5, 6.75) {};
		\node [style=none] (104) at (16.75, 6.75) {};
		\node [style=none] (105) at (16.75, 6) {};
		\node [style=none] (106) at (15.75, 6) {};
		\node [style=none] (107) at (14, 6) {};
		\node [style=none] (108) at (14, 7.75) {};
		\node [style=none] (109) at (14, 5) {};
		\node [style=none] (110) at (13.25, 5) {};
		\node [style=none] (111) at (13.25, 4.25) {};
		\node [style=none] (112) at (14, 4.25) {};
		\node [style=none] (113) at (14, 2.5) {};
		\node [style=none] (114) at (12.25, 2.5) {};
		\node [style=none] (115) at (12.25, 4.25) {};
		\node [style=none] (116) at (10.5, 4.25) {};
		\node [style=none] (117) at (8.75, 4.25) {};
		\node [style=none] (118) at (7, 4.25) {};
		\node [style=none] (119) at (7, 2.5) {};
		\node [style=none] (120) at (7, 0.75) {};
		\node [style=none] (121) at (7, -1) {};
		\node [style=none] (122) at (8.75, -1) {};
		\node [style=none] (123) at (9.75, -1) {};
		\node [style=none] (124) at (9.75, -0.25) {};
		\node [style=none] (125) at (10.5, -0.25) {};
		\node [style=none] (126) at (10.5, 0.75) {};
		\node [style=none] (127) at (14.875, 8.625) {$\bar{1}$};
		\node [style=none] (128) at (16.625, 8.625) {$\bar{1}$};
		\node [style=none] (129) at (16.375, 6.875) {$\bar{2}$};
		\node [style=none] (130) at (14.875, 6.875) {$\bar{2}$};
		\node [style=none] (131) at (13.625, 4.625) {$\scriptstyle 1$};
		\node [style=none] (132) at (13.125, 3.375) {$\bar{3}$};
		\node [style=none] (133) at (11.375, 3.375) {$\bar{3}$};
		\node [style=none] (134) at (9.625, 3.375) {$\bar{1}$};
		\node [style=none] (135) at (7.875, 3.375) {$\bar{1}$};
		\node [style=none] (136) at (7.875, 1.625) {$\bar{2}$};
		\node [style=none] (137) at (7.875, -0.125) {$\bar{3}$};
		\node [style=none] (138) at (9.375, -0.125) {$\bar{3}$};
		\node [style=none] (139) at (9.625, 1.625) {$\bar{2}$};
		\node [style=none] (141) at (10.125, 3.375) {$\uparrow$};
		\node [style=none] (142) at (10.125, 1.625) {$\uparrow$};
		\node [style=none] (143) at (10.125, 0.25) {$\uparrow$};
		\node [style=none] (144) at (10.875, 3.375) {$\downarrow$};
		\node [style=none] (145) at (10.5, 2.5) {};
		\end{pgfonlayer}
		\begin{pgfonlayer}{edgelayer}
		\draw (61.center) to (66.center);
		\draw (59.center) to (58.center);
		\draw (44.center) to (50.center);
		\draw (52.center) to (46.center);
		\draw (64.center) to (70.center);
		\draw (63.center) to (90.center);
		\draw (60.center) to (90.center);
		\draw (55.center) to (56.center);
		\draw [style=Border edge] (45.center) to (47.center);
		\draw [style=Border edge] (47.center) to (48.center);
		\draw [style=Border edge] (48.center) to (49.center);
		\draw [style=Border edge] (49.center) to (51.center);
		\draw [style=Border edge] (57.center) to (90.center);
		\draw [style=Border edge] (90.center) to (69.center);
		\draw [style=Border edge] (69.center) to (68.center);
		\draw [style=Border edge] (68.center) to (67.center);
		\draw [style=Border edge] (67.center) to (65.center);
		\draw (117.center) to (122.center);
		\draw (115.center) to (114.center);
		\draw (100.center) to (106.center);
		\draw (108.center) to (102.center);
		\draw (120.center) to (126.center);
		\draw (119.center) to (145.center);
		\draw (116.center) to (145.center);
		\draw (111.center) to (112.center);
		\draw [style=Border edge] (101.center) to (103.center);
		\draw [style=Border edge] (103.center) to (104.center);
		\draw [style=Border edge] (104.center) to (105.center);
		\draw [style=Border edge] (105.center) to (107.center);
		\draw [style=Border edge] (113.center) to (145.center);
		\draw [style=Border edge] (145.center) to (125.center);
		\draw [style=Border edge] (125.center) to (124.center);
		\draw [style=Border edge] (124.center) to (123.center);
		\draw [style=Border edge] (123.center) to (121.center);
		\draw [style=Border edge] (43.center) to (45.center);
		\draw [style=Border edge] (53.center) to (57.center);
		\draw [style=Border edge] (65.center) to (62.center);
		\draw [style=Grey diagram] (53.center)
		to (54.center)
		to (55.center)
		to (62.center)
		to (42.center)
		to (43.center)
		to cycle;
		\draw [style=Border edge] (99.center) to (101.center);
		\draw [style=Border edge] (109.center) to (113.center);
		\draw [style=Border edge] (118.center) to (121.center);
		\draw [style=Grey diagram] (110.center)
		to (111.center)
		to (118.center)
		to (98.center)
		to (99.center)
		to (109.center)
		to cycle;
		\end{pgfonlayer}
		\end{tikzpicture}
		\caption{Let $d\in \mathbb{N}$. Two pairs of $\lambda/\mu$-tableaux proving $c(\mu,\mu;\lambda)\geq 2$ with $\lambda$ self-conjugate and $\mu=\left( a^{b-1},a-1\right) $ with $a=(k+1)d, b=kd$, where $k=2$ for the top pair and $k=3$ for the bottom pair.}
		\label{Figure without}
	\end{figure}
	
	For the remaining cases suppose that $s_{\lambda}$ is a common constituent of $s_{\nu}\circ s_{\mu}$ and $s_{\bar{\nu}}\circ s_{\mu'}$, and thus
	\begin{align}\label{Eq common constituent}
	c(\mu,\mu;\lambda)=p(\mu,(2);\lambda) + p(\mu,(1^2);\lambda) &\geq 1,\\
	\label{Eq common constituent'}
	c(\mu',\mu';\lambda)=p(\mu',(2);\lambda) + p(\mu',(1^2);\lambda) &\geq 1.
	\end{align}
	Firstly, consider the case $k=1$, or equivalently, $a=2b$ and $\mu=(a+1,a^{b-1})$. We see from (\ref{Eq common constituent}) that $\lambda_1\geq \mu_1=a+1$ and from (\ref{Eq common constituent'}) that $\lambda_1\leq 2\ell(\mu)= a$. Thus no such $\lambda$ exists and we recover (i) from the statement. If instead $\mu=(a^{b-1},a-1)$, similarly, $\ell(\lambda)\leq 2\ell(\mu)=a$ and $\lambda_1\leq 2\ell(\mu)= a$. Hence $\lambda\subseteq (a^a)$. Since $|\lambda|=a^2-2$, we only have two choices: $\lambda_A=(a^{a-2},(a-1)^2)$ and $\lambda_A'$. One readily checks that $c(\mu,\mu;\lambda_A)\leq 1$ and $c(\mu,\mu;\lambda_A')\leq 1$; see the first diagram of Figure~\ref{Figure LR almost rectangles}.
	
	\begin{figure}[!h]
		\begin{tikzpicture}[x=0.5cm, y=0.5cm]
		\begin{pgfonlayer}{nodelayer}
		\node [style=none] (0) at (-17, 8) {};
		\node [style=none] (1) at (-11, 8) {};
		\node [style=none] (2) at (-11, 6) {};
		\node [style=none] (3) at (-12, 6) {};
		\node [style=none] (4) at (-12, 5) {};
		\node [style=none] (5) at (-17, 5) {};
		\node [style=none] (6) at (-17, 2) {};
		\node [style=none] (14) at (-11.5, 3.5) {};
		\node [style=none] (15) at (-12.5, 2.5) {};
		\node [style=none] (17) at (-17.5, 2) {};
		\node [style=none] (18) at (-18, 3.5) {$b$};
		\node [style=none] (19) at (-17, 1.5) {};
		\node [style=none] (20) at (-12, 1.5) {};
		\node [style=none] (21) at (-17.5, 5) {};
		\node [style=none] (22) at (-14.5, 1) {$2b-1$};
		\node [style=none] (26) at (-8, 9) {};
		\node [style=none] (27) at (-8, 5) {};
		\node [style=none] (29) at (-1, 9) {};
		\node [style=none] (30) at (-1, 8) {};
		\node [style=none] (31) at (-2, 8) {};
		\node [style=none] (32) at (-2, 5) {};
		\node [style=none] (33) at (-1, 5) {};
		\node [style=none] (34) at (0, 9) {};
		\node [style=none] (36) at (0, 5) {};
		\node [style=none] (37) at (-8, 1) {};
		\node [style=none] (38) at (-4, 1) {};
		\node [style=none] (39) at (-4, 3) {};
		\node [style=none] (41) at (-3, 4) {};
		\node [style=none] (42) at (-4, 4) {};
		\node [style=none] (44) at (-3, 5) {};
		\node [style=none] (45) at (-2, 5) {};
		\node [style=none] (47) at (-3.5, 3.5) {};
		\node [style=none] (48) at (-2.5, 4.5) {};
		\node [style=none] (49) at (-8.5, 5) {};
		\node [style=none] (50) at (-8.5, 1) {};
		\node [style=none] (51) at (-9, 3) {$b$};
		\node [style=none] (52) at (-8, 0.5) {};
		\node [style=none] (53) at (-4, 0.5) {};
		\node [style=none] (54) at (-6, 0) {$b$};
		\node [style=none] (55) at (0.5, 9) {};
		\node [style=none] (56) at (0.5, 5) {};
		\node [style=none] (57) at (1, 7) {$b$};
		\node [style=none] (58) at (-1, 9.5) {};
		\node [style=none] (59) at (0, 9.5) {};
		\node [style=none] (60) at (-0.5, 10) {$\frac{b-2}{2}$};
		\node [style=none] (61) at (-12, 2) {};
		\node [style=none] (62) at (-12, 4) {};
		\node [style=none] (63) at (-11, 4) {};
		\node [style=none] (64) at (-13, 3) {};
		\node [style=none] (65) at (-12, 3) {};
		\node [style=none] (66) at (-13, 2) {};
		\node [style=none] (67) at (-11, 3) {};
		\node [style=none] (68) at (-3, 3) {};
		\node [style=none] (69) at (-2, 4) {};
		\end{pgfonlayer}
		\begin{pgfonlayer}{edgelayer}
		\draw [style=Grey diagram] (3.center)
		to (4.center)
		to (5.center)
		to (0.center)
		to (1.center)
		to (2.center)
		to cycle;
		\draw [style=Border edge] (6.center) to (5.center);
		\draw [style=measuredots] (21.center) to (17.center);
		\draw [style=measuredots] (19.center) to (20.center);
		\draw [style=Light grey column] (30.center)
		to (33.center)
		to (32.center)
		to (31.center)
		to cycle;
		\draw [style=Grey diagram] (29.center)
		to (30.center)
		to (31.center)
		to (32.center)
		to (27.center)
		to (26.center)
		to cycle;
		\draw [style=Border edge] (29.center) to (34.center);
		\draw [style=Border edge] (34.center) to (36.center);
		\draw [style=Border edge] (36.center) to (45.center);
		\draw [style=Border edge] (38.center) to (37.center);
		\draw [style=Border edge] (37.center) to (27.center);
		\draw [style=measuredots] (49.center) to (50.center);
		\draw [style=measuredots] (52.center) to (53.center);
		\draw [style=measuredots] (58.center) to (59.center);
		\draw [style=measuredots] (55.center) to (56.center);
		\draw [style=Move it, bend right=45, looseness=1.75] (15.center) to (14.center);
		\draw [style=Border edge] (2.center) to (63.center);
		\draw [style=Border edge] (63.center) to (62.center);
		\draw [style=Border edge] (62.center) to (61.center);
		\draw [style=Border edge] (61.center) to (6.center);
		\draw (64.center) to (65.center);
		\draw (64.center) to (66.center);
		\draw [style=Extra box] (65.center) to (67.center);
		\draw [style=Extra box] (67.center) to (63.center);
		\draw [style=Move it, bend right=45, looseness=1.75] (47.center) to (48.center);
		\draw [style=Border edge] (38.center) to (39.center);
		\draw [style=Border edge] (39.center) to (68.center);
		\draw [style=Border edge] (68.center) to (44.center);
		\draw [style=Extra box] (69.center) to (41.center);
		\draw [style=Extra box] (69.center) to (45.center);
		\draw (42.center) to (41.center);
		\draw (42.center) to (39.center);
		\end{pgfonlayer}
		\end{tikzpicture}
		\caption{The first diagram is the Young diagram of $\lambda_A/\mu$ with $\mu=(a^{b-1},a-1), \lambda_A=(a^{a-2},(a-1)^2)$ and $a=2b$. By moving the highlighted box as indicated, we get the Young diagram of $\lambda_A'/\mu$. In both cases our diagram has $2b-1$ columns with $b=\ell(\mu)$ boxes and $1$ column with $b-1$ boxes. Thus, in both cases, there is at most one semistandard tableau with weight $\mu$, establishing $c(\mu,\mu;\lambda_A)\leq 1$ and $c(\mu,\mu;\lambda_A')\leq 1$. In the second diagram we replace $\lambda_A$ with $\lambda_B=((2b)^b,(b+1)^2,b^{b-2})$ and $\mu$ by $(a+1,a^{b-1})$ with $2a=3b$. This time there are $2$ or $3$ columns of length less than $b=\ell(\mu)$. But we still have $c(\mu,\mu;\lambda_B)\leq 1$ and $c(\mu,\mu;\lambda_B')\leq 1$ since the lightly shaded column of length $b-1$ must contain $1,2,\dots,b-1$ from top downwards for a $\lambda_b/\mu$-tableau or a $\lambda_b'/\mu$-tableau to be semistandard with a latticed reading word, which forces the remaining two neighbouring boxes to contain $1$ and $b$.}
		\label{Figure LR almost rectangles}
	\end{figure}
	
	We now consider domino $\mu$-tableaux $T_{\varphi}$ and $T_{\vartheta}$ with $\varphi=\o$ and $\vartheta=(0^{b-1},1)$ from Example~\ref{Example top filling} (here and below, conventions for denoting partitions are used for sequences of integers). We obtain $p(\mu,(1^2);\lambda_A)\geq 1$ and $p(\mu,(2);\lambda_A') \geq 1$ using Theorem~\ref{Theorem domino rule}; see Figure~\ref{Figure almost domino minus} for details. In conclusion $p(\mu,(1^2);\lambda_A)= p(\mu,(2);\lambda_A') = 1$ and $p(\mu,(2);\lambda_A)= p(\mu,(1^2);\lambda_A') = 0$. Since $|\mu|=2b^2-1$ is odd, we recover $p(\mu',(2);\lambda_A')=p(\mu',(1^2);\lambda_A)=1$ and $p(\mu',(1^2);\lambda_A')=p(\mu',(2);\lambda_A)=0$. This gives (ii).
	
	\begin{figure}[!h]
		\begin{tikzpicture}[x=0.5cm, y=0.5cm]
		\begin{pgfonlayer}{nodelayer}
		\node [style=hdomino] (0) at (-9, 6.5) {$\scriptstyle 1$};
		\node [style=hdomino] (1) at (-9, 5.5) {$\scriptstyle 2$};
		\node [style=hdomino] (4) at (0, 6.5) {$\scriptstyle 1$};
		\node [style=hdomino] (5) at (0, 5.5) {$\scriptstyle 2$};
		\node [style=hdomino] (7) at (-2, 0.5) {$\scriptstyle 2b$};
		\node [style=hdomino] (8) at (0, 2.5) {$\scriptstyle 2b-2$};
		\node [style=none] (10) at (-3, 7) {};
		\node [style=none] (11) at (-8, 7) {};
		\node [style=none] (12) at (-8, 5) {};
		\node [style=none] (13) at (1, 5) {};
		\node [style=none] (14) at (1, 3) {};
		\node [style=none] (15) at (-1, 3) {};
		\node [style=none] (16) at (-1, 5) {};
		\node [style=none] (17) at (0, 4.25) {$\vdots$};
		\node [style=none] (19) at (-3, 2) {};
		\node [style=none] (20) at (-1, 1) {};
		\node [style=none] (21) at (-8, 2) {};
		\node [style=none] (22) at (-8, 0) {};
		\node [style=none] (23) at (-3, 0) {};
		\node [style=none] (24) at (-5.5, 0.5) {$\dots$};
		\node [style=none] (25) at (-4.5, 5.5) {$\dots$};
		\node [style=none] (26) at (-4.5, 6.5) {$\dots$};
		\node [style=none] (27) at (-10, 5) {};
		\node [style=none] (28) at (-10, 0.75) {};
		\node [style=none] (29) at (-9, 4.25) {$\vdots$};
		\node [style=none] (30) at (-4.5, 4.25) {$\ddots$};
		\node [style=none] (31) at (-8, 6) {};
		\node [style=none] (32) at (-1, 6) {};
		\node [style=none] (33) at (-10, 7) {};
		\node [style=none] (34) at (1, 7) {};
		\node [style=none] (35) at (1, 2) {};
		\node [style=none] (36) at (-1, 2) {};
		\node [style=none] (37) at (-1, 0) {};
		\node [style=none] (38) at (-10, 0) {};
		\node [style=hdomino] (39) at (3, 6.5) {$\scriptstyle 1$};
		\node [style=hdomino] (40) at (3, 5.5) {$\scriptstyle 2$};
		\node [style=hdomino] (43) at (12, 6.5) {$\scriptstyle 1$};
		\node [style=hdomino] (44) at (12, 5.5) {$\scriptstyle 2$};
		\node [style=hdomino] (46) at (10, 0.5) {$\scriptstyle 2b$};
		\node [style=hdomino] (47) at (12, 2.5) {$\scriptstyle 2b-2$};
		\node [style=none] (52) at (13, 5) {};
		\node [style=none] (53) at (13, 3) {};
		\node [style=none] (54) at (11, 3) {};
		\node [style=none] (55) at (11, 5) {};
		\node [style=none] (56) at (12, 4.25) {$\vdots$};
		\node [style=none] (59) at (11, 1) {};
		\node [style=none] (61) at (4, 0) {};
		\node [style=none] (62) at (9, 0) {};
		\node [style=none] (63) at (7.5, 0.5) {$\dots$};
		\node [style=none] (64) at (8.5, 5.5) {$\dots$};
		\node [style=none] (65) at (8.5, 6.5) {$\dots$};
		\node [style=none] (66) at (2, 5) {};
		\node [style=none] (67) at (2, 0.75) {};
		\node [style=none] (68) at (3, 4.25) {$\vdots$};
		\node [style=none] (69) at (8.5, 4.25) {$\ddots$};
		\node [style=none] (71) at (11, 6) {};
		\node [style=none] (72) at (2, 7) {};
		\node [style=none] (73) at (13, 7) {};
		\node [style=none] (74) at (13, 2) {};
		\node [style=none] (75) at (11, 2) {};
		\node [style=none] (76) at (11, 0) {};
		\node [style=none] (77) at (2, 0) {};
		\node [style=hdomino] (78) at (5, 6.5) {$\scriptstyle 1$};
		\node [style=hdomino] (79) at (5, 5.5) {$\scriptstyle 2$};
		\node [style=hdomino] (80) at (3, 2.5) {$\scriptstyle 2b-2$};
		\node [style=hdomino] (81) at (5, 2.5) {$\scriptstyle 2b-2$};
		\node [style=hdomino] (82) at (5, 1.5) {$\scriptstyle 2b-1$};
		\node [style=hdomino] (83) at (5, 0.5) {$\scriptstyle 2b$};
		\node [style=none] (86) at (6, 6) {};
		\node [style=none] (87) at (6, 5) {};
		\node [style=none] (89) at (4, 3) {};
		\node [style=none] (90) at (4, 5) {};
		\node [style=none] (91) at (6, 3) {};
		\node [style=none] (92) at (5, 4.25) {$\vdots$};
		\node [style=vnew] (94) at (2.5, 1) {$\scriptstyle 2b-1$};
		\node [style=vnew] (95) at (3.5, 1) {$\scriptstyle 2b-1$};
		\node [style=hdomino] (96) at (10, 1.5) {$\scriptstyle 2b-1$};
		\node [style=none] (97) at (9, 2) {};
		\node [style=none] (98) at (9, 1) {};
		\node [style=none] (99) at (6, 1) {};
		\node [style=none] (100) at (6, 2) {};
		\node [style=none] (101) at (7.5, 1.5) {$\dots$};
		\node [style=none] (102) at (8.5, 2.5) {$\dots$};
		\node [style=none] (103) at (-8, 1) {};
		\node [style=none] (104) at (-3, 1) {};
		\node [style=none] (105) at (-8, 3) {};
		\node [style=none] (106) at (-4.5, 2.5) {$\dots$};
		\node [style=none] (107) at (-5.5, 1.5) {$\dots$};
		\node [style=hdomino] (108) at (-9, 2.5) {$\scriptstyle 2b-2$};
		\node [style=hdomino] (110) at (-2, 1.5) {$\scriptstyle 2b-1$};
		\node [style=hnew] (111) at (-9, 1.5) {$\scriptstyle 2b-1$};
		\node [style=hnew] (112) at (-9, 0.5) {$\scriptstyle 2b$};
		\end{pgfonlayer}
		\begin{pgfonlayer}{edgelayer}
		\draw (11.center) to (10.center);
		\draw (16.center) to (15.center);
		\draw (13.center) to (14.center);
		\draw (23.center) to (22.center);
		\draw (21.center) to (19.center);
		\draw (27.center) to (28.center);
		\draw (31.center) to (32.center);
		\draw [style=Border edge] (33.center) to (34.center);
		\draw [style=Border edge] (34.center) to (35.center);
		\draw [style=Border edge] (35.center) to (36.center);
		\draw [style=Border edge] (36.center) to (37.center);
		\draw [style=Border edge] (37.center) to (38.center);
		\draw [style=Border edge] (38.center) to (33.center);
		\draw (55.center) to (54.center);
		\draw (52.center) to (53.center);
		\draw (62.center) to (61.center);
		\draw (66.center) to (67.center);
		\draw [style=Border edge] (72.center) to (73.center);
		\draw [style=Border edge] (73.center) to (74.center);
		\draw [style=Border edge] (74.center) to (75.center);
		\draw [style=Border edge] (75.center) to (76.center);
		\draw [style=Border edge] (76.center) to (77.center);
		\draw [style=Border edge] (77.center) to (72.center);
		\draw (87.center) to (91.center);
		\draw (87.center) to (55.center);
		\draw (91.center) to (54.center);
		\draw (100.center) to (97.center);
		\draw (99.center) to (98.center);
		\draw (12.center) to (16.center);
		\draw (12.center) to (105.center);
		\draw (105.center) to (15.center);
		\draw (103.center) to (104.center);
		\draw (86.center) to (71.center);
		\draw (90.center) to (89.center);
		\end{pgfonlayer}
		\end{tikzpicture}
		\caption{The domino $\mu$-tableaux $T_{\varphi}$ and $T_{\vartheta}$ with $\varphi=\o, \vartheta=(0^{b-1},1)$ and $\mu=(a^{b-1},a-1)$ where $a=2b$. Note the domino tableaux differ only in the pairs of the lightly shaded dominoes. We know they are both semistandard and that the first one has a latticed reading word (as $\varphi$ is a partition). This is true even for the second tableau since the only failure of the latticed condition can occur when reading the bottom left pair of vertical dominoes, however, this clearly does not happen. Observe that the domino tableaux have weights $(a^{a-2},(a-1)^2)=\lambda_A$ and $(a^{a-1},a-2)=\lambda_A'$, respectively. Finally, $|\mu|=2b^2-1$ is odd and so is the first domino tableau. Consequently, the second domino tableau is even and $p(\mu,(1^2);\lambda_A)\geq 1$ and $p(\mu,(2);\lambda_A') \geq 1$.}
		\label{Figure almost domino minus}
	\end{figure}
	
	Finally, if $k=2$, or equivalently, $2a=3b$, and $\mu=(a+1,a^{b-1})$, the inequalities (\ref{Eq common constituent}) and (\ref{Eq common constituent'}) show that $\ell(\lambda)\leq 2\ell(\mu)=2b$ and $\lambda_1\leq 2\ell(\mu)=2b$. By Lemma~\ref{Lemma two boxes}(i) $\lambda$ contains an $(a,b)$-birectangular partition. Therefore $\lambda_{2b}\geq 2a-\lambda_1=3b-\lambda_1\geq b$. Replacing $\lambda$ with $\lambda'$ gives $\lambda'_{2b}\geq b$. Therefore $((2b)^b,b^b)\subseteq\lambda\subseteq((2b)^{2b})$. Since $|\lambda|=3b^2+2$ we only have two choices for $\lambda$: $\lambda_B=((2b)^b,(b+1)^2,b^{b-2})$ and $\lambda_B'$. As in the previous case, $c(\mu,\mu;\lambda_B)\leq 1$ and $c(\mu,\mu;\lambda_B')\leq 1$; see the second diagram of Figure~\ref{Figure LR almost rectangles}.
	
	Letting $\varphi=(b/2-1,(b/2)^{b-2},b/2-1)$ and $\vartheta=(b/2-1,(b/2)^{b-1})$ we consider $T_{\varphi}$ and $T_{\vartheta}$ from Example~\ref{Example top filling}. From Figure~\ref{Figure almost domino plus} we conclude that $p(\mu,(1^2);\lambda_B)\geq 1$ and $p(\mu,(2);\lambda_B') \geq 1$. Thus $p(\mu,(1^2);\lambda_B)= p(\mu,(2);\lambda_B') = 1$ and $p(\mu,(2);\lambda_B)= p(\mu,(1^2);\lambda_B') = 0$. As $|\mu|=2a^2/3 +1$ is odd, we conclude that $p(\mu',(2);\lambda_B')=p(\mu',(1^2);\lambda_B)=1$ and $p(\mu',(1^2);\lambda_B')=p(\mu',(2);\lambda_B)=0$ and recover (iii). 
	
	\begin{figure}[!h]
		\hspace*{-1.5cm}
		\begin{tikzpicture}[x=0.5cm, y=0.5cm]
		\begin{pgfonlayer}{nodelayer}
		\node [style=vdomino] (0) at (-8.5, 9) {$\scriptstyle 1$};
		\node [style=vdomino] (1) at (-5.5, 9) {$\scriptstyle 1$};
		\node [style=vdomino] (2) at (-8.5, 7) {$\scriptstyle 2$};
		\node [style=vdomino] (3) at (-5.5, 7) {$\scriptstyle 2$};
		\node [style=vdomino] (4) at (-4.5, 7) {$\scriptstyle 3$};
		\node [style=vdomino] (5) at (-3.5, 7) {$\scriptstyle 3$};
		\node [style=hdomino] (6) at (-4, 9.5) {$\scriptstyle 1$};
		\node [style=hdomino] (7) at (-4, 8.5) {$\scriptstyle 2$};
		\node [style=hdomino] (8) at (-2, 9.5) {$\scriptstyle 1$};
		\node [style=hdomino] (9) at (-2, 8.5) {$\scriptstyle 2$};
		\node [style=hdomino] (10) at (-2, 7.5) {$\scriptstyle 3$};
		\node [style=hdomino] (11) at (-2, 6.5) {$\scriptstyle 4$};
		\node [style=vdomino] (12) at (-8.5, 1) {$\scriptstyle b$};
		\node [style=vdomino] (13) at (-5.5, 1) {$\scriptstyle b$};
		\node [style=hdomino] (16) at (-2, 1.5) {$\scriptstyle 2b-1$};
		\node [style=hdomino] (17) at (-2, 0.5) {$\scriptstyle 2b$};
		\node [style=hdomino] (18) at (3, 9.5) {$\scriptstyle 1$};
		\node [style=hdomino] (19) at (5, 9.5) {$\scriptstyle 1$};
		\node [style=hdomino] (20) at (3, 8.5) {$\scriptstyle 2$};
		\node [style=hdomino] (21) at (5, 8.5) {$\scriptstyle 2$};
		\node [style=hdomino] (23) at (3, 0.5) {$\scriptstyle 2b$};
		\node [style=none] (24) at (-9, 10) {};
		\node [style=none] (25) at (-9, 0) {};
		\node [style=none] (26) at (4, 0) {};
		\node [style=none] (27) at (4, 8) {};
		\node [style=none] (28) at (6, 8) {};
		\node [style=none] (29) at (6, 10) {};
		\node [style=none] (30) at (-1, 9) {};
		\node [style=none] (31) at (2, 9) {};
		\node [style=none] (32) at (-1, 8) {};
		\node [style=none] (33) at (2, 8) {};
		\node [style=none] (36) at (-1, 1) {};
		\node [style=none] (37) at (-1, 2) {};
		\node [style=none] (38) at (-1, 6) {};
		\node [style=none] (39) at (-3, 6) {};
		\node [style=none] (40) at (-3, 2) {};
		\node [style=none] (41) at (-4, 2) {};
		\node [style=none] (42) at (-5, 2) {};
		\node [style=none] (43) at (-4, 6) {};
		\node [style=none] (44) at (-5, 6) {};
		\node [style=none] (45) at (-6, 6) {};
		\node [style=none] (46) at (-8, 6) {};
		\node [style=none] (47) at (-8, 2) {};
		\node [style=none] (48) at (-6, 2) {};
		\node [style=none] (49) at (-7, 4.25) {$\ddots$};
		\node [style=none] (50) at (-8.5, 4.25) {$\vdots$};
		\node [style=none] (51) at (-7, 7) {$\dots$};
		\node [style=none] (52) at (-7, 9) {$\dots$};
		\node [style=none] (53) at (-7, 1) {$\dots$};
		\node [style=none] (54) at (-5.5, 4.25) {$\vdots$};
		\node [style=none] (55) at (-4.5, 4.25) {$\vdots$};
		\node [style=none] (56) at (-3.5, 4.25) {$\vdots$};
		\node [style=none] (57) at (-2, 4.25) {$\vdots$};
		\node [style=none] (58) at (0.5, 4.25) {$\ddots$};
		\node [style=none] (59) at (0.5, 8.5) {$\dots$};
		\node [style=none] (61) at (0.5, 0.5) {$\dots$};
		\node [style=none] (62) at (0.5, 9.5) {$\dots$};
		\node [style=none] (63) at (3, 4.25) {$\vdots$};
		\node [style=none] (64) at (2, 1) {};
		\node [style=vdomino] (65) at (-24.5, 9) {$\scriptstyle 1$};
		\node [style=vdomino] (66) at (-21.5, 9) {$\scriptstyle 1$};
		\node [style=vdomino] (67) at (-24.5, 7) {$\scriptstyle 2$};
		\node [style=vdomino] (68) at (-21.5, 7) {$\scriptstyle 2$};
		\node [style=vdomino] (69) at (-20.5, 7) {$\scriptstyle 3$};
		\node [style=vdomino] (70) at (-19.5, 7) {$\scriptstyle 3$};
		\node [style=hdomino] (71) at (-20, 9.5) {$\scriptstyle 1$};
		\node [style=hdomino] (72) at (-20, 8.5) {$\scriptstyle 2$};
		\node [style=hdomino] (73) at (-18, 9.5) {$\scriptstyle 1$};
		\node [style=hdomino] (74) at (-18, 8.5) {$\scriptstyle 2$};
		\node [style=hdomino] (75) at (-18, 7.5) {$\scriptstyle 3$};
		\node [style=hdomino] (76) at (-18, 6.5) {$\scriptstyle 4$};
		\node [style=vdomino] (77) at (-24.5, 1) {$\scriptstyle b$};
		\node [style=vdomino] (78) at (-21.5, 1) {$\scriptstyle b$};
		\node [style=hdomino] (81) at (-18, 1.5) {$\scriptstyle 2b-1$};
		\node [style=hdomino] (82) at (-18, 0.5) {$\scriptstyle 2b$};
		\node [style=hdomino] (83) at (-13, 9.5) {$\scriptstyle 1$};
		\node [style=hdomino] (84) at (-11, 9.5) {$\scriptstyle 1$};
		\node [style=hdomino] (85) at (-13, 8.5) {$\scriptstyle 2$};
		\node [style=hdomino] (86) at (-11, 8.5) {$\scriptstyle 2$};
		\node [style=hdomino] (87) at (-13, 0.5) {$\scriptstyle 2b$};
		\node [style=none] (88) at (-25, 10) {};
		\node [style=none] (89) at (-25, 0) {};
		\node [style=none] (90) at (-12, 0) {};
		\node [style=none] (91) at (-12, 8) {};
		\node [style=none] (92) at (-10, 8) {};
		\node [style=none] (93) at (-10, 10) {};
		\node [style=none] (94) at (-17, 9) {};
		\node [style=none] (95) at (-14, 9) {};
		\node [style=none] (96) at (-17, 8) {};
		\node [style=none] (97) at (-14, 8) {};
		\node [style=none] (98) at (-17, 1) {};
		\node [style=none] (99) at (-17, 2) {};
		\node [style=none] (100) at (-17, 6) {};
		\node [style=none] (101) at (-19, 6) {};
		\node [style=none] (102) at (-19, 4) {};
		\node [style=none] (103) at (-20, 4) {};
		\node [style=none] (104) at (-21, 4) {};
		\node [style=none] (105) at (-20, 6) {};
		\node [style=none] (106) at (-21, 6) {};
		\node [style=none] (107) at (-22, 6) {};
		\node [style=none] (108) at (-24, 6) {};
		\node [style=none] (109) at (-24, 2) {};
		\node [style=none] (110) at (-22, 2) {};
		\node [style=none] (111) at (-23, 4.25) {$\ddots$};
		\node [style=none] (112) at (-24.5, 4.25) {$\vdots$};
		\node [style=none] (113) at (-23, 7) {$\dots$};
		\node [style=none] (114) at (-23, 9) {$\dots$};
		\node [style=none] (115) at (-23, 1) {$\dots$};
		\node [style=none] (116) at (-21.5, 4.25) {$\vdots$};
		\node [style=none] (117) at (-20.5, 5.25) {$\vdots$};
		\node [style=none] (118) at (-19.5, 5.25) {$\vdots$};
		\node [style=none] (119) at (-18, 4.25) {$\vdots$};
		\node [style=none] (120) at (-15.5, 4.25) {$\ddots$};
		\node [style=none] (121) at (-15.5, 8.5) {$\dots$};
		\node [style=none] (122) at (-15.5, 0.5) {$\dots$};
		\node [style=none] (123) at (-15.5, 9.5) {$\dots$};
		\node [style=none] (124) at (-13, 4.25) {$\vdots$};
		\node [style=none] (125) at (-14, 1) {};
		\node [style=none] (128) at (2, 7) {};
		\node [style=none] (129) at (2, 6) {};
		\node [style=none] (130) at (-1, 7) {};
		\node [style=none] (131) at (2, 2) {};
		\node [style=hdomino] (132) at (3, 7.5) {$\scriptstyle 3$};
		\node [style=hdomino] (133) at (3, 6.5) {$\scriptstyle 4$};
		\node [style=hdomino] (134) at (3, 1.5) {$\scriptstyle 2b-1$};
		\node [style=none] (135) at (0.5, 7.5) {$\dots$};
		\node [style=none] (136) at (0.5, 6.5) {$\dots$};
		\node [style=none] (137) at (0.5, 1.5) {$\dots$};
		\node [style=none] (138) at (-14, 7) {};
		\node [style=none] (139) at (-14, 6) {};
		\node [style=none] (140) at (-14, 2) {};
		\node [style=none] (141) at (-17, 7) {};
		\node [style=none] (142) at (-15.5, 7.5) {$\dots$};
		\node [style=none] (143) at (-15.5, 6.5) {$\dots$};
		\node [style=hdomino] (144) at (-13, 7.5) {$\scriptstyle 3$};
		\node [style=hdomino] (145) at (-13, 6.5) {$\scriptstyle 4$};
		\node [style=hdomino] (146) at (-13, 1.5) {$\scriptstyle 2b-1$};
		\node [style=none] (147) at (-15.5, 1.5) {$\dots$};
		\node [style=vdomino] (148) at (-20.5, 3) {$\scriptstyle b$};
		\node [style=vdomino] (149) at (-19.5, 3) {$\scriptstyle b$};
		\node [style=hnew] (150) at (-20, 1.5) {$\scriptstyle b+1$};
		\node [style=hnew] (151) at (-20, 0.5) {$\scriptstyle b+2$};
		\node [style=vnew] (152) at (-4.5, 1) {$\scriptstyle b+1$};
		\node [style=vnew] (153) at (-3.5, 1) {$\scriptstyle b+1$};
		\node [style=none] (154) at (-9, 10.5) {};
		\node [style=none] (155) at (-5, 10.5) {};
		\node [style=none] (156) at (-7, 11) {$b-2$};
		\node [style=none] (157) at (-3, 10.5) {};
		\node [style=none] (158) at (4, 10.5) {};
		\node [style=none] (159) at (0.5, 11) {$2b$};
		\node [style=none] (160) at (-8, 8) {};
		\node [style=none] (161) at (-6, 8) {};
		\node [style=none] (162) at (-24, 8) {};
		\node [style=none] (163) at (-22, 8) {};
		\node [style=none] (164) at (-25, 10.5) {};
		\node [style=none] (165) at (-21, 10.5) {};
		\node [style=none] (166) at (-19, 10.5) {};
		\node [style=none] (167) at (-12, 10.5) {};
		\node [style=none] (168) at (-23, 11) {$b-2$};
		\node [style=none] (169) at (-15.5, 11) {$2b$};
		\end{pgfonlayer}
		\begin{pgfonlayer}{edgelayer}
		\draw [style=Border edge] (24.center) to (29.center);
		\draw [style=Border edge] (29.center) to (28.center);
		\draw [style=Border edge] (28.center) to (27.center);
		\draw [style=Border edge] (27.center) to (26.center);
		\draw [style=Border edge] (26.center) to (25.center);
		\draw [style=Border edge] (25.center) to (24.center);
		\draw (46.center) to (45.center);
		\draw (45.center) to (48.center);
		\draw (48.center) to (47.center);
		\draw (47.center) to (46.center);
		\draw (44.center) to (42.center);
		\draw (43.center) to (41.center);
		\draw (39.center) to (40.center);
		\draw (30.center) to (31.center);
		\draw (32.center) to (33.center);
		\draw (38.center) to (37.center);
		\draw (64.center) to (36.center);
		\draw [style=Border edge] (88.center) to (93.center);
		\draw [style=Border edge] (93.center) to (92.center);
		\draw [style=Border edge] (92.center) to (91.center);
		\draw [style=Border edge] (91.center) to (90.center);
		\draw [style=Border edge] (90.center) to (89.center);
		\draw [style=Border edge] (89.center) to (88.center);
		\draw (108.center) to (107.center);
		\draw (107.center) to (110.center);
		\draw (110.center) to (109.center);
		\draw (109.center) to (108.center);
		\draw (106.center) to (104.center);
		\draw (105.center) to (103.center);
		\draw (101.center) to (102.center);
		\draw (94.center) to (95.center);
		\draw (96.center) to (97.center);
		\draw (100.center) to (99.center);
		\draw (125.center) to (98.center);
		\draw (130.center) to (128.center);
		\draw (38.center) to (129.center);
		\draw (129.center) to (131.center);
		\draw (131.center) to (37.center);
		\draw (141.center) to (138.center);
		\draw (100.center) to (139.center);
		\draw (139.center) to (140.center);
		\draw (140.center) to (99.center);
		\draw [style=measuredots] (154.center) to (155.center);
		\draw [style=measuredots] (157.center) to (158.center);
		\draw (160.center) to (161.center);
		\draw (162.center) to (163.center);
		\draw [style=measuredots] (164.center) to (165.center);
		\draw [style=measuredots] (166.center) to (167.center);
		\end{pgfonlayer}
		\end{tikzpicture}
		\caption{The domino $\mu$-tableaux $T_{\varphi}$ and $T_{\vartheta}$ with $\varphi=(b/2-1,(b/2)^{b-2},b/2-1), \vartheta=(b/2-1,(b/2)^{b-1})$ and $\mu=(a+1,a^{b-1})$ where $2a=3b$. Observe that the domino tableaux differ only in the lightly shaded dominoes. We know they are both semistandard. Checking the pair of columns with lightly shaded dominoes carefully, we see that their reading words are latticed. The weight of the first domino tableau is $\lambda_B=((2b)^b,(b+1)^2,b^{b-2})$ from which we easily conclude that the weight of the second one is $((2b)^b,b+2,b^{b-1})=\lambda_B'$. The first domino tableau is odd since $|\mu|-\varphi_1-\varphi_2-\dots-\varphi_b=2a^2/3 +1-2a^2/9+2=4a^2/9+3$ is odd. Consequently, the second domino tableau is even and $p(\mu,(1^2);\lambda_B)$ and $p(\mu,(2);\lambda_B')$ are at least $1$.}
		\label{Figure almost domino plus}
	\end{figure}
\end{proof}

We are ready to provide the classification of the elementary irreducible induced-multiplicity-free characters of $T_{m,2}$ and $\left( S_m\wr S_2 \right)\cap A_{2m}$. We provide two statements with a shared proof using the discussion before Lemma~\ref{Lemma hooks}. 

\begin{corollary}\label{Cor Sm wr S2 cap An class}
	Let $m\geq 37$. All the elementary irreducible induced-multiplicity-free characters of $G=\left( S_m\wr S_2 \right)\cap A_{2m} $ are $\left( \charwr{\chi^{\mu}}{\chi^{\nu}}{2}\right) _G$, where one of the following holds:
	\begin{enumerate}[label=\textnormal{(\roman*)}]
		\item $\mu$ is of the form $(a^b)$ \emph{with} $a-b\nmid a$ or $\nu=(1^2)$,
		\item $\mu$ is of the form $(a+1,1^b)$ with $b>a+1$,
		\item $\mu$ is of the form $(a+1,a^{b-1}), (a^b,1)$ or $(a^{b-1},a-1)$ with $a-b\nmid a$ and $a>b+2$,
		\item $\mu$ is of the form  $(2c+1,(2c)^{c-1}), ((2c)^{c-1}, 2c-1)$ or $(3c+1,(3c)^{2c-1})$,
		\item $\mu$ is of the form $(a^{a-1},a-1)$, provided $m$ is even.
	\end{enumerate}
\end{corollary}

\begin{corollary}\label{Cor Tm,2 class}
	Let $m\geq 37$. All the elementary irreducible induced-multiplicity-free characters of $G=T_{m,2} $ are $\left( \charwr{\chi^{\mu}}{\chi^{\nu}}{2}\right) _G$, where one of the following holds:
	\begin{enumerate}[label=\textnormal{(\roman*)}]
		\item $\mu$ is rectangular,
		\item $\mu$ is a hook,
		\item $\mu$ is of the form $(a+1,a^{b-1}), (a^b,1)$ or $(a^{b-1},a-1)$ with $a-b\nmid a$ and $a>b+2$,
		\item $\mu$ is of the form  $(2c+1,(2c)^{c-1})$,
		\item $\mu$ is of the form $(a^{a-1},a-1)$, provided $m$ is odd.
	\end{enumerate}
\end{corollary}

\begin{proof}[Proof of Corollary~\ref{Cor Sm wr S2 cap An class} and Corollary~\ref{Cor Tm,2 class}]
	We go through all the possible choices for $\mu$ and observe which choices of $\nu\vdash 2$ and $\varepsilon\in\left\lbrace \pm\right\rbrace $ make the character $\left( \charwr{\chi^{\mu}}{\chi^{\nu}}{2}\right) _{G_{\varepsilon}}$ induced-multiplicity-free.
	
	If $\mu$ is a square partition we can take arbitrary $\nu$ and $\varepsilon$ since $s_{\mu}^2$ and $s_{\nu}\circ s_{\mu}$ are always multiplicity-free by Theorem~\ref{Theorem Stembridge}(iv) and Theorem~\ref{Theorem MF plethysms}(i). If $\mu=(a^b)$ is non-square rectangular, Lemma~\ref{Lemma rotate rectangles} shows that the only prohibited choice is given by $a-b\mid a, \varepsilon=+$ and $\nu=(2)$. Since the condition $a-b\mid a$ implies that $m=ab$ is even, we recover (i).
	
	If $\mu$ is a self-conjugate hook, then $m$ is odd and we can allow only $\varepsilon=+$ since $s_{\mu}^2$ is not multiplicity-free by Theorem~\ref{Theorem Stembridge}. On the other hand, if $\mu=(a+1,1^b)$ is a non-self-conjugate hook, without loss of generality suppose that $a<b$ and use Lemma~\ref{Lemma hooks} to obtain the only prohibited choice given by $b=a+1$ and $\varepsilon=+$. For these parameters, $m=a+b+1$ is even and we obtain (ii).
	
	Now suppose that $\mu$ is an almost rectangular partition, say $(a+1,a^{b-1})$, $(a^b,1)$ or $(a^{b-1},a-1)$, and without loss of generality suppose that $a\geq b$. If $a> b$, we obtain (iii) immediately from the non-exceptional choices of $a$ and $b$ in Lemma~\ref{Lemma almost rectangles}. The exceptional choices in Lemma~\ref{Lemma almost rectangles}(i)-(iii) yield (iv) (with $c=b$ for (i) and (ii), and $c=b/2$ for (iii)), because (ii) and (iii) require $\varepsilon=-$ and in both cases $\mu$ has an odd size. Finally if $a=b$, by Lemma~\ref{Lemma square and a box} we conclude that $\mu=(a^{a-1},a-1)$. Since $s_{\mu}^2$ is not multiplicity-free by Theorem~\ref{Theorem Stembridge}, we get an induced-multiplicity-free character only for $\varepsilon=+$ which yields (v).
\end{proof}

\section{Proof of the main theorem}\label{Sec main thm}

Combining our classifications of irreducible induced-multiplicity-free characters, we can verify the main theorem.

\begin{proof}[Proof of Theorem~\ref{Theorem main}]
	The list after the statement of Theorem~\ref{Theorem main} provides references to results which show that the groups from Theorem~\ref{Theorem main}(iii)--(x) are multiplicity-free with the exceptions of $\left( S_m\wr S_2\right)\cap A_{2m}$ and $T_{m,2}$ for $33\leq m\leq 36$. These results and Corollary~\ref{Cor S1 times S2 wr Sm class} (needed for $(S_2\wr S_h)\cap A_{2h}$ embedded in $S_{2h+1}$) also show that the requirements for $k,m$ and $h$ in (iii), (iv) and (x) are necessary. For the two exceptions one may use Lemma~\ref{Lemma hooks} with $a=1$ and $b=m-2$ to obtain an induced multiplicity-free character $\left( \charwr{\chi^{(2,1^{m-2})}}{\chi^{(2)}}{2}\right) _G$, where $G$ is either $\left( S_m\wr S_2\right)\cap A_{2m}$ or $T_{m,2}$.
	
	Using Proposition~\ref{Prop possible MF subgroups}, to finish the proof it remains to show, firstly, that the only non-multiplicity-free groups in Proposition~\ref{Prop possible MF subgroups}(i) and (ii) are $S_k\times A_2$ and $A_k\times A_2$, and secondly, that the groups $S_2\wr A_h, T_{2,h}$ and $T_{2,h}\cap A_{2h}$ embedded in $S_{2h}$ and $S_{2h+1}$, and $S_m\wr A_3$ embedded in $S_{3m}$ are not multiplicity-free. The latter part follows from Corollary~\ref{Cor S2 wr Sm class}(ii), Corollary~\ref{Cor S1 times S2 wr Sm class} and Proposition~\ref{Prop Sm wr S3 class}(ii).
	
	Clearly, $S_n$ and $A_n$ are multiplicity-free, and so are $A_k\times S_2$ and $(S_k\times S_2)\cap A_{k+2}$ (and consequently $S_k\times S_2$), since by Lemma~\ref{Lemma rotate linear} applied with partitions $(k)$ and $(1^k)$, the symmetric function $s_{(k)}s_{\nu}+s_{(1^k)}s_{\bar{\nu}}$ is multiplicity-free for any $\nu,\bar{\nu}\vdash 2$. On the other hand, the groups $S_k\times A_2$ and $A_k\times A_2$ are not multiplicity-free by Lemma~\ref{Lemma sanity lemma}.
	
	Now suppose that $k\geq l\neq 2$ with $k+l\geq 66$. To finish the proof we need to show that $A_k\times A_l$ is multiplicity-free as then all groups containing it are also multiplicity-free. Since $l\neq 2$, there is a self-conjugate partition $\mu\vdash l$. By Theorem~\ref{Theorem Stembridge}(i), we know that $s_{(k)}s_{\mu}$ and $s_{(1^k)}s_{\mu}$ are multiplicity-free, and so is their sum as any common constituent $s_{\lambda}$ would satisfy $\lambda_1\geq k$ and, since $\mu$ is self-conjugate, also $\lambda_1\leq \mu_1+1\leq (l+1)/2+1< k$ which is impossible. Therefore $\chi^{(k)}_{A_k}\boxtimes \chi^{\mu}_{A_l}$ is induced-multiplicity-free, as required.
\end{proof}

\begin{remark}\label{Remark improving n}
	Theorem~\ref{Theorem main} remains true for $n=65$ and can be proved by improving the bound in Corollary~\ref{Cor S1 times S2 wr Sm class}. However, for $n=64$ we need to include the group $T_{2,32}$ for the statement to remain valid.
\end{remark}

\subsection*{Acknowledgements} The author would like to thank Mark Wildon for suggesting this project as well as providing useful comments and {\sc Magma} code used in a couple of proofs, and an anonymous referee for a careful reading and valuable suggestions and corrections.

	\bibliographystyle{alpha}
\bibliography{../../../References/MSNrefs}

\begin{thebibliography}{dBPW21}

\bibitem[AC12]{AkerCanParkingGelfand12}
K\"{u}r\c{s}at Aker and Mahir~Bilen Can.
\newblock From parking functions to {G}elfand pairs.
\newblock {\em Proc. Amer. Math. Soc.}, 140(4):1113--1124, 2012.

\bibitem[AHN21]{AndersonHumphriesNicholsonStrongGelfPairsSn21}
Gradin Anderson, Stephen~P. Humphries, and Nathan Nicholson.
\newblock Strong {G}elfand pairs of symmetric groups.
\newblock {\em J. Algebra Appl.}, 20(4):Paper No. 2150054, 22, 2021.

\bibitem[BBP22]{BessenrodtBowmanPagetMF22}
Christine Bessenrodt, Chris Bowman, and Rowena Paget.
\newblock The classification of multiplicity-free plethysms of {S}chur
  functions.
\newblock {\em Trans. Amer. Math. Soc.}, 375(7):5151--5194, 2022.

\bibitem[BCP97]{BosmaCannonPlayoustMagma97}
Wieb Bosma, John Cannon, and Catherine Playoust.
\newblock The {M}agma algebra system. {I}. {T}he user language.
\newblock {\em J. Symbolic Comput.}, 24(3-4):235--265, 1997.
\newblock Computational algebra and number theory (London, 1993).

\bibitem[Ben88]{BensonSpinModules88}
Dave Benson.
\newblock Spin modules for symmetric groups.
\newblock {\em J. London Math. Soc. (2)}, 38(2):250--262, 1988.

\bibitem[BR18]{BensonRatcliffGelfarndWreath18}
Chal Benson and Gail Ratcliff.
\newblock A family of finite {G}elfand pairs associated with wreath products.
\newblock {\em Colloq. Math.}, 152(1):65--78, 2018.

\bibitem[CL95]{CarreSplitting95}
Christophe Carr\'{e} and Bernard Leclerc.
\newblock Splitting the square of a {S}chur function into its symmetric and
  antisymmetric parts.
\newblock {\em J. Algebraic Combin.}, 4(3):201--231, 1995.

\bibitem[CSS21]{CanSheSpeyerStrongGelfPairsFwrSn21}
Mahir~Bilen Can, Yiyang She, and Liron Speyer.
\newblock Strong {G}elfand subgroups of {$F\wr S_n$}.
\newblock {\em Internat. J. Math.}, 32(2):Paper No. 2150010, 63, 2021.

\bibitem[dBPW21]{deBoeckPagetWildonPlethysms21}
Melanie de~Boeck, Rowena Paget, and Mark Wildon.
\newblock Plethysms of symmetric functions and highest weight representations.
\newblock {\em Trans. Amer. Math. Soc.}, 374(11):8013--8043, 2021.

\bibitem[GM10]{GodsilMeagherMultiplicity-free10}
Chris Godsil and Karen Meagher.
\newblock Multiplicity-free permutation representations of the symmetric group.
\newblock {\em Ann. Comb.}, 13(4):463--490, 2010.

\bibitem[GW15]{GiannelliWildonFoulkesandDecomposition15}
Eugenio Giannelli and Mark Wildon.
\newblock Foulkes modules and decomposition numbers of the symmetric group.
\newblock {\em J. Pure Appl. Algebra}, 219(2):255--276, 2015.

\bibitem[IRS90]{InglisRichardsonSaxlModel90}
N.~F.~J. Inglis, R.~W. Richardson, and J.~Saxl.
\newblock An explicit model for the complex representations of {$S_n$}.
\newblock {\em Arch. Math. (Basel)}, 54(3):258--259, 1990.

\bibitem[Jam78]{JamesSymmetric78}
G.~D. James.
\newblock {\em The representation theory of the symmetric groups}, volume 682
  of {\em Lecture Notes in Mathematics}.
\newblock Springer, Berlin, 1978.

\bibitem[JK81]{JamesKerberSymmetric81}
Gordon James and Adalbert Kerber.
\newblock {\em The representation theory of the symmetric group}, volume~16 of
  {\em Encyclopedia of Mathematics and its Applications}.
\newblock Addison-Wesley Publishing Co., Reading, Mass., 1981.
\newblock With a foreword by P. M. Cohn and an introduction by Gilbert de B.
  Robinson.

\bibitem[JL01]{JamesLiebeckRepTheory01}
Gordon James and Martin Liebeck.
\newblock {\em Representations and characters of groups}.
\newblock Cambridge University Press, New York, second edition, 2001.

\bibitem[Knu98]{KnuthArtofCPVol398}
Donald~E. Knuth.
\newblock {\em The art of computer programming. {V}ol. 3}.
\newblock Addison-Wesley, Reading, MA, second edition, 1998.
\newblock Sorting and searching.

\bibitem[Mac95]{MacdonaldPolynomials95}
I.~G. Macdonald.
\newblock {\em Symmetric functions and {H}all polynomials}.
\newblock Oxford Mathematical Monographs. The Clarendon Press, Oxford
  University Press, New York, second edition, 1995.
\newblock With contributions by A. Zelevinsky, Oxford Science Publications.

\bibitem[Mar02]{MarotiPrimitive02}
Attila Mar\'{o}ti.
\newblock On the orders of primitive groups.
\newblock {\em J. Algebra}, 258(2):631--640, 2002.

\bibitem[Oka98]{OkadaRectangularProducts98}
Soichi Okada.
\newblock Applications of minor summation formulas to rectangular-shaped
  representations of classical groups.
\newblock {\em J. Algebra}, 205(2):337--367, 1998.

\bibitem[PS80]{PraegerSaxlPrimitiveOrder80}
Cheryl~E. Praeger and Jan Saxl.
\newblock On the orders of primitive permutation groups.
\newblock {\em Bull. London Math. Soc.}, 12(4):303--307, 1980.

\bibitem[Sax81]{SaxlMultiplicity-free81}
Jan Saxl.
\newblock On multiplicity-free permutation representations.
\newblock In {\em Finite geometries and designs ({P}roc. {C}onf., {C}helwood
  {G}ate, 1980)}, volume~49 of {\em London Math. Soc. Lecture Note Ser.}, pages
  337--353. Cambridge Univ. Press, Cambridge-New York, 1981.

\bibitem[Sta86]{StanleyPlaneSymmetries86}
Richard~P. Stanley.
\newblock Symmetries of plane partitions.
\newblock {\em J. Combin. Theory Ser. A}, 43(1):103--113, 1986.

\bibitem[Sta99]{StanleyEnumerativeII99}
Richard~P. Stanley.
\newblock {\em Enumerative combinatorics. {V}ol. 2}, volume~62 of {\em
  Cambridge Studies in Advanced Mathematics}.
\newblock Cambridge University Press, Cambridge, 1999.
\newblock With a foreword by Gian-Carlo Rota and appendix 1 by Sergey Fomin.

\bibitem[Sta00]{StanleyPositivity00}
Richard~P. Stanley.
\newblock Positivity problems and conjectures in algebraic combinatorics.
\newblock In {\em Mathematics: frontiers and perspectives}, pages 295--319.
  Amer. Math. Soc., Providence, RI, 2000.

\bibitem[Ste01]{StembridgeMultiplicity-free01}
John~R. Stembridge.
\newblock Multiplicity-free products of {S}chur functions.
\newblock {\em Ann. Comb.}, 5(2):113--121, 2001.

\bibitem[Tou21]{ToutGelfandWreath21}
Omar Tout.
\newblock Gelfand pairs involving the wreath product of finite abelian groups
  with symmetric groups.
\newblock {\em Canad. Math. Bull.}, 64(1):91--97, 2021.

\bibitem[Wil09]{WildonMultiplicity-free09}
Mark Wildon.
\newblock Multiplicity-free representations of symmetric groups.
\newblock {\em J. Pure Appl. Algebra}, 213(7):1464--1477, 2009.

\end{thebibliography}
\end{document}